\documentclass{memo-l}

\usepackage{geometry}                
\usepackage{graphicx}
\usepackage{amssymb}
\usepackage{epstopdf}
\usepackage{hyperref,cleveref}
\usepackage{color,ulem,textcomp}
\usepackage{xcolor}

\def\eps{\varepsilon}
\def\R{{\mathbb R}}
\def\N{{\mathbb N}}
\def\Z{{\mathbb Z}}

\def\C{{\mathbb C}}
\def\R{{\mathbb R}}
\def\N{{\mathbb N}}
\def\Z{{\mathbb Z}}
\def\Sch{{\mathcal S}}

\def\O{\mathcal O}

\def\U{\mathcal U}
\def\wp{{\rm WP}}

\def\bdS{\boldsymbol{S}}
\def\op{{\rm op}}
\def\eps{\varepsilon}
\def\Im{\text{\rm Im}}

\def\e{{\rm e}}

\def\<{{\langle}}
\def\>{{\rangle}}
\def\({{\left(}}
\def\beq{\begin{equation}}   \def\eeq{\end{equation}}
\def\bea{\begin{eqnarray}}  \def\eea{\end{eqnarray}}

\def\di{\displaystyle}
\def\W{Y}
\newcommand{\1}{\mathbb I}

\newtheorem{theorem}{Theorem}[chapter]
\newtheorem{lemma}[theorem]{Lemma}

\theoremstyle{definition}
\newtheorem{definition}[theorem]{Definition}
\newtheorem{example}[theorem]{Example}
\newtheorem{assumption}[theorem]{Assumption}
\newtheorem{proposition}[theorem]{Proposition}
\newtheorem{corollary}[theorem]{Corollary}

\theoremstyle{remark}
\newtheorem{remark}[theorem]{Remark}

\numberwithin{section}{chapter}
\numberwithin{equation}{chapter}

\makeindex

\begin{document}

\frontmatter

\title{Asymptotic initial value representation  of the solutions of semi-classical systems presenting smooth codimension one crossings}


\author[C. Fermanian Kammerer]{Clotilde~Fermanian-Kammerer}
\address{Larema, UMR 6093 du CNRS,
Universit\'e d'Angers,
2 boulevard Lavoisier, 
49045 Angers Cedex 01\\ France}
\email{Clotilde.Fermanian@univ-angers.fr}

\author[C. Lasser]{Caroline Lasser}
\address{Department of Mathematics
Technische Universit\"at M\"unchen
85748 Garching bei M\"unchen, Germany}
\email{classer@ma.tum.de}

\author[D. Robert]{Didier Robert}
\address{Laboratoire de math\'ematiques Jean Leray
UMR 6629 du CNRS
Universit\'e de Nantes
2, rue de la Houssini\`ere  44322 Nantes Cedex 3, France}
\email{didier.robert@univ-nantes.fr}

\thanks{\textbf{Acknowledgements}. Clotilde Fermanian Kammerer acknowledges the support of the Region Pays de la Loire, Connect Talent Project {\it High Frequency Analysis of Schrödinger equations} (HiFrAn 2022 07750),
and from the France 2030 program, Centre Henri Lebesgue ANR-11-LABX-0020-01. Caroline Lasser was supported by the Centre for Advanced Study in Oslo, Norway,
research project Attosecond Quantum Dynamics Beyond the Born--Oppenheimer Approximation, and by the Deutsche Forschungsgemeinschaft (DFG, German Research Foundation) -- TRR 352 -- Project-ID 470903074.}



\keywords{Matrix-valued Hamiltonian, smooth crossings, codimension~1 crossings, wave-packet, Bargmann transform,  Herman-Kluk propagators, thawed and  frozen Gaussian approximations}

\dedicatory{This text is dedicated to the memory of George Hagedorn (1953--2023)
}

\begin{abstract}
This paper is devoted to the  construction of approximations of the propagator associated with a semi-classical matrix-valued   Schrödinger operator with  symbol presenting smooth eigenvalues  crossings. Inspired by 
 the approach of the theoretical chemists Herman and Kluk who propagated continuous superpositions of Gaussian wave-packets for scalar equations, we consider frozen and thawed Gaussian initial value representations that incorporate classical transport and branching processes along a hopping  hypersurface.
Based on the Gaussian wave-packet frame work, our result relies on an accurate analysis of the solutions of the associated Schrödinger equation for data that are vector-valued 
wave-packets. We prove that these solutions are asymptotic to wavepackets at any order in terms of the semi-classical parameter. 
\end{abstract}

\maketitle

\tableofcontents



\mainmatter

\chapter{Introduction}\label{chapter:introduction}

Since the early days of semi-classical analysis, operators {that approximate} the dynamics of a semi-classical propagator {have} been the object of major attention. The theory of Fourier integral operators answers to this question by proposing methods for constructing {approximative} propagators of a scalar semi-classical Schrödinger equation (see~\cite[Chapter 12]{Zwobook} or~\cite{disj}). 
{Pioneering} work about a semi-classical theory of FIOs is \cite{Ch}, extended in \cite{HeRo} and \cite{Rb}.
\smallskip

Few results exist for systems {except} for those that are called {\it adiabatic}, because the eigenvalues of the underlying Hamiltonian {matrix} are of constant multiplicity. The analysis of such systems can be reduced to those of scalar equations through a diagonalization process using the so-called {\it super-adiabatic projectors}. {The super-adiabatic approach} has been carried out by Martinez and Sordoni~\cite{MS} {as well as } Spohn and Teufel~\cite{ST}, {see also}~\cite{emwe,N1,N2,bi} {for earlier results or}~\cite{BGT,PST} for {more recent} results in a similar direction. 
\smallskip

The present study gives the first complete construction of an integral representation of the propagator associated to a Hamiltonian {generating} non-adiabatic {dynamics} in a very general situation. It focuses on those Hamiltonian {matrices} that have smooth eigenprojectors, with smooth eigenvalues, though of non constant multiplicity. The framework {applies} to generic situations where two eigenvalues cross along a hypersurface on points where  the Hamiltonian vector fields associated with these eigenvalues are transverse to the crossing hypersurface. {This set-up} has  already been the one of the work of Hagedorn~\cite[Section 5]{Hag94} and Jecko~\cite{J}. The Fourier integral operators approximating the propagator associated with these non-adiabatic Hamiltonians are based on Gaussian {wave-packets} and  the Bargmann transform, in the spirit of the {\it Herman--Kluk propagator}. 
\smallskip

The Herman--Kluk propagator has been  introduced in theoretical chemistry (see~\cite{Hel,Kay1,HK,Kay2}) for the analysis of molecular dynamics for scalar equations. The mathematical analysis has been performed later by Rousse and Swart~\cite{RS1} and Robert~\cite{R}, independently.  
The action of the Herman--Kluk propagator consists in the {continuous} decomposition of the initial data {into semi-classical Gaussian wave-packets} and the implementation of  the propagation of the wave-packets as studied in the 70s and 80s by {Heller \cite{Hel}}, Combescure and Robert~\cite{corobook}, and Hagedorn~\cite{Hagedorn}.
It involves time-dependent quantities that are called {\it classical quantities} because they can be interpreted in terms of Newtonian mechanics. Such an {approximative} description of the propagator {in terms of several Gaussian wave packets motivates numerical methods that naturally combine with  probabilistic sampling techniques, see \cite{KluHD86} or more recently \cite{LS,KLV}.}
\smallskip

We prove the convergence of two types of approximations, respectively called {\it thawed} and {\it frozen Gaussian approximations}, both {built of continuous superpositions of Gaussian 
wave-packets, the frozen one in the spirit of the original Herman--Kluk propagator}. Their difference mainly consists in the way the {width matrices} resulting from the propagation of the {individual} semi-classical Gaussian wave packets are treated. The presence of crossings requires to add to the {semi-classical Gaussian wave packet propagation} some transitions between the crossing hypersurfaces. Therefore, these Fourier integral operators  incorporate classical transport  along the Hamiltonian trajectories associated with the eigenvalues of the Hamiltonian and a branching process along the crossing hypersurface. Some of these ideas have been introduced in~\cite{FLR1,FLR2}, in particular in~\cite{FLR1} where the  propagation of wave-packets through smooth crossings {has been studied. Here, we revisit and extend these results}, by proving that a wave-packet {propagated} through a smooth generic crossing remains asymptotically a wave-packet to any order in the semi-classical parameter. We {then} prove uniform estimates for the {associated semi-classical} approximations of propagators when acting on families of initial data that are {\it frequency localized} in the sense that their $L^2$-mass does not escape in phase space to $\infty$ when the semi-classical parameter goes to~$0$, neither in position, nor in {momentum}. {This class of initial} data {is typically met for the numerical simulation of molecular quantum systems.}
\smallskip

{\it Codimension one crossing} is the simplest case of crossings. They are classified by the codimension of the set of dicontinuities of the multiplicity of the eigenvalues of the matrix-valued Hamiltonian. In his monograph~\cite{Hag94}, George Hagedorn has settled a first classification  of the crossings involved in molecular dynamiccs that have been improved and extended to general matrix-valued Hamiltonian in 2000 by the normal form results of Colin de Verdière~\cite{CdV1,CdV2}. Codimension~1 crossings have the advantage that, despite the discontinuities of the multiplicity of the eigenvalues,  the latter are smooth after adequate renumbering. In higher codimensions, the eigenvalues display conical singularities and the transitions generated by the crossing are at leading order~\cite{Hag94,HagJ98,HagJ99} (the last two references are treating avoided crossings, i.e. gaps that shrink with the semi-classical parameter, which generates conical singularities in position variables and this parameter). For codimension~1 crossing transitions are of the size of the square root of the semiclassical parameter. Various other questions have been addressed for crossings of codimension 2 and 3, such as resonances asymptotics~\cite{grigis_martinez,fujiie_martinez_watanabe} or propagation with nonlinear potentials~\cite{Hari1}. We do not discuss these aspects in this monograph and focus on the analysis of the  propagator through a (smooth) codimension one crossing. 
\smallskip

The analysis of the propagation through smooth eigenvalue crossings has been 
pioneered by Hagedorn in \cite[Chapter~5]{Hag94}. He considered Schr\"odinger operators with 
matrix-valued potentials and propagated initial data that are known as
semi-classical wave packets or generalized coherent 
states \cite[Chapter~4]{corobook}. The core of the wave-packet had to be chosen such that 
it classically propagates to the crossing. 
In the same framework adjusted to the context of solid states physics, Watson and Weinstein~\cite{WW} analyze the propagation of wave-packets through a smooth crossing of Bloch bands.
The results developed here extend \cite[Chapter~5]{Hag94} and \cite{WW} in two ways. 
The single wave-packet is turned into an initial value representation with uniform control for frequency localized initial data. The Schr\"odinger and Bloch operators are generalized to Weyl 
quantized operators with smooth time-dependent symbol.

\section{{First overview and notations}}

{
The remainder of the introduction specifies the mathematical setting (assumptions on the Hamiltonian 
operator and the initial data), discusses the classical quantities involved in the approximation, 
reviews the known results on the thawed and frozen initial value representations in the adiabatic setting, 
and then presents the main results of this paper: Theorem~\ref{thm:TGeps} on the thawed approximation 
with hopping trajectories, Theorem~\ref{thm:FGeps} and Theorem~\ref{thm:FGeps_av} on the frozen approximation with hopping trajectories, that are pointwise and averaged in time, respectively, 
and Theorem~\ref{th:WPmain} on wave-packet propagation through smooth crossings to arbitrary order.} 
\smallskip 

We prove Theorems \ref{thm:TGeps}, \ref{thm:FGeps} and \ref{thm:FGeps_av} in Chapters~\ref{chap:4} and~\ref{chap:5}. These proofs rely on Theorem~\ref{th:WPmain}, that is proved in 
Chapters~\ref{chap:2} and~\ref{sec:prop}.
\smallskip

{Chapter~\ref{chap:4} recalls elementary facts about the Bargmann transform. 
Then, it introduces the new notion of frequency localization, which will be crucial 
for controlling the remainder estimates for both the frozen and the thawed initial value representations 
in Chapter~\ref{chap:5}.}   
\smallskip 

{
The refined wave-packet analysis of Chapters~\ref{chap:2} and~\ref{sec:prop} 
does not depend on the theory of initial value representations and can be read 
independently from Chapters~\ref{chap:4} and~\ref{chap:5}. It propagates wave-packets through smooth crossings in two steps: using a rough diagonalisation of the Hamiltonian operator in the crossing region and  
super-adiabatic projectors for the outside. Both constructions rely on pseudo-differential calculus 
for matrix-valued symbols that is developed in Chapter~\ref{chap:2} and complemented by 
additional technical points in the appendices. The propagation of a wave packet through regions of small eigenvalue gap and the actual crossing are analyzed in \Cref{sec:prop} and  
\Cref{sec:through}, respectively. }
\bigskip 
 
\noindent{\bf Notations and conventions.}
 All the functional sets that we shall consider can have values in $\C$ (scalar-valued), $\C^m$ (vector-valued) or in $\C^{m,m}$ (matrix-valued). We denote by by $v\cdot w = v_1w_1 + \cdots + v_mw_m$ the bi-linear products of vectors and by 
 \[ \langle g, f\rangle=\int_{\R^d}  f(x) \cdot  \overline g(x) dx\]
 the inner product of $L^2(\R^d,\C^m)$.
 For $k\in\N$, we will work with the spaces 
\beq\label{eq:Sigmak}
\Sigma_\eps^k(\R^d)=\left\{ f\in L^2(\R^d),\;\;\forall \alpha,\beta\in\N^d,\;\; |\alpha|+|\beta| \leq k,\;\; x^\alpha (\eps \partial_x)^\beta f\in L^2(\R^d)\right\}
\eeq
endowed with the norm 
\[
\| f\|_{\Sigma^k_\eps} = \sup_{|\alpha|+|\beta| \leq k}\| x^\alpha (\eps \partial_x)^\beta f\|_{L^2}.
\]
We will use the abbreviation $\Sigma_\eps^k := \Sigma_\eps^k(\R^d)$. We denote by
\[
\{f,g\}=\nabla _\xi f\cdot \nabla_x g -\nabla_x f\cdot \nabla_\xi g
\]
the Poisson bracket of differentiable functions $f$ and $g$ defined on $\R^{2d} = \R^d_x\times\R^d_\xi$, and note that 
$\{f,g\} = -\nabla f\cdot J\nabla g$ with
\[
J = \begin{pmatrix}0 & \1_{d}\\ -\1_{d} & 0\end{pmatrix}.
\]
Let $0<\eps\ll 1$ be the semi-classical parameter. For $a\in{\mathcal C}^\infty(\R^{2d})$ being a smooth scalar-, vector- or matrix-valued function with adequate control on the growth of derivatives, the Weyl operator $\widehat a = \op_\eps(a)$ is defined by 
\beq\label{opweyl}\op^w_\eps(a)f(x):= \widehat a f(x) := (2\pi\eps)^{-d} \int_{\R^{2d}} a\!\left(\frac{x+y} 2, \xi\right) 
{\rm e}^{i\xi\cdot(x-y)/\eps} f(y) \,dy\, d\xi
\eeq
for all Schwartz functions $f\in{\mathcal S}(\R^d)$. In particular, $\op_1^w$ denotes the pseudo-differential quantization with $\eps=1$, that is, the standard non semi-classical one. We set $D_x= \op_1(\xi) = \frac 1 i\partial_x$. If $\pi$ is an orthogonal projector, then $\pi^\perp$ denotes the projector $\pi^\perp=\1 -\pi$.

 \section{The setting }
 
 \subsection{The Schr\"odinger equation} 
 For $\eps>0$, we consider the Schr\"odinger equation
\beq\label{eq:sch}
i\eps\partial_t\psi^\eps(t) =\widehat H^\eps(t)\psi^\eps(t),\;\; \psi^\eps_{|t=t_0}=\psi^\eps_0.
\eeq
in $L^2(\R^d, \C^m)$, $m\geq 2$, where $\widehat H^\eps(t)$ is the semi-classical quantization of a (time-dependent) Hermitian   matrix  symbol $H^\eps(t, z)\in\C^{m,m}$. Here, $t\in\R$ and $z=(x,\xi)\in\R^d\times\R^d$. We are interested in an asymptotic description of the solution $\psi^\eps(t)$ as $\eps\to0$, in particular in asymptotic expansions of $\psi^\eps(t)$ with precise error estimates of the order $\eps$. 
  
\smallskip
In full generality, we could assume that the map $(t,z)\mapsto H^\eps (t,z)$ is a semi-classical observable   in the sense that  the function $H^\eps(t,z)$ is an asymptotic sum of the form $\sum_{j\geq 0} \eps^j H_j(t,z)$. However, 
in this asymptotic sum, the important terms  are the principal symbol $H_0(t,z)$ and the sub-principal one $H_1(t,z)$;  the terms $H_j(t,z)$ for $j\geq 2$  only affect the solution beyond order $\eps$, which is the order of the approximation we are looking for. Therefore, we assume  that the self-adjoint matrix $H^\eps$ writes 
$$H^\eps(t,z):=H_0(t,z)+ \eps H_1(t,z).$$

\subsection{Assumptions on the Hamiltonian}

We work on a time interval of the form 
\[
I:=[t_0,t_0+T], \;\;t_0\in\R\;\;\mbox{and} \;\; T>0
\]
and consider subquadratic matrix-valued Hamiltonians. 

\begin{definition}[Subquadratic] \label{def:subquad}
The $\eps$-dependent Hamiltonian 
\[
H^\eps=H_0+\eps H_1 \in\mathcal C^\infty(I\times \R^{2d},\C^{m,m})
\]
 is {\it subquadratic on the time interval $I$}  if and only if it has the property:
 \beq\label{eq:decayHj}
\forall j\in\{0,1\},\;\;\forall \gamma\in\N^d,\;\;\exists C_{j,\gamma}>0,\;\; 
\sup_{(t,z)\in I\times \R^{2d}}  \vert\partial_z^\gamma H_j(t,z)\vert \leq C_{j,\gamma}\langle z\rangle^{(2-j-\vert\gamma\vert)_+}
\eeq
 \end{definition}

 Assuming that  $H^\eps$ is subquadratic on the time interval~$I$ ensures that the system~\eqref{eq:sch} is well-posed in $L^2(\R^d)$ for $t\in I$, and, more generally also in the functional spaces $\Sigma^k_\eps$ for $k\in\N$. We denote by ${\mathcal U}^\eps_H(t,t_0)$ the  unitary propagator defined by 
\begin{equation}\label{def:propagator}
i\eps\partial_t \mathcal U^\eps_H(t,t_0)= \widehat H^\eps (t) \mathcal U_H^\eps(t,t_0),\;\; \mathcal U_H^\eps(t_0,t_0)=
\1_{m}.
\end{equation}
It is a   bounded operator of the  $\Sigma_\eps^k$ spaces, uniformly in $\eps$ (see~\cite{MaRo}):  there exists $C_{T,k}>0$ such that 
\[
\sup_{t\in I}\|{\mathcal U}^\eps_H(t,t_0)\|_{{\mathcal L}(\Sigma^k_\eps)} \,\le\, C_{T,k}.
\]

\begin{example}\label{ex_referee}
     Take $m=2$, $d=1$,  $H_0=\frac {|\xi|^2}2 \1_{2} +{\rm Diag}(v_1(x),v_2(x))$ with $v_j\in\mathcal C^\infty(\R^d)$, and $H_1(x)= M$ a fixed Hermitian matrix. Assume the boundedness of the derivatives of the function $v_j$ with order higher than 2. then $H_0+\eps H_1$  is subquadratic.
 \end{example}

We assume that the principal symbol $H_0(t,z)$ of $H^\eps(t,z)$ has two distinct eigenvalues $h_1(t,z)$ and $h_2(t,z)$ that present a smooth crossing in the following sense, see also~\cite{FLR1}.  

\begin{definition}\label{def:smooth_cros}
\begin{enumerate}
\item {\rm (Smooth crossing)}.
The matrix $H_0\in\mathcal C^\infty(I\times \R^{2d} , \C^{m,m})$ has a smooth crossing on the set~$\Upsilon\subseteq I\times\R^{2d}$  if there exists  $h_1, h_2\in\mathcal C^\infty (I\times \R^{2d} ) $ and  two orthogonal projectors  $\pi_1,\pi_2\in\mathcal C^\infty(I\times \R^{2d} , \C^{m,m})$ such that $H_0=h_1\pi_1+h_2\pi_2$ and 
\[ 
h_1(t,z)=h_2(t,z)\;\Longleftrightarrow\;  (t,z)\in \Upsilon.
\]
\item
Set 
$\displaystyle{
f(t,z)= \frac 12\left( h_1(t,z) -h_2(t,z) \right) \;\;\mbox{and}\;\; v(t,z)=\frac 12 \left( h_1(t,z) + h_2(t,z) \right)
}$.
\begin{enumerate}
\item {\rm (Non-degenerate crossing)}.  The crossing is non-degenerate at $(t^\flat,\zeta^\flat)\in\Upsilon$  if
\[d_{t,z} \left( H_0-v\, \1_{m}\right)(t^\flat ,\zeta^\flat )  \not=0\]
where $d_{t,z}$ is the one  differential form in the variables $(t,z)$.
\item {\rm (Generic crossing points).}
The crossing is generic at  $(t^\flat, \zeta^\flat)\in\Upsilon$ if one has 
  \begin{equation}\label{hyp:codim1}
  (\partial_t f +\{v, f\})(t^\flat, \zeta^\flat)\neq 0.
  \end{equation}
  Note that there then exists an open set $\Omega\subset I\times \R^{2d}$ containing $(t^\flat, \zeta^\flat)$ such that the set $\Upsilon\cap \Omega$  is a manifold.
  \end{enumerate}
  \end{enumerate}
\end{definition}

\begin{example}\label{ex_referee2}
 Continuing with Example~\ref{ex_referee}, assume  $v_1(x)=xw(x)$ and $v_2(x)=-xw(x)$ with $w(0)\not=0$. We then have $\Upsilon=\{x=0\}$ and the non-degenerate generic crossing points are those included in $ \{\xi\not=0\}$.   
\end{example}

With these definitions at hand, we introduce one of the main assumptions on the crossing points of the Hamiltonian $H^\eps$.

\begin{assumption}[Crossing set]\label{hyp:smooth_cros}
The matrix $H_0$ has a smooth crossing set~$\Upsilon$ and all the  points of~$\Upsilon$ are non degenerate and generic crossing points. 
\end{assumption}

To have well-defined unitary propagators ${\mathcal U}^\eps_{h_1}(t,t_0)$ and $\mathcal U^\eps_{h_2}(t,t_0)$ and adequate control on the eigenprojectors $\pi_1(t,z)$ and $\pi_2(t,z)$, we shall make additional assumptions on the growth of the eigenvalues and of their gap function. Our setting will be the following: 

\begin{assumption}[Growth conditions for smooth crossings]
\label{hyp:growthH}
Let $H^\eps=H_0+\eps H_1\in\mathcal C^\infty(\R\times \R^{2d})$ be 
 subquadratic on the time interval $I$ and have a smooth crossing on the set~$\Upsilon$. 
 We consider the two following assumptions :
 \begin{enumerate}
 \item[(i)] The growth of $H_0(t,x)$  is driven by the function $v(t,z)$, 
 \begin{equation}\label{eq:growth_trace} 
 \forall \gamma\in \N^{2d} \; , \;\;\exists C_\gamma>0,\;\;\forall (t,z)\in I\times \R^{2d},\;\;
 | \partial_z^\gamma( H_0-v \, \1_{m} )(t,z)| +|f(t,z)|  \leq C_\gamma.
 \end{equation}
\item[(ii)] The gap is controlled at infinity, i.e. there exist $R>0$ and $n_0\in\N$ such that 
\begin{equation}\label{eq:growth_gap}
 \forall t\in I,\;\;\forall |z|>R,\;\; |f(t,z)|\geq C \langle z\rangle ^{-n_0},
 \end{equation}
 and in the case $n_0\not=0$, the functions $z\mapsto  \pi_1(t,z), \pi_2(t,z)$ are assumed to have bounded derivatives at infinity.
  \item[(iii)] The eigenvalues $h_1$ and $h_2$ are subquadratic, i.e. for $\ell\in\{1,2\}$ 
 \begin{equation}\label{hyp:hj}
\forall \gamma\in \N^{2d}, \; |\gamma|\geq 2, \;\;\exists C_\gamma>0,\;\;\forall (t,z)\in I\times \R^{2d},\;\;
\left|\partial_z^\gamma h_\ell(t,z)\right|\leq C_\gamma.
\end{equation}
 \end{enumerate}
\end{assumption}

In the context of Assumption~\ref{hyp:growthH}, we shall say that a matrix~$A$ is diagonal if $A=\pi_1 A \pi_1+\pi_2 A \pi_2$ and off-diagonal if  $A=\pi_1 A \pi_2+\pi_2 A \pi_1$.

 \begin{remark}
 \begin{enumerate}
 \item 
In (iii), the fact that the eigenvalues $h_1(t,z)$ and $h_2(t,z)$ are of subquadratic growth guarantees the existence  
of the unitary propagators ${\mathcal U}^\eps _{h_j}(t,t_0)$ for $ j\in\{1,2\}$ and of the classical quantities associated with  the Hamiltonians $h_1$ and $h_2$ that we will introduce below. 
\item The growth conditions of Assumption~\ref{hyp:growthH}
 imply that 
the eigenprojectors $\pi_j(t)$, $j=1,2$, and their derivatives have at most polynomial growth. However, when $n_0\not=0$, they may actually grow. This is proved in  Lemma~\ref{lem:growth_eigen_bis}. It is for this reason that we assume that the projectors have bounded derivatives when $n_0\not=0$ in (ii).
\end{enumerate}
\end{remark}

\begin{example}
\begin{enumerate}
\item The Hamiltonian of Example~\ref{ex_referee} and~\ref{ex_referee2}, satisfies Assumption~\ref{hyp:smooth_cros} in the set $\{\xi\not=0\}$. It also satisfies Assumption~\ref{hyp:growthH} (iii) in $\R^d$. However,  even it does not formally satisfy Assumption~\ref{hyp:growthH} (i) and (ii),  the eigenprojectors $\pi_1$ and $\pi_2$ are constant, and thus satisfy all the conditions required for our analysis. 
\item Examples of matrix-valued Hamiltonian are given in molecular dynamics (as defined in~\cite {Hag94}[Chapter~5]) by
Schr\"odinger operators with matrix-valued potential,
\begin{equation*}\label{ex:hag}
\widehat H_S = -\frac{\eps^2}{2}\Delta_x\, \1_{2} + V(x),\quad V\in
{\mathcal C}^\infty(\R^d,\C^{2,2}).
\end{equation*}
When $V$ presents a codimension 1 crossing (as defined in~\cite{Hag94}, then  the crossing points $(x,\xi)$ are non degenerate and generic when $\xi\not=0$.
\item Another class of  examples appears in
 solid state 
physics in the context of Bloch band decompositions (see~\cite{WW,CFM} for example) with  Hamiltonians of the form
\begin{equation*}\label{ex:WW}
\widehat H _A= A(-i\eps\nabla_x) + W (x) \1_{2},\quad A\in{\mathcal C}^\infty(\R^d,\C^{2,2}),\quad 
W\in{\mathcal C}^\infty(\R^d,\C).
\end{equation*}
\item Finally, in~\cite{FLR2}, the authors have considered the operator 
\[
\widehat H_{k,\theta}=\frac{\eps}{i}\frac{d}{dx}\1_{2} + kx\begin{pmatrix} 0&{\rm e}^{i\theta x}\\ {\rm e}^{-i\theta x}&0\end{pmatrix},
\]
with $d=1$, $m=2$, $\theta\in{\R}_+$, $k\in\R\backslash\{0\}$.
\end{enumerate}
\end{example}

\subsection{Assumptions on the data}
We consider vector-valued initial data $\psi^\eps_0\in L^2(\R^d,\C^m)$ of the form  
\[
\psi^\eps_0=\widehat{\vec V} \phi^\eps_0
\]
 where  $z\mapsto \vec V(z)$ is a smooth $\C^m$ vector-valued function, bounded together with its derivatives, and  
 $\phi^\eps_0\in L^2(\R^d,\C)$ is {\it frequency localized} in the sense of the next definition. For stating it, we denote the Gaussian of expectation~$q$, variance $\sqrt\eps$  that oscillates along~$p$ according to
\begin{equation}\label{def:gepsz}
g^\eps_z(x)=(\pi\eps)^{-d/4} {\rm e}^{-\frac {(x-q)^2}{\eps} +\frac i\eps p\cdot (x-q)},\;\;\forall x\in\R^d, \;{\rm with}\; z=(q,p). 
\end{equation}

 \begin{definition}[Frequency localized functions]\label{def:freq_loc_0}
 Let   
 $(\phi^\eps)_{\eps>0}$ be a  family of functions of $L^2(\R^d)$. The family~$(\phi^\eps)_{\eps>0}$ is frequency localized 
  if the family is bounded in $L^2(\R^d)$ and if there exist $R_0, C_0,\eps_0>0$ and $N_0 > d $ such that for all $\eps\in(0,\eps_0]$,
 \[
(2\pi\eps)^{-d/2} \left| \langle g^\eps_z, \phi^\eps \rangle \right|\leq C_0 \,\langle z\rangle^{-N_0} \;\;\ \text{for all}\ z\in\R^{2d}\ \text{with}\ |z| > R_0.
 \]
 One then says that $(\phi^\eps)_{\eps>0}$ is frequency localized.
 \end{definition}

We will introduce a more precise definition in Chapter~\ref{chap:4}.  We will also discuss in this Chpater the properties of frequency localized families. In particular, we will show that they can be written as continuous sum of Gaussian states with center $(q,p)$ localized in a compact set (see Lemma~\ref{lem:freq_loc}). The analysis of the examples given below is performed in Lemma~\ref{lem:ex_freq_loc}.
 \smallskip 
 
 \begin{example}
\begin{enumerate}
\item The Gaussian wave packets $(g^\eps_{z_0})_{\eps>0}$ are frequency localized functions. 
\item Define  $( {\rm WP}^\eps_{z_0} (u) )_{\eps>0}$ by
 \begin{equation}\label{def:WP}
 {\rm WP}^\eps_{z_0} (u) (x)= \eps^{-d/4} {\rm e}^{\frac i\eps p_0\cdot (x-q_0)} u \left(\frac{x-q_0}{\sqrt\eps}\right) ,\;\; 
x\in\R^d,
 \end{equation} 
for  $u\in\mathcal S(\R^d)$ and $z_0=(q_0,p_0)\in\R^{2d}$.  They are frequency localized functions.
\item Lagrangian (or WKB) states
$\varphi^\eps(x) = a(x){\rm e}^{\frac{i}{\eps}S(x)}$
with $a\in \mathcal C_0^\infty(\R^d,\C)$ and $S\in \mathcal C^\infty(\R^d, \R)$,  are  also frequency localized functions.
\end{enumerate} 
\end{example}

Our vector-valued initial data will have a scalar part consisting in a frequency localized family.

\begin{assumption}\label{hyp:data}
The initial data $\psi^\eps_0$ in~\eqref{eq:sch} satisfies
\begin{equation}\label{eq:psieps0}
\psi^\eps_0(x)= \widehat {\vec V}  \phi^\eps_0(x),\;\; x\in\R^d
\end{equation} 
where 
\begin{enumerate}
\item [(i)] The family $(\phi^\eps_0)_{\eps>0}$  is frequency localized  with  constants  $R_0, N_0,C_0,\eps_0$ in Definition~\ref{def:freq_loc_0}.
\item[(ii)] The function $z\mapsto \vec V(z)$ is a function of $\mathcal C^\infty(\R^{2d}, \C^m)$, bounded together with its derivatives, and valued in the set of normalized vectors.
\end{enumerate}
\end{assumption}

We point out that any vector-valued bounded family in $L^2(\R^d)$ writes as a sum of data of the form~$\widehat {\vec V}  \phi^\eps_0(x)$ for   bounded $(\phi^\eps_0)_{\eps>0}$. As a consequence, assuming the initial datum $\psi^\eps_0$ satisfies~\eqref{eq:psieps0} is not really restrictive.   Of course,  for $k\in \C\setminus\{0\}$,the vector valued function $\vec V$ can be turned into $k\vec V$
by changing~$ \phi^\eps_0$ into~$k^{-1} \phi^\eps_0$.

\section{Classical quantities}\label{subsec:classical}

In this section, we introduce classical quantities associated with the Hamiltonian $H^\eps$. These quantities will be used to construct the approximations of the propagator $\mathcal U^\eps_H(t,t_0)$ that are the subject of this text. They are called {\it classical} because they do not depend on the semi-classical parameter $\eps$ and are obtained by solving $\eps$-independent equations that mainly are ODEs instead of PDEs. Thus, the numerical realization of the resulting propagator's approximations avoids the difficulties induced by the  $\frac 1\eps$-oscillations and is applicable in a high-dimensional setting, see \cite{LL} for a recent review on this topic.  Besides their definition, we shall also recall well-known results about their role in the description of Schr\"odinger propagators. 
\smallskip 

In this section, we assume that $H^\eps=H_0+\eps H_1$ is subquadratic on the time interval $I$ (as defined in Definition~\ref{def:subquad}), with smooth eigenprojectors $\pi_1$ and $\pi_2$, and eigenvalues $h_1$ and $h_2$, the latter being subquadratic (as in (iii) of Assumption~\ref{hyp:growthH}).

\subsection{The flow map}
Let $\ell\in\{1,2\}$, we associate with   $h_\ell:I\times\R^{2d}\to\R$, $(t,z)\mapsto h_\ell(t,z)$ the functions
\[
z_\ell(t) = (q_\ell(t),p_\ell(t))
\]
which denote the {\it classical Hamiltonian trajectory} issued from a phase space point $z_0$ at time $t_0$, that is defined by the ordinary differential equation
$$\dot z_\ell(t) = J \partial_z h_\ell(t,z_\ell(t)),\;\; z_\ell(t_0)=z_0$$
with 
\begin{equation}\label{def:J}
J = \begin{pmatrix}0 & \1_{d}\\ -\1_{d} & 0\end{pmatrix}.
\end{equation}
We note that $J$ is the matrix associated with the symplectic form 
\[
\sigma(z,z')= \langle Jz,z'\rangle =p\cdot q'-p'\cdot q,\qquad z=(q,p),\, z'=(q',p')\in\R^{2d}.
\]
The trajectory $z_\ell(t) = z_\ell(t,t_0,z_0)$ depends on the initial datum and defines 
the associated {\it flow map} $\Phi_{h_\ell}^{t,t_0}$  of the Hamiltonian function $h_\ell$ via
\[ 
z_0\mapsto \Phi_{h_\ell}^{t,t_0}(z_0):= z_\ell(t,t_0,z_0),\qquad z_0\in\R^{2d}.
\]
We will also use the trajectory's {\it action integral}
\begin{equation}
\label{def:S} 
S_\ell(t,t_0,z_0) = \int_{t_0}^t \left(p_\ell(s)\cdot \dot q_\ell(s)-h_\ell(s,z_\ell(s)) \right) ds,
\end{equation}
and the {\it Jacobian matrix of the flow map}, also called {\it stability matrix}
\begin{equation}\label{def:F}
F_\ell(t,t_0,z_0) = \partial_z \Phi_{h_\ell}^{t,t_0}(z_0). 
\end{equation}
Note that $F_\ell(t,t_0,z_0)$ is a symplectic $2d\times 2d$ matrix, that satisfies the linearized flow equation
\begin{equation}\label{eq:lin}
\partial_t F_\ell(t,t_0,z_0) = J {\rm Hess}_zh_\ell(t,z_\ell(t)) \, F_\ell(t,t_0,z_0),\;\;\;\;F_\ell(t_0,t_0,z_0) = \1_{\R^{2d}}.
\end{equation}
We denote its blocks by 
\begin{equation}\label{eq:F}
F_\ell(t,t_0,z_0) = 
 \begin{pmatrix}  A_\ell(t,t_0,z_0) &B_\ell(t,t_0,z_0) \\ C_\ell(t,t_0,z_0) &D_\ell(t,t_0,z_0)\end{pmatrix}.
\end{equation}

\begin{example}
In the case of Example~\ref{ex_referee},  we have $\Phi^{t,t_0}_{h_\ell}(z)=(q_\ell(t-t_0),p_\ell(t-t_0))$ for $\ell\in \{1,2\}$, with 
\begin{equation}
    \dot q_\ell(t)=p_\ell(t),\;\;\dot p_\ell(t)=-\nabla v_\ell(q_\ell(t)),\;\;(q_\ell(0),p_\ell(0))=z.
\end{equation}
The action satisfies $S_\ell(t,t_0,z)=\underline S_\ell(t-t_0,z)$ with $\underline {\dot S}_\ell(t)=\frac{|p_\ell(t)|^2}2 - v_\ell(q_\ell(t))$ and the flow map also is a function of $t-t_0$. Moreover, the system~\eqref{eq:lin} simplifies since 
\[
{\rm Hess} \, h_\ell(t, z)=\begin{pmatrix} 
{\rm Hess}\, v_\ell(q) & 0 \\
0& \1_{\R^d}
\end{pmatrix}
\]
\end{example}

\subsection{The metaplectic transform and Gaussian states} 
It is standard to
associate with the time-dependent symplectic map $F_\ell(t,t_0,\cdot)$ a unitary evolution operator, the {\it metaplectic transformation} that acts on square integrable functions in $L^2(\R^d)$ 
as a unitary transformation.
 \[
  {\mathcal M}[F_\ell(t,t_0,z_0)] : \; u_0\mapsto u(t) 
\]
 and 
 associates with an initial datum $u_0$ the solution at time $t$ of the Cauchy problem 
$$
i\partial_t u(t)= \op_1\left({\rm Hess}_z h_\ell\left(t, \Phi^{t,t_0}_{h_\ell}(z_0)\right)z\cdot z\right) u(t),\;\; u(t_0)=u_0.
$$
This map is called the {\rm metaplectic transformation} associated with the matrix $F_\ell(t,t_0,z_0)$ (see~\cite{MaRo}). 
 It satisfies for all $\eps>0$ and for all symbol $a$ compactly supported or polynomial
 \begin{equation}\label{prop:metaplectic}
   {\mathcal M}[F_\ell(t,t_0,z_0)] ^{-1} {\rm op}_\eps(a)   {\mathcal M}[F_\ell(t,t_0,z_0)] ={\rm op}_\eps(a\circ F_\ell(t,t_0,z_0) ),
 \end{equation} 
 where $a\circ F_\ell(t,t_0,z_0)$ denotes the function $z\mapsto a\left(F_\ell(t,t_0,z_0)z\right)$.
\smallskip

All these classical quantities are involved in the description of the propagation of {\it Gaussian states} by~${\mathcal U}^\eps_{h_\ell}(t,t_0)$, that are a generalization of the Gaussian families $(g^\eps_z)_{\eps>0}$ that we have already seen.  
Gaussian states are wave packets ${\rm WP}^\eps_{z} (g^\Gamma)$ with complex-valued Gaussian profiles~$g^\Gamma$, whose  
covariance matrix~$\Gamma$ is taken in the Siegel half-space ${\mathfrak S}^ +(d)$ of  $d\times d$ complex-valued symmetric matrices with positive imaginary part,
\[
{\mathfrak S}^+(d) = \left\{\Gamma\in\C^{d\times d},\  \Gamma=\Gamma^\tau,\ \Im\Gamma >0\right\}.
\]
Here $\Gamma^\tau$ is the transposed matrix of $\Gamma$ and $\Im\Gamma >0$ means that the matrix 
  $\Im\Gamma$ is positive definite.
 The covariance matrices $\Gamma$ are sometimes called  {\it width} of the Gaussian state. 
 
More precisely, the Gaussian profile  $g^\Gamma$ depends on   $\Gamma\in{\mathfrak S}^+(d)$ according to
\begin{equation}\label{def:Gaussian}
g^\Gamma(x)
 := c_\Gamma\, {\rm e}^{\frac{i}{2}\Gamma x\cdot x},\quad x\in\R^d,
\end{equation}
where 
$c_\Gamma\in\C$ is a normalization constant such that $|c_\Gamma|=\pi^{-d/4} {\rm det}^{1/4}(\Im\Gamma)$. With these notations, the normal centered Gaussian is $g^{i\1_d}$ and the Gaussian wave packets defined in~\eqref{def:gepsz} writes 
$g^\eps_z={\rm WP}_z^\eps(g^{i\1_d})$, according to~\eqref{def:WP}.
\smallskip 
 
 The propagation of Gaussian states by a metaplectic transform  is well-known \cite{corobook}[Chapter~3]: 
for  $\Gamma_0\in{\mathfrak S}^+(d)$, we have 
\begin{equation}\label{eq:action_Gaussian}
 {\mathcal M}[F_\ell(t,t_0,z_0)] g^{\Gamma_0}= g^{\Gamma_\ell(t,t_0,z_0)},
\end{equation}
where the width  $\Gamma_\ell(t,t_0,z_0)\in\mathfrak S^+(d)$ and the corresponding normalization $c_{\Gamma_\ell(t,t_0,z_0)}$ are determined by the initial width $\Gamma_0$ and the Jacobian $F_\ell(t,t_0,z_0)$ according to
\begin{eqnarray}\label{def:Gamma}
\Gamma_\ell(t,t_0,z_0) 
&=& (C_\ell(t,t_0,z_0)+ D_\ell(t,t_0,z_0)\Gamma_0)(A_\ell(t,t_0,z_0) +B_\ell(t,t_0,z_0)\Gamma_0)^{-1}\\
\nonumber
c_{\Gamma_\ell(t,t_0,z_0)} 
&=& c_{\Gamma_0}\,{\rm det}^{-1/2}(A_\ell(t,t_0,z_0)+B_\ell(t,t_0,z_0)\Gamma_0).
\end{eqnarray}
The branch of the square root in ${\rm det}^{-1/2}$ is determined by continuity in time. 
\smallskip 

A quick calculation based on Taylor formula shows that the relation 
\begin{equation}\label{basic_metaplectic}
\mathcal U^\eps_{h_\ell}(t,t_0){\rm WP}^\eps_{z_0} (\varphi)=
{\rm e}^{\frac i\eps S_\ell(t,t_0,z_0)}
{\rm WP}^\eps_{\Phi^{t,t_0}_{h_\ell}(z_0)}(\varphi^\eps(t))
\end{equation}
if and only if $\varphi^\eps(t)$ solves
$$
i\partial_t \varphi^\eps (t)= \op_1\left(\int_0^1{\rm Hess}_z h_\ell\left(t, \Phi^{t,t_0}_{h_\ell}(z_0)+\tau z)\right)z\cdot z(1-\tau)d\tau\right) \varphi^\eps,\;\; \varphi^\eps (t_0)=\varphi.
$$ 
Therefore,  a remarkable  consequence of the fact that the metaplectic transform preserves the Gaussian structure according to~\eqref{eq:action_Gaussian} is that the Gaussian structure is preserved in the quantum evolution. Indeed, 
the action of $\mathcal U^\eps_{h_\ell}(t,t_0)$ on Gaussian wave packets~${\rm WP}^\eps_z (g^\Gamma)$ obeys 
 $$
 \mathcal U^\eps_{h_\ell}(t,t_0){\rm WP}^\eps_{z_0} (g^{\Gamma_0}) ={\rm e}^{\frac i\eps S_\ell(t,t_0,z_0)} {\rm WP}^\eps_{\Phi^{t,t_0}_{h_\ell}(z_0)} (g^{\Gamma_\ell(t,t_0,z_0)})+ \O(\sqrt \eps)
 $$
 in $\Sigma^k_\eps$ for any $k\in\N$.

\subsection{Parallel transport} \label{sec:parallel_trspt}
For systems, the wave function is valued in $L^2(\R^d,\C^m)$ and thus vector-valued. The propagation then involves transformation of the vector part of the eigenfunctions that is called {\it parallel transport}.
\smallskip

Denoting by $\pi_\ell ^\perp$ the projector $\pi_\ell^\perp=\1_{m}-\pi_\ell$, we define    self-adjoint matrices  $H^{{\rm adia}}_{\ell,1} $ by  
  \begin{align}\label{eq:H1adiab}
&   \pi_\ell ^\perp\, H_{\ell,1}^{\rm adia}\pi_\ell ^\perp=0,\;\;\;\;\;\pi_\ell\, H_{\ell,1}^{\rm adia}\pi_\ell= \pi_\ell \left( H_1  + \frac{1}{2i}\{H_0,\pi_\ell\} \right)\pi_\ell,\\
 &\pi_\ell ^\perp\, H_{\ell,1}^{\rm adia}\pi_\ell 
 = \pi_\ell ^\perp\left( i\partial_t \pi_\ell +i \{ h_\ell, \pi_\ell\}\right)\pi_\ell.
  \end{align}

One then  introduces the map ${\mathcal R}_\ell (t,t_0,z_0)$ defined for $\ell\in\{1,2\}$ by -
  \beq\label{def:Rell}
   i \partial_t{\mathcal R}_\ell (t,t_0,z_0) = H^{{\rm adia}}_{\ell,1} \left(t,\Phi_{h_\ell}^{t,t_0}(z_0)\right){\mathcal R}_\ell (t,t_0,z_0),\; \;{\mathcal R}_\ell (t_0,t_0,z_0)=\1_{m}.
\eeq
  The map $t\mapsto H^{{\rm adia}}_{\ell,1} \left(t,\Phi_{h_\ell}^{t,t_0}(z_0)\right)$ is a  locally Lipschitz  map valued in the set of   self adjoint matrices. Therefore, the existence of ${\mathcal R}_\ell (t,t_0,z_0) $ comes from solving  a  linear time dependent ODE by the Cauchy--Lipschitz Theorem.

  \begin{lemma}\label{lem:trsp_par}
  For all $(t_0,z_0)\in I\times \R^{2d}$ and $\ell\in\{1,2\}$, the matrices
${\mathcal R}_\ell (t,t_0,z_0)$ are unitary matrices. Besides, for all $\omega\in\R^{2d}$, they satisfy
\begin{align}\label{eq:trsp_para}
&{\mathcal R}_\ell (t,t_0,z_0)\pi_\ell(t_0,z_0) = \pi_{\ell}\left(t,\Phi^{t,t_0}_{h_\ell}(z_0)\right){\mathcal R}_\ell (t,t_0,z_0),\\
\label{eq:deriv_R}
&\omega\cdot \nabla_z \mathcal R_\ell (t,t_0,z)=\frac 1i \int_{t_0}^t \mathcal R_\ell (t,s,z) (F(s,t_0,z)\omega)\cdot \nabla_z H_{\ell,1}^{\rm adia} (s,\Phi^{s,t_0}_{h_\ell} (z))\mathcal R_\ell (s,t_0,z) ds. 
\end{align}
\end{lemma}

This Lemma is proved in Appendix~\ref{app:A}.
The relation~\eqref{eq:trsp_para} implies that 
 whenever a vector $\vec V_0$  is in the eigenspace of $H_0(t_0,z_0)$ for the eigenvalue $h_\ell(t_0,z_0)$, then  the vector ${\mathcal R}_\ell (t,t_0,z_0) \vec V_0$   is in the range of $\pi_\ell(t,\Phi^{t,t_0}_{h_\ell}(z_0))$.
  In other words, we have constructed a map that preserves the eigenspaces along  the flow:
  \[
  {\mathcal R}_\ell (t,t_0,z_0) : {\rm Ran} \left(\pi_\ell(t_0,z_0)\right) \mapsto {\rm Ran}\left(\pi_\ell(t,\Phi^{t,t_0}_{h_\ell}(z_0))\right).
  \]
  The matrices ${\mathcal R}_\ell (t,t_0,z_0)$ are sometimes referred to as Larmor precession (see~\cite[\S 14.2]{corobook}).
  \smallskip 
  
  For $(t,z)\in I\times \R^{2d}$, 
  the map $\vec V_0 \mapsto {\mathcal R}_\ell (t,t_0,z) \pi_\ell(t_0, z) \vec V_0$ is a parallel transport in the Hermitian vector  fiber bundle $(t,z)\mapsto {\rm Ran}(\pi_\ell(t, z))$ over the phase space $I\times \R^{2d}\subset \R^{1+2d}$,  
  associated with the curves $s\mapsto \gamma(s)=\left(s,\Phi^{s,t_0}_{h_\ell}(z) \right)_{s\in I}$ and the matrix $H^{\rm adia}_{\ell,1}$. Indeed, the covariant derivative along the curve $(\gamma(s))_{s\in I}$ is given by 
  \[
  \nabla_{\dot \gamma(s)}=\partial_t +J dh_\ell \cdot \nabla_z
  \]
and the relation 
 $\vec X\left(s, \Phi^{s,t_0}_{h_\ell}(z) \right) =  {\mathcal R}_\ell (t,t_0,z) \pi_\ell(t_0, z) \vec V_0$ defines a smooth section along the path $\gamma$ that satisfies 
 $ \nabla_{\dot \gamma(s)} \vec X(t,x)= -iH_{\ell,1}^{\rm adia}(t,x) \vec X(t,x)$.

\begin{example}
    In the situation of Example~\ref{ex_referee2}, for $\ell\in\{1,2\}$, the operators $\mathcal R_\ell$ are multiplications by $\e^{-itm_{\ell\ell}}$ where we have set $M=(m_{ij})_{1\leq i,j\leq 2}$.
\end{example}

\subsection{Classical quantities and wave-packets propagation}
The classical maps introduced before, in particular the parallel transport  $\mathcal R_{\ell}$, $\ell\in\{1,2\}$, play a  role on 
 the quantum side  in the  propagation of {\it vector-valued wave packets} by systems. By vector-valued wave packets, we mean initial data colinear to fixed direction $\vec V_0$ with scalar part that is a wave packet as in~\eqref{def:WP}: 
 \[
 \psi^\eps_0(x)= \vec V_0 {\rm WP}^\eps_{z_0}\varphi_0, \;\; x\in\R^d,\;\; \varphi_0\in\mathcal S(\R^d),\;\; z_0\in\R^{2d},\;\;\vec V_0\in\C^m.
 \]
 We will say that $\psi^\eps_0$ is {\it polarized} along the direction $\vec V_0$.
 \smallskip 
 
 Under the assumptions made in this text, the propagation of such data is described at leading order by the classical trajectories and the classical quantities associated with them, including the parallel transport. 

\begin{proposition}[Vector-valued wave packets]\label{prop:transport_1}
Let $k\in\N$ and assume that  $H^\eps=H_0+\eps H_1$ satisfies Assumptions~\ref{hyp:codim1} and~\ref{hyp:growthH}.
Let $\ell\in\{1,2\}$ and $(t_0,z_0)\in I\times \R^d$.
Then, 
for any $\varphi_0\in{\mathcal S}(\R^d,\C)$ and $\vec V_0\in  {\rm Ran}\, \pi_\ell(t_0,z_0)$, there exists a 
constant $C>0$ such that
$$
\sup_{t\in I}\left\|{\mathcal U}^\eps_{H}(t,t_0)\, \widehat{\vec V_0}\,  {{\rm WP}}^\eps_{z_0}\varphi_0 -  
{\rm e}^{\frac i\eps S_\ell(t,t_0,z_0)  } \,
\vec V_\ell(t,t_0,z_0)\, {{\rm WP}}^\eps_{\Phi_{h_\ell}^{t,t_0}(z_0)}\varphi(t) \right\|_{\Sigma^k_\eps} \le C\sqrt \eps,
$$
where the profile function $\varphi(t)$ and the vector of polarization $\vec V_\ell(t,t_0) $ are given by
\[
\varphi(t) = {\mathcal M}[F_\ell(t,t_0,z_0)]
\varphi_0,\;\;\mbox{and}\;\; 
\vec V_\ell(t,t_0,z_0) =  {\mathcal R}_\ell \left(t,t_0,\Phi^{t,t_0}_{h_\ell}(z_0)\right)\vec V_0 .
\]
\end{proposition} 

The proof of Proposition~\ref{prop:transport_1} can be found in    \cite[Chapter 14] {corobook} for Hamiltonians without crossings, and in \cite{FLR1} or in Appendix~\ref{app:C} of this memoir for codimension~1 crossings. Recall that the functional spaces $\Sigma^k_\eps$, $k\in\N$, are defined in~\eqref{eq:Sigmak}.

\section{Thawed and frozen Gaussian approximations}

Thawed and frozen Gaussian approximations have been introduced in the 80's in  theoretical chemistry~\cite{HK,Kay1,Kay2}. The frozen one has become popular as the so-called Herman--Kluk approximation.  They rely on the fact that the 
 family of wave packets $(g^\eps_z)_{z\in\R^{2d}}$ forms a continuous frame and 
provides for all square integrable functions $f\in L^2(\R^d)$ the reconstruction formula
\[
f(x)= (2\pi\eps)^{-d} \int_{z\in\R^{2d}}\langle g^\eps_z,f\rangle g^\eps_z (x) dz.
\]
The leading idea is then to write the unitary propagation of general, square integrable initial data $\psi^\eps_0\in L^2(\R^d)$ as 
\[
{\mathcal U}^\eps_H(t,t_0)\psi^\eps_0=(2\pi\eps)^{-d}  \int_{z\in\R^{2d}} \langle  g^\eps_z,\psi^\eps_0\rangle 
\,{\mathcal U}^\eps_H(t,t_0)g^\eps_z\, dz,
\]
and to take advantage of the specific properties of the propagation of Gaussian states to obtain an integral representation that allows in particular for an efficient numerical realization of the propagator.
\smallskip

Such a program has  completely been  accomplished in the scalar case and
the mathematical proof of the convergence of this approximation is more recent~\cite{RS1,R}. It can be easily extended to the adiabatic setting
 (see~\cite{FLR1}).
We recall in the first subsections the scalar  and  adiabatic results and then we explain how we extend this approach to systems presenting smooth crossings via a hopping process. Surface hopping has been popularized in theoretical chemistry by the algorithm 
of the fewest switches (see~\cite{TP}). The first mathematical result proving convergence of such a hopping algorithm has been performed in~\cite{FL1} in the context of conical intersections. Surface hopping has
        been combined with frozen Gaussian propagation 
in various instances, see for example \cite{WH,Lu}. Here, it is the first time that the combination of a Gaussian approximation and a hopping process are achieved in a fully rigorous manner.  All along this section, we use the Gaussian frame defined in~\eqref{def:gepsz}.

\subsection {The scalar case}
Let us assume for a while that $m=1$ and select one of the scalar Hamiltonians $h_\ell(t)$ that we assume to be subquadratic on $I\times \R^{2d}$. Define the operator 
\begin{equation}\label{def:J_thawed_scalar}
\mathcal J_{\ell,{\rm th}} ^{t,t_0}(f) = 
 (2\pi\eps)^{-d}\int_{\R^{2d}}{\rm e}^{\frac{i}{\eps}S_\ell(t,t_0,z)}\<g^\eps_z, f\>   g_{\Phi^{t,t_0}_{h_\ell}(t,z)}^{\Gamma_\ell(t,t_0,z),\eps}dz.
 \end{equation}

 As first proposed in~\cite{HK}, it is also possible to get rid of the time-dependent variance matrices~$\Gamma_\ell$ by introducing the  {\it Herman-Kluk prefactors  for the $\ell$-th modes}, $a_\ell$, defined by 
\begin{equation}\label{def:prefactor}
a_\ell(t,t_0,z) = 2^{-d/2}{\rm det}^{1/2}\left(A_\ell(t,t_0,z)+D_\ell(t,t_0,z)+i(C_\ell(t,t_0,z)-B_\ell(t,t_0,z)\right),
\end{equation}
where $A_\ell(t,t_0,z)$, $B_\ell(t,t_0,z)$, $C_\ell(t,t_0,z)$ and $D_\ell(t,t_0,z)$ are the $d\times d$ matrices associated with the differential of the flow map according to~\eqref{def:F}.
One then sets 
\begin{equation}\label{def:J_frozen_scalar}
 \mathcal J_{\ell,{\rm fr}}^{t,t_0} (f)= 
 (2\pi\eps)^{-d}\int_{\R^{2d}}{\rm e}^{\frac{i}{\eps}S_\ell(t,t_0,z)}\<g^\eps_z, f \>a_\ell (t,t_0,z)  g_{\Phi_{h_\ell}^{t,t_0}(z)}^\eps dz
 \end{equation}
 
It is proved in~\cite{R,RS1} that the operators $\mathcal J_{{\rm th}} ^{t,t_0}$ and $\mathcal J_{{\rm fr}} ^{t,t_0}$ approximate the propagator $\mathcal U_{h_\ell}^{t,t_0}$ according to 
 \[
\left\| \mathcal U^\eps_{h_\ell} (t,t_0)\left(\phi^\eps_0 \right) - \mathcal J_{\ell, {\rm th/fr}} ^{t,t_0}( \phi^\eps_0 ) \right\| _{L^2} \leq C_T \,  \eps\,  \| \phi^\eps_0 \|_{L^2},
 \]
 for all 
 $\phi^\eps_0 \in L^2 (\R^d)$ and for all $t\in[t_0,t_0+T]$. Here, $C_T>0$ is a constant depending on 
 $T>0$ where $t_0,T$ are chosen in $\R$ with $[t_0,T]\subset I$ and $t_0<T$.
 \smallskip 

The arguments we detail in Chapter~\ref{chap:4} give a proof of this result for frequency localized data. However, the result holds as soon as the initial data are uniformly bounded in $L^2(\R^d)$ (see~\cite{RS1,R}).  

\subsection{The  adiabatic situation}
Whenever the eigenvalues are of constant multiplicity, the classical quantities that we have introduced above are enough to construct an approximation of the propagator. For $\ell\in\{1,2\}$, we define the first order thawed Gaussian approximation for the $\ell$-th mode as the operator  $\mathcal J_{\ell, {\rm th}}^{t,t_0}$ defined on functions of the form $\psi=\widehat{\vec  V }f$ for $\vec V$  a smooth given vector-valued function and  $f$ any function in $L^2(\R^d)$,
\begin{equation}\label{def:J_thawed}
\mathcal J_{\ell,\vec V, {\rm th}} ^{t,t_0}(f) = 
 (2\pi\eps)^{-d}\int_{\R^{2d}}{\rm e}^{\frac{i}{\eps}S_\ell(t,t_0,z)}\<g^\eps_z, f\> \vec V_\ell (t,t_0,z)  g_{\Phi^{t,t_0}_{h_\ell}(t,z)}^{\Gamma_\ell(t,t_0,z),\eps}dz,
 \end{equation}
 with 
 \begin{equation}\label{def:Vl}
 \vec V_\ell (t,t_0,z)= \mathcal R_\ell (t,t_0,z) \pi_\ell (t_0,z) \vec V(z).
 \end{equation}
The family of operators $f\mapsto \mathcal J_{\ell,\vec V, {\rm th}} ^{t,t_0}(f)$ has a Schwartz distribution kernel and defines a Fourier integral operator with an explicit complex phase. It is proved in Corollary~\ref{cor:boundedness} that it is  a bounded family in $\mathcal L\left(L^2(\R^d),\Sigma^k_\eps(\R^d)\right)$.

\begin{theorem} [Thawed Gaussian approximation \cite{Kay1,R,RS1,FLR1}]\label{th:thawed_standard}
Assume $h_\ell$ is an eigenvalue of constant multiplicity of a matrix $H^\eps=H_0+\eps H_1$ of subquadratic growth on the time interval~$I$. Let $t_0,T\in\R$ with $[t_0,T]\subset I$. Then, there exists $C_T>0$ such that for all 
 $\phi^\eps_0 \in L^2 (\R^d)$, $\vec V\in\mathcal C^\infty(\R^{2d}, \C^m)$ bounded with bounded derivatives, for all $t\in[t_0,t_0+T]$
 \[
\left\| \mathcal U^\eps_{H} (t,t_0)\left( \widehat {\pi_\ell(t_0) {\vec V} }\phi^\eps_0 \right) - \mathcal J_{\ell, \vec V,{\rm th}} ^{t,t_0}( \phi^\eps_0 ) \right\| _{L^2} \leq C_T \,  \eps\,  \| \phi^\eps_0 \|_{L^2}.
 \]
\end{theorem}

\begin{remark}\label{rem:pluto8}
 \begin{enumerate}
 \item Note that in Theorem~\ref{th:thawed_standard}, the family $(\phi^\eps_0)_{\eps>0}$ is not supposed to be frequency localized. 
 \item 
 The approach  we develop in this text yields as a by-product the convergence of the thawed propagator in the spaces $\Sigma^\eps_k$ for initial data $(\phi^\eps_0)_{\eps>0} $ which are frequency localized  with $k$ satisfying  $N_0>k+d$  ($N_0$ being associated to $(\phi^\eps_0)_{\eps_0} $ by Definition~\ref{def:freq_loc_0}).
 \end{enumerate}
\end{remark}

\medskip

In a similar manner than in the scalar case, one 
 defines the first order frozen Gaussian approximation for the $\ell$-th mode as the operator $\mathcal J_{\ell,\vec V,{\rm fr}}^{t,t_0}$ defined by 
\begin{equation}\label{def:J_frozen}
 \mathcal J_{\ell,\vec V,{\rm fr}}^{t,t_0} (f)= 
 (2\pi\eps)^{-d}\int_{\R^{2d}}{\rm e}^{\frac{i}{\eps}S_\ell(t,t_0,z)}\<g^\eps_z, f \>a_\ell (t,t_0,z) \vec V_\ell (t,t_0,z)  g_{\Phi_{h_\ell}^{t,t_0}(z)}^\eps dz
 \end{equation}
where $\vec V_\ell (t,t_0,z)$ is defined in~\eqref{def:Vl} and $a_\ell(t,t_0,z)$ is the Hermann-Kluk prefactor associated with the $\ell$-th mode as in~\eqref{def:prefactor}.

Here again, this family of operators is bounded in the functional space $\mathcal L\left(L^2(\R^d),\Sigma^k_\eps(\R^d)\right)$ (see Corollary~\ref{cor:boundedness}).
The next result then is a consequence of Theorem~\ref{th:thawed_standard} and Remark~\ref{rem:pluto8} also holds for this statement.

 \begin{theorem}[Frozen Gaussians approximation \cite {Kay1,R,RS1,FLR1}]\label{th:frozen_standard}
 Assume $h_\ell$ is an eigenvalue of constant multiplicity of a matrix $H=H_0+\eps H_1$ of subquadratic growth on the time interval~$I$. Let $t_0,T\in\R$ with $[t_0,T]\subset I$. Then, there exists $C_T>0$ such that for all 
 $\phi^\eps_0\in L^2(\R^d)$, $\vec V\in\mathcal C^\infty(\R^{2d}, \C^m)$ bounded with bounded derivatives, and for all $t\in[t_0,t_0+T]$,
 \[
\left\| \mathcal U^\eps_{H}(t,t_0)  \left( \widehat {\pi_\ell(t_0) {\vec V}} \phi^\eps_0 \right) - \mathcal J_{\ell, \vec V, {\rm fr}} ^{t,t_0}(\phi^\eps_0) \right\| _{L^2} \leq C_T \,  \eps\,  \| \phi^\eps_0\|_{L^2}.
 \]
    \end{theorem}

The terminology {\it thawed/frozen} for these Gaussian approximations was introduced by Heller~\cite{Hel} to  put emphasis on the fact that, on the first case,  the covariance of the matrix was evolving ``naturally''   by following the classical motion, while, on the other one, the covariance is ``frozen'' (constant). The possibility of freezing the covariance matrix was realized by Herman and Kluk (see~\cite{HK})  by computing the kernel of the time dependent propagator.
\smallskip

In the next sections, we present our results and an extension of these statements to systems with crossings. 
The method we develop also allows  to prove the approximations of Theorems~\ref{th:frozen_standard} and~\ref{th:thawed_standard} in the spaces $\Sigma^k_\eps(\R^d)$, with additional assumptions 
on the initial data.

\subsection{Initial value representations for codimension $1$ crossing at order $\sqrt\eps$. }

Our first result consists in an extension of the range of validity of Theorems~\ref{th:thawed_standard} and~\ref{th:frozen_standard} to Hamiltonians presenting smooth crossing and satisfying~\eqref{hyp:codim1} at the prize of a loss in the accuracy of the approximation. 

\begin{theorem}[Leading order thawed/frozen Gaussian approximation]\label{th:sqrteps}
Let $k\in\N$. Assume 
  $H^\eps=H_0+\eps H_1$ satisfies Assumptions~\ref{hyp:smooth_cros} and~\ref{hyp:codim1}  on the interval $I$.
Then, there exist constants $C_{T,k}>0$,  such that for all initial data 
  $\psi^\eps_0=\vec V \phi^\eps_0$ that satisfies Assumption~\ref{hyp:data} with frequency localization index $N_0>k+d$ and constant $C_0$, there exists $\eps_0>0$ such that 
for all  $t\in I$ and  $\eps\in(0,\eps_0]$, we have
 \[
\left\| \mathcal U^\eps_{H}(t,t_0) \left( \widehat{\pi_\ell(t_0){\vec V}}\phi^\eps_0\right) - \mathcal J_{\ell, {\rm th}/{\rm fr}}^{t,t_0}\phi^\eps_0 \right\| _{\Sigma^k_\eps} \leq C_{T,k} \,  \sqrt \eps\,\left(  \| \phi^\eps_0\|_{L^2} + C_0\right).
 \]
\end{theorem}

The remarks below also hold for Theorems~\ref{thm:TGeps}, \ref{thm:FGeps} and~\ref{thm:FGeps_av}. 
\begin{remark}
\begin{enumerate}
\item Of course the result also holds for all initial data of the form
\[
\psi^\eps_0(x)= \widehat {\vec V}  \phi^\eps_0(x)+r^\eps_0(x),\;\; x\in\R^d
\]
when the family $(r^\eps_0)_{\eps>0}$ satisfies  $\| r^\eps_0\|_{L^2(\R^d)}=\O(\eps)$ in $\Sigma^k_\eps$ for the index $k$ considered in the statement.
\item The fact of being frequency localized  with $N_0>k+d+\frac 12$ implies that $(\phi^\eps_0)_{\eps>0}$ is bounded in $\Sigma^k_\eps$ (see Section~\ref{sec:freq_loc_Sigma_k}). Thus, $( \mathcal U^\eps_H(t,t_0) \psi^\eps_0)_{\eps>0}$ also is  bounded in $\Sigma^k_\eps$ and this space is the natural space where studying the  approximation. 
\item The control of the approximation in terms of the initial datum by  $ \| \phi^\eps_0\|_{L^2} + C_0$  instead of $\Vert \phi^\eps_0\Vert_{\Sigma^\eps_k}$ is due to the method of the proof, which has to account for the presence of the crossing. The constant $C_0$ (and the $L^2$-norm) controls the $\Sigma^k_\eps$-norm. 
\end{enumerate}
\end{remark}

The loss of accuracy of the approximation, in $\sqrt\eps$ instead of $\eps$,  is  due to  the presence of the crossing set~$\Upsilon$. Indeed, the crossing  induces transitions between the modes that are exactly of order~$\sqrt\eps$ and cannot be neglected. If the initial datum is frequency localized in  a domain such  that all the classical trajectories issued from its microlocal support at time $t_0$ do not reach the crossing set before the time $t_0+T$, then an estimate in $\eps$ will hold. However, if these trajectories pass through the crossing, some additional terms of order $\sqrt\eps$ have to be added to obtain an approximation at order~$\eps$.
 Let us now introduce the hopping trajectories that we will consider and the branching of classical quantities that we will use above the crossing set for treating these transitions. 
 \smallskip

\subsection{Hopping trajectories and branching process}

Assume $H^\eps=H_0+\eps H_1$ satisfies condition~\ref{hyp:codim1} and~\Cref{hyp:growthH} on the interval $I$.
For considering initial data $\psi^\eps_0=\widehat{\vec V} \phi^\eps_0$ with $(\phi^\eps_0)_{\eps>0}$ frequency localized in a compact set $K\subset B(0,R_0)$, we are going to make  assumptions on the set~$K$.
\smallskip

We consider sets $K$ that  are connected compact  subsets of  $\R^{2d}$  and that do not intersect the crossing set~$\Upsilon$. If one additionally assumes that the  trajectories $\Phi^{t,t_0}_{h_\ell}(z)$ issued from points $z\in K$ intersect $\Upsilon$ on generic crossing points, 
 then, because of their transversality to~$\Upsilon$,  a given
 trajectory~$\Phi_{h_\ell}^{t,t_0}(z)$ issued from $z\in K$  meets~$\Upsilon$ only a finite number of times.
 We then denote by $(t^\flat_\ell(t_0,z),\zeta^\flat_\ell(t_0,z))$ the first crossing  point in $\Upsilon$:
\begin{equation}\label{def:zflat}
\zeta_\ell^\flat(t_0,z)=  \Phi^{t^\flat_\ell(t_0,z), t_0}_{h_\ell} (z).
\end{equation}

 For the $\ell$-th mode and the compact~$K$, we define
\[
t_{\ell,{\rm max}}^\flat(t_0,K) = \max\{t_\ell^\flat(t_0,z),\; z\in K\}\;\;\mbox{and}\;\;  t_{\ell,{\rm min}}^\flat(t_0,K) = \min\{t_\ell^\flat(t_0,z),\; z\in K\}.
\] 
We shall assume that $K$ is well-prepared for the mode $\ell$  in the sense that all trajectories issued from $K$ for this mode~$\ell$ have passed through $\Upsilon$ (if they do) before the ones for the other mode start to reach $\Upsilon$. 

\begin{assumption}[Well-prepared frequency domain]\label{hyp:compact}
Let $K$ be a connected compact  subset $K$ of  $\R^{2d}$ which  does not intersect the crossing set~$\Upsilon$. We say that $K$ is a 
well-prepared frequency domain on the time-interval~$I$ if for $\ell\in\{1,2\}$, the  trajectories $(\Phi^{t,t_0}_{h_\ell} (z))_{t\in I}$ issued from points $z\in K$ intersect $\Upsilon$ on generic crossing points and one has 
\[
t_{1,{\rm max}}^\flat(t_0,K) < t_{2,{\rm min}}^\flat(t_0,K).
\]
\end{assumption}

A space-time crossing point $(t^\flat_\ell (t_0,z),\zeta^\flat _\ell(t_0,z))$ is  characterized by three parameters
\[
\mu^\flat\in\R, \qquad (\alpha^\flat,\beta^\flat)\in\R^{2d}
\]  given by 
\begin{align}\label{def:mu}
&\mu^\flat (t_0,z) =
\partial_t f+\{v,f\}
\left(t^\flat_\ell(t_0,z), \zeta^\flat_\ell(t_0,z)\right),\\
\label{def:alpha_beta}
&\left(\alpha^\flat (t_0,z),\beta^\flat (t_0,z)\right)  = 2\, J \nabla _z f\left(t^\flat_\ell(t_0,z),\zeta^\flat_\ell(t_0,z)\right).
\end{align}

Let us consider a trajectory for the mode $\ell=1$ that reaches the hypersurface~$\Upsilon$.
The hopping process is affected with a {\it transition coefficient }  $\tau_{1,2}(t,t_0, z)$ that restricts the space time variables $(t,z)$ to trajectories that have met the crossing set~$\Upsilon$
\begin{equation}\label{def:tau}
\tau_{1,2}(t,t_0, z)  = \1_{t\geq t^\flat_1(t_0,z)}\sqrt{\frac{\pi}{i\mu^\flat (t_0,z)}}.
\end{equation} 
Note that when $K$ satisfies  Assumption~\ref{hyp:compact}, then if $t<t^\flat_{1,{\rm min}}(K)$ and $z\in K$, then $\tau_{1,2}(t,t_0,z)=0$. Moreover,  if $t\in \left(t^\flat_{1,{\rm max}}(K),t^\flat_{2,{\rm min}}(K) \right)$, $z\mapsto \tau_{1,2} (t,t_0,z)$ is smooth. 
\smallskip

One also considers a {\it matrix of change of polarization}
\begin{equation}\label{def:W12}
W_{1\to 2}(t_0,z)= W_1\!\left(t_1^\flat(t_0,z), \zeta_1^\flat(t_0,z)\right)^*
\end{equation}
with 
 \begin{equation}
 \label{def:W1}
  W_{1} =  \pi_1 H_1 \pi_2 +i\pi_1 \left( \partial_t \pi_1 +\frac 12 \{ h_1+h_2,\pi_1\} \right)\pi_2
  \end{equation}
The matrix $  W_{1\to 2}$
maps ${\rm Ran}(\pi_1(\zeta_1^\flat(t_0,z))$ to   ${\rm Ran}(\pi_2(\zeta_1^\flat(t_0,z))$.
\smallskip

One then introduces {\it hopping trajectories} by setting
\begin{equation}\label{def:traj12}
 \Phi_{1,2}^{t,t_0}( z)= \Phi_{2}^{t,t^\flat_1(t_0,z)}\left(\zeta_1^\flat(t_0,z)\right), \;\;t>t^\flat_1(t_0,z).
 \end{equation}
This trajectory $\left(\Phi_{1,2}^{t,t_0}( z)\right)_{t>t^\flat(t_0,z)}$ is the branch of  a generalized trajectory that has hopped from the mode~$\ell=1$ to the mode~$\ell=2$
at the crossing point $(t^\flat_1(t_0,z),\zeta^\flat_1(t_0,z))$. 
One could define similarly trajectories hopping  from the mode~$\ell=2$ to~$\ell=1$ by exchanging the role of the indices $1$ and $2$. 
\smallskip

Along these trajectories, one defines classical quantites as follows: 
\begin{enumerate}
\item[(a)]  The function $S_{1,2} (t,t_0,z)$ is the action accumulated along the hopping trajectories, i.e. between times $t_0$ and $t^\flat_1(t_0,z)$ on the mode~$\ell=1$ and then on the mode~$\ell=2$
\begin{equation}\label{def:S12}
S_{1,2} (t,t_0,z)  =S_1(t^\flat_1(t_0,z),t_0,z)+ S_{2}(t,t^\flat_1(t_0,z),\zeta^\flat_1(t_0,z)) ,
\end{equation}
\item[(b)] 
The matrix $\Gamma_{1,2}(t,t_0,z)$ is generated according to~\eqref{def:Gamma} for the mode $\ell=2$ along the trajectory
$\left(\Phi_{1,2}(t,t_0, z)\right)_{t>t^\flat_1(t_0,z)}$   starting at time $t_1^\flat =  t^\flat_1(t_0,z)$  from the matrix 
\begin{equation}\label{def:Gammaflat}
\Gamma^\flat(t_0,z)= \Gamma_1(t^\flat_1, t_0,z) - \frac
{(\beta^\flat -\Gamma_1(t^\flat_1, t_0,z) \alpha^\flat)\otimes (\beta^\flat -\Gamma_1(t^\flat_1, t_0,z))}
{2\mu^\flat -\alpha^\flat \cdot \beta^\flat +\alpha^\flat \cdot \Gamma_1(t^\flat_1, t_0,z) \alpha^\flat},
\end{equation}
\item[(c)] The vector 
$\vec V_{1,2}(t,t_0,z)$ is obtained by propagating the vector $\vec V_1\left(t_1^\flat, t_0,z\right)$ for the mode $\ell=2$ along the trajectory
$\left(\Phi_{1,2}(t,t_0, z)\right)_{t>t^\flat(t_0,z)}$  starting at time $t_1^\flat$  from the vector $\pi_2(t_1^\flat,\zeta_1^\flat) \vec V_1\left(t_1^\flat, t_0,z\right)$ with $\zeta_1^\flat = \zeta_1^\flat(t_0,z)$. One has
\begin{equation}\label{def:vecV12}
\vec V_{1,2}(t,t_0,z) = \mathcal R_2 (t,t_1^\flat,\zeta_1^\flat)\pi_2\left(t_1^\flat,\zeta_1^\flat\right) W_{1\rightarrow 2}(t_0,z)\vec V_1\left(t_1^\flat, t_0,z\right).
\end{equation}
\item[(d)]  The matrices $F_{1,2}(t,t_0,z)$ are associated with the flow maps 
\begin{equation}
F_{1,2}(t,t_0,z) =\partial_z \Phi^{t,t_0}_{1,2}(z)= 
\begin{pmatrix}
A_{1,2}(t,t_0,z) & B_{1,2}(t,t_0,z)\\
C_{1,2}(t,t_0,z) & D_{1,2}(t,t_0,z)
\end{pmatrix}.
\end{equation}
\item[(e)]  The transitional Herman--Kluk prefactors
 depend on $\Gamma_{1,2}(t,t_0,z)$ and $ \tau _{1,2}(t,t_0,z)$ according to 
\begin{align}
\label{eq:prefactor2}
a_{1,2}&= \tau _{1,2} \frac
{{\rm det}^{1/2} (C_{1,2}-iD_{1,2} -i(A_{1,2}-iB_{1,2}))}
{ {\rm det}^{1/2} (C_{1,2}-iD_{1,2}-\Gamma_{1,2}(A_{1,2}-iB_{1,2}))}
\\
 \nonumber
&= \tau _{1,2} \frac
{{\rm det}^{1/2} (A_{1,2}+D_{1,2} +i(C_{1,2}-B_{1,2}))}
{ {\rm det}^{1/2} (D_{1,2}+iC_{1,2}-i\Gamma_{1,2}(A_{1,2}-iB_{1,2}))}
\end{align}
where we have omitted to mark the dependence on $(t,t_0,z)$ for readability. 
\end{enumerate}
\smallskip 

\begin{example}
    With the notations of Example~\ref{ex_referee2}, if $q_{1}(t^\flat)=0$ with $p^\flat:=p_1(t^\flat)\not=0$, we have a generic crossing points with transition parameters
    \[
    \mu^\flat=p^\flat,\;(\alpha^\flat,\beta^\flat)=(0,-w(0))
    ,\;\;\tau_{1,2}(t,0,z)=\1_{t\geq t_1^\flat(0,z)}\sqrt{\frac{\pi}{ip^\flat}},\;\; W_{1\to 2}(t_0,z)= \begin{pmatrix} 0 & 0\\ \overline m_{12}& 0\end{pmatrix}
    \]
    where we have set  $M=(m_{ij})_{1\leq i,j\leq 2}$ and used that  $\pi_1$ and $\pi_2$ are the orthogonal  projectors on the vectors  $(1,0)$ and $(0,1)$ respectively.
\end{example}

With these quantities in hands,  we can define the correction terms of order   $\sqrt\eps$ of the thawed and the frozen approximations and state our main results.

\subsection{Thawed Gaussian approximation at order $\eps$} \label{sec:notation_thaw}
With the notations of the preceding section, one defines the {\it thawed Gaussian correction term  for the mode $\ell=1$} as
\begin{equation}\label{def:J12th}
\mathcal J_{1\to 2,\vec V , {\rm th}}^{t,t_0} (f) =
(2\pi\eps)^{-d}\int_{z\in K}\tau _{1,2}(t,t_0,z)
{\rm e}^{\frac{i}{\eps}S_{1,2}(t,t_0,z)}  \langle g_z^\eps, f\rangle  \vec V_{1,2}(t,t_0,z)  g_{\Phi_{1,2}^{t,t_0}(z)}^{\Gamma_{1,2}(t,t_0,z),\eps}
dz
\end{equation}

The formula~\eqref{def:J12th}  defines a family of operators that is bounded in $\mathcal L(L^2(\R^d), \Sigma^k_\eps(\R^d))$ (see Corollary~\ref{cor:boundedness}). 
The restriction $t>t_1^\flat(t_0,z)$ introduces a localization of the domain of integration on one side of the hypersurface $\{t=t_1^\flat(t_0,z)\}$. 
 \smallskip 
 
The  thawed Gaussian correction term  for the mode $\ell=2$, denoted by $\mathcal J_{2\to 1, \vec V, {\rm th}}^{t,t_0}$ would be defined by exchanging the roles of the indices $1$ and $2$.  These correction terms allow to ameliorate the accuracy of the thawed Gaussian approximation and to obtain an approximation at order~$\eps$. 

\begin{theorem}[Thawed Gaussian approximation with hopping trajectories]\label{thm:TGeps}
Let $k\in\N$. Assume 
  $H^\eps=H_0+\eps H_1$  satisfies Assumption~\ref{hyp:smooth_cros} and s~\ref{hyp:codim1} 
  on the interval $I$.
    Let $K$  be a compact satisfying Assumption~\ref{hyp:compact}. Let $\varsigma\in(0,1)$.
Then, there exists constants $C_{T,k,K,\varsigma}>0$,  such that for all initial datum
  $\psi^\eps_0=\vec V \phi^\eps_0$ that satisfies Assumption~\ref{hyp:data} 
    with frequency localization index $N_0>k+d$ and constants $C_0$, $R_0
    $ with $K\subset B(0,R_0)$, there exists $\eps_0>0$ such that 
for all  $t\in I$
we have for $\eps\in(0,\eps_0]$,
 \begin{align*}
\Bigl\| \mathcal U^\eps_H(t,t_0) \psi^\eps_0 
&  -
 \mathcal J_{1, \vec V, {\rm th}}^{t,t_0} \left(\phi^\eps_0\right) -  \mathcal J_{2, \vec V, {\rm th}}^{t,t_0} \left(\phi^\eps_0\right)
 -\sqrt\eps \mathcal J_{1\to 2, \vec V, {\rm th}}^{t,t_0} \left( \phi^\eps_0\right) 
\Bigr\|_{\Sigma^k_\eps} 
\leq  \eps^{1-\varsigma} \, C_{T,k,K,\varsigma}
\left( C_0 +  \| \phi^\eps_0\|_{L^2} \right).
\end{align*}
\end{theorem}

This result emphasizes  that for systems with smooth crossings, a term of order $\sqrt\eps$ is generated by the crossing. 
\smallskip

\subsection{Frozen Gaussian approximation at order $\eps$}

In order to freeze the covariance of the Gaussians $\Gamma_{1,2}(t,t_0,z)$ that appear in the formula of the thawed Gaussian correction term~\eqref{def:J12th}, we use  the correction prefactor $a_{1,2}$  introduced in~\eqref{eq:prefactor2}
and define the frozen Gaussian correction term for the mode $\ell=1$  as
\begin{equation}\label{def:J12fr}
\mathcal J_{1, 2, \vec V, {\rm fr}}^{t,t_0} ( f) =
(2\pi\eps)^{-d}\int_{z\in K} a _{1,2}(t,t_0,z)
{\rm e}^{\frac{i}{\eps}S_{1,2}(t,t_0,z)}  \langle g_z^\eps, f \rangle  \vec V_{1,2}(t,t_0,z)  g^\eps_{\Phi_{1,2}^{t,t_0}(z)}
dz.
\end{equation}
The formula~\eqref{def:J12fr}  defines a Fourier-integral operator with a complex phase associated with the canonical transformations $\Phi_{1,2}^{t,t_0}$ associated with the hopping flow. It is a family of operators that is bounded in $\mathcal L(L^2(\R^d), \Sigma^k_\eps(\R^d))$ by Corollary~\ref{cor:boundedness} below. 
\smallskip 

We first state a point-wise approximation.

\begin{theorem}[Point-wise time frozen Gaussian approximation with hopping trajectories]
\label{thm:FGeps}
$\;$
Let $k\in\N$. Assume 
  $H^\eps=H_0+\eps H_1$  satisfies Assumptions~\ref{hyp:smooth_cros} and~\ref{hyp:codim1} 
  on the interval $I$.
  Let $K$ be a compact satisfying Assumption~\ref{hyp:compact}. Let $\varsigma\in (0,1)$.
Then, there exists constants $C_{T,k,K,\varsigma}>0$,  such that for all initial datum
  $\psi^\eps_0=\vec V \phi^\eps_0$ that satisfies Assumption~\ref{hyp:data} 
    with frequency localization index $N_0>k+d$ and constant $C_0$, $R_0$ with $K\subset B(0,R_0)$, there exists $\eps_0>0$ such that 
for all  $t\in I$ satisfying
\[
t<t_{1,\rm min}^\flat(t_0,K)\;\;\mbox{or}\;\; t^\flat_{1,{\rm max}}(t_0,K)\leq t < t^\flat_{2,{\rm min}}(t_0,K),
\]
we have for $\eps\in(0,\eps_0]$,
 \begin{align*}
\biggl\| \mathcal U^\eps_H(t,t_0) \psi^\eps_0 
  -\mathcal J_{1,\vec V, {\rm fr}}^{t,t_0}& \left( \phi^\eps_0\right) 
-  \mathcal J_{2, \vec V, {\rm fr}}^{t,t_0} \left( \phi^\eps_0\right)
-\sqrt\eps \,\mathcal J_{1\to 2,  \vec V, {\rm fr}}^{t,t_0} \left(\phi^\eps_0\right)
\Bigr\|_{\Sigma^k_\eps} 
\leq \eps^{1-\varsigma}\, C_{T,k,K,\varsigma}\,  
\left(  \| \phi^\eps_0\|_{L^2} +C_0\right) .
\end{align*}
\end{theorem}

The proof of Theorem~\ref{thm:FGeps} is based on integration by parts and requires differentiability. 
When $t<t^\flat_{1,{\rm min}}(K)$, then the transfer coefficient $\tau_{1,2}(t,t_0,z)=0$ for all $z\in K$. 
Moreover, on the interval $ [t^\flat_{1,{\rm max}}(K),t^\flat_{2,{\rm min}}(K))$, then $z\mapsto \tau_{1,2} (t,t_0,z)$ is smooth. It is for that reason, that we have to restrict the time validity of the approximation. 
\smallskip

Averaging in time allows to overcome this difficulty and to obtain an approximation result that holds almost everywhere on intervals of time such that  the classical trajectories issued from~$K$ and associated with the level $\ell=2$ have not yet reached $\Upsilon$.

\begin{theorem}[Time averaged frozen Gaussian approximation with hopping trajectories]\label{thm:FGeps_av}
Let $k\in\N$. Assume 
  $H^\eps=H_0+\eps H_1$ is of subquadratic growth  and satisfies Assumptions~\ref{hyp:smooth_cros} and~\ref{hyp:codim1} 
  on the interval $I$.
    Let $K$  be a compact satisfying Assumption~\ref{hyp:compact}. Let $\varsigma\in(0,1)$.
Then, there exists constants $C_{T,k,K,\varsigma}>0$,  such that for all initial datum 
  $\psi^\eps_0=\vec V \phi^\eps_0$ that satisfies Assumption~\ref{hyp:data}
    in~$K$ with frequency localization index $N_0>k+d$ and constant $C_0$, there exists $\eps_0>0$  such that 
for all $\chi\in\mathcal C^\infty_0\left(\left(t_0,  t^\flat_{2,{\rm min}}(t_0,K)\right)\right)$,    
 \begin{align*}
&\biggl\| \int_\R \chi(t) \biggl(\mathcal U^\eps_H(t,t_0) \psi^\eps_0 
  -\mathcal J_{1,\vec V, {\rm fr}}^{t,t_0} \left( \phi^\eps_0\right) 
-  \mathcal J_{2,  \vec V,{\rm fr}}^{t,t_0} \,\left( \phi^\eps_0\right)
 -\sqrt\eps \mathcal J_{1\to 2, \vec V, {\rm fr}}^{t,t_0} \left(\phi^\eps_0\right)\biggr) dt
\Bigr\|_{\Sigma^k_\eps} \\
&\qquad\qquad \qquad\qquad \leq  \eps^{1-\varsigma}\,  C_{T,k,K,\varsigma}\,   \|\chi\|_{L^\infty}\,
\left( \| \phi^\eps_0\|_{L^2} +C_0\right).
\end{align*}
\end{theorem}

To go beyond the time~$ t^\flat_{2,{\rm min}}(t_0,K)$, one has to consider new transitions that would now go from the level~$\ell=2$ to the level~$\ell=1$, each time a trajectory for the level~$\ell=2$ hits~$\Upsilon$.  The process can be understood as a random walk: each time a trajectory passes through~$\Upsilon$ a new trajectory arises on the other mode with a transition rate of order $\sqrt\eps$.
\smallskip 

The averaging in time can be understood as the result of a non-pointwise observation, that takes place over some time interval, that might even be a short one. 
\smallskip

For proving Theorems \ref{thm:FGeps} and \ref{thm:FGeps_av} we use an accurate analysis for the propagation of individual wave-packets. We prove that a Gaussian wave-packet stays a generalized  Gaussian wave packet  modulo an error term of order $\eps^\mu$,  for any $\mu\in\N$, and in any space $\Sigma^k_\eps$, $k\in\N$, as well before, or after, hitting the crossing hypersurface $\Upsilon$.
\smallskip

The proof then consists first in proving the thawed approximations and in  deriving the frozen approximation from the thawed one. The arguments developed in  Section~\ref{sec:th_to_fr} will show that one can ``freeze'' the Gaussian on any state $g^{\Gamma_0,\eps}_z$. The choice of some $\Gamma_0$ instead of $i\1_{d}$ will imply a slight modification of the definition of the Herman--Kluk prefactors $a_\ell$, and the 
transitional ones $a_{\ell,\ell'}$, $\ell,\ell'\in\{1,2\}$.


\section{Wave packets propagation at any order through  generic smooth crossings}

Our results crucially rely on the analysis of the propagation of wave-packets (including the ones with Gaussian amplitude functions) through smooth crossings.
We consider  a Hamiltonian~$H^\eps=H_0+\eps H_1$ that satisfies Assumption~\ref{hyp:growthH} on the time interval $I$ and presents a smooth crossing on a set ~$\Upsilon$. We  fix a point $z_0=(q_0,p_0)\notin \Upsilon$ and  times $t_0, T$ such that 
\begin{equation}\label{eq:times}
t_0<t^ \flat_1(t_0,z_0)<t_0+T< t_2^\flat(t_0,z_0).
\end{equation}
We assume that the point $z^\flat_1:=\Phi^{t_1^\flat(t_0,z_0),t_0}_{h_1}(z_0)$ is a non degenerate generic crossing point (see Definition~\ref{def:smooth_cros}) that we denote by $\Phi^{t_1^\flat,t_0}_{h_1}(z_0)$ for simplicity. We use the notations introduced in Section~\ref{sec:notation_thaw} (namely equations~\eqref{def:zflat}, ~\eqref{def:mu}, ~\eqref{def:alpha_beta}, ~\eqref{def:traj12}  and \eqref{def:S12}).
For $\ell\in\{1,2\}$, we set
\[
z_\ell(t)= \Phi^{t,t_0}_{h_\ell}(z_0).
\]

  \begin{theorem}\label{th:WPmain}
Let $(t_0,z_0)$ satisfy~\eqref{eq:times}, let $\psi^\eps_0$ be a polarized wave packet
\begin{equation}\label{def:initial_data}
\psi^\eps_0= \vec V_0\wp^\eps_{z_0} (f_0)\;\;\mbox{with} \;\;f_0\in{\mathcal S}(\R^d)\;\;\mbox{and}\;\;\vec V_0\in\C^m.
\end{equation}
Let 
$\psi^\eps(t)$ be the solution of  \eqref{eq:sch} with initial datum $\psi^\eps_0$.
There exist $\kappa_0\in\N$ and  three  families of  differential operators   $\left(\Vec B_{\ell,j}(t) \right)_{j\in\N} $, $\ell\in\{1,2\}$  and $\left(\Vec B_{1\rightarrow 2,j}(t) \right)_{j\in\N} $ such that setting 
for $\delta>0$ and $t\in I_\delta =[t_0,t_1^\flat(t_0,z_0)-\delta]\cup[t_1^\flat(t_0,z_0) +\delta, t_0+T]$
 \begin{align*} 
&\psi_1^{\eps,N}(t) =   {\rm e}^{\frac i\eps S_{1}(t,t_0,z_0)}
\wp^\eps_{z_1(t)}\left(f^\eps_1(t)\right), \\
&\psi_2^{\eps,N}(t)  =   {\rm e}^{\frac i\eps S_{2}(t,t_0,z_0)}
\wp^\eps_{z_2(t)}\left(f^\eps_2(t)\right)
+{\1}_{t>t^\flat}  {\rm e}^{\frac i\eps S_{1,2}(t,t_0,z_0) }
\wp^\eps_{\Phi_{1,2}(t,t_0,z_0)}\left(f^\eps_{1\rightarrow 2}(t)\right),
\end{align*}
with 
\begin{align*}
f^\eps_\ell(t) & = 
{\mathcal R}_\ell (t,t_0) \, {\mathcal M}[F_\ell( t,t_0,z_0)]  \sum_{0 \leq j\leq N}\eps^{j/2}\,\Vec B_{\ell,j}(t) f_0 ,\;\;\ell\in\{1,2\},\\
f^\eps_{1\rightarrow 2}(t)&={\mathcal R}_2 (t,t^\flat_1) \, {\mathcal M}[F_2( t,t^\flat_1,z^\flat_1)]  \sum_{1 \leq j\leq N}\eps^{j/2}\,\Vec B_{1\rightarrow 2,j}(t) f_0,
\end{align*}
 one has the following property: 
 for all $k,N,M\in\N$, there exists   $C_{M,N,k}>0$ such  for all  $t\in I_\delta$
\[
\left\Vert\psi^\eps(t) - \left(\psi^{\eps,N}_1(t)+\psi^{\eps,N}_2(t) \right)\right \Vert_{\Sigma^k_\eps}
  \leq C_{M,N,k}\left(\left(\frac{\sqrt\eps}{\delta}\right)^{N+1}\delta^{-2\kappa_0-k} +\delta^M\right).
  \]
 Moreover, 
 the operators $\vec B_{\ell, j}(t)$   are  differential operators of degree $\leq 3j$  with time dependent smooth vector-valued coefficients and  satisfy for $\ell\in\{1,2\}$,
  \begin{align}
\label{B0} &  \Vec B_{\ell,0}(t) ={\pi_\ell(t_0,z_0)}\Vec V_0\;\;\mbox{ and}\;\; \Vec B_{\ell,j}(t_0) =0\;\;\forall j\geq 1,\\
\label{B1}& \Vec B_{\ell ,1}(t)= \biggl(\sum_{\vert\alpha\vert=3} \frac{1}{\alpha!} \frac{1}{i} \int_{t_0}^{t} \partial_z^\alpha h_\ell(s,z_\ell(s))\,
\op_1^w[(F_\ell (s,t_0,z_0)z)^{\alpha}] \,ds\,\1_{m}\\
\nonumber
&\qquad\qquad   + \frac 1i\int_{t_0}^t
\mathcal R(t_0,s)\, 
\nabla_z H^{\rm adia}_{\ell,1}(s,z_s)\cdot 
{\rm op}^w_1(F_\ell (s,t_0,z_0)z)\,
\mathcal R(s,t_0)\, ds \biggr) 
{\pi_\ell(t_0,z_0 )}\Vec V_0  \\
\nonumber
& \qquad\qquad +\nabla_z \pi_\ell(t_0,z_0)\cdot \widehat z\, \vec V_0, \\
\label{B12}& \Vec   B_{{1 \rightarrow 2},0} (t)=W_1(t^\flat_1, \zeta_1^\flat)^*\, \mathcal T^\flat_{1\rightarrow 2} \,   \mathcal M [F_1( t^\flat_1 ,t_0,z_1^\flat)] \pi_1(t_0,z_0) \Vec V_0 +\O(\delta)
\end{align}
as operator acting on the Schwartz class, and where the scalar transfer operator $\mathcal T^\flat_{1\rightarrow 2}$ is defined by
\beq\label{transf1}
 {\mathcal T}^\flat _{1\rightarrow 2}\varphi(y) =
  \int_{-\infty}^{+\infty}{\rm e}^{-i(\mu^\flat +\alpha^\flat\cdot\beta^\flat/2)s^2} {\rm e}^{is\beta^\flat\cdot y}\varphi(y-s\alpha^\flat)ds,\;\;\forall \varphi\in\mathcal S(\R^d) 
 \eeq 
 and the  transfer matrix $W_1(t^\flat, \zeta^\flat)$ is given by~\eqref{def:W1}.
   \end{theorem}
   
In other words, Theorem~\ref{th:WPmain} says that if $\psi^\eps_0$ is a polarized wave packet, then, for $t\in I$, $t\not=t^\flat_1(t_0,z_0)$,  the solution~$\psi^\eps(t)$ of~\eqref{eq:sch} is asymptotic at any order to a sum of wave packets. Indeed, if $n\in\N$ is fixed, choosing $\delta=\eps^{1/2-c}$, $0<c<1/2$ and $M, N$ large enough will give an approximation in $\O(\eps^n)$. This fact is crucial for the proofs of Theorems~\ref{thm:TGeps}, \ref{thm:FGeps} and~\ref{thm:FGeps_av}.
\smallskip

The polarization of the wave packets $\psi_\ell^{\eps,N}(t)$ is first described by the vectors $ \Vec B_{\ell,j}(t_0)$ that evolves through ${\mathcal R}_\ell(t,t_0) \, {\mathcal M}[F_\ell( t,t_0)] $. Such evolution preserves the eigenmode. Secondly, in $\psi_2^{\eps,N}(t)$, one sees a $\sqrt\eps$ contribution that comes from  a transfer from the mode~$1$ to the mode~$2$ with a change of polarization  performed by the matrix $W_1^*$. This operator $\vec B_{1\to 2,0}(t)$ coincides with a simple transfer operator given in~\eqref{B12} up to a term of order~$\delta$.  Higher order terms in $\epsilon$ due to the transfer are more intricate.
  \medskip
   
 Theorem~\ref{th:WPmain} was proved in~\cite{FLR1} up to order $o(\eps)$. The notations are compatible. However, 
 in~\cite{FLR1}, a  coefficient~$\gamma^\flat $ appears in the definition of the transfer operator. It corresponds to a normalization process that we avoid here by using the projector $\pi_2(t^\flat_1,\zeta_1^\flat)$ instead of taking the scalar product  with a normalized eigenvector.  
\smallskip

For proving the initial value representations of $\mathcal U^\eps_{H}$ of Theorems~\ref{th:sqrteps} to~\ref{thm:FGeps_av}, we shall use two consequences of Theorem~\ref{th:WPmain}:
   \begin{itemize}
   \item[(i)] the wave packet structure up to any order in $\eps$ of $\mathcal U^\eps_{H} \vec V_0\wp^\eps_{z_0} (g^{i\1_d})$,
   \item[(ii)] the exact value of the action of $\Vec B_{1, 0}$,  $\Vec B_{2, 0}$, and  of $\Vec   B_{1 \rightarrow 2} (t)$ when $f_0$ is the Gaussian~$g^{i\1_d}$. 
   \end{itemize}
We recall that the action of  the operators ${\mathcal R}_\ell(t,s) \, {\mathcal M}[F_\ell( t,s)]   $ on focalized Gaussians preserves the Gaussian structure and the focalization: in view of~\eqref{eq:action_Gaussian} and~\eqref{def:Vl},
\[
{\mathcal R}_\ell(t,s) \,  {\mathcal M}[F_\ell(t,s,z)] \pi_\ell(s,z) \vec V_0 g^{i\1_{\R^d}}= \vec V_\ell(t,s,z) g^{\Gamma_\ell(t,s,z)}
\]
where $\vec V_\ell\in {\rm Ran}(\pi_\ell(t,\Phi^{t,s}_\ell(z)))$ and the matrix $\Gamma_\ell (t,s,z)$ is given by~\eqref{def:Gamma} with $\Gamma_0=i\1_{d}$. 
Besides, regarding the transfer term, 
  with the notations of  Corollary~3.9 of~\cite{FLR1} and those of~\eqref{def:mu}, ~\eqref{def:alpha_beta},  \eqref{def:Gammaflat} and~\eqref{def:vecV12}
   \begin{align*}\label{B12_Gaussian}
 & \Vec   B_{1 \rightarrow 2,0} (t^\flat_1) g^{i\1_d} = \sqrt{ \frac{\pi}{ i\mu^\flat(t_0,z_0)}}\, 
\vec V_{1,2}(t_1^\flat,t_0,z_0) g^{\Gamma^\flat(t_1^\flat,t_0,z_0)}.
   \end{align*}

   These elements may enlighten the construction of the operator 
   $\mathcal J_{1,2,\vec V, {\rm fr/th}}^{t,t_0} $.

\section{Wave packet propagation, an example}

In this section, we give a direct proof of Theorem~\ref{th:WPmain} in a simple special case. It shows how the transfer terms of order $\sqrt\eps$ arise.
We assume that $d=1$ and we consider the Hamiltonian
\[
H(x,\xi)= \xi \1_{\C^2} + x R_{\theta( x)}
\]
where the $2$ by $2$ matrix $R_{\theta(x)}$ is the $x$-dependent matrix 
\[
R_{\theta(x)}=\begin{pmatrix} \cos \theta (x) & \sin\theta(x) \\ \sin\theta(x) & -\cos\theta (x)\end{pmatrix},\;\;\theta\in\mathcal C^\infty(\R,\R),
\]
and $x\mapsto \theta(x)$ has bounded derivatives. 
The Hamiltonian $H$ presents a smooth crossing along $\Upsilon=\{x=0\}$ according to Definition~\ref{def:smooth_cros} and satisfies Assumptions~\ref{hyp:growthH}. Moreover, we have 
\[
v=\xi,\;\; f=x,\;\; h_1=\xi+x,\;\; h_2=\xi-x.
\]
The projector $\pi_1(x)$ is the orthogonal projection on the vector $
\vec e(x)=\Bigl(\cos \left({\theta(x)}/2 \right), \,\sin \left( {\theta( x)}/2\right)\Bigr)$,
and the projector $\pi_2(x)$ on the vector $\vec e(x)^\perp = \left(-\sin \left({\theta(x)}/2 \right), \,\cos \left( {\theta( x)}/2\right)\right)$. We write
\[
\pi_1(x)=|\vec e(x)\rangle \,\langle \vec e(x)| \;\;\mbox{and}\;\;\pi_2(x)=\pi_1(x)^\perp=|\vec e(x)^\perp\rangle\, \langle \vec e(x)^\perp|.
\]

We start by a rough diagonalisation of the system, which shows the role of the matrix~$W_1$ defined in~\eqref{def:W1}.
Then, we will fix the function $\theta$ and  prove Theorem~\ref{th:WPmain} with $N=1$ in this case. The case $N=1$ gives a good understanding of the mechanism of the transfer, even though
the approximation of the unitary propagator by an initial value representation  
 requires the result for any $N\in \N$. 

\subsection{Rough diagonalisation and the role of the matrix $W_1$}

This matrix $W_1$ appears as soon as one aims at decoupling the equations according to the modes.
\begin{lemma} The functions
\[
\psi^\eps_\ell(t,x)=\pi_\ell(x)\psi^\eps(t,x),\;\; (t,x)\in\R\times\R,\;\; \ell \in\{1,2\},
\]
satisfy the system 
 \begin{equation}\label{syst_diag}
 \left\{
 \begin{array}l
i\eps \partial_t\psi^\eps_1(t,x) = ((\eps D_x + x)\1_{\C^2}+\eps \Omega(x)) \psi^\eps_1(t,x) + \eps W_1(x) \psi^\eps_2(t,x) ,\\
i\eps \partial_t\psi^\eps_2(t,x) = ((\eps D_x - x)\1_{\C^2}+\eps \Omega(x) ) \psi^\eps_2(t,x) + \eps W_1(x)^* \psi^\eps_1(t,x).
\end{array}
\right.
\end{equation}
with 
\[
\Omega(x):= i \pi_2(x) \pi'_1(x)\pi_1(x) -i  \pi_1(x) \pi'_1(x)\pi_2(x),\;\; W_1(x):= i\pi_1(x) \pi_1'(x)\pi_2(x),\;\;x\in\R.
\]
\end{lemma}

The arguments developed in Part~2 for proving Theorem~\ref{th:WPmain} in full generality are based on such a `rough' diagonalisation explained in Chapter~\ref{chap:2} (see Section~\ref{sec:rough_diag}). 
The matrix-valued term $\Omega(x)$ is also  the first term obtained far from the crossing when one uses superadiabatic projectors as  in Section~\ref{sec:superadiab}. The matrix $W_1$ is a key element of the transfer between the modes. In particular, if the function $\theta$ is constant, $W_1=0$, the equations are decoupled and there is no transfer. The transfer is mainly due  to the 
variations of the eigenprojectors.

\begin{proof}
By projecting the equations, we obtain 
\begin{align*}
\left\{
\begin{array}l
i\eps \psi^\eps_1(t,x) = (\eps D_x + x+i\eps \pi_2(x) \pi'_1(x)\pi_1(x) ) \psi^\eps_1(t,x) + i\eps \pi_1(x) \pi'_1(x) \pi_2(x)  \psi^\eps_2(t,x) ,\\
i\eps \psi^\eps_2(t,x) = (\eps D_x - x+i\eps \pi_1(x) \pi'_2(x)\pi_2(x) ) \psi^\eps_2(t,x) + i\eps \pi_2(x) \pi'_2(x) \pi_1(x)  \psi^\eps_1(t,x).\end{array}
\right.
\end{align*}
In the computation above, we have used 
\[
\pi'_1(x)= \pi_1(x) \pi'_1(x) \pi_2(x) + \pi_2(x) \pi'_1(x) \pi_1(x) ,\;\; x\in\R.
\]
In order to view this system as governed by a self adjoint operator,  we take advantage of the fact that $\pi_2(x)\pi_1(x)=0$ to replace the matrix 
$i \pi_2(x) \pi'_1(x)\pi_1(x) $ by the self-adjoint matrix $\Omega(x)$. 
Similarly, on the other mode, we replace  $i \pi_1(x) \pi'_2(x)\pi_2(x) $ by $ i \pi_1(x) \pi'_2(x)\pi_2(x) -i\pi_2(x) \pi'_2(x)\pi_1(x)$.  
In view of $\pi_2(x)=1-\pi_1(x)$, we have 
\begin{align*}
& i \pi_1(x) \pi'_2(x)\pi_2(x) -i\pi_2(x) \pi'_2(x)\pi_1(x) = -i \pi_1(x) \pi'_1(x)\pi_2(x) +i\pi_2(x)\pi'_1(x)\pi_1(x)=\Omega(x),\\
& i\pi_2(x) \pi'_2(x) \pi_1(x)= - i\pi_2(x) \pi'_1(x) \pi_1(x)= W_1(x)^*,
\end{align*}
which gives the result.
\end{proof}

\subsection{Propagation of wave packets}

We now assume  $\theta(x)=x$ for all $x\in\R$.
 We fix $\tau>0$ and choose an interval of time $I=[t_0,t_0+T]$ given by 
\[
t_0=-\tau, \;\;t_0+ T=\tau.\]
We are going to use 
the 
 classical trajectories 
 \[
 z_1(t):=\Phi_{h_1}^{t,-\tau}(-\tau,\tau )= \Phi_{h_1}^{t,0}(0,0)=(t,-t)\;\;\mbox{ and}\;\; z_2(t):=\Phi_{h_2}^{t,-\tau}(-\tau ,-\tau)=\Phi_{h_2}^{t,0}(0,0)=(t,t).
 \]
 These trajectories meet on $\Upsilon$ at $(t^\flat,z^\flat)$ with 
\[
t^\flat=0,\;\; z^\flat=(0,0).
\]
As a consequence, the initial data is chosen as  
\[
\psi^\eps_1(0)=
\wp_{z_1(-\tau)} (f_1) \vec e(-\tau)
,\;\;\mbox{and}\;\;
\psi^\eps_2(0)=
\wp_{z_2(-\tau)} (f_2) \vec e(-\tau)^\perp,
\]
with $f_1,f_2\in \mathcal S(\R)$.
 The actions  associated  with $z_1(t)$ and $z_2(t)$ respectively and given for $t\in [-\tau,\tau]$ by 
\[
S_1(t)= \frac {t^2}2 -\frac{\tau^2}2 \;\;\mbox{and}\;\; S_2(t)=-\frac{t^2}2+\frac{\tau^2}2.
\]
In the statement of Theorem~\ref{th:WPmain}, several  quantities are of importance and have simple form for this example:
\begin{enumerate}
    \item 
The {\it metaplectic operator} associated with the classical trajectories are identity:
\[
\mathcal M[F_\ell(t,0, z^\flat)]=\1_{L^2(\R)},\;\; \ell=\{1,2\},\;\; t\in [-\tau,\tau].
\]
\item 
The {\it parallel transports} of the vectors $\vec e(t)$ (resp. $\vec e(t)^\perp$) along $z_1(t)$ (resp. $z_2(t)$) are simple:
\[ 
\mathcal R_1 (t,0, z^\flat)\, \vec e(0)= \vec e(t) \;\;\mbox{and}\;\; 
\mathcal R_2 (t,0, z^\flat) \,\vec e(0)^\perp= \vec e(t) ^\perp.
\]
Indeed, we have for all $t\in[-\tau,\tau]$
\[
i\partial_t  \vec e(t)= \Omega(t) \vec e(t)=\vec  e(t)^\perp\;\;
\mbox{and}\;\;
i\partial_t \vec  e(t)^\perp= \Omega(t) \vec e(t)^\perp= -\vec e(t).
\]
\item The {\it correction terms} inside each mode is given by vector-valued multiplication operators:
\[
\vec B_{1,1}(t,y):= \frac yi \int_{-\tau}^t \Omega'(s) \vec  e(s) \, ds\;\;\mbox{and}\;\; \vec B_{2,1}(t,y):= \frac yi \int_{-\tau}^t \Omega'(s) \vec e(s) ^\perp\, ds ,\;\; \ell\in\{1,2\}\;\; y\in\R.
\]
\item The {\it transfer operators} $\mathcal T_{1\to 2}$ and $\mathcal T_{2\to 1}$ are  given by multiplication operators (see the analysis of Section~\ref{sec:trsf_Gauss}): 
\[
\mathcal T_{1\to 2} =
\int_{-\infty}^{+\infty} \e^{is^2-2isy}  ds = \sqrt{i\pi}\; \e^{-iy^2}\;\;\mbox{and}\;\; \mathcal T_{2\to 1} =
\int_{-\infty}^{+\infty} \e^{-is^2+2isy}ds= \sqrt{-i\pi} \;\e^{iy^2}.
\]
In particular,
\[
\mu^\flat=2, \;\;\alpha^\flat=0 \;\;\mbox{and}\;\; \beta^\flat = -2,
\]
according to~\eqref{def:mu} and~\eqref{def:alpha_beta}.
Moreover, 
\[
W_1(t)^* \vec e(t)=-i \vec e(t)^\perp\;\;\mbox{and}\;\; W_1(t) \vec e(t)^\perp=i\vec  e(t),\;\;
t\in [-\tau,\tau].
\]
\end{enumerate}
As a consequence, 
Theorem~\ref{th:WPmain} with $N=1$ is equivalent to the next Lemma~\ref{lem:prop_ex} that we will prove directly.

\begin{lemma}\label{lem:prop_ex}
With the notation of points (3) and (4) above, define 
\begin{equation}\label{expansion}
\left\{
\begin{array}l
\psi^\eps_{1,\rm app}(t,x)= {\rm e}^{\frac i\eps S_1(t) }\wp_{z_1(t)}\left(f_1 \vec e(t) + \sqrt\eps \,\vec B_{1,1}(t) f_1 +\sqrt\eps\, \e^{\frac i\eps\tau^2} \,\1_{t>0} \mathcal T^\flat _{2\to 1} \, f_{2} \,\vec e(t)\right)  ,
\\
\psi^\eps_{2,\rm app}(t,x)= {\rm e}^{\frac i\eps S_2(t) }\wp_{z_2(t)}\left(f_2 \, \vec e(t)^\perp + \sqrt\eps\, \vec B_{2,1}(t) f_2- \sqrt\eps \, \e^{-\frac i\eps\tau^2} \,\1_{t>0} \,\mathcal T^\flat _{1\to 2} \, f_1 \, \vec e(t)^\perp\right).
\end{array}
\right.
\end{equation}
Then, for all $\delta\in (0,\sqrt\eps)$,  there exists $C_\delta>0$ such that 
\[
\sup_{\delta<|t|\leq \tau}\| \psi^\eps(t) - ( \psi^\eps_{1,\rm app}(t)+ \psi^\eps_{2,\rm app}(t))\|_{L^2(\R)} 
\leq C_\delta \left(\left(\frac{\sqrt\eps} \delta\right)^2+\sqrt\eps \delta\right).
\]
\end{lemma}

\begin{proof}
Let us investigate the construction of the approximate solution. We look for $\psi^\eps_1(t)$ and $\psi^\eps_2(t)$, satisfying~\eqref{syst_diag}, and of the form 
\[
\psi^\eps_\ell(t,x)=  {\rm e}^{\frac i\eps S_\ell(t) }\wp_{z_\ell(t)}(\vec f_\ell^\eps(t)),\;\;\ell\in\{1,2\}
\]
with $\vec f_\ell^\eps(t)\in L^2(\R,\C^2)$, 
\[
\vec f^\eps_\ell(t,y)= \vec f_{\ell,0}(t,y)+\sqrt\eps \, \vec f _{\ell,1}(t,y)+o(\sqrt\eps),\;\; t\in[-\tau,\tau],\;\;y\in\R.
\] 
Since  $z_1(t), z_2(t), S_1(t), S_2(t)$ are fixed, the map $\psi^\eps_\ell(t) \mapsto \vec f_\ell^\eps(t) $ is one to one on $L^2(\R,\C^2)$ for each $\ell\in\{1,2\}$. 
The idea is to plug the ansatz in the equation~\eqref{syst_diag} and to choose the functions $\vec f_{1,0}$, $\vec f_{2,0}$, $ \vec f _{1,1}$ and $\vec f_{2,1}$ so that there only remains terms of order $\eps^2$ in the equation (up to $\delta$-depending coefficients). 

\smallskip 

We start by considering the diagonal part of the system~\eqref{syst_diag}.
Injecting the ansatz in the equation leads to
\begin{align}\label{eq_decomp}
&\left(i\eps\partial_t -\eps D_x -(-1)^\ell x-\eps\,  \Omega\right) \left( \e^{\frac i\eps S_\ell(t)}\wp_{z_\ell(t)}(\vec f^\eps_\ell(t))\right)\\
&\qquad 
\nonumber 
=\eps\,  \e^{\frac i\eps S_\ell(t)}\wp_{z_1(t)}\left(i\partial_t \vec f_\ell^\eps(t) - \Omega(t+\sqrt\eps y)\vec f^\eps_\ell(t)\right).
\end{align}
This equation comes from the matching of the actions and trajectories with the Hamiltonians $h_1$ and~$h_2$. Such a discussion is developed in full generality in Sections~\eqref{prop:scalar} and~\eqref{prop_subprinc}.
\smallskip 

Let us now use~\eqref{eq_decomp} to fix the first elements of the ansatz. We consider the mode $\ell=1$. The
function
\[ 
\vec f _{1,0}(t,y)=f_1(y) \vec e(t),
\]
satisfies 
\[
i\partial_t \vec f_{1,0} (t) - \Omega(t) \vec f_{1,0}(t)=0
\]
and thus 
 contributes to suppressing the terms of degree $\eps$ in~\eqref{eq_decomp}. For getting rid of the terms of order $ \eps^{\frac 32}$, one asks for a first term in $\vec f_{1,1}$, that we denote by $f_{1,1}^{\rm cor}$, and that will   satisfy
\[
i\partial_t  \vec f_{1,1}^{\rm cor}(t,y)=\Omega'(t)y \vec f_{1,0}(t)=  f_1(y) \Omega'(t) \vec e(t).
\]
This equation implies $\vec f _{1,1}^{\rm cor}(t)= \vec B_{1,1}(t) f_1$ and gives the first term  in $\sqrt\eps$  in the expansion~\eqref{expansion}. One argues similarly for the mode $\ell=2$ and obtains  the term $ \vec B_{2,1}(t) f_2$  in the expansion~\eqref{expansion}. 
\smallskip 

We now have to understand  how to use a contribution in $\vec f_{1,1}(t)$ and $\vec f_{2,1}(t) $ for compensating  the terms generated by the off-diagonal terms of the Hamiltonian in~\eqref{syst_diag}. 
We look for a term of the form $\vec f_{1,1}^{\rm tr}(t)= f_{1}^{\rm tr} (t) \vec e(t) $ so that up to $\mathcal O(\sqrt\eps)$,
\[
\wp_{z_1(t)}(i\partial_t f_{1}^{\rm tr}(t)) \, \vec e(t) =\frac 1{\sqrt\eps} \e^{\frac i\eps (S_2(t)-S_1(t))} \wp _{z_2(t)}(W_1(t+\sqrt \eps y) \vec f^\eps_2(t,y)).
\]
Or, equivalently 
\begin{align*}
&i\partial_t f_{1}^{\rm tr}(t,y) \, \vec e(t)  = \frac 1{\sqrt\eps}  \e^{-\frac i\eps( t^2-\tau^2) +\frac{2i}{\sqrt\eps} ty }W_1(t+\sqrt \eps y)  \vec f^\eps_2(t,y))\\
 &\;= \frac 1{\sqrt\eps}  \e^{-\frac i\eps (t^2-\tau^2) +\frac{2i}{\sqrt\eps} ty }\left(f_2(y)  W_1(t) \vec e(t)^\perp + \sqrt\eps (yf_2(y)  W_1(t)' \vec e(t)^\perp + \vec f_{2,1}(t,y))\right) +O(\sqrt\eps)\\
 &\;=  \frac {\e^{\frac i\eps\tau^2}}{\sqrt\eps}  \e^{-\frac i\eps t^2 +\frac{2i}{\sqrt\eps} ty }\left(i \vec e(t) f_2(y) + \sqrt\eps ( W_1(t)' \vec e(t)^\perp yf_2(y) + \vec f_{2,1}(t,y))\right) +O(\sqrt\eps)
\end{align*}
in the Schwartz class. 
It is thus important to understand the action of the operator defined on $\mathcal C^\infty([-\tau,\tau], \mathcal S(\R))$ by 
\begin{align*}
\mathcal T^\eps(t) : \varphi\mapsto 
&\frac 1{\sqrt\eps}  
 \int_{-\tau}^t\e^{-\frac i\eps s^2+\frac {2 i}{\sqrt\eps} ys } 
\varphi(s,y)ds.
\end{align*}
Let $\delta\in (0,\sqrt\eps)\subset (0,\tau)$.
If $t\in[-\tau,-\delta] $, an integration by parts  shows that ~$\mathcal T^\eps\varphi(t,\cdot)$ is of order $\sqrt\eps \delta^{-1} $ in the Schwartz class. 
 This corresponds to the adiabatic situation, when the classical trajectories have not yet reached  the crossing point. However, if $t\in[\delta, \tau] $, this operator is no longer negligible. 
Indeed, 
   performing again integration by parts, we obtain that in the Schwartz space 
  \begin{align*}
\mathcal T ^\eps \varphi (t,y)&=\frac 1{\sqrt\eps} \int_{-\infty}^{+\infty} \e^{-\frac i\eps s^2+\frac {2 i}{\sqrt\eps} ys } 
\varphi(s,y)ds +\mathcal O(\sqrt \eps\delta^{-1}) \\
&=  \int_{-\infty}^{+\infty} \e^{-i\eps \sigma^2+ {2 i}{y\sigma }} 
\varphi(0,y)d\sigma +\mathcal O(\sqrt \eps\delta^{-1}).
\end{align*}
A detailed argument for analyzing these asymptotics is given in Section~\ref{sec:mtom+1} for the general case. 
We deduce 
\begin{align*}
f_{1}^{\rm tr}(t,y) &=\e^{\frac i\eps\tau^2}( \mathcal T^\eps f_2) \vec e(t) + \mathcal O(\sqrt\eps)+ \mathcal O(\eps \delta^{-1}) .
\end{align*} 
We can now conclude the proof. 
Setting $\vec f_{\ell,1}(t)= \vec f_{\ell,1}^{\rm cor}(t)+ \vec f_{\ell,1}^{\rm tr}(t)$, and arguing similarly for the mode $\ell:=2$, yields to 
\[
 \left\{
 \begin{array}l
i\eps \psi^\eps_{1,{\rm app}}(t,x) = ((\eps D_x + x)\1_{\C^2}+\eps \Omega(x)) \psi^\eps_{1,{\rm app}}(t,x) + \eps W_1(x) \psi^\eps_{2,{\rm app}}(t,x)  +\mathcal O(\eps^2\delta^{-1}+\eps^{\frac 32}\delta),\\
i\eps \psi^\eps_{2,{\rm app}}(t,x) = ((\eps D_x - x)\1_{\C^2}+\eps \Omega(x) ) \psi^\eps_{2,{\rm app}}(t,x) + \eps W_1(x)^* \psi^\eps_{1,{\rm app}}(t,x) +\mathcal O(\eps^2\delta^{-1}+\eps^{\frac 32}\delta).
\end{array}
\right.
\]
in $L^2(\R,\C^2)$, which concludes the proof. 
\end{proof}

\section{Detailed overview}

The main results of this paper are Theorems \ref{thm:TGeps}, \ref{thm:FGeps}, \ref{thm:FGeps_av}, and \ref{th:WPmain}. 
\smallskip

We prove Theorems \ref{thm:TGeps}, \ref{thm:FGeps} and \ref{thm:FGeps_av}  in Chapters~\ref{chap:4} and~\ref{chap:5}. These proofs rely on Theorem~\ref{th:WPmain} that is proved later.
 Chapter~\ref{chap:4} starts with Section~\ref{sec:Bargmann} that recalls elementary facts about the Bargmann transform. Then, in Section~\ref{sec:freq_loc}, we analyze the notion of frequency localization that we have first introduced in this text and that is crucial in the setting of frozen and thawed initial value representations.  Indeed, these approximations rely on a class of operators that is studied in Section~\ref{subsec:Barg_ext}. Endowed with these results, we are able to prove the approximations of Theorems \ref{thm:TGeps}, \ref{thm:FGeps}, \ref{thm:FGeps_av} in Chapter~\ref{chap:5}. We describe the general proof strategy in Section~\ref{sec:strategy} and develop the proof of Theorem~\ref{thm:TGeps} in Section~\ref{sec:thawed}. We explain in Section~\ref{sec:th_to_fr} how to pass from a thawed to a frozen approximation, and thus obtain Theorems~\ref{thm:FGeps} and~\ref{thm:FGeps_av}.
\smallskip 

Chapters~\ref{chap:2} and~\ref{sec:prop} are devoted to the proof of Theorem~ \ref{th:WPmain}.  These two chapters are independent of the preceding ones, and one can start reading them, skipping Chapters~\ref{chap:4} and~\ref{chap:5}. 
In Chapter~\ref{chap:2}, we construct  the different diagonalisations of the Hamiltonian $H^\eps$ that we are going to use. In the crossing region, we use a rough diagonalisation (see Section~\ref{sec:rough_diag}), and outside this region we use super-adiabatic projectors as proposed  in~\cite{bi,MS,Te} (see~Section~\ref{sec:superadiab}). These constructions rely on a symbolic calculus that we present in Section~\ref{sec:asympt_series}. The analysis of  the propagation of wave-packets in region of small gap  is then performed in Chapter~\ref{sec:prop}. We carefully analyze the dependence of the approximation in term of the gap's parameter. The propagation through the crossing region itself is performed in Chapter~\ref{sec:through}, it is based on Duhamel formula and the use of Dyson series. 
\smallskip 

The Appendices are devoted to the proof of some technical points used in the proofs of 
Chapters~\ref{chap:2} and~\ref{sec:prop}.

\part{Initial value representations}
  
\chapter{Frequency localized families}\label{chap:4}

  In this chapter, we study frequency localized families. We shall use a more precise definition than the one of the Introduction. We use tha Gaussian wave packets $g^\eps_z$ introduced in~\eqref{def:gepsz},
   $g^\eps_z={\rm WP}^\eps_z(g^{i \1_{\R^d}})$ with the notation~\eqref{def:Gaussian}.
   \smallskip
  
   \begin{definition}[Frequency localized functions of order $\beta$]\label{def:freq_loc}
 Let $\beta\geq 0$. Let 
 $(\phi^\eps)_{\eps>0}$ be a family of functions of $L^2(\R^d)$. The family~$(\phi^\eps)_{\eps>0}$ is frequency localized of order $\beta$
  if the family is bounded in $L^2(\R^d)$ and if there exist $R_\beta, C_\beta,\eps_\beta>0$ and $N_\beta > d $ such that for all $\eps\in(0,\eps_\beta]$,
 \[
(2\pi\eps)^{-d/2} \left| \langle g^\eps_z, \phi^\eps \rangle \right|\leq C_\beta \,\eps ^\beta \,\langle z\rangle^{-N_\beta} \;\;\ \text{for all}\ z\in\R^d\ \text{with}\ |z| > R_\beta.
 \]
 One then says that $(\phi^\eps)_{\eps>0}$ is frequency localized of order $\beta$ 
 on the ball $B(0,R_\beta)$.
 \end{definition}

 With this terminology, the notion introduced in Definition~\ref{def:freq_loc_0} in the Introduction consists in frequency localized functions at order~$0$. 
 \smallskip 
 
  In other words, being  frequency localized of order $\beta$ expresses in terms of the Bargmann transform 
  \begin{equation}\label{def:bargmann}
 \mathcal B[f](z):= (2\pi\eps)^{-\frac d2} \langle g^\eps_z, f \rangle, \;\; z\in\R^{2d}
 \end{equation}
 of the family. More precisely, if
  $(\phi^\eps)_{\eps>0}$ satisfies Definition~\ref{def:freq_loc}, its Bargmann transform  has polynomial decay at infinity with control in  $\eps^\beta$ outside a ball~$B(0,R_\beta)$. 
 \smallskip

In Section~\ref{sec:Bargmann} below,
 we recall some facts about the Bargmann transform, i.e. the map 
 \[
 z\mapsto \mathcal B[f](z)
 .
 \]
 Then, in Section~\ref{sec:freq_loc},  we study some properties of frequency localized families and give some examples. 
 Indeed, the thawed/frozen approximations that we aim at studying are constructed thanks to the Bargmann transform. They belong to a class of 
operators obtained by integrating the Bargmann transform against adapted  families of the form 
\[
\left(z\mapsto \theta^\eps_z\right)\in \mathcal C^\infty(\R^{2d}_z, L^2 (\R^d) ).
\] 
We then denote by~$\mathcal J[\theta^\eps_z]$  the operator acting on  $\phi\in L^2(\R^d)$ according to 
 \begin{equation}\label{def:barg_op}
 \mathcal J[\theta^\eps_z](\phi)(x)=(2\pi\eps)^{-\frac d2}\int_{\R^{2d}} \mathcal B[\phi](z)  \theta^\eps_z (x) dz=  (2\pi\eps)^{-d}\int_{\R^{2d}} \langle g^\eps_z, \phi\rangle \theta^\eps_z (x) dz,\;\; x\in\R^d.
 \end{equation}
 The thawed/frozen operators of equations~\eqref{def:J_thawed}, \eqref{def:J_frozen}, \eqref{def:J12th} and \eqref{def:J12fr}  are of that form.
 We study the properties of these operators in Section~\ref{subsec:Barg_ext}.

 \section{The Bargmann transform}\label{sec:Bargmann}

Recall that the Bargman transform is  the map $\mathcal B$ 
 \[
  \mathcal B: \;\; L^2(\R^d)\ni  f\mapsto  \mathcal B[f] \in L^2(\R^{2d}),
  \]
 defined by~\eqref{def:bargmann}
This map is  an isometry and one has 
 \[
 \int_{\R^{2d} }|  \mathcal B[f](z)|^2 dz =\| f\|^2_{L^2}.
 \]
Indeed,  the Gaussian frame  identity writes
\begin{equation}\label{eq:gaussian_frame}
f(x)= (2\pi\eps)^{-d} \int_{\R^{2d}} \langle g^\eps_z, f\rangle g^\eps_z (x) dz
=\mathcal J[g^\eps _z] f,
\end{equation}
and equation~\eqref{eq:gaussian_frame} is equivalent to 
\begin{equation*}
f(x)= (2\pi\eps)^{-\frac d2} \int_{\R^{2d}}  \mathcal B[f](z) g^\eps_z (x) dz,\;\;\forall f\in L^2(\R^d).
\end{equation*}

More generally, the Bargmann transform characterizes the $\Sigma^k_\eps$ spaces according to the next result that we prove in Section~\ref{sec:freq_loc_Sigma_k} below.

\begin{lemma}\label{lem:Barg_Sobolev}
Let $k\in\N$, there exists a constant $c_k$ such that for all $f\in\mathcal S(\R^d)$,
\[
\| f\|_{\Sigma^k_\eps }\leq  c_k \| \langle z\rangle ^k \mathcal B[f] \|_{L^2(\R^{2d})}.
\]
\end{lemma}


 \smallskip

For proving Lemma~\ref{lem:Barg_Sobolev} and along the next sections of this chapter, we shall use properties of wave packets that we sum-up here.

\begin{lemma}
\begin{enumerate}
    \item 
 Let $f,g\in\mathcal S(\R^{d})$ and $z,z'\in\R^{2d}$, then 
\begin{equation}\label{def:W(f,g)}
\langle {\rm WP}_z^\eps(f), {\rm WP}^\eps_{z'} (g)\rangle
={\rm e}^{\frac i\eps p'\cdot (q-q') }W[f,g]\left( \frac {z'-z}{\sqrt\eps}\right)
\end{equation}
where the function $W[f,g]$ is the Schwartz function on $\R^{2d}$ defined by 
\[
W[f,g](\zeta)=\int_{\R^d} \overline f(x) g(x-q) {\rm e}^{ip\cdot x} dx,\;\;\zeta=(q,p).
\]
Moreover, for all $n\in\N$, there exists a constant $C_n>0$ such that 
\begin{equation}\label{estimate_wigner}
\forall \zeta\in\R^{2d},\;\; \langle \zeta\rangle ^n \left| W[f,g](\zeta)\right| \leq C_n  \| f\|_{\Sigma_1^{n}} \| g\|_{\Sigma_1^{n}}.
\end{equation}
\item Let $z=(q,p)\in\R^{2d}$ and $\alpha,\gamma\in\N^d$, then  
\begin{equation}\label{def:g_alpha_gamma}
g_{\eps,z}^{\alpha,\gamma}:= x^\alpha (\eps D_x)^\gamma g^\eps_z={\rm WP}^\eps _z\left((q+\sqrt\eps y)^\alpha (p+\sqrt\eps D_y)^\gamma g^{i\1_{\R^d}})\right).
\end{equation}
Moreover, for all $k\in\N$, there exists $c_{k}>0$ such that for all $z\in\R^{2d}$
and for all $\alpha,\gamma\in\N^d$,
\begin{equation}\label{norm:g_alpha_gamma}
\| g_{\eps,z}^{\alpha,\gamma}\|_{\Sigma^k_\eps} \leq c_{k} \langle z\rangle^{k+|\alpha|+|\beta|}.
\end{equation}
\end{enumerate}
\end{lemma}

\begin{proof}
(1) The formula for $\langle {\rm WP}_z^\eps(f), {\rm WP}^\eps_{z'} (g)\rangle$ comes from a simple computation. Then, 
for $\alpha,\gamma \in\N^d$ and $z=(q,p)\in\R^{2d}$, we observe 
\begin{align*}
\left| q ^\gamma  p^\alpha  W[f,g] (z)  \right|& =\left| q^\gamma  \int _{\R^d}  D_x^\alpha (\overline f(x) g(x-q) ){\rm e}^{ix\cdot p} dx\right|\\
&\leq \langle q\rangle ^{|\gamma|}  \int _{\R^d} \left| D_x^\alpha (\overline f(x) g(x-q) )\right| dx\\
&\leq 2^{\frac{ |\gamma|} 2}   \int _{\R^d}\langle x\rangle ^{|\gamma|}\langle x-q\rangle ^{|\gamma|} \left| D_x^\alpha (\overline f(x) g(x-q) )\right| dx
 \end{align*}
 where we have used Peetre inequality
 \begin{equation}\label{eq:Peetre}
\forall t\in\R,\;\; \forall \ell\in\Z,\;\; \frac{ \langle t\rangle ^\ell} {\langle t'\rangle^\ell }\leq 2^{\frac{|\ell|}{ 2}} \langle t-t'\rangle^{|\ell|}.
\end{equation}
The conclusion then follows.
\smallskip 

(2) Equation~\eqref{def:g_alpha_gamma} comes from the relation 
\[
{\rm op}_\eps(a){\rm WP}_z^\eps(\varphi)=
{\rm WP}_z^\eps(\op_1(a(z+\sqrt\eps\cdot))\varphi)
\]
and~\eqref{norm:g_alpha_gamma} follows.
\end{proof}

We are now in position of proving Lemma~\ref{lem:Barg_Sobolev}.

\begin{proof}[Proof of Lemma~\ref{lem:Barg_Sobolev}]
Let $k\in\N$ and $\alpha,\gamma\in\N^d$ such that $|\alpha|+|\gamma|=k$. 
we consider the operator 
\[
T_{\alpha,\gamma}=\mathcal B\circ (x ^\alpha ( \eps D_x)^\gamma ) \circ \mathcal B^{-1}\circ \langle z\rangle ^{-k}.
\]
The kernel of this operator is the function 
\[
\R^{4d}\ni (X,Y)\mapsto  k^\eps (X,Y)=(2\pi\eps)^{-d} \langle g^\eps_X, x ^\alpha ( \eps D_x)^\gamma g^\eps_Y \rangle \langle Y\rangle^{-k}.
\]
Therefore, with the notation of~\eqref{def:g_alpha_gamma} and  by~\eqref{def:W(f,g)},
\begin{align*}
\int_{Y\in \R^{2d} }\sup_{X\in\R^{2d} } |k^\eps(X,Y)| dY & = (2\pi\eps)^{-d}
\int_{Y\in \R^{2d} }\sup_{X\in\R^{2d} } \left|W[g^{i{\1_{\R^d}} },g^{\alpha,\gamma}_{\eps,Y}]\left(\frac{X-Y}{\sqrt\eps}\right)\right| \langle Y\rangle ^{-k} dY\\
&=\int_{\zeta\in \R^{2d} }\sup_{X\in\R^{2d} } \left|W[g^{i{\1_{\R^d }} },g^{\alpha,\gamma}_{\eps,Y}]\left(\zeta\right)\right| \langle X-\sqrt\eps \zeta \rangle ^{-k} d\zeta.
\end{align*}
We deduce from~\eqref{estimate_wigner}  and~\eqref{norm:g_alpha_gamma} the existence of a constant $c_k>0$ such that 
\[
\int_{Y\in \R^{2d} }\sup_{X\in\R^{2d} } |k^\eps(X,Y)| dY
\leq c_k.
\]
Similarly, we have 
\begin{align*}
\int_{X\in \R^{2d} }\sup_{Y\in\R^{2d} } |k^\eps(X,Y)| dX &= (2\pi\eps)^{-d}
\int_{X\in \R^{2d} }\sup_{Y\in\R^{2d} } \left|W[g^{i{\1_{\R^d}} },g^{\alpha,\gamma}_{\eps,Y}]\left(\frac{X-Y}{\sqrt\eps}\right)\right| \langle Y\rangle ^{-k} dX\leq c_k.
\end{align*}
Therefore, the Schur test yields the boundedness of $T_{\alpha,\gamma}$. One then deduces that for $f\in\mathcal S(\R^d)$, one has 
\begin{align*}
\| x^\alpha( \eps D_x)^\gamma f\|_{L^2(\R^d)}& 
=  \| \mathcal B[x^\alpha( \eps D_x)^\gamma f]\|_{L^2(\R^{2d})}\\
&=  \left\|T^{\alpha,\gamma} \left( \langle z\rangle ^k \mathcal B[ f] \right) \right\|_{L^2(\R^{2d})}
\leq c_k   \|  \langle z\rangle ^k \mathcal B[ f] \|_{L^2(\R^{2d})},
\end{align*}
which concludes the proof. 
\end{proof}

 \section{Frequency localized families}\label{sec:freq_loc}
 
 We investigate here the properties of families that are frequency localized in the sense of Definition~\ref{def:freq_loc}.
 

   \subsection{First properties of  frequency localized functions}
   The first properties are straightforward.
   
   \begin{proposition}
   The set of frequency localized function is a  subspace of $L^2(\R^d)$. Moreover, we have the following properties: 
   \begin{enumerate}
   \item If $(\phi^\eps_1)_{\eps>0}$ and  $(\phi^\eps_2)_{\eps>0}$ are   two frequency localized families at the scales $\beta_1$ and $\beta_2$ respectively, then for all $a,b\in\C$, the family  $(a\phi^\eps_1+b\phi^\eps_2)_{\eps>0}$ is frequency localized at the scale ${\rm min}(\beta_1,\beta_2)$. 
    \item If  $(\phi^\eps)_{\eps>0}$ is frequency localized at the scale $\beta\geq 0$, then
     it is also frequency localized at the scale $\beta'$ for all $\beta'\in[0,\beta]$. 
      \end{enumerate}
   \end{proposition}

Moreover, the notion of frequency localized functions is microlocal. Indeed, defining the 
    $\eps$-Fourier transform  by 
   \[
   \mathcal F^\eps f(\xi)=(2\pi\eps)^{-\frac d2} \int_{\R^d} \e^{\frac i\eps \xi\cdot x} \phi^\eps(x) dx =(2\pi\eps)^{-\frac d2} \widehat f \left(\frac \xi\eps\right) ,\;\; f\in \mathcal S^2(\R^d),
   \]
   we have the following result.
   
   \begin{proposition}
   Let $(\phi^\eps)_{\eps>0}$ be a  bounded family in $L^2(\R^d)$. Then, $(\phi^\eps)_{\eps>0}$ is frequency localized at the scale $\beta\geq 0$ if and only if 
    $(\mathcal F^\eps\phi^\eps)_{\eps>0}$ is frequency localized  at the scale $\beta\geq 0$. 
   \end{proposition}

   \begin{proof}
This comes from the observation that for all $z\in\R^{2d}$,
   \[
 \left|  \langle g^\eps_z,\phi^\eps\rangle\right| =  \left|  \langle g^\eps_{Jz},\mathcal F^\eps \phi^\eps\rangle\right| 
 \]
 where $J$ is the matrix defined in~\eqref{def:J}. Thus it is equivalent to state the fact of being frequency localized for a family or for the family of its $\eps$-Fourier transform.
   \end{proof}

 \subsection{Frequency localized families and Bargmann transform}

The Gaussian frame identity~\eqref{eq:gaussian_frame} allows to decompose a function of $L^2(\R^d)$ into a (continuous) sum of Gaussians. After discretization of the integral, this sum may be turned into a finite one, which opens the possibility of approximation's strategies (see~\cite{LS} where this observation is used for numerical purposes).
It is thus important to identify assumptions  that allow to compactify the set of integration in $z$. The notion of frequency localized families plays this role according to the next result. 

\begin{lemma}\label{lem:freq_loc}
Let $(\phi^\eps)_{\eps>0}$  be a frequency localized family at the scale~$\beta\geq 0$. Let  $R_\beta, C_\beta$ and~$N_\beta$ be  the constants associated to $(\phi^\eps)_{\eps>0}$ according to Definition~\ref{def:freq_loc}.  Let $k\in\N$ with $N_\beta>d +k$. Then,  for all $\chi\in L^\infty(\R)$ supported in $[0,2]$ and equal to $1$ on $[0,1]$, there exists $C>0$ such that for $R>R_\beta$ and $\eps\in (0,1]$,
 \begin{equation*}
 \left\|
 \phi^\eps -  \phi^\eps_{<,R}
  \right\|_{\Sigma^k_\eps(\R^d)} 
  \leq C \, C_\beta\, \eps^\beta  \left(\int_{|z|>R} \langle z\rangle ^{-2(N_\beta-k)}{dz}\right)^{1/2},
\end{equation*}
where
\begin{equation}\label{def:phi_R_<}
\phi^\eps_{R,<} := \mathcal J\left[g^\eps_z \chi\left(\frac{|z|} R\right) \right](\phi^\eps)= \mathcal B^{-1} \left( \1 _{|z|<R} \mathcal B[\phi^\eps](z)\right).
\end{equation}
\end{lemma}

\begin{remark}
Lemma~\ref{lem:freq_loc} can be interpreted in different manners. 
\begin{enumerate}
\item If $\beta>0$, then $ \phi^\eps_{R,<}$ approximates $\phi^\eps$ in $L^2(\R^d)$ as $\eps$ goes to $0$  in any space  $\Sigma^k_\eps(\R^d)$ with $k\in\N$ such that $   N_\beta >k+ d$, and uniformly with respect to  $R>R_\beta$.
\item If  $\beta\geq 0$ (which includes $\beta=0$), then the same approximation holds by letting $R$ go to~$+\infty$, and it is uniform with respect to $\eps$. In particular, when $\beta=0$ we have 
\[
 \limsup_{\eps\rightarrow 0} 
 \|
 \phi^\eps-  \phi^\eps_{R,<}
  \|_{\Sigma^k_\eps(\R^d)}
  \leq C  R^{-(N_\beta-k-d) }.
\]
\end{enumerate}
\end{remark}

\begin{proof} We   set
 \[
 r^\eps(x)= (2\pi\eps)^{-d}  \int_{|z|\geq R} \langle g^\eps_z, \varphi^\eps \rangle g^\eps_z (x) dz
 \]
 and consider $k\in\N$. 
For $R>R_\beta$, $\alpha,\gamma\in\N^d$ with $|\alpha|+|\gamma|=k$, we have
\begin{align*}
&\| x^\alpha (\eps D_x)^\gamma r^\eps\|^2_{L^2(\R^d)}\\
\nonumber
&\;  \leq
 (2\pi\eps)^{-2d}\int_{\R^d}  \int _{|z|>R} \int _{|z'|>R} 
 \langle g^\eps_z, \varphi^\eps \rangle \, \overline{ \langle g^\eps_{z'}, \varphi^\eps \rangle }
g_{\eps,z}^{\alpha,\gamma}(x) \, \overline{ g_{\eps,z'}^{\alpha,\gamma}(x)} \,dx\, dz\, dz'
\end{align*}
with the notation~\eqref{def:g_alpha_gamma}.
%
By~\eqref{def:W(f,g)} and the frequency localization of order $_beta$ of $(\phi^\eps)_{\eps>0}$, we obtain 
\begin{align*}
& \| x^\alpha (\eps D_x)^\gamma r^\eps\|^2_{L^2(\R^d)} \\
 &\qquad   \leq
  C_\beta^2\, \eps^{2\beta}  \, (2\pi\eps)^{-d}  \int _{|z|>R} \int _{|z'|>R} 
\langle z\rangle ^{-N_\beta} \langle z'\rangle ^{-N_\beta}\, W[g_{\eps,z}^{\alpha,\gamma},g_{\eps,z'}^{\alpha,\gamma}] \left(\frac{z-z'}{\sqrt\eps}\right) dz\, dz'.
\end{align*}
Besides, by~\eqref{estimate_wigner} and~\eqref{norm:g_alpha_gamma}, there exists a constant $C'_{n,k}$ such that 
\[
\left| W[g_{\eps,z}^{\alpha,\gamma},g_{\eps,z'}^{\alpha,\gamma}] (\zeta )\right|\leq   C'_{n,k}  \, \langle\zeta \rangle ^{-n}  \langle z\rangle^k\langle z'\rangle^k.
\]
We deduce the   existence of $c>0$ such that 
\begin{align*}
 \|x^\alpha (\eps D_x)^\gamma  r^\eps\|^2_{L^2(\R^d)} &  \leq c\, C_\beta^2 \,
\, \eps^{2\beta}  \,  \eps^{-d}  \int _{|z|>R} \int _{|z'|>R} 
\langle z\rangle ^{-N_\beta+k} \langle z'\rangle ^{-N_\beta+k} \left\langle \frac{z-z'}{\sqrt\eps}\right\rangle^{-n} dzdz'\\
&  \leq
c\, C_\beta^2\, \eps^{2\beta}  \,  \int _{|z|>R} \int  
\langle z\rangle ^{-N_\beta+k} \langle z+\sqrt\eps  \zeta \rangle ^{-N_\beta+k} \langle \zeta \rangle^{-n}  dzd\zeta.
\end{align*}
Since $-N_\beta+k\leq 0$, Peetre inequality~\eqref{eq:Peetre} gives
\[
\langle z+\sqrt\eps  \zeta\rangle ^{-N_\beta+k} 
\leq  2^{\frac{N_\beta-k}2} \langle \sqrt\eps  \zeta\rangle ^{ N_\beta-k} \langle z\rangle ^{-N_\beta+k}
\leq  2^{\frac{N_\beta-k}2}\langle  \zeta\rangle ^{ N_\beta-k} \langle z\rangle ^{-N_\beta+k},
\]
by restricting ourselves to $\eps\leq 1$. 
Therefore, there exists a constant $c>0$ such that
\begin{align*}
 \|x^\alpha (\eps D_x)^\gamma  r^\eps\|^2_{L^2(\R^d)} 
&  \leq c\, C_\beta^2\,
\eps^{2\beta}    \left( \int _{\R^{2d}} \langle \zeta\rangle^{-n+N_\beta -k } d\zeta\right)
\left(\int _{|z|>R}\langle z\rangle ^{-2(N_\beta -k)}  dz\right)
\end{align*}
and we   conclude the proof by choosing $n=N_\beta+k+2d+1$. 
 \end{proof}

 \subsection{Examples}
In this section, we  analyze the examples given in the Introduction.

  \begin{lemma}\label{lem:ex_freq_loc}
  \begin{enumerate}
  \item 
  Let $u\in\mathcal S(\R^d)$ and $z_0=(q_0,p_0)\in\R^{2d}$. Then,  the family $({\rm WP}^\eps_{z_0}(u))_{\eps>0}$ is frequency localized at the scale $\beta$ for any $\beta\geq 0$.
  \item 
Let $a\in\mathcal C_0^\infty(\R^d)$ and $S\in\mathcal C^\infty(\R^d)$. Then,  the family $({\rm e}^{\frac i\eps S(x)} a)_{\eps>0}$ is frequency localized at the scale $\beta$ for any $\beta\geq 0$. 
  \end{enumerate}
  \end{lemma}
  
  \begin{proof}
1-   By~\eqref{def:W(f,g)}, we have for $z\in\R^{2d}$
\[  
\mathcal B[ {\rm WP}^\eps_{z_0}(u)] (z) 
=(2\pi\eps)^{-\frac d2} \langle {\rm WP}_z^\eps(g^{i\1_{\R^d}}),{\rm WP}^\eps_{z_0}(u)\rangle 
=(2\pi\eps)^{-\frac d2} \, {\rm e}^{\frac i\eps p_0\cdot (q-q_0)} \, W[g^{i\1_{\R^d}},u] \left(\frac {z_0-z}{\sqrt\eps}\right).
\]
Let $N\in\N$, the estimate~\eqref{estimate_wigner} implies the existence of a constant $C_N'$ such that 
\[
| \mathcal B[ {\rm WP}^\eps_{z_0}(u)] (z) | \leq C_N' \, \eps^{-\frac d2}\, \left\langle \frac {z_0-z}{\sqrt\eps} \right\rangle^{-N}.
\]
Choosing $|z|>\max(2|z_0|,1)$, we have $2|z_0-z|\geq 2(|z|-|z_0|)\geq |z|$ and we deduce 
\[
\left\langle \frac {z_0-z}{\sqrt\eps} \right\rangle^{-N}
= \left( \frac\eps{\eps+|z-z_0|^2}\right)^{\frac N2} 
\leq  \left(\frac{4\eps}{4\eps+|z|^2}\right)^{\frac N2} \leq   {(4\eps)^{\frac {N}2}} \, {| z|^{-N}},
\]
 whence
 the existence of a constant $c_N>0$ such that for $|z|>\max(2|z_0|,1)$ and $N\in\N$,
\[
| \mathcal B[ {\rm WP}^\eps_{z_0}(u)] (z) | \leq c_N \, \eps^{\frac {N-d} 2}  \, \langle z\rangle^{-N}.
\]
    \medskip 
    
    2- One has for $z\in\R^{2n}$
 \begin{align*} \nonumber
 \mathcal B[{\rm e}^{\frac i\eps S(x)} a] (z)
 & \nonumber= (2\pi)^{-d/2} (\pi\eps) ^{-d/4} 
 \int_{\R^d} a(q+\sqrt\eps y)
  {\rm Exp}\left(  - \frac i{\sqrt\eps}  p\cdot y + \frac i{\eps} S(q+y\sqrt \eps) 
\right)  {\rm e}^{-\frac{ |y|^2}2} dy.
 \end{align*}
 This term has a very specific structure involving the symbol $y\mapsto a(y)$, an exponentially decaying function  $y\mapsto  {\rm e}^{-\frac{ |y|^2}2} $ and an oscillating phase 
  \[
  y\mapsto  \Lambda^\eps(y):= - \frac 1{\sqrt\eps}  p\cdot y + \frac 1{\eps} S(q+y\sqrt \eps) .
  \] 
We are going to show that the terms defined for $j\in\{1,\cdots,d\}$ by 
  \[
A_j^\eps:=q_j   \int_{\R^d } a(q+\sqrt\eps y)
 {\rm e}^{-\frac{ |y|^2}2} \e^{ i\Lambda^\eps(y)}  dy\;\;\;\mbox{and}\;\;\;
 B_j^\eps:=p_j \int_{\R^d } a(q+\sqrt\eps y)
 {\rm e}^{-\frac{ |y|^2}2} \e^{ i\Lambda^\eps(y)}  dy,
 \]
 have the same structure. Then, it will be enough to consider only one of these terms and to prove that they are controlled by a power of $\eps$, this will imply an adequate control on $| \mathcal B[{\rm e}^{\frac i\eps S(x)} a] (z)|$ by a recursive argument.
 \smallskip 
 
 Let us first transform  $A_j^\eps$  and $B_j^\eps$. 
 Indeed, we have 
 \begin{align*}
 A_j ^\eps&= \int_{\R^d } (q_j+\sqrt\eps y_j) a(q+\sqrt\eps y)
 {\rm e}^{-\frac{ |y|^2}2} \e^{ i\Lambda^\eps(y)}  dy
 -\sqrt\eps \int_{\R^d } a(q+\sqrt\eps y)
\left(y_j  {\rm e}^{-\frac{ |y|^2}2} \right) \e^{ i\Lambda^\eps(y)}  dy.
 \end{align*}
 The first integral of the right hand side has the same structure with the symbol $y\mapsto a(y)$ and the second one with the exponentially  decaying function $y\mapsto y_j {\rm e}^{-\frac{ |y|^2}2} $.
 Besides, observing 
\[
p_j \e^{i\Lambda^\eps(y)}=-i\sqrt\eps \partial_{y_j} (\e^{i\Lambda^\eps(y)})-\partial_{y_j} S(q+\sqrt\eps y) \e^{i\Lambda^\eps(y)},
\]
we obtain with an integration by parts
\begin{align*}
 B_j^\eps &= -\int_{\R^d } \partial_{y_j} S(q+\sqrt\eps y)  a(q+\sqrt\eps y)
 {\rm e}^{-\frac{ |y|^2}2} \e^{ i\Lambda^\eps(y)}  dy\\
&\;\;\;\; +i \sqrt\eps \int_{\R^d } \partial_{y_j} \left(a(q+\sqrt\eps y)
  {\rm e}^{-\frac{ |y|^2}2} \right) \e^{ i\Lambda^\eps(y)}  dy.
 \end{align*}
Here again the right hand side has the same structure with different symbols and rapidly decaying term. 
\smallskip

We now focus in proving that one typical term   written for  $\varphi$ a polynomial function  as 
\begin{equation}\label{mickey}
L^\eps:=  \int_{\R^d } a(q+\sqrt\eps y)
 \varphi(y)\e^{-\frac{|y|^2}2}\e^{ i\Lambda^\eps(y)}  dy
  \end{equation}
 is of order $\eps^N$ for all $N\in\N$.
  The decay of $y\mapsto \e^{-\frac{|y|^2}2  } $  allows to reduce the set of integration. Indeed, we have 
  \[
\left| 
 \int_{|y| >\eps^{-\frac 14} } a(q+\sqrt\eps y)
 {\rm e}^{-\frac{ |y|^2}2} \e^{ i\Lambda^\eps(y)} \varphi(y) dy\right| 
\leq \e^{-\frac {\sqrt\eps } 4} \| a\|_{L^\infty} \int_{\R^d}   {\rm e}^{-\frac{ |y|^2}4} |\varphi(y)|dy.
\]
Therefore, there exists a constant $c>0$ such that 
\[
|L^\eps|
 \leq c \left(  \left|
 \int_{|y|\leq \eps^{-\frac 14} } a(q+\sqrt\eps y)
  {\rm e}^{-\frac{ |y|^2}2} \e^{ i\Lambda^\eps(y)} \varphi(y)dy\right|
+   \e^{-\frac 1 {4\sqrt\eps } }\right).
\]
 We now use the oscillations of the phase for treating the integral with domain $\{|y|\leq \eps^{-\frac 14}\}$. We observe that there exists $R_0>0$ such that if $|z|>R_0$, then 
  \[
z\notin\{ |p- \nabla S(q)|\leq 1,\;{\rm dist}( q,  {\rm supp}(a))\leq 1 \}.
 \]
We
choose  $|z|>R_0$ and  we have the following alternative: 
\[
\mbox{either}\;  {\rm dist}( q,  {\rm supp}(a))>1,\;\mbox{  or}\;\left(   {\rm dist}( q,  {\rm supp}(a))\leq 1\;\mbox{ and}\;  |p- \nabla S(q)|> 1\right).
\]
 If 
 ${\rm dist}( q,  {\rm supp}(a))>1$, there exists $\eps_0>0$ such that if $\eps\in(0,\eps_0]$ and $|y|\leq \eps^{1/4} $, then $q+\sqrt\eps y\notin {\rm supp}(a)$. The integral thus is zero and we are reduced to the case where  ${\rm dist}( q,  {\rm supp}(a))\leq 1$ and $ |p- \nabla S(q)|> 1$. One can find $\eps_1>0$ such that for $\eps\in (0,\eps_1]$ and $|y|\leq \eps^{\frac 14}$, 
 \[
\left| \nabla S(q+y\sqrt\eps)-p\right|>\frac 1{2}.
\]
We then consider the differential operator 
\[
L^\eps= \sqrt\eps \frac{\nabla S(q+y\sqrt\eps)-p}{| \nabla S(q+y\sqrt\eps)-p|^2} \cdot \nabla_y
\]
and we  write 
\begin{align*}
& \int_{|y| \leq \eps^{-\frac 14} } a(q+\sqrt\eps y)
  {\rm e}^{-\frac{ |y|^2}2} \e^{ i\Lambda^\eps(y)}\varphi(y)  dy \\
&= \int_{|y| \leq \eps^{-\frac 14} } a(q+\sqrt\eps y) 
  {\rm e}^{-\frac{ |y|^2}2}\varphi(y)(L^\eps)^N\left( \e^{ i\Lambda^\eps(y)} \right) dy\\  
&= \int_{|y| \leq \eps^{-\frac 14} } {(L^\eps)^N}^*\left( a(q+\sqrt\eps y)|\nabla S(q+y\sqrt\eps)-p|^{-2N} 
  {\rm e}^{-\frac{ |y|^2}2} \varphi(y)\right)\e^{ i\Lambda^\eps(y)}  dy.
\end{align*}
There exists a constant $C>0$ independent of $z$ such that for all $\eps\in(0,\eps_1]$ and $|y|\leq \eps^{\frac 14}$,
\[
\left|  {(L^\eps)^N}^*\left( a(q+\sqrt\eps y)|\nabla S(q+y\sqrt\eps)-p|^{-2N} 
  {\rm e}^{-\frac{ |y|^2}2} \varphi(y)\right)\right|\leq C \eps^{\frac N2} {\rm e}^{-\frac{ |y|^2}2}.
  \]
  We deduce 
  \begin{align*}
 \left|\int_{|y| \leq \eps^{-\frac 14} } a(q+\sqrt\eps y)
  {\rm e}^{-\frac{ |y|^2}2} \e^{ i\Lambda^\eps(y)} \varphi(y) dy \right|
  \leq C \eps^{\frac N 2}
  \end{align*}
  and~\eqref{mickey} writes 
  \[
  |L^\eps|\leq c\left(\eps^{\frac N 2} +\e^{-\frac1{4\sqrt\eps}}\right)
  \]
  for some constant $c>0$. 
  This terminates the proof. 
   \end{proof}

   \subsection{Characterization of frequency localized families}
   
The characterization of frequency localized families  can be done by using other families of wave packets than Gaussian ones and the cores $z$ can be  distributed in different manners. 
   
   \begin{proposition}\label{prop:charac_freq_loc}
    The family $(\phi^\eps)_{\eps>0}$ is frequency localized at the scale $\beta\geq 0$ if and only if  for all $\mathcal C^1$-diffeomorphism $\Phi$ satisfying 
   \[
 \exists a,b>0,\;\;  \forall z\in\R^{2d},\;\; a|z|\leq \Phi(z)\leq b|z|,
   \]
for all   $\theta\in\mathcal S(\R^d)$,
 there exists  $C_\beta$, $N_\beta$, $R_\beta$ and $\eps_\beta$ such that for all $\eps\in(0,\eps_\beta]$ and  for $|z|>R_\beta$
    \[
    (2\pi\eps)^{-\frac d2} |\langle{\rm WP}^\eps_{\Phi(z)}(\theta), \phi^\eps\rangle |\leq C_\beta \,\eps^\beta\, \langle z\rangle ^{-N_\beta}\,\max\left(1, \frac 1a\right)^{N_\beta}  \,\|\theta\|_{\Sigma^{2d+1+N_\beta}}.
    \]
    Moreover, for all family $(\lambda^\eps)_{\eps>0}$ bounded in $L^\infty(\R^{2d})$, $\eps\in(0,\eps_\beta]$ and $R>R_\beta$,
    \[
    \left\| \mathcal J\left[{\rm e}^{\frac i\eps \lambda^\eps(z)}{\rm WP}^\eps_{\Phi(z)}(\theta)\1_{|z|>R}\right](\phi^\eps)\right\|_{\Sigma^k_\eps} 
    \leq C\, C_\beta\, \eps^\beta \left(\int_{|z|>R} \langle z\rangle^{-2(N_\beta-k)} dz\right)^{\frac 12}.
    \]
      \end{proposition}

    \begin{proof}
   We only have to prove that if $(\phi^\eps)_{\eps>0}$ is frequency localized at the scale $\beta\geq 0$, then the property holds for some given profile $\theta$ and diffeomorphism $\Phi$. Then, the equivalence will  follow. 
   We consider the constants $C_\beta$, $N_\beta$, $R_\beta$ and $\eps_\beta$ given by Definition~\ref{def:freq_loc} and we take  $\eps\in(0,\eps_\beta]$.
   We observe 
   \begin{align*}
     (2\pi\eps)^{-\frac d2} \langle{\rm WP}^\eps_{\Phi(z)}(\theta), \phi^\eps\rangle & = 
       (2\pi\eps)^{-\frac {3d}2} \int_{\R^d} \langle{\rm WP}^\eps_{\Phi(z)}(\theta), g^\eps_{z'}\rangle \langle g^\eps_{z'} , \phi^\eps\rangle dz'
       = I_1+I_2
       \end{align*}
       with 
       \[
       I_1=    (2\pi\eps)^{-\frac {3d}2} \int_{|z'|>R_\beta} \langle{\rm WP}^\eps_{\Phi(z)}(\theta), g^\eps_{z'}\rangle \langle g^\eps_{z'} , \phi^\eps\rangle dz'.
       \]
       
       Let us study $I_1$. Using~\eqref{def:W(f,g)},~\eqref{estimate_wigner}   and  that  $(\phi^\eps)_{\eps>0}$ is frequency localized at the scale $\beta\geq 0$, we deduce the  existence of $c_\beta, N_\beta >0$ such that we have 
          \begin{align*}
     |I_1| & 
     \leq c_\beta \,\eps^\beta\,  (2\pi\eps)^{-d} \int_{|z'|>R_\beta} \left| W[\theta, g^{i\1_{\R^d}}] \left(\frac{z'-\Phi(z)}{\sqrt\eps}\right)\right| \langle z'\rangle^{-N_\beta} dz'\\
           & \leq c_\beta  \,\|\theta\|_{\Sigma^{n}} \,\eps^\beta\,   (2\pi\eps)^{-d} \int_{\R^d} \left\langle \frac{z'-\Phi(z)}{\sqrt\eps}\right\rangle^{-n}  \langle z'\rangle^{-N_\beta} dz'\\
          & \leq c_\beta \,\|\theta\|_{\Sigma^{n}}  \,\eps^\beta\,   \int_{\R^d} \left\langle \zeta\right\rangle^{-n}  \langle \Phi(z)+\sqrt\eps \zeta \rangle^{-N_\beta} d\zeta,
       \end{align*} 
       where the constant $c_\beta$ may have changed between two successive lines.
       We observe that Peetre's inequality~\eqref{eq:Peetre} yields, assuming $\eps\in (0,1]$,
       \[
        \langle \Phi(z)+\sqrt\eps \zeta \rangle^{-N_\beta}\leq 2^{\frac{N_\beta}2}\langle \Phi(z)\rangle ^{-N_\beta} \langle \sqrt \eps \zeta\rangle^{N_\beta} \leq 
        2^{\frac{N_\beta}2}\langle \Phi(z)\rangle ^{-N_\beta} \langle \zeta\rangle^{N_\beta} ,
        \]
        whence by choosing $n=2d+1+N_\beta$,
          \begin{align*}
     |I_1|
  & \leq c_\beta  \,\|\theta\|_{\Sigma^{2d+1+N_\beta}} \,\eps^\beta\,  \langle \Phi(z) \rangle^{-N_\beta}  \int_{\R^d}\left\langle \zeta\right\rangle^{-(2d+1)}  d\zeta,
  \end{align*}
  for some new constant $c_\beta>0$. We conclude by observing that 
  \[
  \langle \Phi(z)\rangle^{-N_\beta} \leq \max\left(1, \frac 1a\right)^{N_\beta} \langle z\rangle^{-N_\beta},
\]
whence, by modifying $c_\beta$, 
\[
     |I_1|\leq c_\beta  \,\|\theta\|_{\Sigma^{2d+1+N_\beta}} \,\eps^\beta\, \max\left(1, \frac 1a\right)^{N_\beta} \langle z\rangle^{-N_\beta}.
     \]

  We now study $I_2$. Using~\eqref{def:W(f,g)},~\eqref{estimate_wigner}, we write for $n\in\N$
\[
  |I_2| \leq \|\phi^\eps\|_{L^2} (2\pi\eps)^{-\frac {3d}2}
  \int_{|z'|\leq R_\beta}  \left\langle\frac{z'-\Phi(z)}{\sqrt\eps}\right\rangle^{-n} dz'.
  \]
  We observe that if $|z|>2 aR_\beta$, then for $|z'|\leq R_\beta\leq \frac 1{2a} |z|$, we have 
  \[
  |z'-\Phi(z)|\geq |\Phi(z)|-|z'| \geq \frac 1{2a} |z|.
  \]
  Therefore
  \[
   \left\langle\frac{z'-\Phi(z)}{\sqrt\eps}\right\rangle^{-n} = \left(\frac{\eps}{\eps+|z'-\Phi(z)|^2}\right)^{\frac n2}\leq { (2a)^n} {\eps^{\frac n2}} |z|^{-n}.
   \]
   Using that $(\phi^\eps)_{\eps>0}$ is a bounded family in $L^2$, we obtain that there exists a constant $c'$ such that for $|z|>2a R_\beta$ and any $n\in\N$,
   \begin{align*}
  |I_2|& \leq c' \eps^{\frac{n-3d}2}|z|^{-n}.
  \end{align*}  
  
  The proof of the last property follows the line of the proof of Lemma~\ref{lem:freq_loc} combined with  adapted change of variables.  
  This terminates the proof. 
      \end{proof}

    \subsection{Frequency localized families and semi-classical pseudodifferential calculus}
  Frequency localized families also enjoy properties with respect to pseudodifferential calculus. 
    
      \begin{proposition}\label{prop:freq_loc_pseudo}
      Let $( \phi^\eps)_{\eps>0}$ be a frequency localized family at the scale $\beta\geq 0$.
           \begin{enumerate}
  \item For all  semi-classical symbol $a\in\mathcal C^\infty_c(\R^{2d})$, the family $\bigl(\widehat a \, \phi^\eps \bigr)_{\eps>0}$ is frequency localized at the scale $\beta\geq 0$. 
  \item For all  subquadratic Hamiltonian $h\in\mathcal C^\infty(\R\times \R^{2d},\R)$,  for all $t,t_0\in\R$,    the vector-valued family $\bigl(\mathcal U^\eps_h(t,t_0) \phi^\eps \bigr)_{\eps>0}$ is frequency localized at the scale $\beta\geq 0$. 
 \end{enumerate}
   \end{proposition}

\begin{proof}
(1) We  can   assume without loss of generality that $a$ is real-valued. We write 
\[
\mathcal B[ \widehat a \phi^\eps] = (2\pi\eps)^{-d/2} \langle \widehat a g^\eps_z,\phi^\eps\rangle.
\]
Since $g^\eps_z$ is a wave packet, we have 
\[
\widehat a g^\eps_z= \widehat a\,  {\rm WP}^\eps_z (g^{i\1_{\R^d}})= {\rm WP}^\eps _z (\mathfrak g^\eps_a),\;\; \mathfrak g^\eps_a= \op_1(a(z+\sqrt\eps \cdot) ) g^{i\1_{\R^d}}.
\]
The function  $\mathfrak g^\eps_a$ is of Schwartz class on $\R^d$ and its Schwartz semi-norms are uniformly bounded in~$\eps$ because $a$ is compactly supported. We deduce from Proposition~\ref{prop:charac_freq_loc}, 
\[
 \left| \mathcal B[ \widehat a \phi^\eps] \right|\leq C_\beta \,\eps^\beta\, \langle z\rangle^{-N_\beta} \| \mathfrak g^\eps_a\|_{2d+1+N_\beta},
 \]
 which concludes the proof.
 \medskip 
 
 (2) We write 
 \[
 \mathcal B\left[ \mathcal U^\eps_h(t,t_0)\phi^\eps\right]= (2\pi\eps)^{-d/2} \langle \mathcal U^\eps_h(t,t_0) g^\eps_z,\phi^\eps\rangle.
\]
Since $g^\eps_z$ is a wave packet, we have 
\begin{equation}\label{robin1}
\mathcal U^\eps_h(t,t_0) g^\eps_z={\rm e}^{\frac i\eps S(-t,z)}  {\rm WP}^\eps_{\Phi^{-t,0}_h(z)}  ( g^{\Gamma(-t,z) }+ \sqrt\eps \, r^\eps_z(t))
\end{equation}
with the notations of the introduction. Besides, for all  $n\in\N$, there exists a constant $C=C(n,t)$ such that
$\| r^\eps_z(t)\|_{\Sigma^n} \leq C(n,t) $.
We deduce  from Proposition~\ref{prop:charac_freq_loc}, 
\[
 \left|  \mathcal B\left[  \mathcal U^\eps_h(t,t_0) \phi^\eps\right]\right|\leq C_\beta \,\eps^\beta\, \langle z\rangle^{-N_\beta} \| g^{\Gamma(t,z)} + \sqrt\eps r^\eps_z(t)\|_{2d+1+N_\beta},
 \]
 which concludes the proof.
\end{proof}
  
  \begin{remark}\label{rem:robin2}
  \begin{enumerate}
  \item 
  The proof of Proposition~\ref{prop:freq_loc_pseudo} (1) extends to smooth functions $a$  with polynomial  growth 
  \[\exists N_0\in\N,\;\;
  \forall \gamma\in\N^d,\;\;\forall z\in\R^{2d},\;\; \left|\partial^\gamma a(z)\right| \leq \langle z\rangle^{N_0-|\gamma|}
  \]
   provided the integer $N_\beta $ associated with the frequency localisation at the scale $\beta\geq 0$ of the family $(\phi^\eps_0)_{\eps>0}$ verifies $N_\beta>d+N_0$. 
   
   \item The proof of Proposition~\ref{prop:freq_loc_pseudo} (2) also extends to adiabatic smooth matrix-valued Hamiltonian~$H$ that are  subquadratic according to Definition~\ref{def:subquad}.
However, it is not clear whether the same result holds   for Hamiltonians with crossings, either they are smooth as in this article or conical  as in the Appendix of~\cite{FGH}. Indeed, even though 
 one knows that $\mathcal U^\eps_H(t,t_0) (g^\eps_z \vec V)$ is asymptotic to a 
wave packet, it is not clear that the remainder of the approximation has a wave-packet structure as in~\eqref{robin1}. 
  \end{enumerate}
   \end{remark}

  \subsection{Frequency localized families and  $\Sigma^k_\eps$-regularity}\label{sec:freq_loc_Sigma_k}

  The size of $N_\beta$ in Definition~\ref{def:freq_loc} gives an information about the regularity of the family. 
  
  \begin{lemma}\label{lem:Nbeta_sobolev}
  Let $(\phi^\eps)_{\eps>0}$ be  a frequency localized family at the scale $\beta\geq 0$, let $C_\beta, N_\beta,\eps_\beta $  are the constants associated by Definition~\ref{def:freq_loc}. Assume   that $k\in\N$ is such that $N_\beta> d+k$, then $(\phi^\eps )_{0<\eps<\eps_\beta}$ is uniformly bounded in $\Sigma^k_\eps$ and there exists $c>0$ independent of $\eps$
  \[
  \| \phi^\eps\|_{\Sigma^k_\eps} \leq c ( C_\beta+ \|\phi^\eps_0\|_{L^2}).
  \]
  \end{lemma}

 \begin{proof}
   This Lemma is a simple consequence of 
   Lemma~\ref{lem:Barg_Sobolev}.
 Since $N_\beta>d+k$, we have for $|z|>R_\beta$, 
 \[
 \langle z\rangle ^k| \mathcal B[\phi^\eps](z)|\leq \eps^\beta C_\beta \langle z\rangle ^{-N_\beta+k} \in L^{2}(\R^{2d}).
 \]
Moreover, $z\mapsto  \langle z\rangle ^{2k}| \mathcal B[\phi^\eps](z)|^2$ is locally integrable and we can write 
\begin{align*}
\|  \langle z\rangle ^k \mathcal B[\phi^\eps](z)\|_{L^2(\R^{2d})} ^2& \leq 
\langle R_\beta\rangle^{2k} \int_{|z|\leq R_\beta}  | \mathcal B[\phi^\eps](z)|^2 dz + \eps^{2\beta} C_\beta ^2 \int_{\R^{2d}}\langle z\rangle ^{-2(N_\beta-k)} dz\\
& \leq \langle R_\beta\rangle^{2k}\|  \mathcal B[\phi^\eps](z)\|^2 _{L^2(\R^{2d}} + \eps^{2\beta} C_\beta ^2 \int_{\R^{2d}}\langle z\rangle ^{-2(N_\beta-k)} dz\\
&\leq \langle R_\beta\rangle^{2k} \| \phi^\eps\|^2_{L^2(\R^d)}+ \eps^{2\beta} C_\beta ^2 \int_{\R^{2d}}\langle z\rangle ^{-2(N_\beta-k)} dz,
\end{align*}
whence the conclusion since the right hand side is bounded for $2(N_\beta-k)>2d$. 
 \end{proof}

   \section{Operators built on Bargmann transform}\label{subsec:Barg_ext}
   
We investigate here the properties of the operators defined in~\eqref{def:barg_op}. 
We shall investigate two cases :
\begin{enumerate}
\item[(a)] The case where the family $(\theta^\eps_z)_{\eps>0}$ is only uniformly bounded in $L^2(\R^d)$, which is a light assumption, but with uniform bounds in $z$ on adequate semi-norms or norms.
\item[(b)] The case where the family $(\theta^\eps_z)_{\eps>0}$ is a wave packet (up to a phase), which is a stronger assumption on the family. 
\end{enumerate}
The thawed/frozen approximation operators belong to the  type (b). We will consider operators of type (a) in the proofs of 
Theorems~\ref{th:sqrteps} and~\ref{thm:TGeps}, when taking for the family $(\theta^\eps_z)_{\eps>0}$ a term of rest appearing in the expansion of the action of the propagator on a Gaussian wave packet. 
\smallskip 

In Section~\ref{sec:loulou1}, we analyze the action of these operators on $\Sigma^k_\eps$ spaces. In Section~\ref{sec:loulou2}, we prove special properties of the operators corresponding to families of the  type 
 (b) involving classical quantities linked with the propagation of Gaussian wave packets by a  Schrödinger evolution.

\subsection{Action in $\Sigma^k_\eps$ of operators  built on Bargmann transform}\label{sec:loulou1}

  This section is devoted to the proof of the following result. 
 
\begin{theorem}\label{lem:Barg_ext}
Let $\eps_0>0$.
\begin{enumerate}
\item Let $R>0$. 
There exists $c_0>0$ such that for all measurable $z$-dependent family $(\theta^\eps_z)_{\eps>0}$, for all $k\in\N$, $\eps\in (0,\eps_0]$, for all $\phi\in L^2(\R^d)$ 
\[
\|\mathcal J[\theta^\eps_z \1 _{|z|<R}](\phi )\|_{\Sigma^k_\eps}\leq    (2\pi  \eps)^{-d} \, c_0
   \| \phi \| _{L^2} \, R ^{2d} \sup_{|z|\leq R} \| \theta^\eps_z\|_{\Sigma^k_\eps}.
\] 
\item Assume $\theta^\eps_z=\lambda^\eps(z){\rm WP}^\eps_{\Phi(z)} (\theta)$ with $\theta\in\mathcal S(\R^d)$,  $(\lambda^\eps)_{\eps>0}$  a bounded family in~$L^\infty(\R^{2d},\C)$ and $\Phi$ a smooth diffeomorphism of $\R^{2d}$ such that 
\[
\exists c>0,\; \exists \ell\in\N,\;\;\forall z\in\R^{2d},\;\;|J_\Phi(z) |+|J_\Phi(z)^{-1} |\leq c\langle z\rangle ^{\ell}.
\]
Then, there exists $c'_0>0$ such that for all  $\phi\in L^2(\R^d)$, $k\in\N$, $\eps\in (0,\eps_0]$,
\[
\|\mathcal J[ \theta^\eps_z](\phi)\|_{\Sigma^k_\eps}\leq  c'_0\, \| \lambda_\eps\|_{L^\infty}  \| \phi\| _{L^2} 
\| \theta\|_{\Sigma^{k+\ell+2d+1}}.
\] 
\end{enumerate}
\end{theorem}

These results also hold for the adjoint of ${\mathcal J}[ \theta^\eps_z]$:
\begin{equation}\label{eq:tildeJ}
{\mathcal J}[ \theta^\eps_z]^*: \phi\mapsto (2\pi\eps)^{-d} \int_{\R^{2d} }\langle \theta^\eps_z, \phi \rangle g^\eps_z dz.
\end{equation}

\begin{corollary}\label{cor:tildeJ}
Under the assumptions of Theorem~\ref{lem:Barg_ext} (2), the family of operators ${\mathcal J}[ \theta^\eps_z]^*$ given by in~\eqref{eq:tildeJ}
satisfies the same kind of estimates than the family ${\mathcal J}[ \theta^\eps_z]$.
\end{corollary}

A straightforward consequence of Theorem~\ref{lem:Barg_ext} and of Lemme~\ref{lem:freq_loc} is given in the next statement. 

\begin{corollary}\label{cor:J_R}
Assume $(\theta^\eps_z)_{\eps>0}$ satisfies the assumptions of Theorem~\ref{lem:Barg_ext}. Let $(\Phi^\eps)_{\eps>0}$ be a frequency localized family at the scale $\beta\geq 0$ and $C_\beta>0$, $N_\beta\in\N$ be the constants associated with Definition~\ref{def:freq_loc}.   Then, for all $k>0$ such that $N_\beta >k+d$, there exists a constant $c_k$ such that for all  $R>0$, 
\[
\| \mathcal J[\theta^\eps_z] (\phi^\eps -\phi^\eps_{R,<})\|_{\Sigma^\eps_k} \leq c_k \, C_\beta\,\eps^\beta 
R^{-(N_\beta -k-d)} .
\]
where the family $(\phi^\eps_{R,<})_{\eps>0}$ is introduced in~\eqref{def:phi_R_<}.
\end{corollary}

\begin{proof} [Proof of Theorem \ref{lem:Barg_ext}]

 Let us first prove the $L^2$-estimate ($k=0$) in both cases. 

\noindent (1)  By the Cauchy-Schwarz inequality, for $x\in\R^d$, we have  
 \begin{align*} 
\| \mathcal J[\theta^\eps_z\1 _{|z|\leq R}](\phi^\eps) \|_{L^2}^2
 & \leq (2\pi \eps)^{-2d} \| \phi^\eps\|_{L^2} ^2  \int_{|z|,|z'|\leq R } \left| \int_{x\in\R^d} \theta^\eps_z(x) \overline{ \theta^\eps _{z'}}(x) dx\right| \, dz\, dz' \\
&  \leq (2\pi \eps)^{-2d} \| \phi^\eps \|_{L^2} ^2  \int_{|z|,|z'|\leq R }\| \theta^\eps_z\|_{L ^2} \|   \theta^\eps _{z'}\|_{L^2}   \, dz\, dz'\\
&  \leq c_1 R^{4d} (2\pi \eps)^{-2d} \| \phi^\eps\|_{L^2} ^2 \left(\sup_{|z| \leq 2R} \| \theta^\eps_z\|_{L ^2} \right)^2
\end{align*}
where $c_1>0$ is a universal constant.
\smallskip

\noindent (2) Let $(x,y)\mapsto k^\eps(x,y)$ be the integral kernel of the operator~$\mathcal J[\theta^\eps_z]$.
Since the Bargmann transform is an isometry, it is equivalent to consider the operator $\mathcal B\circ \mathcal J[\theta^\eps_z]\circ  \mathcal B^{-1}$, the kernel of which is the function
 $(\R^{2d})^2\ni (X,Y)\mapsto  k^\eps_{\mathcal B}(X,Y)$ defined by 
\begin{align*}
 k^\eps_{\mathcal B}(X,Y)&  =(2\pi\eps)^{-d} \int _{\R^{2d}} \overline g^\eps_X(x) g^\eps_Y(y) k^\eps(x,y) dxdy = (2\pi\eps)^{-2d} \int_{z\in\R^{2d} }  \langle g^\eps_z, g^\eps_Y\rangle \langle g^\eps_X,\theta^\eps_z\rangle dz.
 \end{align*}
Therefore, by~\eqref{def:W(f,g)}, $ k^\eps_{\mathcal B}(X,Y)$ satisfies
\[
| k^\eps_{\mathcal B}(X,Y)|\leq (2\pi\eps)^{-2d}\int_{z\in\R^{2d} } \left| \lambda^\eps(z) W[g^{i\1_{\R^d}},g^{i\1_{\R^d}}]\left(\frac{Y-z}{\sqrt\eps}\right) W[g^{i\1_{\R^d}},\theta]\left(\frac{\Phi(z)-X}{\sqrt\eps}\right)\right|dz.
 \]
 We deduce 
 \begin{align*}
 \int_{\R^{2d}} |k^\eps_{\mathcal B}(X,Y)| dX & \leq   M \| \lambda^\eps\|_{L^\infty} \left( \int_{\R^{2d}} |W[g^{i\1_{\R^d}},g^{i\1_{\R^d}}](z)|dz\right) \left( \int_{\R^{2d}} |W[g^{i\1_{\R^d}},\theta](X)|dX\right),\\
  \int_{\R^{2d}} |k^\eps_{\mathcal B}(X,Y)| dY & \leq M  \| \lambda^\eps\|_{L^\infty} \left( \int_{\R^{2d}} |W[g^{i\1_{\R^d}},g^{i\1_{\R^d}}](Y)|dY\right) \left( \int_{\R^{2d}} |W[g^{i\1_{\R^d}},\theta](z)J_\Phi^{-1} (z) | dz\right) ,
 \end{align*}
   with $M=\sup_{\eps\in(0,1]}\|\lambda^\eps\|_{L^\infty}$, and, by equations \eqref{estimate_wigner},  we deduce the existence of $C>0$ such that   
 \begin{align*}
 \int_{\R^{2d}} |k^\eps_{\mathcal B}(X,Y)| dX +   \int_{\R^{2d}} |k^\eps_{\mathcal B}(X,Y)| dY  &
\leq  C M  \| \theta \|_{\Sigma^{2d+\ell+1}}.
 \end{align*}
 We  then conclude by Schur Lemma
and obtain
 \[
 \|\mathcal B\circ \mathcal J[\theta^\eps_z]\circ  \mathcal B^{-1}\|_{\mathcal L( L^2(\R^{2d}))} \leq C M\  \| \lambda^\eps\|_{L^\infty} \, \| \theta \|_{\Sigma^{2d+\ell+1}},
 \]
 and so it is for $\mathcal J[\theta^\eps_z]$. 
 \smallskip

\noindent (3) For concluding the proof when $k\not=0$, we   use  that for $\alpha,\gamma\in\N^d$  and $\phi\in\mathcal S(\R^d)$,
 \[
 x^\alpha (\eps\partial_x)^\gamma \mathcal J[\theta^\eps_z]= \mathcal J[  x^\alpha (\eps\partial_x)^\gamma\theta^\eps_z],
\]
and the additional  observation
\[ 
x^\alpha(\eps\partial_x)^\gamma{\rm WP}^\eps_z(\theta)= {\rm WP}^\eps \left( (q+\sqrt\eps x)^\alpha (p+\sqrt\eps D_x)^\gamma \theta\right).
\]
We then conclude by observing that, as in the estimate~\eqref{norm:g_alpha_gamma}, we have for all $n\in\N$,
\[
\| (q+\sqrt\eps x)^\alpha (p+\sqrt\eps D_x)^\gamma \theta\|_{\Sigma^n}\leq \langle z\rangle ^k\| \theta\|_{\Sigma^{n+k}}.
\]
This finishes the proof. 
\end{proof}

Theorem~\ref{lem:Barg_ext} has consequences for the thawed/frozen approximation operators introduced in Chapter~\ref{chapter:introduction}.
We extend the notation $\mathcal J[\vec \theta^\eps_z] $ to vector valued functions $\vec \theta^\eps_z$ by considering the operator coordinate by coordinate.

 \begin{corollary}\label{cor:boundedness}
 Assume the Hamiltonian $H^\eps=H_0+\eps H_1$ satisfies Assumptions~\ref{hyp:codim1} and~\ref{hyp:growthH}.  Let $k\in\N$ and $t\in I$.
 \begin{enumerate}
\item 
 The families of operators $\left(\mathcal J_{\ell, \vec V, {\rm th/fr}}^{t,t_0}\right)_{\eps>0}$ 
  defined in~\eqref{def:J_thawed} and \eqref{def:J_frozen} 
  are bounded families in 
 $\mathcal L(L^2(\R^d,\C^m), \Sigma^k_\eps(\R^d,\C^m))$. 
 \item  Assume moreover that the compact $K$ satisfies Assumptions~\ref{hyp:compact}. Then,  the family of operators  $\left(\mathcal J_{1,2, \vec V, {\rm th/fr}}^{t,t_0}\right)_{\eps>0}$ defined in~\eqref{def:J12th}  and~\eqref{def:J12fr} are  bounded families  in  the space 
 $\mathcal L(L^2(\R^d,\C^m),\Sigma^k_\eps(\R^d,\C^m))$. 
 \end{enumerate}
\end{corollary}

\begin{remark}
If one assumes that $(t,z)\mapsto \partial_t f+\{v,f\}$ is bounded from below and $\partial_t f$ is bounded, then one can replace the compact $K$ by $\R^{2d}$ in the definition of $\mathcal J_{1,2,\vec V, {\rm th}}^{t,t_0}$ and one obtains  a  bounded family in 
 $\mathcal L(L^2(\R^d,\C^m),\Sigma^k_\eps(\R^d,\C^m))$. 
\end{remark}

 \begin{proof} Let $\ell\in\{1,2\}$. Let us first discuss $\mathcal J_{\ell,\vec V, {\rm th}}^{t,t_0}$.
 We write $\mathcal J^{t,t_0} _{\ell, \vec V, {\rm th}}= \mathcal J[\theta^\eps_z]$ with 
 \begin{align*}
 \theta^\eps_z & = \lambda^\eps(z) {\rm WP}^\eps_{\Phi_\ell^{t,t_0}(z) } (g^\Gamma_\ell (t,t_0,z))\;\;\mbox{and}\;\;
 \lambda^\eps (z)  ={\rm e}^{\frac i\eps S_\ell(t,t_0,z)} \vec V_\ell (t,t_0,z).
 \end{align*}
 We observe that for all $t\in I$ and $z\in\R^{2d}$,
 \[
 \|\vec  V_\ell(t,t_0,z)\|_{\C^m} = \| \vec V_\ell(t_0,t_0,z)\|_{\C^m} = \| \pi_\ell(t_0)\vec V\|_{\C^m} \leq \| \vec V\|_{\C^{m,m}} 
 \]
 which is independent of $z$.
 Therefore,   the family $(\lambda^\eps)_{\eps>0}$ is bounded in $L^\infty(\R^{2d})$.
 
 Besides, by Proposition~\ref{lem:growth_eigen_bis}, the flow map $(t,z)\mapsto \Phi^{t,t_0}_{h_\ell}(z)$ satisfies the assumptions of (2) of Theorem~\ref{lem:Barg_ext}. Similarly, the map $(t,z) \mapsto \Gamma_\ell (t,t_0,z)$ is bounded on $ I \times \R^{2d}$. Therefore, for any $N\in\N$, there exists $c_{N,t_0,T}>0$ such that 
\[
\forall t\in I, \;\;  \| x^\alpha\partial_x^\beta g^\Gamma_\ell (t,t_0,\cdot )\|_{\Sigma_1^N}\leq c_{N,t_0,T}.
\]
 We then conclude by  
  (2) of Theorem~\ref{lem:Barg_ext}.
    The proof for $\mathcal J_{\ell, {\rm fr}}^{t,t_0}$ follows exactly the same lines.  
  \smallskip 
  
The proof for $\mathcal J_{1,2,\vec V, {\rm th/fr}}^{t,t_0}$ requires additional observations. 
  We need to consider  the transition coefficient map $(t,z)\mapsto \tau_{1,2}(t,t_0,z)$ (see~\eqref{def:tau}) and the matrix-valued maps  $z\mapsto \Gamma^\flat(t_0,z)$ (see~\eqref{def:Gammaflat}), which requires the analysis of the function parametrizing the crossing (see~\eqref{def:mu} and~\eqref{def:alpha_beta}),
   \begin{equation}\label{minnie1}
  z\mapsto\left( \alpha^\flat(t_0,z), \beta^\flat(t_0,z), \mu^\flat(t_0,z)\right).
  \end{equation}
By the condition~\eqref{eq:growth_trace} of Assumption~\ref{hyp:codim1}, with $n_0=0$, the derivatives  of $(t,z)\mapsto f(t,z)$ are uniformly bounded in $z$. Moreover,   if one takes $z$ in a compact $K$ that satisfies Assumptions~\ref{hyp:compact}, one has the additional properties that $\partial_t f$ and $\mu^\flat$  are bounded below. 
As a consequence $z\mapsto \alpha^\flat(t_0,z)$, $z\mapsto\beta^\flat(t_0,z)$ and $z\mapsto\mu^\flat(t_0,z)$ are bounded functions on $\R^{2d}$ for all $t\in I$, 
the map defined in~\eqref{minnie1} is smooth. One then argues  as before by including the  coefficient $\tau_{1,2}(t,t_0,z)$ in the definition of $\lambda^\eps$
and the result follows from   Theorem~\ref{lem:Barg_ext}~(2).
 \end{proof}

 \subsection{Some properties of operators built on Bargmann transform via families with wave packet structure}\label{sec:loulou2}

In this section we analyze the properties of the  operators $\mathcal J[\theta^\eps_z]$ when $(\theta^\eps_z)_{\eps>0}$ is of the form 
\begin{equation}\label{eq:formula1}
  \theta^\eps_z= {\rm e}^{\frac i\eps S(z)} u(z) {\rm WP} ^\eps_{\Phi(z)}(\theta(z,\cdot)),
  \end{equation}
     with $\theta\in\mathcal C^\infty (\R^{2d}, \mathcal S(\R^d))$, $S\in\mathcal C^\infty(\R^{2d},\R)$,  $u\in\mathcal C^\infty(\R^{2d},\C)$ and $\Phi$ a smooth diffeomorphism satisfying the assumptions of Theorem~\ref{lem:Barg_ext}. We are interested  in the case where~$S$ and $\Phi$ are linked in the same manner as when they are the flow map and the action associated with classical trajectories.
    Therefore,  we consider the following set of Assumptions. 
    
    \begin{assumption} \label{assumption:SPhi}
    Let $S\in\mathcal C^\infty(\R^{2d}_z,\R)$,  $u\in\mathcal C^\infty(\R^{2d}_z,\C)$ and $\Phi$ a smooth diffeomorphism. We assume the following properties:
    \begin{enumerate}
    \item[{\rm (i)}] There  exists $c>0$ and $\ell\in\N$ such that 
    \[
     \forall z\in\R^{2d},\;\;|J_\Phi(z)|+|J_\Phi(z)^{-1}|\leq c\langle z\rangle ^\ell.
    \]
  \item[{\rm (ii)}]  Setting $\Phi(z)=( \Phi_q(z), \Phi_p(z))$ and 
 \[ \partial_z\Phi =  \begin{pmatrix}  A(z) &B(z)\\ C(z) &D(z)\end{pmatrix},\]
 we have
 \begin{equation*}\label{eq:diff_S}
 \nabla_q S(z)= -p + A (z) \Phi_p(z) \;\;\mbox{and}\;\; \nabla_p S(z) = B(z) \Phi_p(z),\;\; z=(q,p).
 \end{equation*}
 \item[{\rm (iii)}] For all $k\in\N$, the quantities $ \sup_{z\in\R^{2d}}\|\theta(z,\cdot)\|_{\Sigma^k} $,  
$\sup_{|\alpha|\leq k}\|\partial_z^\alpha S\|_{L^\infty} $ and $\sup_{|\alpha|\leq k}\|\partial_z^\alpha u\|_{L^\infty} $,
 are uniformly bounded in $z$. 
 \end{enumerate}
 \end{assumption}

The next technical lemma will be useful for proving our main results. It contains all the information needed to pass from the thawed approximation to the frozen one.

 \begin{lemma}\label{lem:mathfrak_d}
 Let  $\mathfrak d=\partial_q-i\partial_p$. 
Let  $\theta\in\mathcal C^\infty (\R^{2d}, \mathcal S(\R^d))$,  $S\in\mathcal C^\infty(\R^{2d},\R)$,  $u\in\mathcal C^\infty(\R^{2d},\C)$ and~$\Phi$  a smooth diffeomorphism satisfying Assumptions~\ref{assumption:SPhi}.
 Then, the following equality between operators in $\mathcal L(L^2(\R^d), \Sigma^k_\eps)$ holds for $k\in\N$:
\[
 \mathcal J \left[u\, {\rm e}^{\frac i\eps S}\,  {\rm WP}^\eps_{\Phi}\left((\mathfrak d \Phi_p x-\mathfrak  d \Phi_q D_x)\theta\right)\right]
 =- i\sqrt\eps\, \mathcal J \left[ \mathfrak d u\,  {\rm e}^{\frac i\eps S} \,{\rm WP}^\eps_{\Phi}(\theta)\right]
 -i\sqrt\eps\, \mathcal J \left[ u \, {\rm e}^{\frac i\eps S} \,{\rm WP}^\eps_{\Phi}( \mathfrak d \theta)\right].
\]
 \end{lemma}
 
 Note that with the notation of Lemma~\ref{lem:mathfrak_d} and Assumptions~\ref{assumption:SPhi}, we have 
 \begin{equation}\label{eq:mathfrak_Phi}
 \mathfrak d \Phi_p(z)= C(z) -i D(z)\;\;\mbox{and}\;\;  \mathfrak d \Phi_q(z)= A(z) -i B(z).
 \end{equation}
Besides, condition (ii) of Assumption~\ref{assumption:SPhi} implies that the equality of Lemma~\ref{lem:mathfrak_d} holds formally. The conditions (i) and (iii) ensure the boundedness of the operators involved in the estimates.

 \begin{proof}
 The integral kernel of the operator $\mathcal J\left[ u \, {\rm e}^{\frac i\eps S}\, {\rm WP}^\eps_{\Phi}( \theta)\right]$ is the 
 function 
 \[(x,y)\mapsto \int_{z\in\R^{2d}} k(z,x,y) dz\]
 defined by 
 \[
 k(z,x,y)= u(z,x) {\rm e}^{\frac i\eps S(z)} \overline{g_z^\eps(y)} {\rm WP}^\eps_{\Phi(z)} (\theta(z,\cdot))(x),\;\;(x,y)\in\R^d,\;\;z\in\R^{2d}. 
 \]
 We aim at calculating $\mathfrak d k$. We  observe for $z=(q,p)\in\R^{2d}$, $y\in\R^d$ and 
 \begin{align*}
 \mathfrak d S(z) & =-p+(A(z)-iB(z)) \mathfrak d  \Phi_p(z),\\
 \mathfrak d\left(\overline{ g^\eps_z(y)}\right) & =\frac i\eps [ (\mathfrak  d p +i\mathfrak d q) (y-q)+p\mathfrak d q ]g^\eps_z(y)  =\frac i\eps p\,  g^\eps_z(y),\\
 \mathfrak d \left({\rm WP}^\eps_{\Phi(z)}(\theta(z,\cdot)) \right) &= {\rm WP}^\eps_{\Phi(z)}(\mathfrak d \theta(z,\cdot)) 
 +\frac i{\sqrt\eps} {\rm WP}^\eps_{\Phi(z)}\left((\mathfrak d \Phi_p (z)x-\mathfrak  d\Phi_q(z) D_x)\theta(z,\cdot)\right)\\
 & \;\; -\frac i\eps {\rm WP}^\eps_{\Phi(z)}\left( (A(z)-iB(z))\mathfrak d\Phi_p(z )\theta(z,\cdot)\right).
 \end{align*}
 We obtain
 \begin{align*}
 \mathfrak d k(z,x,y) &= {\rm e}^{\frac i\eps S(z)}\Bigl(\mathfrak d u(z,x) \, \, \overline{g_z^\eps(y)} \, {\rm WP}^\eps_{\Phi(z)}(\theta(z,\cdot))(x)
 +  u (z,x) \,  \overline{g_z^\eps(y)} \, {\rm WP}^\eps_{\Phi(z)} ( \mathfrak d \theta(z,\cdot))(x)\\
&  \qquad \qquad +\frac {i}{\sqrt\eps} u(z,x)\, \overline{g_z^\eps(y)} \,  {\rm WP}^\eps_{\Phi(z)}\left((\mathfrak  d\Phi_p(z) x-\mathfrak d\Phi_q (z) D_x)\theta(z,\cdot)\right)(x)\Bigr)
 \end{align*}
  The result then follows from the integration in $z\in\R^{2d}$. 
 \end{proof} 
 
 The case of Gaussian functions $\theta$ is of particular interest. Indeed, if $\theta(z,\cdot)=g^{\Theta(z)}$ with  $\Theta\in\mathcal C^\infty(\R^{2d},\mathfrak S^+(d))$, we have for $x\in\R^d$ and $z\in\R^{2d}$,
  \begin{equation}\label{eq:formule2}
 (\mathfrak d \Phi_p (z) x-\mathfrak  d \Phi_q (z) D_x) g^{\Theta(z)}(x) =  (\mathfrak d \Phi_p (z) -\mathfrak  d \Phi_q(z) \Theta (z) ) x \, g^{\Theta(z)}(x).
\end{equation}
We set 
\[
M_\Theta(z):= \mathfrak d \Phi_p (z)- \mathfrak  d \Phi_q(z)  \Theta (z) .
\]
By~\eqref{eq:mathfrak_Phi}, we have the equality between matrix-valued functions 
\begin{equation}\label{def:M}
M_\Theta= (C-iD) - (A-iB) \Theta =(A-iB) \left[(A-iB)^{-1}(C-iD)-\Theta\right] .
\end{equation}
Note that  this matrix $M_\Theta$ is invertible because $(A+iB)^{-1}(C+iD)-\overline\Theta\in\mathfrak S^+(d)$ (as the sum of two elements of $\mathfrak S^+(d)$). 
These observations are in the core of the proof of the next  result which is a corollary of Lemma~\ref{lem:mathfrak_d}, when applied to Gaussian profiles.

\begin{corollary}\label{cor:chemin}
 Let $k\in\N$. Let  $\Theta\in\mathcal C^\infty(\R^{2d},\mathfrak S^+(d))$ such that $M_\Theta$ is bounded together with its inverse,
 let  $S\in\mathcal C^\infty(\R^{2d},\R)$,  $u\in\mathcal C^\infty(\R^{2d},\C)$ and $\Phi$ a smooth diffeomorphism satisfying Assumptions~\ref{assumption:SPhi}. Then, 
in $\mathcal L(L^2(\R^d),\Sigma^k_\eps(\R^d))$, we have
\begin{equation}
\label{formule de base1}
\mathcal J\left[ u\, {\rm e}^{\frac i\eps S}\,  {\rm WP}^\eps_{\Phi}\left(x g^\Theta\right)\right] = \O(\sqrt\eps).
\end{equation}
Besides, for all $L \in\mathcal C^\infty(\R^{2d}, \C^{d,d})$, in  $\mathcal L(L^2(\R^d),\Sigma^k_\eps(\R^d))$, we have
\begin{equation}
\label{formule de base2}
\mathcal J\left[ u\, {\rm e}^{\frac i\eps S}\,  {\rm WP}^\eps_{\Phi}\left(Lx\cdot x g^\Theta\right) \right]
= \frac 1 i \, \mathcal J \left[ u\, {\rm Tr} \left( L \, M_\Theta^{-1}\,\mathfrak d \Phi_q \right)\, {\rm e}^{\frac i\eps S}\,  {\rm WP}^\eps_{\Phi}\left(g^\Theta\right)\right]+\O(\eps).
\end{equation}
with 
\[
L M_\Theta^{-1}\mathfrak d \Phi_q = L \left[(A-iB)^{-1}(C-iD)-\Theta\right] ^{-1}.
\] 
\end{corollary}

\begin{remark}\label{rem:aetJ}
\begin{enumerate}
\item This result has an interesting consequence concerning the pseudodifferential calculus. Indeed, for $a$ real-valued and with bounded derivatives, in view of~\eqref{eq:tildeJ}
\[
\op_\eps(a)  = {\mathcal J}(\op_\eps(a) g^\eps_z)=  {\mathcal J} [a(z) g^\eps_z]+ \sqrt\eps {\mathcal J} \left({\rm WP}^\eps_z\left(\nabla a(z)\cdot 
\begin{pmatrix} x \\ D_x \end{pmatrix} g^{i\1_{\R^d}}\right)\right) +\O(\eps),
\]
where we have used the properties of wave packets.
The equality $\nabla g^{i\1_{\R^d}}=xg^{i\1_{\R^d}}$ allows to conclude by Corollary~\ref{cor:chemin} 
\[\op_\eps(a) =  {\mathcal J} [a(z) g^\eps_z] +\O(\eps)\] in any space $\Sigma^k_\eps$. This was already proved in~\cite{RS1}.
\item 
{
More can be said about the $\sqrt\eps$-order term on the right-hand side of~\eqref{formule de base1}. By revisiting the proof below, one sees that there exists a real-valued smooth function $z\mapsto c(z)$ such that 
\[
\mathcal J\left[ u\, {\rm e}^{\frac i\eps S}\,  {\rm WP}^\eps_{\Phi}\left(x g^\Theta\right)\right] = -i \sqrt\eps \mathcal J\left[ u\, {\rm e}^{\frac i\eps S}\,  {\rm WP}^\eps_{\Phi}\left( c(z)g^\Theta\right)\right] +\O(\eps).
\]
}
\end{enumerate}
\end{remark}


\begin{proof}[Proof of Corollary~\ref{cor:chemin}]
 One uses~\eqref{eq:formule2} and the first relation of Lemma~\ref{lem:mathfrak_d} that we apply to $\theta=g^\Theta$. It gives that in  $\mathcal L(L^2(\R^d),\Sigma^k_\eps(\R^d))$, 
\begin{align*}
 \mathcal J\left[u\, {\rm e}^{\frac i\eps S}\,  {\rm WP}^\eps_{\Phi}\left(xg^\Theta\right)\right] &= 
 \mathcal J\left[u\, {\rm e}^{\frac i\eps S}\,  {\rm WP}^\eps_{\Phi}\left( M_\Theta^{-1}(\mathfrak d \Phi_p x-\mathfrak  d \Phi_q D_x)g^\Theta\right)\right]= \O(\sqrt\eps), 
\end{align*}
whence~\eqref{formule de base1}.
\smallskip 

Secondly, with $L \in\mathcal C^\infty(\R^{2d}, \C^{d,d})$, we associate the matrix  $L'$ such that 
 $L= \,^tL' M_\Theta$. We 
have
\begin{align*}
(\mathfrak d \Phi_p x -\mathfrak d \Phi_q D_x )\cdot  (L' x g^\Theta)
&= \left( (\,^tL' (\mathfrak d \Phi_p -  \mathfrak d \Phi_q\Theta))x \cdot x - {\rm Tr} (\,^tL' \mathfrak d \Phi_q ) \right) g^\Theta.
\end{align*}
It remains to prove that in $\mathcal L(L^2(\R^d),\Sigma^k_\eps(\R^d))$, we have
\begin{equation}\label{aim}
 \mathcal J\left[u\, {\rm e}^{\frac i\eps S}\,  {\rm WP}^\eps_{\Phi}\left(
 (\mathfrak d \Phi_p x -\mathfrak d \Phi_q D_x )\cdot  (L' x g^\Theta)
 \right)\right] 
 =\O(\eps).
 \end{equation}
 We first apply
Lemma~\ref{lem:mathfrak_d}  to the function  $\theta=L'x g^\Theta$ and we write 
\begin{align*}
\nonumber
& \mathcal J\left[u\, {\rm e}^{\frac i\eps S}\,  {\rm WP}^\eps_{\Phi}\left(
 (\mathfrak d \Phi_p x -\mathfrak d \Phi_q D_x )\cdot  (L' x g^\Theta)
 \right)\right] \\
&\;\;=- i\sqrt\eps \mathcal J\left[ 
\mathfrak du\, {\rm e}^{\frac i\eps S}\,  {\rm WP}^\eps_{\Phi}\left(L'x g^\Theta\right)+ u {\rm e}^{\frac i\eps S}\,  {\rm WP}^\eps_{\Phi}\left(L'x\,\mathfrak d (g^\Theta)\right) 
\right] .
\end{align*}
We use the relation~\eqref{formule de base1} and obtain in  $\mathcal L(L^2(\R^d),\Sigma^k_\eps(\R^d))$, 
\begin{equation} \label{dodo}
\mathcal J\left[u\, {\rm e}^{\frac i\eps S}\,  {\rm WP}^\eps_{\Phi}\left(
 (\mathfrak d \Phi_p x -\mathfrak d \Phi_q D_x )\cdot  (L' x g^\Theta)
 \right)\right] =
 - i\sqrt\eps \mathcal J\left[ u {\rm e}^{\frac i\eps S}\,  {\rm WP}^\eps_{\Phi}\left(L'x\,\mathfrak d (g^\Theta)\right)\right]
+\O(\eps).
\end{equation}
We calculate
\[
\mathfrak d (g^\Theta) = \frac {\mathfrak d c_\Theta}{c_\Theta} g^\Theta + ( \mathfrak d \Theta x\cdot x ) g^\Theta,
\]
with 
\begin{align*}
(\mathfrak d \Theta x\cdot x)\, x\,  g^\Theta &= (\mathfrak d \Theta x\cdot x)  M_\Theta^{-1} (\mathfrak d\Phi_p x-\mathfrak d\Phi_q D_x)g^\Theta\\
&= M_\Theta^{-1} (\mathfrak d\Phi_p x-\mathfrak d\Phi_q D_x)\left( (\mathfrak d \Theta x\cdot x)  g^\Theta\right) - 2 M_\Theta^{-1}\mathfrak d\Phi_q \mathfrak d \Theta\, x\, g^\Theta.
\end{align*}
Therefore, there exists matrices $L_1$ and $L_2$ such that, setting $\tilde\theta =  (\mathfrak d \Theta x\cdot x)  g^\Theta$, we have 
\[
L' x\,\mathfrak d (g^\Theta)=L_1x g^\Theta + L_2 (\mathfrak d \Phi_p x -\mathfrak d \Phi_q D_x )\tilde \theta.
\]
We deduce
\[
 \mathcal J\left[u\, {\rm e}^{\frac i\eps S}\,  {\rm WP}^\eps_{\Phi}\left(
 (\mathfrak d \Phi_p x -\mathfrak d \Phi_q D_x )\cdot  (L' x g^\Theta)
 \right)\right] =-i\sqrt\eps  \mathcal J[L_1 x g^\Theta] -i\sqrt\eps  \mathcal J[ L_2(\mathfrak d \Phi_p x -\mathfrak d \Phi_q D_x )\tilde \theta]
\]
and we obtain~\eqref{aim} by 
 Lemma~\ref{lem:mathfrak_d}  applied to the function  $\tilde \theta$, and by the relation~\eqref{formule de base1},
 which concludes the proof
\end{proof}

 \subsection{Operators built on Bargmann transform via classical quantities}
 
 We now apply the results of the preceding section to the diffeomorphism $\Phi$ given by a flow map associated to a Hamiltonian $h$.  We are going to derive the results induced by Lemma~\ref{lem:mathfrak_d} and Corollary~\ref{cor:chemin} for time dependent quantities after  integration in time. We will use the resulting formula for the Hamiltonians $h_1$ and $h_2$ associated with the matrix-valued Hamiltonian~$H^\eps$.
 
 \begin{lemma}\label{lem:cor_t_dep}
 Let $k\in\N$. Let $h$ be a subquadratic Hamiltonian on $I\times \R^{2d}$, $I=[t_0,t_0+T]$. We consider  
 \begin{enumerate}
 \item   the classical quantities associated to $h$ as in Section~\ref{subsec:classical} on the interval $I$:
 \[
 z\mapsto S(t,z), \Phi^{t,t_0}(z), F(t,t_0,z),
 \]
\item a smooth  function defined on $I\times \R^{2d}$, bounded and with bounded derivatives,
$(t,z)\mapsto u(t,z)$,
\item a smooth map from $I\times \R^{2d}$ into $\mathcal S(\R^d)$,  $(t,z)\mapsto\theta(t,z)$,
\item a smooth function from $\R^{2d}$ into $I$,
 $z\mapsto t^\flat(z)$.
 \end{enumerate}
  Then, for all   $\chi\in\mathcal C^\infty_0(I)$, we have the following equality between operators  in  $\mathcal L(L^2(\R^d),\Sigma^k_\eps(\R^d))$, 
  \begin{align*}
\int_{\R} \chi(t)&  \mathcal J\left[  \1_{t>t^\flat(z)}  u(t)\, {\rm e}^{\frac i\eps S(t)}\,  {\rm WP}^\eps_{\Phi^{t,t_0}}\left((\mathfrak d \Phi_p^t x-\mathfrak  d \Phi^{t,t_0}_q D_x)\theta(t)\right)\right]dt\\
& = i\sqrt\eps \int_{\R} \mathcal J \left[  \1_{t>t^\flat(z)}  \mathfrak d u(t)\,  {\rm e}^{\frac i\eps S(t)} \,{\rm WP}^\eps_{\Phi^{t,t_0}}(\theta(t))\right]dt\\
 &+i\sqrt\eps \int_{\R} \chi(t)  \mathcal J \left[ \1_{t>t^\flat(z)}  u(t) \, {\rm e}^{\frac i\eps S(t)} \,{\rm WP}^\eps_{\Phi^{t,t_0}}( \mathfrak d \theta(t))\right]dt \\
 & -i\sqrt\eps \mathcal J\left[ \chi(t^\flat)\, \mathfrak dt^\flat \, u(t^\flat)\, \, {\rm e}^{\frac i\eps S( t^\flat)} \,{\rm WP}^\eps_{\Phi^{t^\flat,t_0}}( \theta(t^\flat))\right]. 
 \end{align*}
 \end{lemma}
 
 Note that the result of this lemma is an equality. Thus, we have not emphasized assumptions that make these operators bounded. One could for example assume global boundedness of all the quantities involve and of their derivatives, or, what would be enough, that  $\theta$ is compactly supported in $z$. 
 \smallskip 
 
 We also emphasize that the functions denoted by $(\chi\circ t^\flat)\, u(t^\flat)\, \, {\rm e}^{\frac i\eps S( t^\flat)} \,{\rm WP}^\eps_{\Phi^{t^\flat}}( \theta(t^\flat))$ is the map 
 \[
 z\mapsto \chi\left( t^\flat(z)\right)\, u\left(t^\flat(z),z\right)\, \, {\rm e}^{\frac i\eps S\left( t^\flat(z),z\right)} \,{\rm WP}^\eps_{\Phi^{t^\flat(z),t_0}(z)}\left( \theta\left(t^\flat(t),z\right)\right).
 \]
 Note also that, by construction, the flow map $\Phi^{t,t_0}$ and the action $S$ satisfy Assumptions~\ref{assumption:SPhi} (see~\cite{corobook,R,LL}).
 
 \begin{proof}
The proof follows the lines of the one of Lemma~\ref{lem:mathfrak_d}, using the relation
 \begin{equation}\label{dirac}
 \mathfrak d\left( \1_{t>t^\flat(z)}\right) = \mathfrak d t^\flat(z)  \delta (t-t^\flat(z))
 \end{equation}
that produces an additional term.
\end{proof}

As a Corollary, for Gaussian profiles, we have the following Corollary.

\begin{corollary}\label{cor_chemin_av}
 With the same assumptions as in Lemma~\ref{lem:cor_t_dep}, we additionally assume  $\theta(t)=g^{\Theta(t)}$, with $\Theta\in\mathcal C^\infty(I\times \R^{2d}, \mathcal S^+(d))$. Then,  for all $L \in\mathcal C^\infty(\R^{2d}, \C^{d,d})$, we have the following equality  in  $\mathcal L(L^2(\R^d),\Sigma^k_\eps(\R^d))$, 
 \begin{align*}
&\int \chi(t)  \mathcal J\left[ \1_{t>t^\flat(z)}  u\, {\rm e}^{\frac i\eps S(t,t_0)}\,  {\rm WP}^\eps_{\Phi^{t,t_0}}\left(Lx\cdot x g^{\Theta(t,t_0)}\right) \right] dt\\
&\;\;= \frac 1 i \, \int \chi(t)\mathcal J \left[ \1_{t>t^\flat(z)}  \tilde  u(t,t_0)\, {\rm e}^{\frac i\eps S(t,t_0)}\,  {\rm WP}^\eps_{\Phi^{t,t_0}}\left(g^{\Theta(t,t_0)}\right)\right]dt+\O(\eps)
\end{align*}
with 
$\widetilde u(t,t_0) ={\rm Tr} \left(  L \left[(A(t,t_0)-iB(t,t_0))^{-1}(C(t,t_0)-iD(t,t_0))-\Theta(t,t_0)\right] ^{-1}\right)$. 
\end{corollary}

\begin{proof}
The proof follows the lines of the one of Corollary~\ref{cor:chemin}, using the relation
\eqref{dirac}.
\end{proof}

\chapter{Convergence of the thawed and the frozen Gaussian approximations}\label{chap:5}

Our aim in this chapter  is to prove the initial value representations of  Theorems~\ref{th:sqrteps},~\ref{thm:TGeps}, \ref{thm:FGeps} and~\ref{thm:FGeps_av}.
In Section~\ref{sec:strategy}, we explain the strategy in the context of Theorem~\ref{th:sqrteps}, and prove  the thawed approximation at order $\sqrt\eps$. Then, in Section~\ref{sec:thawed}, we prove the thawed Gaussian approximation with transfer's term of  Theorem~\ref{thm:TGeps}. Finally Section~\ref{sec:th_to_fr} is devoted to the passage from thawed approximations to frozen ones, which allow to terminate the proof of Theorem~\ref{th:sqrteps}, and to obtain Theorems~\ref{thm:FGeps} and~\ref{thm:FGeps_av}.

\section{Strategy of the proofs
} \label{sec:strategy}

Our aim in this section is to prove the initial value representation of  Theorems~\ref{th:sqrteps} for the thawed Gaussian approximation. It is also the opportunity  to explain the overall strategy that is also used for proving Theorems~\ref{thm:TGeps}, \ref{thm:FGeps} and~\ref{thm:FGeps_av}.
\smallskip

Let $k\in\N$.
Let $\psi^\eps_0=\widehat{\vec V} \phi^\eps_0$ be as in Assumption~\ref{hyp:data} with $(\phi^\eps_0)_{\eps>0}$ a frequency localized family of order $\beta\geq 0$ with  $N_\beta>d+k$. By Lemma~\ref{lem:Nbeta_sobolev}, the latter condition which implies the boundedness of the family  $(\phi^\eps_0)_{\eps>0}$  in $\Sigma^k_\eps$.
Without loss of generality, we assume 
\[
\vec V(z)=\pi_\ell(t_0,z) \vec V(z),\;\;
\forall z\in\R^{2d},
\]
for some $\ell\in\{1,2\}$ that is now fixed.
\smallskip

We start with the Gaussian frame equality~\eqref{eq:gaussian_frame}
\[
 \psi^\eps_0  = (2\pi\eps)^{-d}\int_{z\in\R^{2d}}\< g^\eps_z, \widehat{\vec{V}}\phi^\eps_0\>  g^\eps_z \, dz.
 \]
 Writing 
 $\< g^\eps_z, \widehat{\vec{V}}\phi^\eps_0\> =\<\widehat{\vec{\overline V}}
 g^\eps_z, \phi^\eps_0\> $ and formula~\eqref{eq:tildeJ}, we have in $\Sigma^k_\eps$,
 \begin{align*}
 \psi^\eps_0 & = \mathcal J[ \widehat{\vec{\overline V}}
 g^\eps_z]^*(\phi^\eps_0) + \O(\eps \|\phi^\eps_0\|_{L^2}).
 \end{align*}
By  Remark~\ref{rem:aetJ} (1), we obtain in $\Sigma^k_\eps$,
 \begin{align*}
 \psi^\eps_0 &=
  \mathcal J[ \vec{\overline V}(z)
 g^\eps_z]^*(\phi^\eps_0) + \O(\eps\|\phi^\eps_0\|_{L^2})
 \end{align*}
 In view of the formula~\eqref{eq:tildeJ}, we also have 
  \begin{align*}
 \psi^\eps_0 
&= \mathcal J[ \vec{V}(z)
 g^\eps_z](\phi^\eps_0) + \O(\eps\|\phi^\eps_0\|_{L^2}).
 \end{align*}
 Corollary~\ref{cor:J_R} yields that, in  $\Sigma^k_\eps$, we have 
  \begin{align*}
 \psi^\eps_0 & =
 \mathcal J[ \vec{ V}(z)
 g^\eps_z]((\phi^\eps_0)_{R,<}) + \O(\eps\|\phi^\eps_0\|_{L^2}) +\O(\eps^\beta C_\beta R^{-n_\beta})
 \end{align*}
 with the notations of  Corollary~\ref{cor:J_R} and setting $n_\beta=N_\beta-k-d>0 $.
\smallskip

Now that the data has been written in a convenient form, we apply the propagator $\U^\eps_H(t,t_0)$ and we take advantage of its boundedness in $\mathcal L (\Sigma^k_\eps)$ to write 
\begin{align*}
\U^\eps_H(t,t_0) \psi^\eps_0 & = (2\pi\eps)^{-d}\int_{|z|\leq R }\< g^\eps_z,\phi^\eps_0\> \left(\U^\eps_H(t,t_0) \vec{V}(z) g^\eps_z\right)dz +\O(\eps\|\phi^\eps_0\|_{L^2})+\O(\eps^\beta\,  C_\beta \,R^{-n_\beta}  )  \\
&= \mathcal J\left[\U^\eps_H(t,t_0) {\vec{V}}(z) g^\eps_z\right]((\phi^\eps_0)_{R,<})  +\O(\eps\|\phi^\eps_0\|_{L^2})+\O(\eps^\beta\,  C_\beta \,R^{-n_\beta}  ) .
\end{align*} 
We then use the 
description of the propagation of wave packets by $\U^\eps_H(t,t_0)$, as stated in Theorem~\ref{th:WPmain}. We choose $\varsigma\in (0,1)$, $\delta= \eps^{\frac 12-\varsigma}$, $M$  and $N$ such that 
\begin{equation}\label{choice:delta}
\left(\frac {\sqrt\eps}{\delta} \right)^{N+1}\delta^{-2\kappa_0-k} + \delta^M\leq \eps ^{N'}
\end{equation}
where $N'\in N$ with $N'\geq d+1$. 
Then,  in $\Sigma^k_\eps(\R^d)$, we have
\[
\U^\eps_H(t,t_0){\vec{V}}(z) g^\eps_z=\psi^{\eps,N}_\ell (t) +\O(\eps^{N'}).
\]
Therefore, by (1) of Theorem~\ref{lem:Barg_ext} in  $\mathcal L(L^2(\R^d),\Sigma^k_\eps(\R^d,\C^m))$, 
\begin{equation}\label{eq:minnie5}
\U^\eps_H(t,t_0) \psi^\eps_0=
  \mathcal J\left[\1_{|z|\leq R}\, \, \psi^{\eps,N}_\ell (t)\right] + \O(\eps^{N-d}\,R^d\,\|\phi^\eps_0\|_{L^2}) +\O(\eps^\beta\,  C_\beta \,R^{-n_\beta}  ) 
  \end{equation}
  For simplifying notation, we write $N'=N$.
  \smallskip 

Using that  $\psi^{\eps,N}_\ell (t)$  is a linear combination of wave packets and considering the explicit formula of Theorem~\ref{th:WPmain}, Point (2)  of Theorem~\ref{lem:Barg_ext} implies that  in $\mathcal L(\Sigma^k_\eps(\R^d,\C^m))$, 
\begin{align*}
\U^\eps_H(t,t_0) \psi^\eps_0& =
  \mathcal J_{\ell,\vec V, {\rm th}}^{t,t_0} \left((\phi^\eps_0)_{R,<}\right) + \O(\sqrt \eps \|\phi^\eps_0\|_{L^2}) +\O(\eps^{N-d}\,R^d\,\|\phi^\eps_0\|_{L^2}) +\O(\eps^\beta\,  C_\beta \,R^{-n_\beta}  ) \\
  & =
  \mathcal J_{\ell,\vec V, {\rm th}}^{t,t_0} \left(\phi^\eps_0\right) + \O(\sqrt \eps \|\phi^\eps_0\|_{L^2}) +\O(\eps^{N-d}\,R^d\,\|\phi^\eps_0\|_{L^2}) +\O(\eps^\beta\,  C_\beta \,R^{-n_\beta}  ) 
  \end{align*}
  where we have used again Corollary~\ref{cor:J_R}. 
  If $\beta<\frac 12$, we perform
 an appropriate choice of $R$ and $N$: we choose $R=\eps^{-\gamma}$ with $\gamma \geq \frac  1 {n_\beta}(\frac 12-\beta)$, and $N\geq \frac 12 +d(1+\gamma)$. 
  \smallskip
  
    At this stage of the description, the thawed Gaussian approximation of   Theorem~\eqref{th:sqrteps} is proved. For obtaining the frozen one, we shall  argue as  in the scalar case considered in \cite{R} (Lemma 3.2 and Lemma 3.4). We will detail this argument later  in Section~\ref{sec:th_to_fr} below. 
  \smallskip 
  
  The proofs of the order $\eps$ approximations of Theorems~\ref{thm:TGeps}, \ref{thm:FGeps} and \ref{thm:FGeps_av} start with the same lines. However, one includes in the  approximation the two first terms of the asymptotic expansion of  $\psi^{\eps,N}_\ell (t)$: the one of order~$\eps^0$ and the one of order~$\eps^{\frac 12}$. The terms of order $\sqrt\eps$ are twofold:
  \begin{enumerate}
\item[(i)]  The one along the same mode as the initial data, here denoted by $\ell$. This term will be proved to be of lower order because its structure  allows to use the first part of Corollary~\ref{cor:chemin}. 
\item[(ii)]The one generated by the crossing along the other mode. This one is not negligible and involves the transfer's term.
\end{enumerate}
At that stage of the proofs, one will be left with the thawed approximation. 
\smallskip

The derivation of the frozen approximation from the thawed one involves the second part of Corollary~\ref{cor:chemin}.   However, complications are induced in the treatment of the transfer's term described in (ii) above because of the singularity in time that it contains. This difficulty is overcome by averaging in time and using Corollary~\ref{cor_chemin_av}.
We implement this strategy in the next sections.

  \section{Thawed Gaussian approximations with transfer terms }\label{sec:thawed}

  We prove here the higher order approximation of Theorem~\ref{thm:TGeps} for initial data $\psi^\eps_0=\widehat{\vec V }\phi^\eps_0$ with $(\phi^\eps_0)_{\eps>0}$ frequency localized at the scale $\beta\geq 0$ in a compact set $K$. As in the preceding section, we assume $\vec V=\pi_\ell (t_0) \vec V$ and, without loss of generality, we suppose $\ell=1$. 
  \smallskip
  
  We start as in the preceding section and transform equation~\eqref{eq:minnie5} by taking the terms of order~$\eps^0$ and~$\eps^{\frac 12}$ in the expansion of $\psi_\ell^{\eps, N}$ (see Theorem~\ref{th:WPmain}).
  We obtain  
\begin{align*}
  \U^\eps_H(t,t_0) \psi^\eps_0(x) =&  \mathcal J \left[  \psi^{\eps,1}_1(t)+ \psi^{\eps,1}_2(t)\right] ((\phi^\eps_0)_{R,<})
   +  \O\left(\eps^\beta\,  C_\beta \,R^{-n_\beta}  \right) \\
   &+\O(\eps^{N-d} R^d \|\psi^\eps_0 \|_{L^2}) 
   +  \O(\eps \|\psi^\eps_0 \|_{L^2}) .
\end{align*}
  The rest in $\O(\eps^{N-d} R^d \|\psi^\eps_0 \|_{L^2})$ comes from the remainder of the approximation of  $\U^\eps_H(t,t_0) g^\eps_z $ (case~(1) of Theorem~\ref{lem:Barg_ext}), while the term $\O(\eps\|\psi^\eps_0\|_{L^2})$ is generated by the terms of order $\eps^j$ for $j\geq 1$ of the approximation, these terms having a wave packet structure (case (2) of Theorem~\ref{lem:Barg_ext}).
  \smallskip 

{  
We write for $\ell\in\{1,2\}$, taking $N=1$ in Theorem~\ref{th:WPmain},
\[
\psi^{\eps,N=1}_\ell(t)=\sum_{j=0}^1 \eps^{\frac j2}   \psi^{\eps,N=1}_{\ell,j}(t).
\]
We make the following observations:
\begin{enumerate}
\item[(i)]
Because the assumptions on $K$ induce that there is only one passage through the crossing, and because $\vec V=\pi_1(t_0) \vec V$, Theorem~\ref{th:WPmain} implies 
that  $\psi^{\eps,N=1}_{2,0}(t)=0$ and $\psi^{\eps,N=1}_{2,1}(t)$ only depends on the transfer profile $f^\eps_{1\rightarrow 2}$.
\item [(ii)]
For the mode 
$\ell=1$, the term of order $\sqrt\eps$
involves 
\[
  \psi^{\eps,N=1}_{\ell=1,j=1}(t)= 
 {\rm e}^{\frac i\eps S_\ell(t,t_0,z_0)}
\wp^\eps_{z_\ell (t)}\left({\mathcal R}_\ell (t,t_0,z_0) \, {\mathcal M}[F_\ell( t,t_0,z_0)] \, \Vec B_{\ell,j}(t) g^{i\1_{\R^d}}\right).
\]
\end{enumerate}

We first prove that the operator $\mathcal J \left[   \psi^{\eps,1}_{1,1}(t)\right]$ has a contribution of lower order, i.e. of order $\O(\sqrt\eps)$. 
In that purpose, 
we analyze the  form $\psi^{\eps,1}_{1,1}(t)$ and we show that we can  use Corollary~\ref{cor:chemin}. These ideas are inspired from ~\cite{Rcimpa}[section 3] and the reader can also refer to the book~\cite{corobook}. 
We  consider the  term  
\[
{\mathcal R}_\ell (t,t_0) \, {\mathcal M}[F_\ell( t,t_0,z_0)] \,  \Vec B_{\ell,1}(t) g^{i\1_{\R^d}}.
\]
The idea is to analyze the effect of each of the operation performed on the Gaussian $g^{i\1_{\R^d}}$.
The operator $\vec B_{\ell,1}$ is given by~\eqref{B1}. It is a differential operators which is a linear combination of terms of the form $x^\alpha \partial_x^\beta$ with $|\alpha|+|\beta|=3$, and of operators of multiplication by $x_j$ or of derivation $\partial_{x_j}$, $1\leq j\leq d$. When applied to $g^{i\1_{\R^{d}}}$,  we obtain a linear combination of terms of the form $x_jg^{i\1_{\R^{d}}}(x)$ for $1\leq j\leq d$, or  $x^\alpha g^{i\1_{\R^{d}}}(x)$ for $|\alpha|=3$. One then applies the metaplectic operator ${\mathcal M}[F_\ell( t,t_0,z_0)]$ to this function. By~\eqref{prop:metaplectic}, if $P$ is a polynomial function,  
\[
{\mathcal M}[F_\ell( t,t_0,z_0)] (P g^{i\1_{\R^{d}}})(x)= P(A_\ell(t,t_0,z_0)x){\mathcal M}[F_\ell( t,t_0,z_0)]g^{i\1_{\R^{d}}}(x)
\]
where $A_\ell(t,t_0,z_0)$ is the upper block-diagonal coefficient of $F_\ell(t,t_0,z_0)$ (see~\eqref{eq:F}).
We deduce that 
\[ {\mathcal R}_\ell (t,t_0,z_0) \, {\mathcal M}[F_\ell( t,t_0,z_0)] \,  \Vec B_{\ell,1}(t) g^{i\1_{\R^d}}(x)= \vec a(t) x g^{\Gamma_\ell(t,t_0,z_0)}(x)\]
for some smooth and bounded vector-valued map $(t,z)\mapsto \vec a(t,z)$. 
Therefore, we can apply Corollary~\ref{cor:chemin}. It yields 
\[
 \mathcal J \left[\psi^{\eps,1}_{1,1}(t)
 \right]  ( (\phi^\eps_0)_{R,<})= \O(\sqrt\eps \|\phi^\eps_0\|_{L^2}).
\]
\smallskip 

We can now conclude the proof. Indeed, we
have now only two terms in the expansion and we are left with 
\begin{align*}
  \U^\eps_H(t,t_0) \psi^\eps_0(x)
   &= \mathcal J \left[  \psi^{\eps,1}_{1,0}(t)+\sqrt\eps\, \psi^{\eps,1}_{2,1}(t)\right] ((\phi^\eps_0)_{R,<})
   +  \O\left(\eps^\beta\,  C_\beta \,R^{-n_\beta}  \right) +\O(\eps^{N-d} R^d \|\psi^\eps_0 \|_{L^2})\\
   &= 
    \mathcal J \left[  \psi^{\eps,1}_{1,0}(t)+\sqrt\eps\, \psi^{\eps,1}_{2,1}(t)\right] (\phi^\eps_0)
   +  \O\left(\eps^\beta\,  C_\beta \,R^{-n_\beta}  \right) +\O(\eps^{N-d} R^d \|\psi^\eps_0 \|_{L^2}).
\end{align*} 
We can now identify the terms, using in particular the formula for the transfer terms given in~\eqref{B12}. Since we have chosen $\delta=\eps^{\frac 12-\varsigma}$ at the beginning of the proof (see around~\eqref{choice:delta}), we obtain
  \begin{align}\label{first_app}
 \U^\eps_H(t,t_0) \psi^\eps_0(x) & =
   \mathcal J_{1,\vec V, {\rm th}}^{t,t_0} \left(  \phi^\eps_0\right) 
   +\sqrt\eps\, \mathcal J_{1,2,\vec V, {\rm th}}^{t,t_0} \left(\phi^\eps_0\right) \\
   \nonumber
   &\qquad +\O(\eps^{1-\varsigma}\|\psi^\eps_0 \|_{L^2})
+ \O\left(\eps^\beta\,  C_\beta \,R^{-n_\beta}  \right)+ \O(\eps^{N-d} R^d \|\psi^\eps_0 \|_{L^2}) .
\end{align}
If $\beta<1$, we  choose   $R=\eps^{-\gamma}$, $N=1+d(\gamma+1)$ with 
$\gamma\geq \frac 1{n_\beta} (1-\beta).$
 This gives 
Theorem~\ref{thm:TGeps}.
\smallskip

\begin{remark}
More generally, if we had started with $\vec V=\pi_1(t_0)\vec V+\pi_2(t_0)\vec V$, we would have obtained 
\begin{align*}
 \U^\eps_H(t,t_0) \psi^\eps_0(x) & =
   \mathcal J_{1,\vec V, {\rm th}}^{t,t_0} \left( \phi^\eps_0\right) +  \mathcal J_{2, \vec V, {\rm th}}^{t,t_0} \left( \phi^\eps_0\right)
   +\sqrt\eps\, \mathcal J_{1,2,\vec V,  {\rm th}}^{t,t_0} \left(\phi^\eps_0\right) \\
\nonumber
&\qquad +\O(\eps^{1-\varsigma}\|\psi^\eps_0 \|_{L^2})
+ \O\left(\eps^\beta\,  C_\beta \,R^{-n_\beta}  \right)+ \O(\eps^{N-d} R^d \|\psi^\eps_0 \|_{L^2}) .
\end{align*}
\end{remark}

\section{Frozen Gaussian approximations with transfer terms }\label{sec:th_to_fr}
  
It remains to  pass from the thawed to the frozen approximation.   As we have already mentioned,  we follow  the arguments developed in Lemma 3.2 and~3.4 of ~\cite{R} that we detail explicitly, in the same spirit than in~\cite{FLR2}. It is based on  an evolution argument which crucially uses Corollary~\ref{cor:chemin}. We now explain that step.
  
    \begin{proof}[End of the proof of Theorem~\ref{th:sqrteps}]
We start from the approximation given by the first part of Theorem~\ref{th:sqrteps}: in $\Sigma^k_\eps(\R^d)$, we have 
\[
 \U^\eps_H(t,t_0) \psi^\eps_0(x)  =
   \mathcal J_{1, \vec V, {\rm th}}^{t,t_0} \left(\phi^\eps_0\right) +\O(\sqrt \eps (C_\beta + \|\phi^\eps_0\|_{L^2}))
   \]
 and our aim is to prove that  in $\Sigma^k_\eps(\R^d)$
 \[
   \mathcal J_{1, \vec V, {\rm th}}^{t,t_0} \left(\phi^\eps_0\right)=   \mathcal J_{1,\vec V, {\rm fr}}^{t,t_0} \left(\phi^\eps_0\right)+\O(\eps (C_\beta + \|\phi^\eps_0\|_{L^2}))).
   \]
   Of course, a remainder of size $\O(\sqrt\eps)$ would be enough for proving  Theorem~\ref{th:sqrteps}; however, it will be usefull to have a reaminder of size $\O(\eps)$  in order to prove Theorems~\ref{thm:FGeps} and~\ref{thm:FGeps_av}.
   \smallskip 
    
  We set for $s\in[0,1]$
  \[
  \Theta(s,z)= (1-s) \Gamma_\ell (t,t_0,z) + is \1_{\R^d}
  \]
  where $\Gamma_\ell$ is given by~\eqref{def:Gamma}.
  We consider the partially normalised Gaussian function 
  \[
    \widetilde g(t,s)=
  (\pi)^{-d/4}{\rm e}^{\frac i{2} \Theta(s,z) x\cdot x},\;x\in\R^d
  \]
 and we set
   \[
  \widetilde g^{\,\Theta(s),\eps}_{\Phi^{t,t_0}_{h_\ell}(z)}(x)=
  {\rm WP}^\eps_{\Phi^{t,t_0}_{h_\ell}(z)} \left( \widetilde g(t,s)\right),\\
  \]
  The aim 
  is to  construct a map 
 $s\mapsto a(s,z)$  such that  for all $s\in [0,1]$ in $\mathcal L(L^2(\R^d), \Sigma^k_\eps)$,
 \[
 \frac d{ds} \mathcal J\left[a(s,z) \vec V_\ell (t,t_0,z)  \widetilde g^{\,\Theta(s),\eps}_{\Phi^{t,t_0}_{h_\ell}(z)}\right] = \O(\eps).
 \]
 Choosing  $a(0,z)=1$,  
 we have 
 \[
  \mathcal J^{t,t_0}_{\ell,\vec V, {\rm th}}   = \mathcal J \left[ a(0,z) \vec V_\ell (t,t_0,z)  \widetilde g^{\,\Theta(0),\eps}_{\Phi^{t,t_0}_{h_\ell}(z)}\right],
  \]
 and we will obtain that 
 for any $f\in L^2(\R^d)$, we have   in $\Sigma^k_\eps(\R^d)$
 \begin{align*}
 \mathcal J^{t,t_0}_{\ell,\vec V, {\rm th}} (f) 
 & = \mathcal J \left[ a(1,z) \vec V_\ell (t,t_0,z)  \widetilde g^{\,\Theta(1),\eps}_{\Phi^{t,t_0}_{h_\ell}(z)}\right]  (f)+\O(\eps)\\
 &=  \mathcal J^{t,t_0}_{\ell,\vec V, {\rm fr}} (f) +\O(\eps) 
 \end{align*}
 provided $a(1,z)=a_\ell(t,t_0,z)$ as defined in~\eqref{def:prefactor}.
 \smallskip 
 
 For constructing the map $s\mapsto a(s,z)$, we compute 
 \begin{align*}
 \frac d{ds} \mathcal J\left[a(s,z) \vec V_\ell (t,t_0,z)  \widetilde g^{\,\Theta(s),\eps}_{\Phi^{t,t_0}_{h_\ell}(z)}\right] & = 
 \mathcal J\left[\partial_s a(s,z) \vec V_\ell (t,t_0,z)  \widetilde g^{\,\Theta(s),\eps}_{\Phi^{t,t_0}_{h_\ell}(z)}\right] \\
& +\frac i2  \mathcal J\left[a(s,z) \vec V_\ell (t, t_0,z) {\rm WP}^\eps_{\Phi^{t,t_0}_{h_\ell}(z)} \left( \partial_s \Theta(s) x\cdot x \widetilde  g^{\,\Theta(s),\eps}\right)\right] .
 \end{align*}
We use equation~\eqref{formule de base2} of Corollary~\ref{cor:chemin} to transform the second term of the right-hand side and obtain 
  \begin{align*}
 &  \mathcal J\left[a(s,z) \vec V_\ell (t, t_0,z) {\rm WP}^\eps_{\Phi^{t,t_0}_{h_\ell}(z)} ( \partial_s \Theta(s) x\cdot x g^{\,\Theta(s)})\right] \\
& \qquad  = \frac 1 i
    \mathcal J\left[a(s,z) \vec V_\ell (t, t_0,z) {\rm WP}^\eps_{\Phi^{t,t_0}_{h_\ell}(z)} ( {\rm Tr}( \Theta_1(s) )g^{\,\Theta(s)})\right] +\O(\eps)
   \end{align*}
   in $\mathcal L(L^2(\R^d), \Sigma^k_\eps)$ and 
   with
   \[
   \Theta_1(s)=\partial_s \Theta(s) \left[ (A_\ell -iB_\ell)^{-1} (C_\ell -iD_\ell) - \Theta\right]^{-1}
   \]
     where $M_\ell(s,z)$ is associated to $\Theta(s,z)$ according to~\eqref{def:M}. In particular, we have 
     \[
     \partial_s M(s)=-(A_\ell-iB_\ell)\partial_s \Theta(s).
     \]
   We deduce 
    \[
   \Theta_1(s)
   =- (A_\ell -iB_\ell )^{-1}\partial_s M(s) M(s)^{-1} (A_\ell-iB_\ell)
   \]
   and 
   \[
   {\rm Tr} (  \Theta_1(s))=-{\rm Tr} (\partial_s M(s) M(s)^{-1} )=- {\rm det} M(s) ^{-1}\, \partial_s \left( {\rm det} M(s)\right).
   \]
   Therefore, the condition
  \[
  \partial_s a(s,z) -\frac 12 {\rm Tr}(\partial_s M(s,z)M(s)^{-1} ) a(s,z)=0
  \]
that we have to fulfilled, is realized by
  \[
  a(s,z) = \frac{{\rm det}  M(s)}{{\rm det} M(0)} a(0,z)=a_\ell (t,t_0,z).
  \]
    \end{proof}

\begin{proof}[Proof of Theorem~\ref{thm:FGeps}]
We now start from the result of Theorem~\ref{thm:TGeps}, that is equation~\eqref{first_app}. In view of what has been done in the end of the proof of Theorem~\ref{th:sqrteps}, we only have to prove
\[
 \mathcal J_{1,2, \vec V,{\rm th}}^{t,t_0} \left(\phi^\eps_0\right) =
  \mathcal J_{1,2, \vec V, {\rm fr}}^{t,t_0} \left(\phi^\eps_0\right)  +\O\left(\sqrt\eps(C_\beta+\|\phi^\eps_0\|_{L^2}\right).
\]
As noticed in the introduction, when $t<t_{1,{\rm min}}(K)$, then  $\tau_{1,2}(t,t_0,z)=0$ for all $z\in K$ and   when   $t\in [t_{1,{\rm max}}(K),t_{2,{\rm min}}(K))$, $z\mapsto \tau_{1,2} (t,t_0,z)$ is smooth. Therefore, one can use the perturbative argument allowing to froze the covariances of the Gaussian terms as in the proof of  Theorem~\ref{th:sqrteps} and one obtains the formula~\eqref{eq:prefactor2}.
\end{proof}

\begin{proof}[Proof of Theorem~\ref{thm:FGeps_av}]
We now have to deal with
the discontinuity of  the transfer coefficient $\tau_{1,2}(t,t_0,z)$, we  use Lemma~\ref{lem:cor_t_dep}. 
\end{proof}

\part{Wave-packet propagation through smooth crossings}

\chapter{Symbolic calculus and diagonalization of  Hamiltonians with smooth crossings}\label{chap:2}

In this section, we revisit the diagonalization of Hamiltonians in the case of the smooth crossings in which we are interested.  We
settle the algebraic setting that we will use in Section~\ref{sec:prop} for the propagation of wave packets.
\smallskip

 \smallskip
 
 We will use the {\it Moyal product} about which we  recall some facts:
 if $A^\eps, B^\eps$  are semi-classical series,  their Moyal product is the formal series 
 \begin{align}
 \nonumber
& C^\eps:=A^\eps\circledast B^\eps\;\;  \mbox{where}\;\;
 \di{C^\eps =\sum_{j\geq 0}\eps^jC_j}\\
 \label{fifi2}
& C_j(x,\xi) = \frac{1}{2^j}\sum_{\vert\alpha+\beta\vert=j}
\frac{(-1)^{\vert \beta\vert}}{\alpha!\beta!}
(D^\beta_x\partial^\alpha_\xi A).( D^\alpha_x\partial^\beta_\xi B)(x,\xi), \;\;j\in\N.
\end{align}
We also introduce the {\it Moyal bracket}
\begin{equation}\label{def:moyal_bracket}
\{A^\eps, B^\eps\}_{\circledast} := A^\eps\circledast B^\eps - B^\eps\circledast A^\eps.
\end{equation}

 Let us now consider a smooth matrix-valued symbol $H^\eps=H_0+\eps H_1$, where the principal symbol $H_0 = h_1\pi_1 + h_2\pi_2$ has two  smooth eigenvalues $h_1$ and $h_2$ with smooth eigenprojectors $\pi_1$ and $\pi_2$. We allow for a non-empty crossing set $\Upsilon$ as in Definition~\ref{def:smooth_cros}.
  By standard symbolic calculus with smooth symbols, we have for $\ell\in\{1,2\}$ the relations 
 \begin{equation}\label{riri3}
  \pi_\ell \circledast (i\eps\partial_t -H^\eps)= (i\eps \partial_t -h_\ell)\circledast \pi_\ell = O(\eps).
  \end{equation}
 We are going to see two manners to replace the projector $\pi_\ell$ and the Hamiltonian $h_\ell$ by asymptotic series so that the relation above holds at a better order. 
 \smallskip 
 
 We call   ``rough" the first diagonalization process that we propose. It  will hold everywhere, including~$\Upsilon$ and is comparable to the reduction  performed in~\cite{BGT} for avoided crossings. It is the subject of Section~\ref{sec:rough_diag}. 
 \smallskip
 
  The second one, more sophisticated, will require to work in a domain that does not meet $\Upsilon$. Based on the use of superadiabatic projectors, as developed in~\cite{bi,MS,ST,Te}. This strategy is implemented in Section~\ref{sec:superadiab}. The new element comparatively to the references  that we have mentioned, is that we keep a careful memory of the dependence of the constructed  elements  with respect to the distance of their support from $\Upsilon$ (which will be monitored by the size of the gap between the eigenfunctions of the Hamiltonian). For this reason, we will use a pseudodifferential setting that we precise in the next Section~\ref{sec:asympt_series}.

\section{The pseudodifferential setting}\label{sec:asympt_series}

 We consider formal semi-classical series 
 $$A^\eps =\sum_{j\geq 0}\eps^jA_j$$ where all the functions $A_j$ are smooth (matrix-valued) in an open set  $\mathcal D\subset \R\times \R^{2d}$, that is, $A_j\in \mathcal C^\infty(\mathcal D,\C^{m,m})$.
 \smallskip
 
 \noindent{\bf Notation.}
 If $A^\eps =\sum_{j\geq 0}\eps^jA_j$ is a formal series and $N\in\N$,  we denote by $A^{\eps,N}$ the function
 \begin{equation}\label{def:Nseries}
 A^{\eps, N}=\sum_{0\leq j\leq N}\eps^jA_j.
 \end{equation}

The formal series that we will consider in Section~\ref{sec:superadiab} will present two small parameters: the semi-classical parameter $\eps>0$ and another parameter $\delta>0$ that controls the growth of the symbol and of its derivatives. For our intended application, $\delta$ is related to the size of the gap 
 between the eigenvalues of the Hamiltonian's symbol.

 \begin{definition}[Symbol spaces]\label{def:symbol_Sdelta}
 Let $\mu\in\R$
 and $\delta\in (0,1]$.
 \begin{enumerate}
 \item [(i)]  
We denote by $\bdS^\mu_\delta(\mathcal D)$ the set of smooth (matrix-valued)  functions in $\mathcal D$ such 
 $$
 \vert\partial^\gamma_zA(t,z)\vert\leq C_{\gamma}\delta^{\mu-\vert\gamma\vert},\;\; \forall (t,z)\in \mathcal D.
 $$
 Notice that  the set $\bdS_\delta(\mathcal D):=\bdS^0_\delta(\mathcal D)$  has the algebraic structure of a ring.
 \item [(ii)]  
We shall say that a formal series  $\di{A^\eps =\sum_{j\geq 0}\eps^jA_j}$  is in  $\bdS_{\eps,\delta}^\mu(\mathcal D)$ if 
$A_j\in\bdS^{\mu-2j}_\delta(\mathcal D)$ for all $j\in\N$. 
We set $\bdS_{\eps,\delta}(\mathcal D):=\bdS^0_{\eps,\delta}(\mathcal D)$.
\end{enumerate}
 \end{definition}
 
 \begin{remark}\label{rem:Sdelta}
 \begin{enumerate}
 \item If $A\in \mathbf S^\mu_\delta(\mathcal D)$ and $B\in\mathbf S^\mu_{\delta}(\mathcal D)$, then $AB\in \mathbf S^{\mu+\mu'}_\delta(\mathcal D)$ while $\{A,B\}\in \mathbf S^{\mu+\mu'-2}_\delta(\mathcal D)$.  Besides, if $A\in \bdS^\mu_\delta(\mathcal D)$, then $\partial_z^\gamma A\in \bdS_\delta^{\mu-|\gamma|}(\mathcal D)$. 
 \item 
When $\delta=1$, as in the next Section~\ref{sec:rough_diag}, then for all $\mu\in\R$, $\mathbf S^\mu_1(\mathcal D)= \mathbf S^0_1(\mathcal D)$ coincides with the standard class of  Calder\'{o}n--Vaillancourt symbols, those smooth functions that are bounded together with their derivatives. Similarly, $\mathbf S^\mu_{\eps,1}(\mathcal D)=\mathbf S^0_\eps(\mathcal D)$ coincides with asymptotic series of symbols. 
\item The parameter $\delta$ can be understood as quantifying the loss that appears at each differentiation. However, in the asymptotic series, one looses $\delta^{-2}$ when passing from some $j$-th term of to the $(j+1)$-th one.  In Section~\ref{sec:superadiab}, $\delta$ will monitor the size of the gap function $f$ in the domain~$\mathcal D$.
 \end{enumerate}
 \end{remark}
 
 The Moyal bracket~\eqref{fifi2} satisfies the property stated in the next lemma.

 \begin{lemma}\label{thm:prodest}
 Let $\delta_{A},\,\delta_B\in]0, 1]$.
 If $A^\eps$ and $B^\eps$ are formal series of ${\bf S}_{\eps, \delta_A}(\mathcal D)$  and  ${\bf S}_{\eps,\delta_B}(\mathcal D)$,  respectively.  then $A^\eps\circledast B^\eps$  is a formal series  of  ${\bf S}_{\eps,\min( \delta_A,\delta_B)}(\mathcal D)$. Besides, for $N\in\N$,
\[
A^{\eps, N}\circledast B^{\eps, N} = \sum_{0\leq j\leq N}\eps^jC_j +\eps^{N+1}R_{A,B}^{\eps, N}
\]
where for all $\gamma\in\N^{2d}$, there exists $C_{N,\gamma}$ independent on $\delta_{A,B}$ and $\eps$ such that 
\[
\vert\partial_z^\gamma R_{A,B}^{\eps, N}(t,z) \vert\leq C_{N,\gamma}\,[\min(\delta_A,\delta_B)]^{-N-\kappa_0},\; \forall \eps\in]0, 1], \;\; \forall (t,z)\in\mathcal D,
\]
where $\kappa_0>0$ is a universal constant depending only on the dimension $d$. 
\end{lemma}

\begin{proof}
The estimate is a direct consequence of~\cite[Theorem~A1]{BR}. We give a detailed proof of this result in Appendix~\ref{prodest}, see Theorem~\ref{thm:moyalest}.
\end{proof}

When $\delta_A=1$ and $\delta_B=\delta\in (0,1]$, $\min(\delta_A,\delta_B)=\delta$. This shows that 
\[
{\bf S}_{\eps, 1}(\mathcal D)\circledast {\bf S}_{\eps, \delta}(\mathcal D)\subset  {\bf S}_{\eps, \delta}(\mathcal D).
\]
\medskip

{
Let us conclude this section by comments on the quantization of  symbols of the classes ${\bf S}_\delta^\mu(\mathcal D)$. Further results are discussed in Appendix~\ref{sec:app_fifi4}.  
The Calder\'{o}n-Vaillancourt Theorem for pseudo-differential operators (see~\cite{disj,Zwobook,Folland}) states that there exists a constant $C>0$ such that for all $a\in\mathcal C^\infty(\R^{2d})$, 
\begin{equation*}
\| {\rm op}_\eps (a) \|_{\mathcal L(L^2(\R^d))} \leq C \sup_{0\leq |\gamma|\leq 2d+1}\,\eps^{\frac{|\gamma|}2} \sup_{z\in\R^d}  | \partial _z^\gamma a(z)|.
\end{equation*}
Actually, the original article~\cite{CV} treats the case $\eps=1$ and the estimate in the semi-classical case comes from the observation that 
\[
{\rm op}_\eps(a)= \Lambda_\eps^*{\rm op}_1(a(\sqrt\eps\cdot)) \Lambda_\eps
\]
where $\Lambda_\eps$ is the  $L^2$-unitary scaling operator defined on function $f\in\mathcal S(\R^d)$ by 
\[
\Lambda_\eps f(x)=\eps^{-\frac d 4} f\left(\frac x{\sqrt\eps}\right),\;\; x\in\R^d.
\]
One can derive an estimate in the sets $\Sigma^k_\eps$ by observing 
\[
\left( x^\alpha(\eps\partial_x)^\beta\right)\circ  {\rm op}_\eps(a)
=  \sum_{|\gamma_1|+|\gamma_2|+|\gamma_3|\leq  k}\eps^{\frac {|\gamma_1|}2}
c_{\gamma_1,\gamma_2,\gamma_3}(\eps)\,  {\rm op}_\eps( \partial_z^{\gamma_1} a) \circ \left( {x^{\gamma_2}}(\eps\partial_x)^{\gamma_3}\right)
\]
for some  coefficients $c_{\gamma_1,\gamma_2,\gamma_3}(\eps)$, uniformly bounded with respect to $\eps\in[0,1]$. This  implies 
 the boundedness of ${\rm op}_\eps (a)$ in weighted Sobolev spaces: for all $k\in\N$, there exists a constant $C_k>0$ such that for all $a\in\mathcal S(\R^{2d})$, 
\begin{equation} \label{eq:norm_pseudo}
\| {\rm op}_\eps (a) \|_{\mathcal L(\Sigma^k_\eps)} \leq C_k 
\sum_{0\leq |\gamma|\leq 2d+ k+ 1} \eps^{\frac{|\gamma|}2}
\, \sup_{z\in\R^d}  | \partial _z^{\gamma} a(z)|.
\end{equation}
}
\begin{proposition}\label{prop:fifi3}
Let $ A\in \mathbf S^{\mu-2j}_\delta(\mathcal D)$ for  $\mu\in\R$, $j\in\N$. Then, for $k\in\N$ 
\[
 \| {\rm op}_\eps (A ) \|_{\mathcal L(\Sigma^k_\eps)}
\le c_{k} \  \delta^{\mu-k-2j} 
\sup_{0\leq |\gamma|\leq 2d+1} \left(\frac{\sqrt\eps}{\delta}\right)^{|\gamma|},
\]
where the constant $c_k>0$ depends on the symbol $A$ and $k$, but is independent of $\eps,\delta,j,\mu$. 
In particular, if  $\delta \geq \sqrt\eps$, then
\begin{equation}\label{est:CV}
 \| {\rm op}_\eps (A ) \|_{\mathcal L(\Sigma^k_\eps)}
\leq c_k \,  \delta^{\mu-k-2j } .
\end{equation}
\end{proposition}

\begin{proof}
For any $\gamma$, there exists $c_\gamma = c_\gamma(A) >0$ such that 
$\vert\partial^\gamma_z A(z)\vert\leq c_{\gamma}\,\delta^{\mu-2j-\vert\gamma\vert}$. 
Hence, 
\[
\| {\rm op}_\eps (A ) \|_{\mathcal L(\Sigma^k_\eps)} \le C_k 
\sum_{0\leq |\gamma|\leq 2d+ k+ 1} \eps^{\frac{|\gamma|}2} c_{\gamma}\, \delta^{\mu-2j-\vert\gamma\vert} 
\le c_{k} \  \delta^{\mu-2j-k} 
\sup_{0\leq |\gamma|\leq 2d+1} \left(\frac{\sqrt\eps}{\delta}\right)^{|\gamma|}
\]
for some symbol dependent constant $c_k>0$. 
\end{proof}

 \section{The `rough' reduction}\label{sec:rough_diag}

 The next result gives a reduction of the Hamiltonian in a block diagonalized form. We will use this reduction on small interval of times.

  \begin{theorem}\label{thm:rough_reduc}
  Assume a Hamiltonian $H^\eps=H_0+\eps H_1$ with values in the set of $m\times m$ matrices such that $H_0$ has smooth eigenprojectors and eigenvalues. 
  There exist matrix-valued asymptotic series 
  \begin{align*}
  \pi^{\eps}_1&= \pi_1+ \sum_{j\geq 1} \eps^j\pi_{1,j},\;\;
  h^\eps_\ell = h_\ell\,{\1_m}+ \sum_{j\geq 1}  \eps^j h_{\ell,j},\;\; 
  W^\eps= \sum_{j\geq 1} \eps^j W_j, \;\;\ell\in\{1,2\}
  \end{align*}
  such that for all $N\in\N$, $ \pi^{\eps,N}_1$ and $\pi_2^{\eps, N}=\1_m-\pi^{\eps,N}_1$ are approximate projectors
  \begin{equation} \label{eq:piellN}
   \pi^{\eps,N}_\ell \circledast   \pi^{\eps,N}_\ell=  \pi^{\eps,N}_\ell+\O(\eps^{N+1}),\;\;\ell\in\{1,2\}
  \end{equation}
  and $H^\eps=H_0+\eps H_1$ reduces according to 
 \begin{align}\label{eq:h1N}
  \pi^{\eps,N} _1  \circledast   (i\eps \partial_t - H^\eps)&= 
  (i\eps \partial_t -h^{\eps,N}_1)  \circledast  \pi_1^{\eps, N}
  + W^{\eps, N} \circledast\pi_2^{\eps,N}+\O(\eps^{N+1}), \\
  \label{eq:h2N}
    \pi^{\eps,N} _2  \circledast   (i\eps \partial_t - H^\eps)&=
 (i\eps \partial_t -h^{\eps,N}_2)  \circledast  \pi_2^{\eps, N} +(W^{\eps, N}) ^* \circledast \pi_1^{\eps,N}+\O(\eps^{N+1}).
  \end{align}
  Moreover, for all $\ell\in\{1,2\}$ and $j\geq 1$,  the symbols $\pi_{\ell,j}$ and $h_{\ell, j}$ are self-adjoint, the matrices $W_j$ are  off-diagonal, with $W_1$ given by~\eqref{def:W1}, and 
  \begin{align}
  \label{def:hell1}
  h_{\ell,1}& = \pi_\ell H_1\pi_\ell +(-1)^\ell \frac i2 (h_1-h_2) \pi_\ell \{\pi_1,\pi_1\} \pi_\ell.
  \end{align}
  If $H^\eps$ also satisfies Assumption~\ref{hyp:growthH} on the time interval $I$, then the  $2m\times 2m$ matrix-valued Hamiltonian 
\[
 \underline H^\eps := \begin{pmatrix} {h_1^\eps} & W^\eps \\ 
 (W^\eps)^* &  {h_2^\eps} \end{pmatrix}
 \]
 is subquadratic according to Definition~\ref{def:subquad}. 
    \end{theorem}
  
   Note that in $ \underline H^\eps$, the off-diagonal blocks are of lower order than the diagonal ones since the asymptotic series $W^\eps$ has no term of order~$0$. 
   \smallskip
  
  Theorem~\ref{thm:rough_reduc} allows to put the equation~\eqref{eq:sch} in a reduced form by setting 
\[
\underline \psi^{\eps} = \, ^t( \underline \psi^{\eps}_1,\underline \psi^{\eps}_2)\;\;\mbox{with}\;\;
  \underline{\psi^\eps_\ell} = \widehat{\pi_\ell^\eps} \psi^\eps.
  \]
    Indeed,  we then have 
 \begin{equation}\label{eq:underline_schro}
  i\eps \partial_t\underline \psi^{\eps}
 = \widehat{\underline H}^\eps 
\underline \psi^{\eps}+\O(\eps^\infty),\;\;\underline \psi^{\eps}_{|t=0}=\, ^t\left(\widehat{\pi_1^\eps}\, \psi^\eps_0, \widehat{\pi_2^\eps}\, \psi^\eps_0\right).
 \end{equation}
We deduce the following corollary. 
 
  \begin{corollary}
  Formally, we have for $t\in I$, 
  \[
  \mathcal U^\eps_H(t,t_0)\psi^\eps_0=   \underline{\psi^\eps_1}+  \underline{\psi^\eps_2},
  \]
  where $\underline \psi^{\eps} $ solves~\eqref{eq:underline_schro}. 
  \end{corollary}

  \begin{proof}[Proof of Theorem~\ref{thm:rough_reduc}]
  The proof relies on a recursive argument.
  \smallskip 
  
  The case $N=0$  is equivalent to ~\eqref{riri3}

 The case $N=1$ has been proved in Lemma~B.2 in~\cite{FLR1}.
However, we revisit the proof in order to have a justification of the value of~$W_1$ in this text. 
We first compute $\pi_{\ell,\ell}$ by requiring 
$\pi^{\eps,(\ell)} \circledast \pi^{\eps,(\ell)}=\pi^{\eps,(\ell)}+\O(\eps^2)$, which admits the solution for $\ell=1$
$$\pi_{1,1}=-\pi_{2,1}= -\frac 1{2i}\pi_1 \{\pi_1,\pi_1\}\pi_1 +\frac 1{2i} \pi_2\{\pi_1,\pi_1\}\pi_2.$$  
We recall that $\{\pi_1,\pi_1\}$ is diagonal  and skew-symmetric (see Lemma~B.1 in~\cite{FLR1}).
Then, we observe 
  \[
   \pi _\ell  \circledast   (i\eps \partial_t - H^\eps)= 
  (i\eps \partial_t -h_\ell ) \circledast  \pi_\ell 
  + \eps \Theta_{\ell} +\O(\eps^{2})
  \]
  with
  \[
  \Theta_\ell= -\frac 1{2i}\{\pi_\ell,H_0\}- \pi_\ell H_1-i\partial_t \pi_\ell+\frac 1{2i}\{h_\ell,\pi_\ell\}
  \]
  or, equivalently 
  \begin{align*}
  \Theta_1 &=\frac 1{2i} (h_2-h_1)\{\pi_1,\pi_1\} -i\partial_t \pi_1 -\pi_1 H_1+\frac 1i \{h_1,\pi_1\}\pi_1  +\frac 1{2i} \{h_1+h_2,\pi_1\}\pi_2,\\
  \Theta_2 &= \frac 1{2i} (h_1-h_2)\{\pi_2,\pi_2\} -i\partial_t \pi_2 -\pi_2 H_1+\frac 1i \{h_2,\pi_2\}\pi_2+\frac 1{2i} \{h_1+h_2,\pi_2\}\pi_1 \\
& = -\frac 1{2i} (h_2-h_1)\{\pi_1,\pi_1\} +i\partial_t \pi_1 -\pi_2 H_1 -\frac 1i \{h_2,\pi_1\}\pi_2 -\frac 1{2i} \{h_1+h_2,\pi_2\}\pi_1
  \end{align*}
  We observe 
  \[
  \Theta_2^*= -\frac 1{2i} (h_2-h_1)\{\pi_1,\pi_1\} -i\partial_t \pi_1 -H_1\pi_2  +\frac 1i  \pi_2\{h_2,\pi_1\} +\frac 1{2i} \pi_1\{h_1+h_2,\pi_2\}
  \]
  and 
  \begin{equation}\label{riri5}
  \pi_1 \Theta_2^* \pi_2= \pi_1 \Theta_1 \pi_2.
  \end{equation}
  Thus, we have to solve 
  \begin{align*}
  -\pi_{1,1} H_0&=-h_{1,1} \pi_1 -h_1 \pi_{1,1} +i\partial_t \pi_1 + \Theta_1 +W_1 \pi_2,\\
    -\pi_{2,1} H_0&=-h_{2,1} \pi_2 -h_2 \pi_{2,1} +i\partial_t \pi_2 + \Theta_2 +W_1^* \pi_1.
    \end{align*}
    Multiplying on the right the first equation by $\pi_1$ and the second by $\pi_2$, we obtain that $h_{1,1}$ and~$h_{2,1}$ have to satisfy
\[
h_{1,1} \pi_1 = i\partial_t \pi_1 \pi_1 +\Theta_1 \pi_1\;\;\mbox{and}\;\;h_{2,1} \pi_2 = i\partial_t \pi_2 \pi_2 +\Theta_2 \pi_2
\]
which is solved by taking  the self-adjoint matrices
  \begin{align*}
h_{1,1}& = i\partial_t \pi_1 \pi_1 +\Theta_1 \pi_1 -i\pi_1 \partial_t \pi_1 \pi_2+\pi_1\Theta_1^*\pi_2,\\
h_{2,1}& = i\partial_t \pi_2 \pi_2 +\Theta_2 \pi_2 -i\pi_2 \partial_t \pi_2 \pi_1+\pi_2\Theta_2^*\pi_1.
\end{align*}
    Multiplying on the right the first equation by $\pi_2$ and the second by $\pi_1$, we obtain that $W_1$ has to verify
\[
    W_1\pi_2= -(h_2-h_1)\pi_{1,1} \pi_2 -\Theta_1\pi_2\;\;\mbox{and}\;\;
     W_1^*\pi_1= (h_2-h_1)\pi_{2,1} \pi_1 -\Theta_2\pi_1.
\]
Using the relations $\pi_{1,1}^*=\pi_{1,1}=-\pi_{2,1}$, $\pi_1 \partial_{t,z} \pi_1 =\partial_{t,z}\pi_1\pi_2$ and~\eqref{riri5}, we obtain
\[ 
W_1\pi_2=\pi_1W_1=\pi_1\left(H_1+i\partial_t\pi_1+\frac 12\{h_1+h_2,\pi_1\}\right)\pi_2,
\]
whence~\eqref{def:W1}.

\medskip 

One can now perform the recursive argument. Assume that we have obtained~\eqref{eq:piellN}, \eqref{eq:h1N} and~\eqref{eq:h2N} for some $N\geq 1$ and let us look for $\pi_{1,N+1}$,  $h_{1,N+1}$, $h_{2,N+1}$ and $W_{N+1}$ such that the relations for $N+1$ too.
\smallskip

{\it Construction of $\pi_{1,N+1}$}. We start with $\pi_{1,N+1}$. We write 
\[ \pi_{1}^{\eps,N} \circledast \pi_1^{\eps,N} =  \pi_1^{\eps,N}+\eps^{N+1} R^\eps,\]
where $R^\eps$ is an asymptotic series with first term $R_N$.
We first observe that $R_N$ is diagonal. Indeed, we have 
\[ 
(1- \pi_{1}^{\eps,N})\circledast \pi_{1}^{\eps,N}\circledast \pi_{1}^{\eps,N} = \pi_{1}^{\eps,N}\circledast \pi_{1}^{\eps,N}\circledast(1- \pi_{1}^{\eps,N})
\]
and 
\[ 
(1- \pi_{1}^{\eps,N})\circledast \pi_{1}^{\eps,N}\circledast \pi_{1}^{\eps,N} = -\eps^{N+1} R^\eps \circledast\pi_{1}^{\eps,N},\;\;
 \pi_{1}^{\eps,N}\circledast \pi_{1}^{\eps,N}\circledast(1- \pi_{1}^{\eps,N})= -\eps^{N+1} \pi_{1}^{\eps,N}\circledast R^\eps.
 \]
 This yields $ \pi_{1}^{\eps,N}\circledast R^\eps=R^\eps \circledast\pi_{1}^{\eps,N}$ and imply $ \pi_1 R_N =R_N\pi_1$. 
We now look to $\pi_{1,N+1}$ that must satisfy 
\[ \pi_{1,N+1}= R_N + \pi_{1,N+1}\pi_1 +\pi_1 \pi_{1,N+1}.\]
This relation fixes the diagonal part of $\pi_{1,N+1}$ according to
\[ 
\pi_1\pi_{1,N+1}\pi_1=-\pi_1 R_N\pi_1\;\;\mbox{and}\;\; \pi_2\pi_{1,N+1}\pi_2= \pi_2 R_N \pi_2,
\]
We will see later that we do not need to prescribe  
 off-diagonal components to $\pi_{1,N+1}$. 

\medskip

{\it Compatibility relations.} For constructing $h_{1,N+1}$, $h_{2,N+1}$ and $W_{N+1}$, we need information about the asymptotic series $\Theta^\eps_1 $ defined by 
\[
 \pi^{\eps,N} _1  \circledast   (i\eps \partial_t - H^\eps)= 
  (i\eps \partial_t -h^{\eps,N}_1)  \circledast  \pi_1^{\eps, N}
  + W^{\eps, N} \circledast \pi_2^{\eps,N}+\eps^{N+1}\Theta^\eps_1.
   \]
Let us denote by $\Theta_{1,N}$ the first term of the asymptotic series $\Theta^\eps_1$. 
  For obtaining information about $\Theta_{1,N}$, we compute 
$$\pi^{\eps,N}_\ell\circledast (i\eps\partial_t -H^\eps)\circledast \pi^{\eps,N}_{\ell'}$$
for different choices of $\ell,\ell'\in\{1,2\}$. 
\begin{itemize}
\item[(i)] Taking $\ell\not=\ell'$ gives two relations 
\begin{align*}
\pi^{\eps, N}_1&  \circledast (i\eps\partial_t -H^\eps)\circledast \pi^{\eps,N}_{2} \\
& = 
\pi^{\eps, N}_1  \circledast W^{\eps, N} \circledast \pi^{\eps,N}_{2}+  
\eps^{N+1} \pi^{\eps, N}_1 \circledast \Theta ^\eps_1 \circledast \pi^{\eps,N}_{2}+\O(\eps^{N+2})\\
\pi^{\eps, N}_2 &  \circledast (i\eps\partial_t -H^\eps)\circledast \pi^{\eps,N}_{1} \\
& = 
\pi^{\eps, N}_2 \circledast  (W^{\eps, N })^*\circledast \pi^{\eps,N}_{1}+ \eps^{N+1} 
\pi^{\eps, N}_2 \circledast \Theta ^\eps_2 \circledast \pi^{\eps,N}_{1}+\O(\eps^{N+2}),
\end{align*}
from which we deduce 
\begin{equation}\label{eq:compa_Theta}
\pi_2\Theta_{2,N} \pi_1=\left( \pi_1 \Theta_{1,N} \pi_2\right)^*=\pi_2\Theta_{1,N}^*\pi_1.
\end{equation}
\item[(ii)] Taking $\ell=\ell'$ gives the relations
\begin{align*}
\pi^{\eps, N}_\ell  & \circledast (i\eps\partial_t -H^\eps)\circledast \pi^{\eps,N}_{\ell} \\
& = 
\pi^{\eps, N}_\ell \circledast (i\eps\partial_t - h^{\eps,N}_\ell)\circledast \pi^{\eps,N}_{\ell}+
\eps^{N+1} \pi^{\eps, N}_\ell \circledast \Theta ^\eps_\ell \circledast \pi^{\eps,N}_{\ell}+\O(\eps^{N+2}),
\end{align*}
whence the self-adjointness of the diagonal part of $ \Theta ^\eps_\ell$. 
\end{itemize}
\medskip 

{\it Construction of  $h_{1,N+1}$ and  $h_{2,N+1}$, first properties of $W_{N+1}$.} 
  We write the asymptotic series
  \begin{align*}
& \pi^{\eps,N+1} _1  \circledast   (i\eps \partial_t - H^\eps)=    \pi^{\eps,N} _1  \circledast   (i\eps \partial_t - H^\eps) -\eps^{N+1} \pi_{1,N+1} H_0+\O(\eps^{N+2}),\\
  & (i\eps \partial_t -h^{\eps,N+1}_1)  \circledast  \pi_1^{\eps, N+1}
  + W^{\eps, N+1} \circledast \pi_2^{\eps,N+1}=  (i\eps \partial_t -h^{\eps,N}_1)  \circledast  \pi_1^{\eps, N}
  + W^{\eps, N} \circledast \pi_2^{\eps,N} \\
  &\qquad\qquad+\eps^{N+1} \left(i\partial_t\pi_{1,N}-h_{1,N+1}\pi_1-h_1\pi_{1,N+1}+W_{N+1} \pi_2\right) +\O(\eps^{N+2}).
  \end{align*}
  Therefore, the terms $h_{1,N+1}$ and $W_{N+1}$ have to satisfy
 \[
 - \pi_{1,N+1} H_0 =i\partial_t \pi_{1,N}- h_{1,N+1} \pi_1 - h_1 \pi_{1,N+1}+ W_{N+1} \pi_2 + \Theta_{1,N} 
\]
or equivalently 
\[
0 =i\partial_t \pi_{1,N}- h_{1,N+1} \pi_1 +(h_2- h_1) \pi_{1,N+1}\pi_2+ W_{N+1} \pi_2+ \Theta_{1,N} .
\]
By multiplying on the right by $\pi_1$, then $\pi_2$, we are left with the two equations
\begin{align}
\nonumber 
&h_{1,N+1} \pi_1= \Theta_{1,N}\pi_1+i\partial_t \pi_{1,N}\pi_1 ,\\
\label{riri1}
& W_{N+1}\pi_2=\Theta_{1,N} \pi_2+(h_2-h_1) \pi_{1,N+1} \pi_2+i\partial_t \pi_{1,N}\pi_2 .
\end{align}
Considering similarly the conditions for the mode $h_2$, we obtain that  $h_{2,N+1}$ and $W_{N+1}^*$ have to satisfy
\[ 
h_{2,N+1} \pi_2= \Theta_{2,N}\pi_2+i\partial_t \pi_{2,N}\pi_2 \;\;\mbox{and}\;\; W_{N+1}^*\pi_1=\Theta_{2,N} \pi_1-(h_2-h_1) \pi_{2,N+1} \pi_1+i\partial_t \pi_{2,N}\pi_1.
\] 
Since $\pi_{2,N}=-\pi_{1,N}$ for $N\geq 1$, we are left with the relation
\begin{align}
\nonumber
&h_{2,N+1} \pi_2= \Theta_{2,N}\pi_2-i\partial_t \pi_{1,N}\pi_2 ,\\
\label{riri2}
&W_{N+1}^*\pi_1=\Theta_{2,N} \pi_1+(h_2-h_1) \pi_{1,N+1} \pi_1-i\partial_t \pi_{1,N}\pi_1.
\end{align}
We set 
\begin{align*}
h_{1,N+1} &= \Theta_{1,N} \pi_1 + i\partial_t \pi_{1,N}\pi_1+\pi_1 \Theta_{1,N}^*\pi_2-i\pi_1\partial_t \pi_{1,N} \pi_2,\\
h_{2,N+1} &=\Theta_{2,N} \pi_2 + i\partial_t \pi_{2,N}\pi_2+\pi_2 \Theta_{2,N}^*\pi_1-i\pi_2\partial_t \pi_{2,N} \pi_1.
\end{align*}
Then, $h_{1,N+1} $ and $h_{2,N+1} $ are self-adjoint and satisfy the desired properties. 
\smallskip

{\it End of the construction of $W_{N+1}$.} The construction of $W_{N+1}$ requires to be more careful because there is a compatibility condition that has to be satisfied in order to have simultaneously~\eqref{riri1} and \eqref{riri2}.
Let us look  for $W_{N+1}$ of the form 
\[
W_{N+1}= \Theta_{1,N} \pi_2+(h_2-h_1) \pi_{1,N+1} \pi_2+i\partial_t \pi_{1,N}\pi_2+ U_{N+1}\pi_1.
\]
This guarantees~\eqref{riri1}.
Then, one has 
\[
W_{N+1}^* =\pi_2 \Theta_{1,N} ^*+(h_2-h_1) \pi_2\pi_{1,N+1} -i\pi_2 \partial_t \pi_{1,N}+\pi_1 U_{N+1}^*
\]
and 
\begin{align*}
W_{N+1}^*\pi_1& =\pi_2 \Theta_{1,N} ^*\pi_1
-i\pi_2 \partial_t \pi_{1,N}\pi_1+\pi_1 U_{N+1}^*\pi_1\\
&= \pi_2 \Theta_{2,N} \pi_1
-i\pi_2 \partial_t \pi_{1,N}\pi_1+\pi_1 U_{N+1}^*\pi_1
\end{align*}
where we have used the first property of the matrices $\Theta_{1,N}$ and $\Theta_{2,N}$ that we have exhibited (see~\eqref{eq:compa_Theta}), together with the fact that $\pi_{1,N+1}$ is diagonal, as noticed in its construction. 
It is then enough to choose 
\[
U_{N+1}= \pi_1 \left(\Theta_{2,N}^*+(h_2-h_1)\pi_{1,N+1} +i\partial_t \pi_{1,N} \right) \pi_1
\]
since it implies
\begin{align*}
W_{N+1}^*\pi_1& =
\pi_2 \Theta_{2,N} \pi_1-i\pi_2 \partial_t \pi_{1,N}\pi_1
+  \pi_1(\Theta_{2,N} +i\partial_t \pi_{1,N} +(h_2-h_1)\pi_{1,N+1}) \pi_1\\
&= \Theta_{2,N} \pi_1 +(h_2-h_1)\pi_{1,N+1}\pi_1-i\pi_2 \partial_t \pi_{1,N}\pi_1,
\end{align*}
where we have used that $\pi_{1,N+1}$ is diagonal (whence $\pi_{1,N+1}\pi_1=\pi_1 \pi_{1,N+1}\pi_1$). As a consequence, 
 the second part of~\eqref{riri2} is satisfied. 
This concludes the recursive argument and the proof of the Theorem~\ref{thm:rough_reduc} since the growth properties of the matrices that we have constructed  come with the recursive equations. 
  \end{proof}

\section{Superadiabatic projectors and diagonalization}\label{sec:superadiab}

We now want to get rid of the off-diagonal elements $\widehat W^\eps$ that appears in the reduction of Theorem~\ref{thm:rough_reduc}. This is  only possible outside $\Upsilon$. We are going to take into account how far from the crossing set we are by introducing a gap assumption.

 \begin{assumption}[Gap assumption]\label{hyp:NCdelta} 
 Let  $t_0<t_1$, $I$ an open  interval of $\R$ containing $[t_0,t_1]$ and~$\Omega$ an open subset of $\R^{2d}$.
  We say that the eigenvalue $h$ has a gap larger than $\delta\in (0,1]$ in  $\mathcal D :=I\times\Omega$  if one has 
  \begin{enumerate}
\item[$({NC}_\delta)$]\qquad 
 $\displaystyle{ d\left( h(t,z) , {\rm Sp} (H_0(t,z) )\right) \geq \delta,\;\;\forall (t, z)\in \mathcal D .}$
 \end{enumerate}
 \end{assumption}
 
 The construction of superadiabatic projectors dates to~\cite{bi}  which was inspired by the paper~\cite{emwe}. It has then been carefully developed in~\cite{MS} and~\cite{ST} (see also the book~\cite{Te}). 
 We revisit here the construction of superadiabatic projectors, in order to control their norms with respect to the parameter $\delta$. 
 \smallskip
 
 We follow the construction of the Section~14.4 of the latest edition of~\cite{corobook} (2021), that we adapt to our context. 
One proceeds in two steps: first by defining the formal series for the projectors and then for the Hamiltonians. 
In order to simplify the notations in the construction, we just consider an eigenvalue $h$ and we will then apply the result to the eigenvalues $h_1$ and~$h_2$.

\subsection{Formal superadiabatic projectors}

  \begin{theorem}[semiclassical projector evolution]\label{quantpro} Assume the eigenvalue $h$ of the Hamiltonian~$H_0$ satisfies Assumption~\ref{hyp:NCdelta} in $\mathcal D$. 
   Then, there exists a 
 unique formal series $\sum_{j\geq 1}\eps^{j} \Pi_j$ in  $  \bdS_{\eps,\delta}^{-1}(\mathcal D)$ such that  setting    $\Pi_0(t,z) = \pi(t,z)$, 
 the formal series 
 $$\di{\Pi^\eps(t,z)=\sum_{j\geq 0}\eps^j\Pi_j(t,z)}$$
 is a formal projection 
 and 
  \beq\label{project}
  i\eps\partial_t\Pi^\eps (t) = \{H^\eps(t),\Pi^\eps(t)\}_\circledast.
  \eeq
  Moreover the sub-principal term $\Pi_1(t)$ is an Hermitian matrix given by the following formulas:
 \begin{align}
 \label{cacPi1}
 &\pi(t)  \Pi_1(t) \pi(t) =  -\frac{1}{2i}\pi(t)\{\pi(t), \pi(t)\}\pi(t),\\
 \nonumber
 &\pi(t)^\perp  \Pi_1(t)\pi(t)^\perp =  \frac{1}{2i}\pi(t)^\perp\{\pi(t), \pi(t)\}\pi(t)^\perp,\\
\nonumber
&  \pi(t)^\perp  \Pi_1(t) \pi(t) = \pi(t)^\perp (H_0(t)- h(t))^{-1} \pi(t)^\perp R_1(t)\pi(t),   
   \end{align}
   where
   $$
   R_1(t) =    i\partial_t\pi(t) -\frac{1}{2i}\left(\{H_0(t), \pi(t)\}-\{\pi(t), H_0(t)\}\right)-[H_1(t),\pi(t)].
   $$
    \end{theorem}
    
    \begin{proof}
    With Notations~\ref{def:Nseries}, and working in space time variable,
    \begin{align}\label{eq:N}
     \Pi^{\eps,N}\circledast\Pi^{\eps,N} - \Pi^{\eps,N}&= \eps^{N+1} S_{N+1} + \O(\eps^{N+2}),
    \\
    \label{eq:N'}
- \left\{\tau + H_0+\eps H_1,\Pi^{\eps,N}\right\}_{\circledast} &= \eps^{N+1}R_{N+1} + \O(\eps^{N+2}),
    \end{align}
    where the Moyal bracket has been defined in~\eqref{def:moyal_bracket} and $\left\{\tau ,\Pi^{\eps,N}\right\}_{\circledast}=\partial_t \Pi^{\eps,N}$.
    \medskip

   \noindent {\bf Step 1.} We start with $N=0$. We have $\Pi^{(0)} = \pi \in \bdS^{0}_\delta(\mathcal D)$. Since $\pi^2=\pi$ and $[H_0,\pi]=0$, we obtain 
\[
    S_1 = \tfrac{1}{2i}\{\pi,\pi\}\;\mbox{and}\;
    R_1 = i\partial_t\pi - \tfrac{1}{2i}(\{H_0,\pi\}-\{\pi,H_0\}) - [H_1,\pi],
\]
and we have $R_1,S_1\in   \bdS^{0}_\delta(\mathcal D)$. Before constructing $S_1$, 
    two structural observations are in order: 
    \begin{enumerate}
    \item The matrix $S_1$ is symmetric and satisfies $\pi S_1\pi^\perp = \pi^\perp S_1\pi =0$.
    \item The matrix $R_1$ is skew-symmetric. It satisfies 
    \[
    \pi R_1\pi = 0\quad\text{and}\quad
    \pi^\perp R_1\pi^\perp = [H_0,\pi^\perp S_1\pi^\perp].
    \]
    \end{enumerate}
    If $H_0$ has only two eigenvalues, $H_0$ expresses only in terms of $\pi$ and 
the expression  of~$R_1$ given above  shows that $R_1$ is off-diagonal. The situation is more complicated if $H_0$ has strictly more than two distinct eigenvalues. For verifying (2) in that case, one uses two times
   the Poisson bracket rule 
    \[
    \{A,BC\}-\{AB,C\}=\{A,B\}C-A\{B,C\}.
    \]
 We obtain
    \begin{align*}
    \{H_0,\pi\}-\{\pi,H_0\} &= \{H_0,\pi^2\}-\{\pi^2,H_0\}\\
     &= \{h\pi,\pi\} + \{H_0,\pi\}\pi - H_0\{\pi,\pi\} - \{\pi,h\pi\} + \{\pi,\pi\}H_0 - \pi\{\pi,H_0\}\\
     &= \pi\{h,\pi\} - \{\pi,h\}\pi + [\{\pi,\pi\},H_0] + \{H_0,\pi\}\pi - \pi\{\pi,H_0\},
    \end{align*}
    which implies
    \[
    \pi^\perp(\{H_0,\pi\}-\{\pi,H_0\})\pi^\perp = \pi^\perp [\{\pi,\pi\},H_0]\pi^\perp.
    \]
    For determining the $\pi$-diagonal component, we choose $A=\pi$, $B=H_0\pi^\perp$, and $C=\pi$ 
    to obtain
   \begin{align*}
   0 &= \{\pi,H_0\pi^\perp\}\pi - \pi\{H_0\pi^\perp,\pi\}\\
   &= \{\pi,H_0\}\pi - \{\pi,h\pi\}\pi-\pi\{H_0,\pi\} + \pi\{h\pi,\pi\}\\
   &= \{\pi,H_0\}\pi - \{\pi,h\}\pi-\pi\{H_0,\pi\} + \pi\{h,\pi\}.
   \end{align*}
   This relation implies
   \[
    \pi(\{H_0,\pi\}-\{\pi,H_0\})\pi = 0.
    \] 
    For constructing the matrix $\Pi_1$ that defines $\Pi^{(1)} = \pi+\eps\Pi_1$, we need to satisfy      
    \[
    \pi\Pi_1 + \Pi_1\pi - \Pi_1 =  - S_1  \;\;\mbox{and}\;\; -[H_0,\Pi_1] = -R_1.
    \]
    The first of these two equations uniquely determines the diagonal blocks of $\Pi_1$, while the second 
    equation uniquely determines the off-diagonal blocks. We obtain
    \begin{align*}
    &\pi\Pi_1\pi = -\pi S_1\pi \;\;\mbox{and}\;\;
    \pi^\perp\Pi_1\pi^\perp = \pi^\perp S_1\pi^\perp,\\
     &\pi\Pi_1\pi^\perp = -\pi R_1\pi^\perp (H_0-h)^{-1}
      \;\;\mbox{and}\;\;
\pi^\perp\Pi_1\pi = (H_0-h)^{-1}\pi^\perp R_1\pi .
    \end{align*}
    For concluding this first step, we deduce from $R_1,S_1\in   \bdS^{0}_\delta(\mathcal D)$ that
    $\Pi_1\in \bdS_\delta ^{-1}(\mathcal D)$.

    \medskip 
    
    \noindent{\bf Step $N\geq 1$.}
    Next we proceed by induction and assume that we have constructed the matrices $\Pi_j(t)\in  \bdS_\delta ^{1-2j}(\mathcal D)$ for $1\leq j\leq N$ such that~\eqref{eq:N} and~\eqref{eq:N'} hold. Note that by Lemma~\ref{thm:prodest}, this implies 
    $$ R_{N+1}(t)\in  \bdS_\delta ^{-2N}(\mathcal D)\;\;\mbox{and}\;\;S_{N+1} (t)\in  \bdS_\delta ^{-2N}(\mathcal D).$$
    Indeed, 
        $- \left\{\tau +H_0+\eps H_1,\Pi^{\eps,N}\right\}_{\circledast}$  is a formal series of $\eps\, \bdS^{-2}_{\eps,\delta}(\mathcal D)$ while $  \Pi^{N}\circledast\Pi^{\eps,N} - \Pi^{\eps,N}$ is a formal series of $\eps \bdS^{-1}_{\eps,\delta}(\mathcal D)$.  In order to go one step further, we see that $\Pi_{N+1}$ has to 
 be chosen so that first 
   \begin{align*}
   \O(\eps^{N+2})&= 
   - \left\{\tau +H_0+\eps H_1,\Pi^{\eps,N+1}\right\}_{\circledast}\\
   &=
   - \left\{\tau +H_0+\eps H_1,\Pi^{\eps,N}\right\}_{\circledast} -\eps^{N+1} [H_0,\Pi_{N+1}]+\O(\eps^{N+2})\\
   &= \eps^{N+1} R_{N+1}  -\eps^{N+1} [H_0,\Pi_{N+1}]+\O(\eps^{N+2}),
   \end{align*}
   and second
   \begin{align*}
     \O(\eps^{N+2})&= \Pi^{N+1}\circledast\Pi^{\eps,N+1} - \Pi^{\eps,N+1}\\
     &=  \Pi^{N}\circledast\Pi^{\eps,N} - \Pi^{\eps,N} +  \eps^{N+1}(  \pi\Pi_{N+1} + \Pi_{N+1}\pi - \Pi_{N+1}) +\O(\eps^{N+2})\\
     &= \eps^{N+1} S_{N+1} + \eps^{N+1}(  \pi\Pi_{N+1} - \Pi_{N+1}\pi ^\perp) +\O(\eps^{N+2})
   \end{align*}

   As a conclusion, we have to find $\Pi_{N+1}$ such that 

       \begin{equation}\label{eq:syst_to_solve}
    \pi\Pi_{N+1} - \Pi_{N+1}\pi^\perp =  - S_{N+1}  \;\;\mbox{and}\;\; [H_0,\Pi_{N+1}] = R_{N+1}.
    \end{equation}
    The first relation in~\eqref{eq:syst_to_solve} determines the diagonal part of $\Pi_{N+1}$:
    \[
     \pi\Pi_{N+1}\pi  =  - \pi S_{N+1}\pi\;\;\mbox{and}\;\; \pi^\perp \Pi_{N+1}\pi^\perp =  - \pi^\perp S_{N+1}  \pi^\perp
     \]
     provided a first compatibility relation is satisfied, namely 
     \begin{equation}\label{eq:compatibility1}
      \pi S_{N+1}  \pi^\perp= \pi^\perp S_{N+1}\pi =0.
     \end{equation}
     Then the second equation in~\eqref{eq:syst_to_solve} fixes the off diagonal part of $\Pi_{N+1}$, via
     \[
     \pi H_0 (\pi\Pi_{N+1}\pi^\perp )=\pi  R_{N+1}\pi^\perp \;\;\mbox{and}\;\;    \pi ^\perp H_0 (\pi ^\perp \Pi_{N+1}\pi)=\pi^\perp  R_{N+1}\pi,
     \]
     up to a second compatibility about the diagonal part of the commutators that are now fixed: in view of  $   [\pi H_0, \pi\Pi_{N+1}\pi  ]   = - [\pi H_0, \pi S_{N+1}  \pi]$ and $ [\pi^\perp H_0, \pi ^\perp \Pi_{N+1}  \pi^\perp]=  [\pi^\perp H_0, \pi ^\perp S_{N+1}  \pi^\perp]$, one must have 
    \begin{equation}\label{eq:compatibility2}
 [\pi H_0, \pi S_{N+1}  \pi]=- \pi R_{N+1} \pi \;\;\mbox{and}\;\;
 [\pi^\perp H_0, \pi ^\perp S_{N+1}  \pi^\perp]=\pi ^\perp R_{N+1} \pi^\perp.
              \end{equation}
    
   This process allows to  to construct $\Pi_{N+1}(t)\in  \bdS_\delta ^{-2N-1)}(\mathcal D)$ and we will also obtain  $ R_{N+2}(t)\in  \bdS_\delta ^{-2N-2}(\mathcal D)$, $S_{N+2} (t)\in  \bdS_\delta ^{-2N-2}(\mathcal D)$ because of equations~\eqref{eq:N},~\eqref{eq:N'} and Lemma~\ref{thm:prodest}.
    
  \smallskip

  It remains to prove the two compatibility relations~\eqref{eq:compatibility1} and~\eqref{eq:compatibility2}.
    For proving~\eqref{eq:compatibility1}, we take advantage of the fact that 
    \[
Z:=     \Pi^{\eps,N}\circledast\left((\Pi^{\eps,N})^{2\circledast} -\Pi^{\eps, N}\right)\circledast(\1 -\Pi^{\eps, N}) = \eps^{N+1}  \pi S_{N+1}\pi^\perp +\O(\eps ^{N+2}),
     \]
     while one also has by construction
     \[
     Z=    \left((\Pi^{\eps,N})^{2\circledast}-\Pi^{\eps,N}\right)\circledast\left(\Pi^{\eps,N}-(\Pi^{\eps,N})^{2\circledast}\right)={\mathcal O}(\eps^{2N+2}).
     \]
     This implies that 
     $\pi S_{N+1} \pi ^\perp =0$ and, using that $S_{N+1}$ is Hermitian, we deduce that it is diagonal, whence~\eqref{eq:compatibility1}.
     \smallskip 
 For proving the first relation in~\eqref{eq:compatibility2}, we argue similarly with
    \[
Z':=   -  \Pi^{\eps, N}\circledast 
\left(
 \left\{\tau + H_0+\eps H_1,\Pi^{\eps,N}\right\}_{\circledast} 
\right)\circledast \Pi^{\eps,N} = \eps^{N+1}  \pi R_{N+1}\pi +\O(\eps ^{N+2}),
     \]
which also satisfies 
\begin{align*}
Z'&=   
-\Pi^{\eps,N}\circledast   (\tau + H_0 +\eps H_1) \circledast  (\Pi^{\eps,N})^{2\circledast} +(\Pi^{\eps,N})^{2\circledast}\circledast   (\tau + H_0+\eps H_1)\circledast   \Pi^{\eps,N}\\
&=-  \Pi^{\eps,N}\circledast  (\tau + H_0+\eps H_1) \circledast  (\Pi^{\eps,N}+\eps^{N+1} S_{N+1})
\\
&\qquad + (\Pi^{\eps,N}+\eps^{N+1} S_{N+1})\circledast( \tau+  H_0 +\eps H_1)\circledast   \Pi^{\eps,N}+\O(\eps^{N+2})\\
&= \eps^{N+1}\left(- \Pi^{\eps,N}\circledast (\tau+ H_0) \circledast S_{N+1} + S_{N+1} \circledast (\tau+ H_0) \circledast  \Pi^{\eps,N}\right)+\O(\eps^{N+2}).
\end{align*}
This implies 
\[
\pi R_{N+1} \pi=-\pi H_0 S_{N+1}+ S_{N+1} H_0 \pi.
\]
Since $[S_{N+1},\pi]=0$, we deduce 
\[ 
\pi R_{N+1} \pi=-[\pi H_0 , \pi S_{N+1}\pi].
\]
One argues similarly with $\1-\Pi^{\eps,N}$ for obtaining  the other relation in~\eqref{eq:compatibility2}. This finishes the proof.  
             \end{proof}

 \subsection{Formal adiabatic decoupling}             
            
    The second (and decisive) part of the analysis  is a formal adiabatic decoupling using the superadiabatic projectors introduced before.

  \begin{theorem}[formal adiabatic decoupling]\label{adia1} 
  Assume the eigenvalue $h$ of the Hamiltonian $H_0$ satisfies Assumption~\ref{hyp:NCdelta} in $\mathcal D$. 
    There exists a formal time dependent Hermitian   Hamiltonian in $\mathcal D$, 
    $$\di{H^{{\rm adia},\eps}=\sum_{j\geq 0}\eps^jH^{\rm adia}_j}$$
    such that 
  \beq\label{adiaform}
 \Pi^\eps \circledast (i\eps \partial_t - H^\eps) =  (i\eps\partial_t - H^{{\rm adia},\eps} )\circledast \Pi ^\eps
   \eeq
    with the following properties:
  \begin{enumerate}
  \item The principal symbol is    $H^{\rm adia}_0 = h\,\1_{m} $.
  \item  The subprincipal term $H_1^{\rm adia}$ is a Hermitian matrix satisfying
   \begin{align*}
 \pi^\perp\, H_1^{\rm adia}\pi
 = \pi ^\perp\left( i\partial_t \pi +i \{ h, \pi\}\right)\pi 
 \;\;\mbox{and}\;\; \pi\, H_1^{\rm adia}\pi & = \pi H_1\pi  + \frac{1}{2i}\pi\{H_0,\pi\}\pi
  \end{align*}
(see~\eqref{eq:H1adiab}) and we can choose 
$\pi^\perp H^{{\rm adia},\eps}_1\pi^\perp =0.$
\item We have 
 $$\eps^{-2}\left(H^{{\rm adia},\eps}-h\1_{m}-\eps H^{\rm adia}_1\right)\in{\bf S}_{\eps,\delta}^{-1}(\mathcal D).$$ 
\item
Finally,  $\pi(t)$  satisfies   a transport   equation along the classical flow for $h(t)$.
   \beq\label{adiatransp}
\partial_t\pi +\{h, \pi\}=  \frac{1}{i}[H^{{\rm adia}}_1, \pi].
      \eeq      
         \end{enumerate}
    \end{theorem}

    \begin{remark}
     Note that equations~\eqref{eq:H1adiab} imply
   that  $H^{\rm adia}_1(t, z)$ is smooth everywhere, also  on the crossing set.
    \end{remark}

    
The above construction singles out a smooth eigenvalue $h$ with smooth eigenprojector $\pi$. If it is applied to the Hamiltonian $H^\eps$ for two smooth eigenvalues~$(h_1,h_2)$ and two smooth eigenprojectors $(\pi_1,\pi_2)$, then we obtain two pairs of formal series
\begin{equation}\label{def:superadiab_ell}
\Pi^{\eps}_\ell =\sum_{j\geq0} \eps^j \Pi_{\ell, j}\;\;{\rm and}\;\; H^{{\rm adia},\eps}_\ell = \sum_{j\geq0} \eps^jH^{\rm adia}_{\ell,j}.
\end{equation}
We note, that if $\pi_1^\perp = \pi_2$, then $\Pi^\eps_2=\mathrm{Id}-\Pi^\eps_1$.

    \begin{corollary}
    At the level of the evolution operator, the result implies
\[
\mathcal U^\eps_H(t,t_0) = \widehat {\Pi_1^\eps}(t)  \mathcal U^{\rm adia,\eps}_1(t,t_0) \widehat {\Pi_1^\eps} (t_0)
+  \widehat {\Pi_2^\eps}(t)  \mathcal U^{\rm adia,\eps}_2(t,t_0) \widehat {\Pi_2^\eps} (t_0)
\]
where for $\ell\in\{1,2\}$, $\mathcal U^{\rm adia,\eps}_{\ell}(t,t_0) $ are the evolution operators associated with the Hamiltonian~$H^{{\rm adia},\eps}_\ell $
    \end{corollary}

    \begin{proof}[Proof of Theorem~\ref{adia1}]
   {
    This result is Theorem~80 of Chapter~14 in~\cite{corobook} combined with Lemma~\ref{thm:prodest}. We first observe that 
    equation~\eqref{adiaform} reduces to proving 
   \begin{equation}\label{adia:form_new}
 \Pi^\eps \circledast  (i\eps\partial_t -H^\eps )= (i\eps \partial_t - H^{{\rm adia},\eps} )\circledast  \Pi^\eps.
    \end{equation}
  For proving the latter relation, one first observes that if $H_0^{\rm adia}= h$, then we have 
  \begin{align*} 
  (H^{\rm adia,\eps}-H_0-\eps H_1)\circledast  \Pi^\eps&=\eps \left( 
(h-H_0) \Pi_1 +(H_1^{\rm adia}-H_1 )\pi +\frac 1{2i} \{ h-H_0,\pi \} \right)+i\partial_t\pi +\mathcal \O(\eps^2).
  \end{align*}
  Therefore, 
  $H_1$ has to be chosen so that 
 $$(H_1^{\rm adia}-H_1)\pi= (H_0-h) \Pi_1+\frac 1{2i} \{ H_0-h,\pi\}+i\partial_t\pi.$$
 In view of~\eqref{cacPi1}, this requires 
 $$\pi (H_1^{\rm adia}-H_1)\pi= \frac 1{2i} \pi\{H_0,\pi\}\pi,$$
 which is given by the second relation of~\eqref{eq:H1adiab}, and, using again~\eqref{cacPi1}
\[
 \pi^\perp (H_1^{\rm adia}-H_1)\pi = \pi^\perp( R_1+ \frac 1{2i} \{H_0-h,\pi\} +i\partial_t\pi)\pi= \pi ^\perp (i\partial_t\pi +\frac 1{2i} \{\pi,H_0\} -\frac 1{2i}\{h,\pi\}-H_1)\pi    
 \]
 which is also given by the first relation of~\eqref{eq:H1adiab} in view of 
the observation that 
  $$\pi^\perp \{H_0,\pi\}\pi=-\pi^\perp \{h,\pi\}\pi.$$
  For proving this relation, one uses   the Poisson bracket rule 
    \[
    \{A,BC\}-\{AB,C\}=\{A,B\}C-A\{B,C\}
    \]
     several times.  First, one gets 
    $$\pi \{\pi,\pi\}\pi^\perp = 0= \pi^\perp\{\pi,\pi\}\pi.$$    
    Then, taking $A=\pi^\perp$, $B=\pi$, $C=H_0$,
     one gets 
    \begin{align*}
\{\pi^\perp, h\pi\} - 0 & = \{\pi^\perp,\pi\}H_0 -\pi^\perp\{\pi,H_0\},
    \end{align*} 
  whence 
  $-\pi^\perp\{\pi,h\}\pi = -\pi^\perp \{\pi,H_0\}\pi.$
Finally, for concluding the construction of $H_1^{\rm adia}$, It remains to check that 
$$ \left( (H_0-h) \Pi_1+ \frac 1{2i} \{ H_0-h,\pi\} +i\partial_t\pi\right) \pi^\perp=0$$
which comes from the latter observation about $\{H_0,\pi\}$.
}

\medskip 

Now that $H_0^{\rm adia}$ and $H_1^{\rm adia}$ are constructed, one uses a recursive argument: assume that one has constructed $H_j^{\rm adia}$ for $0\leq j\leq N$ with $H_j^{\rm adia}\in \bdS ^{-2j}_\delta$ for $j\in\{2,\cdots ,N\} $ and such that    has~\eqref{adia:form_new} holds up to $\mathcal \O(\eps^{N+1})$. Let us construct~$H_{N+1}^{\rm adia}$. Setting as in Notation~\ref{def:Nseries}
$$H^{{\rm adia},\eps, N}=\sum_{j=1}^N \eps^j H_j^{\rm adia}$$
we write 
$$ (H^{{\rm adia},\eps,N} -H_0-\eps H_1)\circledast \Pi^{\eps,N} = \eps^{N+1}T_{N}+\O(\eps^{N+2})$$
with $T_N\in \bdS_{\delta}^{-2N-3}$ and we look for $H^{\rm adia}_{N+1}$ such that 
$\pi  H^{\rm adia}_{N+1}= T_{N}$.  This is doable as long as $\pi^\perp T_{N}=0$, which comes form the observation that 
$$ (H_0+\eps H_1- H^{{\rm adia},\eps,N})\circledast \Pi^{\eps,N} \circledast (1- \Pi^{\eps,N})= \eps^{N+1}  \pi^\perp T_{N}+ \O(\eps^{N+1})=\O(\eps^{N+1})$$
by the properties of superadiabatic projectors. Besides, $H^{\rm adia}_{N+1}\in \bdS_{\delta}^{-2N-3}$, which fits with~(ii) of Definition~\ref{def:symbol_Sdelta}

\medskip 

For concluding the proof, we observe that Point (3) comes from Lemma~\ref{thm:prodest}, and Point 
  (4)  from~\eqref{eq:H1adiab}.
   \end{proof}


 \chapter{{Propagation in adiabatic regimes with small gap}}\label{sec:prop}

  In this section, we prove Theorem~\ref{th:WPmain} outside the gap region, i.e. the following statement:

  \begin{theorem}\label{th:WPmain_hors_gap}
Let $(t_0,z_0)$ and $T>0$ be  such that for $\ell\in\{1,2\}$, we have $\Phi^{t,t_0}_{h_\ell}\in {|f(t,z)|> \delta}$ on the interval $[t_0,t_0+T]$. Let $\psi^\eps_0$ be a polarized wave packet as in~\eqref{def:initial_data} and
$\psi^\eps(t)$ be the solution of  \eqref{eq:sch} with initial datum $\psi^\eps_0$.
There exist $\kappa_0\in\N$ and  two  families of  differential operators  of degree $\leq 3j$  with time dependent smooth vector-valued coefficients,  $\left(\Vec B_{\ell,j}(t) \right)_{j\in\N} $, $\ell\in\{1,2\}$, satisfying~\eqref{B0} and~\eqref{B1},   such that setting 
for $\ell\in\{1,2\}$,
 \begin{align} \label{eq:psiepsell}
&\psi_\ell^{\eps,N}(t) =   {\rm e}^{\frac i\eps S_{\ell }(t,t_0,z_0)}
\wp^\eps_{z_\ell(t)}\left(f^\eps_\ell(t)\right), 
\end{align}
with 
\begin{align*}
f^\eps_\ell(t) & = 
{\mathcal R}_\ell (t,t_0) \, {\mathcal M}[F_\ell( t,t_0)]  \sum_{0 \leq j\leq N}\eps^{j/2}\,\Vec B_{\ell,j}(t) f_0 ,\;\;\ell\in\{1,2\},
\end{align*}
 one has the following property: 
 for all $k,N,M\in\N$, there exists   $C_{M,N,k}>0$ such  for all  $t\in I_\delta$
\[
\left\Vert\psi^\eps(t) - \left(\psi^{\eps,N}_1(t)+\psi^{\eps,N}_2(t) \right)\right \Vert_{\Sigma^k_\eps}
  \leq C_{M,N,k}\left(\frac{\sqrt\eps}{\delta}\right)^{N+1}\delta^{-2\kappa_0} .
  \]
\end{theorem}

In the next chapter, we will  prove Theorem~\ref{th:WPmain}, by applying this result on the intervals $[t_0,t_1^\flat(t_0,z_0)-\delta]$ and on $[t_1^\flat(t_0,z_0]+\delta,t_0+T]$, with different data at $t=t_1^\flat(t_0,z_0]+\delta$ that will be calculated in the next chapter. Assumption~\eqref{eq:times} guarantees that the hypothesis  of  Theorem~\ref{th:WPmain_hors_gap} are satisfied on these intervals.  The parameter $\delta$ will be taken afterwards as $\delta\approx\eps^\alpha$ with $\alpha<\frac 12$ as suggested by the statement.
\medskip

The proof relies on the  use of the superadiabatic projectors constructed in Section~\ref{sec:superadiab}.
In order to explain carefully each step of the proof, we start by proving the propagation
 faraway from the crossing area in Section~\ref{sec:faraway}. That allows us  to settle the arguments, before doing it  precisely close to $\Upsilon$ in Section~\ref{CCA}. 
 All along Section~\ref{sec:prop}, we will use that in the zone where the wave packets $(\psi^{\eps,N}_\ell(t))_{\eps>0}$ are microlocalized, the  Assumption~\ref{hyp:NCdelta} is satisfied. This comes from  the following dynamical Assumption~\ref{hyp:DA} that will be satisfied in different domains of interest.  

\begin{assumption}[Dynamical assumption]\label{hyp:DA}
  We say that $\Omega_1$ and $t_1$ satisfy the dynamical assumption $(DA)$  for the mode $h_\ell$ and the set $\Omega$ if we have 
 \begin{enumerate}
 \item[$(DA)$] 
   $\Phi^{t,t_0}_{h_\ell}(\Omega_1)\subset\Omega$ 
for all  $t\in[t_0, t_1]$.
\end{enumerate}
\end{assumption}

\section{Propagation far away from the crossing area}\label{sec:faraway}
In this section, we analyze the propagation of wave packets in a region where the gap is bounded from below. It gives the opportunity to introduce the method that we shall use in the next section for a small gap region. So, 
we fix $\delta=\delta_0$, $\delta_0>0$  small  but independent  on $\eps$ and we
work in the open set 
$$\Omega:=\{z\in\R^{2d},\; \vert h_2(t,z)-h_1(t,z)\vert>\delta_0,\;\;\forall t\in[t_0, t_1]\},\;\;\mathcal D=I\times \Omega$$
where $I$ is an open interval of $\R$ containing $[t_0,t_1]$ and where the gap condition is also satisfied.
We associate with~$H^\eps$ the formal series  of Theorems~\ref{adia1} and~\ref{quantpro} for each of the modes: 
$$\Pi^\eps_\ell=\sum_{j\geq0} \eps^j \Pi_{\ell, j}\;\;{\rm and}\;\; H^{{\rm adia},\eps}_\ell = \sum_{j\geq 0} \eps^jH^{\rm adia}_{\ell,j}$$
and we will use the notation introduced in~\eqref{def:Nseries}. 

With $z_0\in\Omega$, we associate the open sets $\Omega_0$, $\Omega_1$, $\Omega_2$ and $\Omega_3$ such that 
$$z_0\in \Omega_0\Subset\Omega_2\Subset\Omega_1\Subset\Omega_3\Subset\Omega $$
where $\Omega_0$ and $\Omega_2$ are constructed so that 
for any initial data  $z\in\Omega_0$  the flows
are staying in $\Omega_2$: 
$$\Phi_{h_\ell}^{t,t_0}(z)\in\Omega_2,\;\; \forall t\in[t_0, t_1],\;\;\forall z\in\Omega_0,\;\;\forall \ell\in\{1,2\}.$$

We associate cut-offs to these subsets. We take 
 $\chi_0\in C_0^\infty(\Omega_0)$  with $\chi_0=1$ near $z_0$. Then, we choose  $K_0$  a compact neighborhood  of $z_0$ in $\Omega_0$, which implies that  $\Omega_2$ is a neighborhood of 
\begin{equation}\label{def:tildeK0}
\widetilde K_{\ell,0}:= \{\Phi_{h_\ell}^{t,t_0}(K_0),\; t_0\leq t\leq t_1\},\;\;\ell\in\{1,2\}.
\end{equation}
So we can choose $\chi_2\in C_0^\infty(\Omega)$ with $\chi_2=1$ on $\widetilde K_0=\widetilde K_{1,0}\cup\widetilde K_{2,0}\subset\Omega_2$. 
Finally, we take $\chi_1\in C_0^\infty(\Omega_1)$ and $\chi_3\in C_0^\infty(\Omega_3)$ with $\chi_1=1$ on ${\rm supp}(\chi_2)\subset\Omega_2$  and  $\chi_3=1$ on ${\rm supp} (\chi_1)\subset \Omega_1$.

\smallskip 

For $\ell\in\{1,2\}$, we set 
\[
\widetilde H^{{\rm adia},\eps,N}_\ell(t)=\chi_3 \, H^{{\rm adia},\eps,N}_\ell(t),
\]
which is a smooth subquadratic Hamitonian, 
and we consider ${\mathcal U}^{{\rm adia},\eps,N }_{\ell}(t,s)$  the propagator associated  with the Hamiltonian $\chi_3\, H_\ell^{{\rm adia}, \eps, N}(t)$.

\smallskip

The next result is the usual adiabatic decoupling that results from the analysis of Chapter~\ref{chap:2}. We sometimes shorten ${\rm op}_\eps(\chi_j)$ in $\widehat \chi_j$, $j\in\{0,1,\cdots,3\}$.

\begin{proposition}[adiabatic decoupling - I]\label{far} 
Let $k, N\in\N$ and $t\in[t_0,t_1]$.
\begin{enumerate}
\item[(i)] For any $\ell\in\{1,2\}$, we have in $\mathcal L(\Sigma^k_\eps)$,
  \begin{align}\label{appadia}
   &  \left(i\eps\partial_t -\widehat H^\eps(t)\right){\rm op}_\eps\left(\chi_1\Pi^{\eps,N}_\ell (t)\right)\,{\rm op}_\eps(\chi_2) =\\
   \nonumber
&\qquad      {\rm op}_\eps\left(\chi_1\Pi^{\eps,N} _\ell (t)\right)\, 
\left(i\eps\partial_t -{\rm op}_\eps\left(\widetilde H^{{\rm adia},\eps,N}_\ell(t)\right)\right) \,{\rm op}_\eps(\chi_2) +{\mathcal O}(\eps^{N+1}).
   \end{align}
  \item[(ii)]  Let $\psi^\eps_0\in \Sigma^k_\eps$ such that 
  $\widehat\chi_0\psi^\eps_0 =\psi^\eps_0 +{\mathcal O}(\eps^{N+1})$ in $\Sigma^k_\eps$.
   Set
    $$\psi^{\eps,N}_\ell (t) = {\rm op}_\eps\left(\chi_2\Pi_\ell^{\eps,N}(t)\right)\, {\mathcal U}^{{\rm adia}, \eps,N}_\ell (t,t_0)
    \, {\rm op}_\eps\left(\chi_0\Pi_\ell^{\eps,N}(t_0)\right)\psi^\eps_0,\;\;\ell\in\{1, 2\}.$$ 
   Then we have in $\Sigma^k_\eps$,
 \beq\label{addec}
   {\mathcal U}^\eps_H(t,t_0)\psi^\eps_0 = \psi^{\eps,N}_1 (t) + \psi^{\eps,N}_2 (t) + {\mathcal O}(\eps^{N+1}),\; \forall t\in[t_0, t_1].
   \eeq
\end{enumerate}
 \end{proposition}

 \begin{remark}\label{expand}
 \begin{enumerate}
 \item   The assumption satisfied by $(\psi^\eps_0)_{\eps>0}$ in (ii) of  Proposition~\ref{far} is sometimes referred in the literature as  having  a compact  semi-classical wave front  set, or being mcrolocalized in $\Omega_0$.
 \item 
  In the proof below, the reader will notice that  we do not need to assume  that $H^\eps(t)$ is sub-quadratic,  we only need to know that $\widehat H^\eps(t)$  defines a unitary Schr\"odinger propagator in $L^2(\R^d,\C^m)$. However, we use the boundedness of the derivatives of the projectors.
  \end{enumerate}
    \end{remark}

  \begin{proof}   
     (i) We fix $\ell\in\{1,2\}$. Using Theorem~\ref{adia1} and the fact that $\chi_3\chi_1=\chi_1$, we obtain
   \begin{align*}
      (i\eps\partial_t - H^\eps (t))\circledast \left(\chi_1\Pi^{\eps,N }_\ell(t) \right)& = \chi_1     (i\eps\partial_t - H^\eps (t))\circledast \left(\chi_1\Pi^{\eps,N }_\ell(t) \right)
     +\eps^{N+1} \widehat R^\eps_N\\
     &= \chi_1  \left(\Pi^{\eps,N }_\ell(t) \right)\circledast \left(i\eps\partial_t -H^{{\rm adia},\eps,N}_\ell (t)\right)
      +\eps^{N+1}\widehat R^\eps_N\\
      & =   \left(\chi_1\Pi^{\eps,N }_\ell(t) \right)\circledast \left(i\eps\partial_t -\widetilde H^{{\rm adia},\eps,N}_\ell (t)\right)
  +\eps^{N+1}\widehat R^\eps_N
  \end{align*}
  where the operators $\widehat R_N^\eps$-s are rest terms that may change from one line to the other one. Their symbols  all satisfy $\chi_2 R_N^\eps=0$. When quantizing the symbols, the relation~\eqref{appadia} follows from Theorem~\ref{thm:moyalest} and Corollary~\ref{cor:kappa0} (with $\delta=1$).
      \medskip 
            
 (ii) We calculate for $\ell\in\{1,2\}$ the quantity $(i\eps\partial_t -\widehat H^\eps) \psi^{\eps,N}_\ell(t)$. By Point (i), and because $\chi_2\chi_1=\chi_2$, we have in $\Sigma^k_\eps$
\begin{align*}
    (i\eps\partial_t -\widehat H^\eps) \,\widehat{\chi_ 2\Pi^{\eps,N}_\ell}&=(i\eps\partial_t -\widehat H^\eps) \, \widehat{\chi_1\Pi^{\eps,N}_\ell}\,  \widehat \chi_2+{\mathcal O}(\eps^{N+1})\\
    &= \widehat{\chi_1\Pi^{\eps,N}_\ell} \,(i\eps\partial_t -\widetilde H_\ell^{{\rm adia},\eps,N})\, \widehat \chi_2+{\mathcal O}(\eps^{N+1}).
\end{align*} 
Therefore, 
\begin{align*}
(i\eps\partial_t -\widehat H^\eps) \psi^{\eps,N}_\ell(t)  = 
\widehat{\chi_1\Pi^{\eps,N}_\ell} \, (i\eps\partial_t -\widetilde H_\ell^{{\rm adia},\eps,N})\, \widehat \chi_2\, 
{\mathcal U}^{{\rm adia}, \eps,N}_{\ell}(t,t_0)\, \widehat{\chi_0\Pi_\ell^{\eps,N}} \,\psi_0^\eps.
\end{align*}
We now study $\widehat \chi_2\,
{\mathcal U}^{{\rm adia}, \eps,N}_{\ell}(t,t_0)\,\widehat{\chi_0\Pi_\ell^{\eps,N}}$. 
 We  recall that
      $\chi_2$ is identically equal to~$1$ on the compact set~$\widetilde K_0$, defined in~\eqref{def:tildeK0}, with  $K_0={\rm supp}(\chi_0)$. 
      Hence, using Egorov Theorem of Appendix \ref{app:egorov} and its notation, we have for $\ell\in\{1,2\}$
      \begin{align}
      \nonumber
      \widehat\chi_2\, {\mathcal U}^{{\rm adia}, \eps,N}_{\ell}(t,t_0) \, \widehat{\chi_0\Pi_\ell^{\eps,N}} & \, = 
        {\mathcal U}^{{\rm adia}, \eps,N}_\ell (t,t_0) \, {\mathcal U}^{{\rm adia}, \eps,N}_\ell(t_0,t)\, \widehat\chi_2
      \, {\mathcal U}^{{\rm adia}, N, \eps}_\ell(t,t_0) \, \widehat{\chi_0\Pi_\ell^{\eps,N}} \\
      \nonumber
      & \, = {\mathcal U}^{{\rm adia}, \eps,N}_\ell(t,t_0) \, 
      \, \widehat{\chi_1^{\eps, N}(t,t_0)}\,\widehat{\chi_0\Pi_\ell^{\eps,N}} +{\mathcal O}(\eps^{N+1})
    \\
    \label{uadia1}
    & \, = {\mathcal U}^{{\rm adia}, \eps,N}_\ell(t,t_0) \, 
      \, \widehat{\chi_0\Pi_\ell^{\eps,N}}  +{\mathcal O} (\eps^{N+1}).
\end{align}
We are left with 
\begin{align*}
(i\eps\partial_t -\widehat H^\eps) \psi^{\eps,N}_\ell(t) & = 
\widehat{\chi_1\Pi^{\eps,N}_\ell} (i\eps\partial_t -\widetilde H_\ell^{{\rm adia},\eps,N})\,
{\mathcal U}^{{\rm adia}, \eps,N}_{\ell}(t,t_0)\, \widehat{\chi_0\Pi_\ell^{\eps,N}} \,\psi_0^\eps +{\mathcal O} (\eps^{N+1})\\
&=
{\mathcal O} (\eps^{N+1}).
\end{align*}
Since $\psi^\eps(t)$ solves $(i\eps\partial_t -\widehat H^\eps) \psi^\eps(t)=0  $, we now compare the data at $t=t_0$. 
By Theorem~\ref{adia1} and because $\chi_0\chi_2=\chi_0$, we have
      $$
      \widehat\chi_0 ={\rm op}_\eps\left(\chi_2\Pi_1^{\eps,N}+\chi_2\Pi_2^{\eps,N}\right)\widehat\chi_0+{\mathcal O}(\eps^{N+1}).
      $$
      Therefore, 
      \[
      \psi^\eps(t_0)= \widehat \chi_0 \psi^\eps(t_0)+  {\mathcal O}(\eps^{N+1})
      = \psi_\ell^{\eps,N} (t_0)+\psi_\ell^{\eps,N} (t_0)+{\mathcal O}(\eps^{N+1}).
      \]
 We deduce 
\[
\psi^\eps(t)=  \psi^{\eps,N}_1(t)  +(i\eps\partial_t -\widehat H^\eps) \psi^{\eps,N}_2(t)  +{\mathcal O}(\eps^{N+1}),
\]
which concludes the proof.
 \end{proof}

The adiabatic decoupling of Proposition~\ref{far}  and Egorov Theorem (see Proposition~\ref{EgD}) allow to give an explicit description at any order of the solutions of equation~\eqref{eq:sch} for initial data that are focalized wave packets.

\subsection{Application to adiabatic propagation of wave packets}\label{sec54first}

We use  the techniques of Appendix~\ref{app:C} that are classic when $\delta=1$, see for example the recent edition of~\cite{corobook}. One constructs 
 two maps ${\mathcal R}_1 (t,t_0,z)$ and ${\mathcal R}_2(t,t_0,z)$ as introduced in~\eqref{def:Rell} and obtains 
the following result, see also Proposition~\ref{evadia}).

\begin{theorem}\label{WPexp1}
 Assume that $\psi^\eps_0$ is a polarized wave packet:  
$$\psi^\eps_0= \wp^\eps_{z_0}(f_0)\,\vec V_0,\;\;\mbox{with} \;\;f_0\in{\mathcal S}(\R^d)\;\;\mbox{and}\;\;\vec V_0\in\C^m.$$
Let $N\geq 1$ and $k\geq 0$. Then, there exists a constant $C_{N,k}>0$ such that  
 the solution $\psi^\eps(t)$ of  \eqref{eq:sch} satisfies for all $t\in [t_0,t_1]$, 
\[
\left\Vert\psi^\eps(t) - \left(\psi^{\eps,N}_1(t)+\psi^{\eps,N}_2(t) \right)\right\Vert_{\Sigma^k_\eps}
  \leq C_{N,k} \, \eps^{\frac{N+1}{2}},\;
\]
  with for $\ell\in\{1,2\}$ and for all $N\geq 0$, 
  \begin{equation}\label{wpapproxfar}
\psi_\ell ^{\eps,N}(t)  =   {\rm e}^{\frac i\eps S_\ell (t,t_0,z_0)}\, 
\wp^\eps_{z_\ell (t)}\left({\mathcal R}_\ell(t,t_0,z_0) {\mathcal M}[F_\ell (t,t_0)]\sum_{0 \leq j\leq N}\eps^{j/2}\Vec B_{\ell,j}(t) f_0\right),
\end{equation}
   where $\Vec B_{\ell,j}(t)$  are differential operators of degree $\leq 3j$ with vector-valued  time-dependent coefficients  satisfying~\eqref{B0} and~\eqref{B1}.
 \end{theorem}

 Theorem~\ref{WPexp1} is a simple version of Theorem~\ref{th:WPmain_hors_gap} in the case where we have no crossing in the time interval $[t_0, t_1]$. However, the main ideas of the proof of Theorem~\ref{th:WPmain_hors_gap} are already there. 

 \begin{proof}
       By Proposition~\ref{far}, we have 
       \[
       \psi^\eps(t)=\mathcal U_H^\eps (t,t_0)\psi^\eps_1=
       \sum_{\ell=1,2}
       {\rm op}_\eps\left(\chi_2\Pi_\ell ^{\eps,N }(t)\right)
{\mathcal U}^{{\rm adia}, \eps,N}_\ell (t,t_0)
\, {\rm op}_\eps\left(\chi_0\Pi_\ell^{\eps,N}(t_0)\right)
\psi^\eps_1+{\mathcal O}\left( \eps^{N+1}\right).       \]
It is thus enough to prove that for $\ell\in\{1,2\}$, $\psi_\ell^{\eps,N}(t)$ given by~\eqref{wpapproxfar} satisfies 
\[
\psi_\ell^{\eps,N}(t)={\rm op}_\eps\left(\chi_2\Pi_\ell ^{\eps,N }(t)\right)
{\mathcal U}^{{\rm adia}, \eps,N}_\ell (t,t_0)
\, {\rm op}_\eps\left(\chi_0\Pi_\ell^{\eps,N}(t_0)\right)
\psi^\eps_1+{\mathcal O}\left( \eps^{N+1}\right).
\]
The micro-localization of wave packets implies 
\[
{\rm op}_\eps\left(\chi_0\Pi_\ell^{\eps,N}(t_0)\right)
\psi^\eps_1=\wp^\eps_{z_0}(\vec f^\eps_\ell)
+{\mathcal O}\left( \eps^{N+\frac 12}\right)
\]
with 
\begin{align*}
    &
\vec f^\eps_\ell=\sum_{j=0}^{N} \eps^{\frac j2} \vec f_{\ell,j},\;\;\vec f_{\ell,j}\in\mathcal S(\R^{d},\C^m),\\
&\vec f_{\ell,0}=f_0\,\pi_\ell(t_0,z_0)\vec V_0\;\;\mbox{and}\;\;\vec f_{\ell,1}=\nabla\pi_\ell(t_0,z_0)\cdot\widehat z \,f_0 \,\vec V_0.
\end{align*}
Applying Theorem~\ref{evadia} with $k = h_\ell$ and $K_1 = H^{\rm adia}_{\ell,1}$, we obtain that
\begin{align*}
    {\mathcal U}^{{\rm adia}, \eps,N}_\ell (t,t_0)&
\, {\rm op}_\eps\left(\chi_0\Pi_\ell^{\eps,N}(t_0)\right)
\psi^\eps_1 = {\mathcal U}^{{\rm adia}, \eps,N}_\ell (t,t_0)
\, \wp^\eps_{z_0}(\vec f^\eps_\ell)
+{\mathcal O}\left( \eps^{N+\frac 12}\right)\\
&={\rm e}^{\frac i\eps S_\ell(t,t_0,z_0)}
\wp^\eps_{z_\ell(t)}\left(\sum_{j=0}^{N}
\eps^{\frac j2}{\mathcal R}_\ell(t,t_0)\,
         {\mathcal M}[F_\ell(t,t_0)]\vec U_{\ell, j}(t)\right)
+{\mathcal O}\left( \eps^{\frac{N+1}2}\right),
\end{align*}
   with
 \begin{align*}
\vec U_{\ell,0} (t) &=  \vec f_{\ell,0}
         =\pi_\ell(t_0,z_0)\vec V_0 f_0
         =\vec B_{\ell,0}f_0,
\end{align*}
using~\eqref{B0}, and 
\begin{align*}
    \vec U_{\ell,1} (t)   & =  \vec f_{\ell,1}+ 
         {\bf b}_1(t,t_0)\vec f_{\ell,0} = \nabla\pi_\ell(t_0,z_0)\cdot\widehat z \,f_0 \,\vec V_0+ 
         {\bf b}_1(t,t_0)\pi_\ell(t_0,z_0)\vec V_0 f_0
  \end{align*}
  where ${\bf b}_1(t,t_0)$ is given \begin{align*}
  {\bf b}_1(t,t_0)& = \frac 1i \sum_{\vert\alpha\vert=3}\frac{1}{\alpha!}\int_{t_0}^t\partial_z^\alpha h_\ell(s,z_\ell(s)){\rm op}_1^w((F_\ell(s, t_0)z)^\alpha )ds \, \1_{\C^m}\nonumber\\
 &\qquad + \frac 1i\int_{t_0}^t 
 \mathcal R(t_0,s)\, 
 \partial_z H_{\ell,1}^{\rm adia}(s,z_\ell(s)\, 
 {\rm op}_1^w(F_\ell(s, t_0)z ) \, 
 \mathcal R(s,t_0) \, ds.
\end{align*}
In view of~\eqref{B1}, we obtain 
\[
 \vec U_{\ell,1} (t) =\vec B_{\ell,1}(t) f_0,
\]
whence 
\[
 {\mathcal U}^{{\rm adia}, \eps,N}_\ell (t,t_0)
\, {\rm op}_\eps\left(\chi_0\Pi_\ell^{\eps,N}(t_0)\right)
\psi^\eps_1=
 \psi_\ell^{\eps,N}(t)
+{\mathcal O}\left( \eps^{N+\frac 12}\right)
\]
and we conclude using again the microlocalisation of wave packets and observing that
\[
{\rm op}_\eps\left(\chi_2\Pi_\ell ^{\eps,N }(t)\right)\psi_\ell^{\eps,N}(t)=\psi_\ell^{\eps,N}(t)+{\mathcal O}\left( \eps^{N+\frac 12}\right).
\]
   \end{proof}


\section{Propagation close to the crossing area}\label{CCA}
Our goal in this section is to extend the analysis of the preceding section
up to a time $t^\flat-c\delta$ for some $c>0$ and $\delta\ll 1$.
We follow the same strategy as in the preceding section and carefully check the dependence in~$\delta$ of the estimates. 
\smallskip

For proving Theorem~\ref{th:WPmain_hors_gap},  we consider a wave packet at initial time $t_0$ that is focalized along the mode $h_1$ at some point $z_0$. 
We let it evolve along that mode according to Theorem~\ref{WPexp1}, up to $(t_1, z_1)$ conveniently chosen and we consider $\psi^\eps(t_1)$  as a  new initial data, knowing that it is a  wave-packet, modulo ${\mathcal O}(\eps^\infty)$. 
We  choose a point $(t_1, z_1)$  close enough to $(t^\flat, \zeta^\flat)$ such that $\vert t_1-t^\flat\vert +\vert z_1-\zeta^\flat\vert\leq \eta_0$ where $\eta_0$ is defined in the next Lemma.

\begin{lemma} \label{tras_lem}
Assume $(t^\flat,\zeta^\flat)$ is a generic smooth crossing point as in Definition~\ref{def:smooth_cros} and consider $\delta\in (0,1]$.
\begin{enumerate}
\item 
There exist
     $\eta_0>0$ and $c_0>0$  such that we have
        $$
    \vert f(t, \Phi_{h_1}^{t,t_1}(z))\vert \geq c_0\vert t-t^\flat\vert\; \mbox { if}\; \;\vert t-t^\flat\vert +\vert z-\zeta^\flat\vert\leq \eta_0. 
    $$
    \item There exists $c, M>0$ such that for all $(t,z)$ satisfying $\left\vert z-\Phi_{h_1}^{t,t_1}(z_1)\right\vert\leq c\delta$, we have 
    $$\vert f(t,z)\vert \geq c_0\delta - Mc\delta\geq \frac{c_0}{2}\delta .$$
    \end{enumerate}
    \end{lemma}

\begin{proof}
The result comes readily from  the transversality of the curve $t\mapsto  \Phi_{h_1}^{t,t_1}(z)$ to the set $\Upsilon=\{f=0\}$. Recall that this transversality is due to Point (b) of Assumption~\ref{def:smooth_cros}.
\end{proof}

Thanks to Lemma~\ref{tras_lem}, we now work close to the point $(t^\flat, \zeta^\flat)$ and our goal  is  to prove accurate estimates  for the evolution of the solution $\psi^\eps(t)$ of the Schr\"odinger equation  with the initial data $\psi^\eps(t_1)$, 
for $t\in [t_1, t^\flat-c\delta]$.
\smallskip 

We follow the strategy  of the previous section. Therefore, in order to obtain the analogue of Proposition~\ref{far}, we need to be able to evaluate the rest term in the Egorov Theorem in terms of~$\delta$, and to construct adapted $\delta$-dependent cut-off functions.  For estimating the remainders  in the small parameter $\delta$, we use  the symbolic calculus in the classes~${\bf S}_{\eps,\delta}$, in particular 
  the Moyal product rules as stated in Lemma~\ref{thm:prodest} and  the  Egorov Theorem of Proposition~\ref{EgD}. Finally,  the construction of the cut-off functions relies on the fact that due to Point (b) of Assumption~\ref{def:smooth_cros}, we can apply a straightening theorem for the Hamiltonian vector fields associated with the modes.

\subsection{Localization up to the crossing region}
We construct the cut-off functions by using thin tubes along  the classical trajectories.
 We use a  straightening  theorem for non singular vector fields. We set  $D(z_1,\rho_1)=\{|z-z_1|\leq \rho_1\}$  and consider a branch of  trajectory
 \[
 {\mathcal T}_1:=\{\Phi_h^{t,t_1}(z_1),\; t\in[t_1, t_1^+]\},\;\;t_1^+>t^\flat.
 \]

  \begin{lemma}\label{SVF}\cite{arno}
  Let be $P_1$  a transverse hyperplane to the curve  ${\mathcal T}_1$ in $z_1$. 
  There exist $\rho_1>0$  and  $t_1^-<t_1<t^\flat<t_1^+$   such that the map 
  $$(t, z)\mapsto\Phi^{t,t_1}_h(z)$$
   is a diffeomorphism from
  $]t_1^-, t_1^+[\times D(z_1,\rho_1)$ onto a neighborhood ${\mathcal W}_1$ of  ${\mathcal T}_1$  in $ P_1$.
  \end{lemma}
  
 Hence for any $z$ in the tube ${\mathcal W}_1$, we have 
 \[
 z=\Phi_h^{\tau(z), t_1}(Y(z))
 \]
  where  $\tau$ and $Y$ are smooth functions of $z$, $\tau(z)\in [t_1, t_1^+]$, $Y(z)\in D(z_1,\rho_1)$.
  \smallskip 
  
  We then define the cut-off functions as follows: consider 
  \begin{itemize}
  \item  $\zeta\in C_0^\infty(]-2,2[)$ equal to 1 in $[-1, 1]$, 
  \item  $\theta\in C^\infty(\R)$ with $\theta(t)=0$ if  $t\leq-1$ and  $\theta(t)=1$ if $t\geq1$,
  \end{itemize}
  we set for $\delta>0$,
 $$
 \chi^\delta(z) = \theta\left(\frac{\tau(z)-t_1^-}{\eta}\right)\left(1-\zeta\right)\left(\frac{\tau(z)-t^\flat}{c\delta}\right)\zeta\left(\frac{\vert z-\Phi_h^{\tau(z),t_1}(z_1)\vert^2}{(C\delta)^2}\right),
 $$
 where $c>0, C>0$,  and $\eta>0$ is  a small enough constant. 
 \smallskip 
 
 By choosing adapted constants $c_j$ and $C_j$} for $j\in\{1,2,3\}$, 
 conveniently, 
 we construct some functions 
 \[
 \chi^\delta_ j\in\mathcal C_0^\infty(\R^{2d},[0,1]), \;\;j\in\{1,2,3\},
 \]
  such that 
 \begin{enumerate}
 \item   $\chi^\delta_j=1$ on $\di{\bigcup_{t_0\leq t\leq t_1}B\left (\Phi_h^{t_0,t}(z_0),c_j\delta\right)}$
 and $\chi^\delta_j$ is supported in $\di{\bigcup_{t_0\leq t\leq t_1}B\left(\Phi_h^{t_0,t}(z_0),2c_j\delta\right)}$,
\item for all $\gamma\in\N^{2d}$, there exists $C_\gamma$ such that   for all $z\in\R^{2d}$ 
 $$\vert\partial_z^\gamma\chi_j^\delta(z)\vert \leq C_\gamma \, \delta^{-\vert\gamma\vert},$$
\item 
  $ \chi_3^\delta=1$ on supp$\chi_1^\delta$ and 
  $\chi_1^\delta=1$ on supp$\chi_2^\delta$.  
  \end{enumerate}
\smallskip
 
Finally, with  $\chi_0\in \mathcal C_0^\infty(\R^{2d},[0,1])$ satisfying $\chi_0=1$
  on  $B(0, 1)$ and $\chi_0(z) =0$ for $\vert z\vert\geq 2$, we associate 
$$\chi^\delta_0(z) = \chi_0\left(\frac{z_1-z}{\delta}\right).$$
And we consider $\chi_4$  a smooth fixed cut-off ($\delta$-independent).

     \subsection{Adiabatic decoupling close to the gap}
  Omitting the mode index, we set 
  \begin{equation}\label{def:tilde_adia}
  \widetilde H^{{\rm adia},\eps,N }_\ell(t) = \chi_4\left(h_\ell(t) +\eps H_{\ell,1}^{\rm adia}(t)\right)+\chi_3^\delta\left(\sum_{2\leq j \leq N}\eps^j H^{\rm adia}_{\ell,j}(t)\right).
  \end{equation}
  Notice that, because the  crossing is smooth, the eigenvalues $h_\ell$, $\pi_\ell$ and the first adiabatic correctors $H_{\ell,1}^{{\rm adia}}$ are  smooth, even  in a neighborhood of $(t^\flat, \zeta^\flat$). 
  \smallskip 
  
  Let
  ${\mathcal U}^{{\rm adia}, \eps,N }_\ell(t,s)$ be  the quantum propagator
 associated with the Hamiltonian  $\widetilde H^{{\rm adia},\eps,N}_\ell(t) $.
  The following result is a consequence of the sharp estimates given in \cite{BR} concerning propagation of quantum observables and proved in the Appendix, see (ii) of the Egorov Theorem~\ref{EgD}.

\begin{proposition}\label{prop:Ego1} Consider the cut-off functions $\chi_0^\delta$ and $\chi_2^\delta$ defined above and set for $t\in [t_1, t^\flat-c\delta]$
 $$
 {\rm op}_\eps(\chi_0^{\delta,\eps}(t,t_1)) :=  {\mathcal U}^{{\rm adia}, \eps,N}_\ell(t_1,t) \, {\rm op}_\eps(\chi^\delta_0)\; {\mathcal U}^{{\rm adia}, \eps ,N}_\ell(t,t_1).$$
Then,   for any $M\geq 1$, $z\in\R^{2d}$ and  $t\in[t_1, t^\flat-c\delta]$,   we have:
$$
(1-\chi_2^\delta)\circledast\chi_0^{\delta,\eps,M}(t,t_1) =  \left(\frac{\eps}{\delta^2}\right)^M\zeta_M(t)\;\;\mbox{with}\;\;  \zeta_M(t,\cdot )\in \bdS_{\delta^2}(\mathcal D).
$$
\end{proposition}

  Revisiting  the proof of Proposition \ref{far}, using Lemma~\ref{thm:prodest} for the formal series  $\Pi^\eps_\ell \in \bdS_{\delta^2}^0(\mathcal D) $ and  using (3) of Theorem~\ref{adia1} about $H^{\eps,{\rm adia}}_\ell$, we obtain the following result.

 \begin{proposition} [adiabatic decoupling - II]\label{prop:close}
 With the previous notations, we have the following properties.
 
 (ii) For $t_1\leq t\leq t^\flat-\delta$, we have 
 \begin{align} \label{closenadia}
  &  \left(i\eps\partial_t -\widehat H^\eps(t) \right) \,   
    {\rm op}_\eps\left(\chi_1^\delta\, \Pi^{\eps,N }_\ell(t) \right) \, {\rm op}_\eps(\chi^\delta_2)  \\
    \nonumber 
    & \qquad ={\rm op}_\eps \left(\chi_1^\delta \, \Pi^{\eps,N}(t) \right)
     \left(i\eps\partial_t -{\rm op}_\eps(\widetilde H^{{\rm adia},\eps,N}_\ell(t) )\right)
      \, {\rm op}_\eps(\chi_2^\delta )
      +{\mathcal O}\left(\left(\frac{\eps}{\delta^2}\right)^{N+1}\delta^{-\kappa_0}\right)
   \end{align}
     where $\kappa_0\in \N$ is $N$-independent.
     
(ii)    Consider
$$\psi^{\eps,N}_\ell (t) = {\rm op}_\eps\left(\chi_2^\delta\Pi_\ell ^{\eps,N }(t)\right)
{\mathcal U}^{{\rm adia}, \eps,N}_\ell (t,t_1)
\, {\rm op}_\eps\left(\chi_0^\delta\Pi_\ell^{\eps,N}(t_1)\right)
\psi^\eps(t_1)$$
where   $\psi^\eps(t_1) = {\mathcal U}^\eps_H(t_1,t_0)\psi^\eps(t_0)$.
   Then we have, for $N\geq 2$ and for all $t\in[t_1, t^\flat-\delta]$,
   $$
   {\mathcal U}^\eps_H(t,t_0)\psi^\eps_0 = \psi^{\eps,N}_1 (t) + \psi^{\eps,N}_2 (t) + {\mathcal O}\left(\left(\frac{\eps}{\delta^2}\right)^{N+1}\delta^{-\kappa_0}\right).
   $$
where
   ${\mathcal U}^{{\rm adia}, \eps,N}_\ell $ is the propagator associated  with the Hamiltonian $\widetilde H^{{\rm adia},\eps,N}(t) $.
 \end{proposition}

              The Remark \ref{expand} is valid also for this Proposition.
       Note that the integer $\kappa_0$ stems from the symbolic calculus estimates of Theorem~\ref{thm:moyalest}.

      \subsection{Proof of Theorem~\ref{th:WPmain_hors_gap}}\label{sec54}
     The previous results  have consequences for wave packets propagation and allow to prove   Theorem~\ref{th:WPmain_hors_gap}, which provides 
     an asymptotic expansion modulo ${\mathcal O}(\eps^\infty)$ for any $\alpha<1/2$ if 
     $\delta\approx \eps^{\alpha}$. In  other words,  
     the super-adiabatic approximation is valid for   times $t$  such that $\vert t-t^\flat\vert \geq \eps^{1/2-\eta}$,
      for any $\eta>0$.

   \begin{proof}[Proof of Theorem~\ref{th:WPmain_hors_gap}]
   We shall use for $\ell\in\{1,2\}$ that since
$\Phi_{h_\ell}^{t_1,t_0}(z_0)=z_\ell(t_1)=z_1$, we have 
   \begin{align}\nonumber
&\Phi_{h_\ell}^{t,t_1}(z_1)= \Phi_{h_\ell}^{t,t_0}(z_0)=z_\ell(t)\qquad 
,\\ \nonumber
&S_\ell (t,t_1,z_1)=S_\ell (t,t_0,z_0),\\ \label{eq:Fgroup}
&F_\ell (t,t_1,z_1)F_\ell(t_1,t_0,z_0)=F_\ell(t,t_0,z_0).
\end{align}
We start from $\psi^\eps(t)= \mathcal U^\eps_H(t,t_1) \psi^\eps(t_1)$. By Theorem~\ref{WPexp1}, $\psi^\eps(t_1)= \psi^\eps_1+ \mathcal O(\eps^{N+\frac 12})$ with 
\begin{align*}
&\psi^\eps_1=  {\rm e}^{\frac i\eps S_\ell (t_1,t_0,z_0)}\, 
\wp^\eps_{z_\ell (t_1)}\left(\sum_{0 \leq j\leq 2N}\eps^{j/2}\Vec g_{\ell,j}(t_1) f\right), \\
&\vec g_{\ell,j}(t_1)= 
{\mathcal R}_\ell(t_1,t_0,z_0) {\mathcal M}[F_\ell (t_1,t_0,z_0)]\Vec B_{\ell,j}(t_1) f_0,
\end{align*}
   where $\Vec B_{\ell,0}(t_1)$ and $\Vec B_{\ell,1}(t_1)$ are  given by~\eqref{B0} and~\eqref{B1} for $t=t_1$.
Therefore, for $t\in[t_1, t^\flat -\delta)$,
\[
\psi^\eps(t)= \mathcal U^\eps_H(t,t_1) \psi^\eps_1+\mathcal O(\eps^{N+\frac 12})
\]
and
we  have to analyze  the propagation of the wave packet $\psi^\eps_1 $ from time $t_1$ to $t\in(t_1,t^\flat-\delta)$ by $\mathcal U^\eps_H(t,t_1)$.
\smallskip 

   We follow the arguments of the proof of Theorem~\ref{WPexp1}. The only difference is the size of the rest terms. We deduce 
 \begin{align*}
    \mathcal U^\eps_H(t,t_1) \psi^\eps_1&=\sum_{\ell=1,2}  {\rm op}_\eps\left(\chi_1^\delta\Pi_\ell ^{\eps,N }(t)\right) {\mathcal U}^{{\rm adia}, \eps,N}_\ell (t,t_1)
\, {\rm op}_\eps\left(\chi_0^\delta\Pi_\ell^{\eps,N}(t_1)\right)
\psi^\eps_1 +{\mathcal O}\left( \left(\frac{\eps}{\delta^2}\right)^{N+1}\delta^{-\kappa_0}\right)  \\
&= \widetilde \psi^{\eps,N}_1(t) + \widetilde \psi^{\eps,N}_2(t) 
+{\mathcal O}\left( \left(\frac{\eps}{\delta^2}\right)^{N+1}\delta^{-\kappa_0}\right)      
\end{align*}
with
   \begin{align*}
\widetilde \psi^{\eps,N}_\ell(t)& =  {\rm e}^{\frac i\eps S_\ell (t,t_1,z_1)}\, 
\wp^\eps_{z_\ell (t)}\left({\mathcal R}_\ell(t,t_1,z_1) {\mathcal M}[F_\ell (t,t_1,z_1)]\sum_{0 \leq j\leq 2N}\eps^{j/2}\Vec U_{\ell,j}(t) \right),
\end{align*}
where
 \begin{align}
 \nonumber
 & \vec U_{\ell,0} (t) = \pi_\ell(t_1,z_1)\vec g_{\ell, 0}(t_1),\\
 \label{alessandra}
 & \vec U_{\ell,1} (t) = (a)+(b)+(c)+(d),
\end{align}
with
 \begin{align}
 \nonumber
&\;\;(a)= \pi_\ell(t_1,z_1)\vec g_{\ell, 1}(t_1),\\
\nonumber
&\;\;(b)= 
 \sum_{\vert\alpha\vert=3} \frac{1}{\alpha!} \frac{1}{i} \int_{t_1}^{t} \partial_z^\alpha h_\ell(s,z_\ell(s))\,
\op_1^w[(F_\ell (s,t_1,z_1)z)^{\alpha}] \,ds\,\pi_\ell(t_1,z_1)\vec g_{\ell, 0}(t_1),\\
\nonumber
&\;\;(c)= \frac 1i\int_{t_1}^t
\mathcal R(t_1,s,z_1)\,
\nabla_z H^{\rm adia}_{\ell,1}(s,z_s)\cdot 
(F_\ell (s,t_1,z_1)\widehat z)\,
\mathcal R(s,t_1,z_1)\, ds\,  \pi_\ell(t_1,z_1)\vec g_{\ell, 0}(t_1),  \\
\nonumber
&\;\;(d)= \widehat z\cdot \nabla \pi_\ell(t_1,z_1)\,\vec g_{\ell, 0}(t_1),
\end{align}
and the $\vec U_{\ell,j} (t)$ for $j\geq 2$  are  linear combinations of differential operators of degree $\leq 3k$  with time dependent smooth vector-valued coefficients, applied to the $\vec g_{\ell,k'}(t_1)$, with $k+k'=j$. Thus they can be written in the adequate form after application of the (exact) Egorov Theorem in~\eqref{prop:metaplectic}. 
\smallskip

We now study more precisely $ \vec U_{\ell,0} (t)$. We have
\begin{align}\nonumber
\vec U_{\ell,0}(t)=\pi_\ell(t_1,z_1)\vec g_{\ell, 0}(t_1)&=\pi_\ell(t_1,z_1)
    {\mathcal R}_\ell(t_1,t_0,z_0) {\mathcal M}[F_\ell (t_1,t_0,z_0)]\Vec B_{\ell,0}(t_1) f_0\\ \label{eq:pig}
    &={\mathcal R}_\ell(t_1,t_0,z_0) {\mathcal M}[F_\ell (t_1,t_0,z_0)] \pi_{\ell}(t_0,z_0) \vec V_0f_0\\
    \nonumber
    &= {\mathcal R}_\ell(t_1,t_0,z_0) {\mathcal M}[F_\ell (t_1,t_0,z_0)] 
    B_{\ell,0}(t) f_0,
\end{align}
where we have used the commutation property of the matrix ${\mathcal R}_\ell(t_1,t_0,z_0)$ stated in Lemma~\ref{lem:trsp_par}, the value of $\vec B_{\ell,0}(t) = {\pi_\ell(t_0,z_0)}\Vec V_0$, which is independent of $t$, and the fact that the scalar operator ${\mathcal M}[F_\ell (t_1,t_0,z)]$ commutes with the constant matrix $ {\mathcal R}_\ell(t_1,t_0,z_0)$.  
\smallskip 

Let us then focus on 
 $ \vec U_{\ell,1} (t)$. We analyze each of the four terms in~\eqref{alessandra}. The first term is is
 \begin{align*}
(a)&= \pi_\ell(t_1,z_1)
 {\mathcal R}_\ell(t_1,t_0,z_0) {\mathcal M}[F_\ell (t_1,t_0,z_0)]\Vec B_{\ell,1}(t_1) f_0\\*[1ex]
 &= {\mathcal R}_\ell(t_1,t_0,z_0) {\mathcal M}[F_\ell (t_1,t_0,z_0)]\\
 &\left( \sum_{\vert\alpha\vert=3} \frac{1}{\alpha!} \frac{1}{i} \int_{t_0}^{t_1} \partial_z^\alpha h_\ell(s,z_\ell(s))\,
\op_1^w[(F_\ell (s,t_0,z_0)z)^{\alpha}] \,ds\,\pi_\ell(t_0,z_0)\vec V_0 f_0\right.\\
\nonumber
&\;+ \frac 1i\int_{t_0}^{t_1}
\mathcal R(t_0,s,z_0)\, 
\nabla_z H^{\rm adia}_{\ell,1}(s,z_s)\cdot 
{\rm op}^w_1(F_\ell (s,t_0,z_0)z)\,
\mathcal R(s,t_0,z_0)\, ds\,  \pi_\ell(t_0,z_0 )\vec V_0 f_0 \\
&\;+\left.\pi_\ell(t_0,z_0)\widehat z\cdot\nabla \pi_\ell(t_0,z_0) \,\vec V_0 f_0\right),
 \end{align*}
where we have used again Lemma~\ref{lem:trsp_par} and the precise form of $B_{\ell,1}(t_1)$ as given in \eqref{B1}. We apply ${\mathcal R}_\ell(t,t_1,z_1) {\mathcal M}[F_\ell (t,t_1,z_1)]$ to $(a)$. 
Using the group properties
\[
{\mathcal R}_\ell(t,t_1,z_1) {\mathcal R}_\ell(t_1,t_0,z_0) = {\mathcal R}_\ell(t,t_0,z_0)\;\;\mbox{and}\;\;
{\mathcal M}[F_\ell (t,t_1,z_1)]{\mathcal M}[F_\ell (t_1,t_0,z_0)] =  {\mathcal M}[F_\ell (t,t_0,z_0)],
\]
we obtain
\begin{align}\label{alex}
&{\mathcal R}_\ell(t,t_1,z_1) {\mathcal M}[F_\ell (t,t_1,z_1)](a)
=
 {\mathcal R}_\ell(t,t_0,z_0) {\mathcal M}[F_\ell (t,t_0,z_0)]\\ \nonumber
&
 \Bigr( \sum_{\vert\alpha\vert=3} \frac{1}{\alpha!} \frac{1}{i} \int_{t_0}^{t_1} \partial_z^\alpha h_\ell(s,z_\ell(s))\,
\op_1^w[(F_\ell (s,t_0,z_0)z)^{\alpha}] \,ds\,\pi_\ell(t_0,z_0)\vec V_0 f_0\\
\nonumber
& + \frac 1i\int_{t_0}^{t_1}
\mathcal R(t_0,s,z_0)\,
\nabla_z H^{\rm adia}_{\ell,1}(s,z_s)\cdot 
{\rm op}^w_1(F_\ell (s,t_0,z_0)z)\,
\mathcal R(s,t_0,z_0)\, ds\, \pi_\ell(t_0,z_0 )\vec V_0 f_0 \\\nonumber
&+ \pi_\ell(t_0,z_0)\widehat z\cdot \nabla \pi_\ell(t_0,z_0)\,\vec V_0 f_0\Bigr).
\end{align}
At this stage, we observe that \eqref{alex} involves two integrals that might combine with $(b)$ and $(c)$ in~\eqref{alessandra}, while the third term might match with $(d)$. We first analyze $(b)$ and write
\begin{align*}
    &{\mathcal R}_\ell(t,t_1,z_1) {\mathcal M}[F_\ell (t,t_1,z_1)](b) \\
    & = {\mathcal R}_\ell(t,t_1,z_1) {\mathcal M}[F_\ell (t,t_1,z_1)] \sum_{\vert\alpha\vert=3} \frac{1}{\alpha!} \frac{1}{i} \int_{t_1}^{t} \partial_z^\alpha h_\ell(s,z_\ell(s))\,
\op_1^w[(F_\ell (s,t_1,z_1)z)^{\alpha}] \,ds\,\pi_\ell(t_1,z_1)\vec g_{\ell, 0}(t_1)\\
&= {\mathcal R}_\ell(t,t_0,z_0) {\mathcal M}[F_\ell (t,t_0,z_0)] 
\sum_{\vert\alpha\vert=3} \frac{1}{\alpha!} \frac{1}{i} \int_{t_1}^{t} \partial_z^\alpha h_\ell(s,z_\ell(s))\,
\op_1^w[( F_\ell (s,t_0,z_0)z)^{\alpha}] \,ds\,
\pi_{\ell}(t_0,z_0) \vec V_0f_0,
\end{align*}
where we first have used \eqref{eq:pig} and then applied the exact Egorov theorem~\ref{prop:metaplectic} and the group property \eqref{eq:Fgroup}. We then observe that the first integral in \eqref{alex} and the above one combine to an integral from $t_0$ to $t$ as claimed. An analogous argument also 
brings the term (c) in the appropriate form, since
\begin{align*}
    &{\mathcal R}_\ell(t,t_1,z_1) {\mathcal M}[F_\ell (t,t_1,z_1)](c)\\
    &= 
    {\mathcal R}_\ell(t,t_1,z_1) {\mathcal M}[F_\ell (t,t_1,z_1)] \ \frac 1i\int_{t_1}^t
\mathcal R(t_1,s,z_1)\,
\nabla_z H^{\rm adia}_{\ell,1}(s,z_s)\cdot 
(F_\ell (s,t_1,z_1)\widehat z)\,\mathcal R(s,t_1,z_1)\, ds\\  
&\qquad{\mathcal R}_\ell(t_1,t_0,z_0) {\mathcal M}[F_\ell (t_1,t_0,z_0)] \pi_{\ell}(t_0,z_0) \vec V_0f_0\\
&= 
    {\mathcal R}_\ell(t,t_0,z_0) {\mathcal M}[F_\ell (t,t_0,z_0)] \ \frac 1i\int_{t_1}^t
\mathcal R(t_0,s,z_0)\,
\nabla_z H^{\rm adia}_{\ell,1}(s,z_s)\cdot 
(F_\ell (s,t_0,z_0)\widehat z)\,
\mathcal R(s,t_0,z_0)\, ds\\
&\qquad \pi_{\ell}(t_0,z_0) \vec V_0f_0.
\end{align*}
We finally consider $(d)$. By the exact Egorov Theorem \eqref{prop:metaplectic}, 
\begin{align*}
 &{\mathcal R}_\ell(t,t_1,z_1) {\mathcal M}[F_\ell (t,t_1,z_1)] (d) \\   
 &={\mathcal R}_\ell(t,t_1,z_1) {\mathcal M}[F_\ell (t,t_1,z_1)]
 \nabla \pi_\ell(t_1,z_1)\cdot \widehat z \,{\mathcal R}_\ell(t_1,t_0,z_0) {\mathcal M}[F_\ell (t_1,t_0,z_0)]\pi_\ell(t_0,z_0)\vec V_0 f_0\\
    &= 
    {\mathcal R}_\ell(t,t_1,z_1) {\mathcal M}[F_\ell (t,t_0,z_0)]
    \nabla \pi_\ell(t_1,z_1)\cdot(F_\ell(t_1,t_0,z_0) \widehat z) \,{\mathcal R}_\ell(t_1,t_0,z_0)\pi_\ell(t_0,z_0)\vec V_0 f_0.
\end{align*}
Next we apply \Cref{lem:trsp_par} to pull the matrix ${\mathcal R}_\ell(t_1,t_0,z_0)$ to the left. We differentiate  
\eqref{eq:trsp_para} and use the result for $t=t_1$ and $z_1 = \Phi^{t_1,t_0}_{h_\ell}(z_0)$. We thus obtain that for any $\omega\in \R^{2d}$,
\begin{align*}
    &\omega\cdot\nabla {\mathcal R}_\ell (t_1,t_0,z_0) \pi_\ell(t_0,z_0)  
    + {\mathcal R}_\ell (t_1,t_0,z_0)\omega\cdot \nabla \pi_\ell(t_0,z_0) \\
    &= 
    (F_\ell(t_1,t_0,z_0)\omega) \cdot \nabla \pi_{\ell}\left(t_1,z_1\right) {\mathcal R}_\ell (t_1,t_0,z_0) 
    + \pi_{\ell}\left(t,z_1\right) \omega\cdot \nabla {\mathcal R}_\ell (t_1,t_0,z_0).
\end{align*}
Using \eqref{eq:deriv_R}, we arrive at 
\begin{align*}
&(F_\ell(t_1,t_0,z_0)\omega)\cdot \nabla \pi_\ell(t_1,z_1)\,{\mathcal R}_\ell(t_1,t_0,z_0) - 
{\mathcal R}_\ell(t_1,t_0,z_0)\, \omega\cdot \nabla \pi_\ell(t_0,z_0)  \\*[1ex]
&= 
\omega\cdot \nabla {\mathcal R}_\ell(t_1,t_0,z_0) \pi_\ell(t_0,z_0) -
\pi_\ell(t_1,z_1) \omega\cdot \nabla {\mathcal R}_\ell(t_1,t_0,z_0)\\
&=
\frac1i \int_{t_0}^{t_1} {\mathcal R}_\ell(t_1,s,z_0) (F_\ell(s,t_0,z_0)\omega)\cdot
\nabla H_{\ell,1}^{\rm adia} (s,\Phi^{s,t_0}_{h_\ell} (z_0)) {\mathcal R}_\ell(s,t_0,z_1) \pi_\ell(t_0,z_0) 
\, ds\\
&- 
\frac1i \int_{t_0}^{t_1} \pi_\ell(t_1,z_1) {\mathcal R}_\ell(t_1,s,z_0) (F_\ell(s,t_0,z_0)\omega)\cdot
\nabla H_{\ell,1}^{\rm adia} (s,\Phi^{s,t_0}_{h_\ell} (z_0)) {\mathcal R}_\ell(s,t_0,z_1)\, ds.
\end{align*}
Therefore, after quantization, 
\begin{align*}
& {\mathcal R}_\ell(t,t_1,z_1) {\mathcal M}[F_\ell (t,t_1,z_1)](d) =    
    {\mathcal R}_\ell(t,t_0,z_0) {\mathcal M}[F_\ell (t,t_0,z_0)] \,\widehat z\cdot \nabla \pi_\ell(t_0,z_0) \pi_\ell(t_0,z_0) \vec V_0 f_0\\*[1ex]
    &+ {\mathcal R}_\ell(t,t_1,z_1) {\mathcal M}[F_\ell (t,t_0,z_0)] \left(1-\pi_\ell(t_1,z_1)\right)\\
    &\frac1i \int_{t_0}^{t_1} {\mathcal R}_\ell(t_1,s,z_0) (F_\ell(s,t_0,z_0)\omega)\cdot
\nabla H_{\ell,1}^{\rm adia} (s,\Phi^{s,t_0}_{h_\ell} (z_0)) {\mathcal R}_\ell(s,t_0,z_1) \pi_\ell(t_0,z_0)\vec V_0 f_0
\, ds
\end{align*}
Using that 
\[
\widehat z \cdot \nabla \pi_\ell(t_0,z_0) \pi_\ell(t_0,z_0)+ \pi_\ell(t_0,z_0)\widehat z\cdot \nabla \pi_\ell(t_0,z_0)=\widehat z \cdot\nabla \pi_\ell(t_0,z_0),
\]
we see that the last term in \eqref{alex} combines with the first one above to produce
\[
 {\mathcal R}_\ell(t,t_0,z_0) {\mathcal M}[F_\ell (t,t_0,z_0)] \,\widehat z\cdot \nabla \pi_\ell(t_0,z_0) \vec V_0 f_0,
\]
which matches the last term in the final formula \eqref{B1} defining $\Vec B_{\ell,1}(t)$. 

In view of  $S_\ell (t,t_1,z_1)=S_\ell (t,t_0,z_0)$,  we finally obtain $\widetilde \psi^{\eps,N}_\ell(t)=\psi^{\eps,N}_\ell(t)$, for all $t\in[t_0,t^\flat -\delta)$, which concludes the proof.
   \end{proof}


   \chapter{{Propagation of wave packets through the crossing set}}\label{sec:through}

   In this chapter, we aim at concluding the proof of Theorem~\ref{th:WPmain}. For this, it remains to analyze the propagation of a wave packet through the smooth crossing set~$\Upsilon$. We  consider an incoming wave packet at time $t^\flat -\delta$ along a trajectory intersecting $\Upsilon$ at time~$t^\flat$, and the aim is to describe the outgoing wave packet at time $t^\flat +\delta$.
   \smallskip 
   
    We use the rough reduction of Section~\ref{sec:rough_diag} to treat the zone around the crossing.  We fix the point $(t^\flat,\zeta^\flat)\in\Upsilon$ and consider trajectories $z_1(t)$ and $z_2(t)$ reaching simultaneously at time $t^\flat$ in the point~$\zeta^\flat$. 
    We fix $N\in\N$ large enough and we set, with the notations of Theorem~\ref{thm:rough_reduc},
   $$\underline \psi_\ell^{\eps,N}(t)=\widehat{\pi_\ell^{\eps, N}(t)} \psi^\eps(t),\;\;\ell\in\{1,2\}.$$
   By Theorem~\ref{thm:rough_reduc},  if $k\in\N$, the solution 
 $\psi^\eps(t)$ of the Schr\"odinger equation~\eqref{eq:sch} satisfies in $\Sigma^k_\eps$, 
   \begin{equation}\label{newpsi}
\psi^\eps(t) =\underline\psi^{\eps,N}_1(t)+\underline\psi^{\eps,N}_2(t) +\O(\eps^{N+1}).
\end{equation}
Our aim in this section is to determine $\psi^\eps(t^\flat+\delta)$ in terms of $\psi^\eps(t^\flat-\delta)$ by using the description~\eqref{newpsi} of $\psi^\eps(t)$.  
  \smallskip

The family 
$\underline \psi^{\eps,N} = \, ^t( \underline \psi^{\eps,N}_1,\underline \psi^{\eps,N}_2)$
satisfies 
\begin{equation}\label{newsch}
i\eps \partial_t \underline \psi^{\eps,N} = \widehat{\underline H^{\eps,N}}(t)\underline \psi ^{\eps,N}
\end{equation}
with 
$$ \underline H^{\eps,N}(t,z) := \begin{pmatrix}h_1^{\eps,N}(t,z) & 0\\0 & h_2^{\eps,N}(t,z)\end{pmatrix}
+ \begin{pmatrix} 0&   W^{\eps,N}(t,z) \\W^{\eps,N}(t,z)^*&0 \end{pmatrix},
$$
where 
\[
h_\ell^{\eps,N} = h_\ell {\1}_m + \sum_{j=1}^N \eps^j h_{\ell,j}\quad\text{and}\quad 
W^{\eps,N}= \sum_{j=1}^{N+1} \eps^{j} W_j
\]
with matrix-valued $h_{\ell,j}$ and $W_j$ for $j\ge 1$. Note that writing $W^{\eps,N}$, we are making a small abuse of notation comparatively with~\eqref{def:Nseries}.
According to Theorem~\ref{thm:rough_reduc},
the Hamiltonian $\underline H^{\eps,N}$ is subquadratic (see Definition~\ref{def:subquad}). Thus, there exists a constant $C_{k,N}>0$ such that for all $\eps>0$ and $t\in I$,
\begin{equation}\label{eq:growth_adiag}
\| \widehat W^{\eps,N}(t) \|_{\mathcal L(\Sigma^{k+1}_\eps,\Sigma^{k}_\eps)}\leq C_{k,N}.
\end{equation}
Here we have used that the asymptotic series $W^{\eps}$ starts with $\eps W_1$, and we recall that 
\[
W_1 =  \pi_1 H_1 \pi_2 +i\pi_1 \left( \partial_t \pi_1 +\frac 12 \{ h_1+h_2,\pi_1\} \right)\pi_2,
\]
see~\eqref{def:W1}. Thus, it is appropriate to write 
\[
\underline H^{\eps, N}= H_{\rm diag}^{\eps,N}+\eps H_{\rm adiag}^{\eps,N}
\]
with 
\begin{equation}\label{def:HadiagN}
H_{\rm diag}^{\eps,N}(t,z) := \begin{pmatrix}h_{1}^{\eps,N} (t,z) & 0\\0 & h_{2}^{\eps,N}(t,z)\end{pmatrix},\qquad
\eps H_{\rm adiag}^{\eps,N}(t,z) :=\begin{pmatrix} 0& W^{\eps,N}(t,z) \\W^{\eps,N}(t,z)^*&0 \end{pmatrix}
\end{equation}

\smallskip

Let us summarize the information about the data that comes from the preceding section. Let $\delta>0$, for all $s\in(t^\flat-\delta, t^\flat-\frac \delta 2)$,
 \begin{equation}\label{data_cros}
\underline \psi^{\eps,N} (s) 
=\, ^t\left( {\rm WP}^\eps_{z_1(s)} (\varphi^{\eps,N}_1(s)) , {\rm WP}^\eps_{z_2(s)} (\varphi^{\eps,N}_2(s))\right), 
\end{equation}
with for $\ell=1,2$, 
\[
\varphi^{\eps,N}_\ell(s)=\sum_{j=0}^N \eps^{\frac j 2} \varphi_{j,\ell}(s),\;\;\varphi_{j,\ell}(s)\in \mathcal S(\R^d).
\]
Our aim is to prove that the description of $\underline \psi^{\eps,N} (s)$ given in Equation~\eqref{data_cros} extends to $s=t^\flat+\delta$ and to derive precise formula for $ \varphi_{j,\ell}(t^\flat+\delta)$ when $j=\{0,1\}$ and $\ell\in\{1,2\}$. This entails a precise analysis on how the wave packets leave the crossing set, in particular with respect 
to the exchanges between modes $\varphi_{1,1}$ and $\varphi_{1,2}$ at the crossing.

\begin{theorem}\label{thm:through_cros}
Let $k,N,M\in\N$ with $M\leq N$.  
We denote by $\U^\eps_{\underline H}(t,s)$ and $\U^\eps_{\rm diag}(t,s)$ the propagators associated with 
$\underline H^{\eps,N}(t)$ and  $H^{\eps,N}_{\rm diag}(t)$, respectively.
Then, there exists $C>0$, $\ell\in\N$ and a universal constant $\kappa_0$, and an operator $\Theta^{\eps,\delta}_{M}$ such that 
for all $\eps\in(0,1)$ and $\delta>0$ with $\delta\geq \sqrt\eps$, 
\[
\Theta^{\eps,\delta}_{M} =\1 + \sum_{1\leq m\leq M} \, \Theta^{\eps,\delta}_{m,M}
\]
and
\begin{align}\label{eqdef:dyson}
&\left\| \U^\eps_{\underline H}(t^\flat +\delta,t^\flat -\delta ) 
- \U^\eps_{\rm diag}(t^\flat +\delta,t^\flat  )\; \Theta^{\eps,\delta}_{M} \; 
\U^\eps_{\rm diag}(t^\flat,t^\flat -\delta ) \right\|_{\mathcal L (\Sigma_\eps^{k+\ell},\Sigma^k_\eps )}\\
\nonumber 
&\qquad\qquad\qquad \qquad  \leq C\left(\delta^{M+1} + (\eps\delta^{-2})^{\frac{M+1}2}\delta^{-2\kappa_0-k}\right).
\end{align}
Moreover, there exists $\eps_0>0$, 
and a family of operators $(\mathcal T_{m,M}^{\eps,\delta})_{m\geq 1}$ and constants $c_{k,m,M}>0$
such that for all  $\vec \varphi\in \mathcal S(\R^d,\C^2)$, $m\geq 1$, $\eps\in(0,\eps_0)$,
\begin{equation}\label{eq:pluto1}
\Theta^{\eps,\delta}_{m,M} {\rm WP}^\eps_{\zeta^\flat} (\vec \varphi)=
 {\rm WP}^\eps_{\zeta^\flat} \left(
 \mathcal T^{\eps,\delta}_{m,M} \vec\varphi +\eps^{\frac{M+1}2} \vec{r^\eps}_{m,M}\right),
\end{equation}
with $\vec{r^\eps}_{m,M}\in\mathcal S(\R^d,\C^m)$ and 
\begin{equation}\label{norm_estimate}
\| \mathcal T^{\eps,\delta}_{m,M}\vec \varphi\|_{\Sigma^k} \leq c_{k,m,M} \, \eps^{\frac m 2} |\log \eps|^{\max(0,m-1)} \, \| \vec\varphi\|_{\Sigma^{k+2m+1}}.
\end{equation}
 Besides, the operator $\mathcal T^{\eps,\delta}_{1,M}$ maps $\mathcal S(\R^d,\C^m)$ into itself and   
\begin{equation}\label{eq:pluto3}
\mathcal T^{\eps,\delta}_{1,M}= \sqrt\eps \, 
\begin{pmatrix} 0 & W_1(t^\flat,\zeta^\flat) \mathcal T^\flat_{2\rightarrow 1}\\
W_1(t^\flat,\zeta^\flat) ^*\mathcal T^\flat_{1\rightarrow 2} & 0
\end{pmatrix} +\O(\delta\sqrt\eps),\end{equation}
where the transfer operators $\mathcal T^\flat_{1\rightarrow 2}$ and $\mathcal T^\flat_{2\rightarrow 1}$ are defined according to~\eqref{transf1}.
\end{theorem}

A few remarks are in order. First, we
 point out that some additional action effects will appear when applying $\U^\eps_{\underline H}(t^\flat +\delta,t^\flat -\delta )$ to a wave packet via the operators $\U^\eps_{\rm diag}(t^\flat +\delta,t^\flat  )$
 and $\U^\eps_{\rm diag}(t^\flat,t^\flat -\delta ) $.  When applied to a Gaussian wave packets, i.e. when  $\varphi_{\ell, 1}= g^\Gamma_\ell$ in~\eqref{data_cros}, the leading order correction term  at time $t^\flat +\delta$ due to the crossing is 
 \begin{align*}
 \sqrt\eps \begin{pmatrix} 
   {\rm e}^{\frac i\eps S_1(t^\flat+\delta, t^\flat ,\zeta^\flat)+\frac i\eps S_2(t^\flat, t_0,z_0)}  {\rm WP}^\eps_{z_1(t^\flat+\delta )} (\varphi_1)\\
    {\rm e}^{\frac i\eps S_2(t^\flat+\delta, t^\flat ,\zeta^\flat)+\frac i\eps S_1(t^\flat, t_0,z_0)}  {\rm WP}^\eps_{z_2(t^\flat+\delta )} (\varphi_2)
    \end{pmatrix}
  \end{align*}
  with 
  \begin{align*}
\varphi_1 &= 
\mathcal M[F_1(t^\flat+\delta,t^\flat, \zeta^\flat)] W_1(t^\flat, \zeta^\flat) \mathcal T^\flat_{2\rightarrow 1} \mathcal M[F_2(t^\flat, t_0,z_0)]g^\Gamma_2 ,\\
\varphi_2 &= 
\mathcal M[F_2(t^\flat+\delta,t^\flat, \zeta^\flat)] W_1(t^\flat, \zeta^\flat)^* \mathcal T^\flat_{1\rightarrow 2} \mathcal M[F_1(t^\flat, t_0,z_0)]g^\Gamma_1. 
 \end{align*}
Recall that $W_1$ is the off-diagonal matrix described in~\eqref{def:W1}.
 \smallskip

The remainder of this section is devoted to the proof of Theorem~\ref{thm:through_cros}. In Section~\ref{sec:dyson}, we use  Dyson series  to construct the interaction operators $\Theta^{\eps,\delta}_{m,M} $, obtain the decomposition~\eqref{eqdef:dyson}.  We prove that their action on wave packets obey to~\eqref{eq:pluto1}. Then,   in Section~\ref{sec:mtom+1}, we study the operators $\mathcal T^{\eps,\delta}_{m,M}$, proving the estimate~\eqref{norm_estimate} and calculating the quantity $ \mathcal T^{\eps,\delta}_{1,M}$, as in~\eqref{eq:pluto3} (which was already done in~\cite{FLR1}). 


\section{Construction of the interaction operators} 
\label{sec:dyson}

\subsection{Dyson expansions}
We perform a Dyson expansion via the Duhamel formula and uses~\eqref{def:HadiagN}. A first use of Duhamel formula gives for $t_1,t_2\in\R$,
\beq\label{duh}
\U^\eps_{\underline H}(t_2, t_1) = \U^\eps_{H_{\rm diag}}(t_2,t_1) +\frac 1i \int_{t_1}^{t_2} \U^\eps_{\underline H}(t_2,s_1)\widehat{H^{\eps, N}_{\rm adiag}}(s_1) \,\U^\eps_{H_{\rm diag}}(s_1, t_1)ds_1.
\eeq
With  one iteration of the Duhamel formula, we obtain
\begin{align*}
\U^\eps_{\underline H}(t_2, t_1)&  = \U^\eps_{H_{\rm diag}}(t_2,t_1) +\frac 1i \int_{t_1}^{t_2} \,\U^\eps_{H_{\rm diag}}(t_2,s_1)\widehat{H^{\eps,N}_{\rm adiag}}(s_1)\U^\eps_{H_{\rm diag}}(s_1, t_1)ds_1 \\
&\qquad 
 -\int_{t_1}^{t_2}\int_{s_1}^{t_2} \U^\eps_{\underline H}(t_2,s_2)\widehat{H^{\eps,N}_{\rm adiag}}(s_2)\,\U^\eps_{H_{\rm diag}}(s_2,s_1)\widehat{H^{\eps,N}_{\rm adiag}}(s_1)\,\U^\eps_{H_{\rm diag}}(s_1, t_1)ds_2ds_1.
\end{align*}
With two iterations, we have 
\begin{align*}
\U^\eps_{\underline H}(t_2,t_1)  =\;&
 \U^\eps_{H_{\rm diag}}(t_2,t_1) 
 +\frac 1i \int_{t_1}^{t_2} \,\U^\eps_{H_{\rm diag}}(t_2,s_1)\widehat{H^{\eps,N}_{\rm adiag}}(s_1)\U^\eps_{H_{\rm diag}}(s_1, t_1)ds_1 \\
&\; 
 -\int_{t_1}^{t_2}\int_{s_1}^{t_2} \U^\eps_{\rm diag }(t_2,s_2)\widehat{H^{\eps,N}_{\rm adiag}}(s_2)\,\U^\eps_{H_{\rm diag}}(s_2,s_1)\widehat{H^{\eps,N}_{\rm adiag}}(s_1)\,\U^\eps_{H_{\rm diag}}(s_1, t_1)ds_1ds_2\\
&\; -\frac 1 i 
 \int_{t_1}^{t_2}\int_{s_1}^{t_2} \int_{s_2}^{t_2}
 \,\U^\eps_{\underline H}(t_2,s_3)\widehat{H^{\eps,N}_{\rm adiag}}(s_3)
 \U^\eps_{\rm diag }(s_3,s_2)\widehat{H^{\eps,N}_{\rm adiag}}(s_2)\\
 &\qquad\qquad\qquad \times \U^\eps_{H_{\rm diag}}(s_2,s_1)\widehat{H^{\eps,N}_{\rm adiag}}(s_1)\,\U^\eps_{H_{\rm diag}}(s_1, t_1)ds_3ds_2ds_1
\end{align*}
and the last term of the right hand side contains three integrations in time, which implies that any of its (adequate) operator norms is of size $|t_1- t_2|^3$.
\smallskip

Our aim is to transform $\U^\eps_{\underline H}$ for $t_1=t^\flat-\delta $, $t_2=t^\flat +\delta$ by performing $M$ iterations. For simplifying the notation, we set 
\begin{equation}\label{def:E(s,t)}
 E(s) =  \U^\eps_{H_{\rm diag}}(t^\flat ,s)\widehat{H^{\eps}_{\rm adiag}}(s)\,\U^\eps_{H_{\rm diag}}(s, t^\flat).
 \end{equation}
In doing that, we choose the time $t^\flat$ as a reference time (an initial time in terms of Cauchy data) for the evolution operators. This choice will be present in all the arguments of the proof. 
In that manner, after $M$ iterations, $M\in\N$, we have the Dyson formula
 \begin{align*}
 &\U^\eps_{\underline H}(t^\flat+\delta, t^\flat-\delta) = \U^\eps_{H_{\rm diag}}(t^\flat+\delta , t^\flat)
  \Bigl(\1 \, +\\ \nonumber & 
  \sum_{1\leq m\leq M}(i)^{-m}
  \int _{t^\flat-\delta}^{t^\flat+\delta}ds_1\int_{s_1}^{t^\flat+\delta}ds_2\cdots\int_{s_{m-1}}^{t^\flat+\delta}ds_m
 E(s_m) E(s_{m-1})\cdots E(s_1)\Bigr)
 \U^\eps_{H_{\rm diag}}(t^\flat, t^\flat-\delta )\\
 &\qquad +
 R_M^{\eps,\delta}(t^\flat) 
 \end{align*}
and $R^{\eps,\delta}_M(t^\flat)$
satisfies
\[
\| R^{\eps,\delta}_M(t^\flat)\|_{\mathcal L (\Sigma_\eps^{k+M+1},\Sigma^k)}\leq C\delta^{M+1}.
\]
This gives 
\begin{equation}\label{eqdef:dyson_bis}
\U^\eps_{\underline H}(t^\flat +\delta,t^\flat -\delta )= \U^\eps_{H_{\rm diag}}(t^\flat +\delta,t^\flat  )\;\widetilde \Theta^{\eps,\delta}_{M} \; \U^\eps_{H_{\rm diag}}(t^\flat,t^\flat -\delta ) +R^{\eps,\delta}_M(t^\flat)
\end{equation}
with
\[
\widetilde\Theta^{\eps,\delta}_{M} =\1+\sum _{m=1}^M\widetilde\Theta_m (t^\flat+\delta)
\]
where the family $\widetilde\Theta_m(s)$ is given for $s\in [t^\flat-\delta, t^\flat+\delta]$, $m\geq 1$ by 
\begin{equation}\label{eq_de_Theta}
\widetilde\Theta_m(s) =(i)^{-m}
  \int _{t^\flat-\delta}^{s}ds_1\int_{s_1}^{t^\flat+\delta}ds_2\cdots\int_{s_{m-1}}^{t^\flat+\delta}ds_m
 E(s_m) E(s_{m-1})\cdots E(s_1)
\end{equation}
%
The family of operators $\widetilde\Theta_{m}(s) $ contains all the information about the interactions  between the modes~$h_1$ and~$h_2$ modulo $\O(\delta^\infty)$ when $M$ goes to $+\infty$.  We prove in the next section that by adequate linear combination of the 
$\widetilde\Theta_{m}(t^\flat + \delta) $, one can find operators  $\Theta_{m,M}^{\eps,\delta} $ satisfying~\eqref{eqdef:dyson} with~\eqref{eq:pluto1}, ~\eqref{norm_estimate} and~\eqref{eq:pluto3}. 

\subsection{Action of the interaction operator  on wave packets}

We focus now in understanding the action of $\widetilde\Theta^{\eps}_{m}(s) $ on wave packets of the form 
\[
{\rm WP}^\eps_{\zeta^\flat} (\vec \varphi) = \begin{pmatrix}
{\rm WP}^\eps_{\zeta^\flat}(\varphi_1)\\
{\rm WP}^\eps_{\zeta^\flat}(\varphi_2)
\end{pmatrix}.
\]
This will allow to define the operators $\Theta^{\eps,\delta}_{m,M}(s) $ with the expected properties. 
\smallskip 


 Let us start by analyzing the function $E(s)$  involved in the definition of $\Theta_{m} (s)$ and introduced in~\eqref{def:E(s,t)}. 
According to~\eqref{def:HadiagN}, we have
 $$
E(s)  =
  \begin{pmatrix} 0 &{ I}(s)\\{I}^*(s)
 & 0\end{pmatrix},\;\;s\in[t^\flat-\delta, t^\flat +\delta]
$$
with
\begin{equation}\label{def:pluto5}
{I}(s):=
\frac 1\eps \, \U^\eps_{h_1^{\eps,N}}(t^\flat, s)  W^{\eps,N}(s)\, \U^\eps_{h_2^{\eps,N}}(s,t^\flat).
\end{equation}
Note that $I(s)$, as $E(s)$,  depends on $\eps$, $\delta$ and $N$. 
\smallskip 

The operator $I(s)$ combines conjugation of the pseudo-differential operator $\chi_\delta (s) W^{\eps,N}(s)$  by the propagator  $\U^\eps_{h_2^{\eps,N}}(s,t^\flat)$ and composition by two different propagators $\U^\eps_{h_1^{\eps,N} }(t^\flat, s)$ and $\U^\eps_{h_2^{\eps,N} }(t^\flat, s)$.
Indeed, we can write, 
\[
{ I}(s)= \frac 1\eps\,\left(\U^\eps_{h_1^{\eps,N} }(t^\flat, s) W^{\eps,N}(s)\, \U^\eps_{h_1^{\eps,N}}(s,t^\flat)\right)\circ \left(\U^\eps_{h_1^{\eps,N} }(t^\flat, s) \U^\eps_{h_2^{\eps,N} }(s,t^\flat)\right).
\]
Similarly, for ${I}^*(s)$, one writes 
\begin{equation}\label{eq:I*(s)}
{I}^*(s)=\frac 1\eps\,\left(\U^\eps_{h_2^{\eps,N} }(t^\flat, s) W^{\eps,N}(s)^* \U^\eps_{h_2^{\eps,N}}(s,t^\flat)\right)\circ \left(\U^\eps_{h_2^{\eps,N} }(t^\flat, s) \U^\eps_{h_1^{\eps,N} }(s,t^\flat)\right).
\end{equation}
Both operations are well understood:
\begin{enumerate}
    \item First, 
the conjugation of a pseudo-differential operator by a propagator, that is generated by a subquadratic, matrix-valued Hamiltonian with scalar principal symbol, is  described by the Egorov Theorem of Appendix~\ref{app:C} (see Proposition~\ref{EgD}). Since $\widehat{W^{\eps,N}(s)}$ has a vanishing principal symbol, 
we have an asymptotic expansion of the form
\[
 \U^\eps_{h_\ell^{\eps,N} }(t^\flat, s) \widehat{W^{\eps,N}(s)}\, \U^\eps_{h_\ell^{\eps,N}}(s,t^\flat)=\eps \, \sum_{j\geq 0} \eps^{j}\widehat{\underline W_{j,\ell}(s)},\qquad \ell\in\{1,2\}.
\]
\item Second,
the action on wave packets of two different propagators acting one backwards and the other one forwards  has been studied in~\cite{FLR1}[Section~5.2]. The  
 precise computation of the asymptotic behavior of the action of the operator $\U^\eps_{h_1^{\eps,N} }(t^\flat, s) \,\U^\eps_{h_2^{\eps,N} }(s,t^\flat)$ on wave packets   is detailed in Appendix~\ref{forwards_backwards}. This involves the canonical transformation of the phase space $z\mapsto \Phi_1^{t^\flat,s}\circ\Phi_2^{s,t^\flat}(z)$) (see Lemma~\ref{lem:composed}). 
\end{enumerate}
As a consequence,
one obtains the 
analogue of Lemma~5.3 of~\cite{FLR1}, which writes in our context as follows.

\begin{lemma}\label{lem_FLR}
Let $k\in\N$.
Using notation~\eqref{def:mu},
\begin{itemize}
\setlength{\itemsep}{1ex}
\item[-]
A smooth real-valued map $s\mapsto \Lambda(s)$ with
$\Lambda(0) = 0$, $\dot\Lambda(0) = 0$, and satisfying the relation \[
\frac 12(\ddot\Lambda(0)-\dot p(0)\cdot \dot q(0)) = -\mu^\flat,
\]
\item[-] 
A smooth vector-valued map $s\mapsto z(s) = (q(s),p(s))$ with 
$z(0) = 0$, $\dot z(0)= (\alpha^\flat,\beta^\flat)$,
\item[-]
A smooth map $s\mapsto {Q}^\eps(s)$ of pseudo-differential operators, 
that maps Schwartz functions to Schwartz functions, 
\end{itemize}
such that for all $s\in[t^\flat-\delta,t^\flat+\delta]$ and all $\vec\varphi\in\mathcal S(\R^d,\C^m )$,
\[
 { I}^*(s){\rm WP}^\eps_{\zeta^\flat}(\vec \varphi)(y)=
{\rm WP}^\eps_{\zeta^\flat}\left(
Q^\eps(s)
\vec \varphi  
\right) +\O\left((\eps\delta^{-2})^{\frac{M+1}2}\delta^{-2\kappa_0-k}
\right).
\]
Moreover, $Q^\eps(s)=\sum_{j=0}^M\eps^{\frac j2} Q^\eps_j(s)$, 
and,  with   the scaling notation $z_\eps(s) = z(s)/\sqrt\eps:=(q_\eps(s),p_\eps(s))$, we have 
\begin{align}\label{def:Qj}
& Q^\eps_j(s)\vec\varphi(y)={\rm e}^{\frac{i}{\eps} \Lambda(s-t^\flat)} 
{ Q}_j(s-t^\flat) 
{\rm e}^{i p_\eps(s-t^\flat)\cdot (y-q_\eps(s-t^\flat))} \vec\varphi(y-q_\eps(s-t^\flat)),\\
\label{def:Q0}
&  Q_0(0)= W_1(t^\flat,\zeta^\flat)^*,
\end{align}
where the $\eps$-independent operators $Q_j(s)$ map $\mathcal S(\R^d,\C^m)$ into itself, uniformly in $s\in [t^\flat-\delta, t^\flat+\delta]$.
\end{lemma}  
 
\begin{proof}
We start from~\eqref{eq:I*(s)}. We first use the Egorov theorem of Proposition~\ref{EgD}  (see also point (1) above). In $\mathcal L(\Sigma ^ k_\eps)$, we have 
\[
\frac 1\eps\,
\U^\eps_{h_2^{\eps,N} }(t^\flat, s)\widehat  {W^{\eps,N}(s)}^*\,
\U^\eps_{h_2^{\eps,N}}(s,t^\flat)=\sum_{j= 0}^ M \eps^j\widehat{\underline W_{j,2}(s)}^*+\mathcal O((\eps\delta^{-2})^ {M+1}\delta^{-2\kappa_0-k}),
\]
with $\underline W_{1,2}(0)=W_1(t^\flat,\zeta^\flat)$. 
Moreover, by Lemma~\ref{lem:composed} and with its notation, we obtain 
\begin{align*}
\frac 1\eps\left(\U^\eps_{h_2^{\eps,N} }(t^\flat, s) \U^\eps_{h_1^{\eps,N} }(s,t^\flat)\right){\rm WP}^\eps_{\zeta^\flat}(\vec \varphi)&=\sum_{j=0}^M \eps^{\frac j 2}\wp^\eps_{\zeta^\flat} \left(
{\rm e}^{\frac{i}{\eps} \widetilde\Lambda( s)} 
{\rm e}^{i \widetilde p_\eps( s)\cdot (y-\widetilde q_\eps( s))}\vec \varphi_j(s,y-\widetilde q_\eps( s)) \right)\\
&\qquad +
\mathcal O((\eps\delta^{-2})^ {\frac{M+1}2}\delta^{-k-2\kappa_0}).
\end{align*}
Injecting these relations in~\eqref{eq:I*(s)}, we obtain 
\begin{align*}
    I^ *(s)\wp^ \eps_{\zeta^\flat}(\vec \varphi) & = 
    \sum_{j_1=0}^ M    \sum_{j_2=0}^ M \eps^{j_1}\eps^{\frac {j_2}2}\widehat{\underline W_{j_1,2}(s)}^*\wp^\eps_{\zeta^\flat} \left(
{\rm e}^{\frac{i}{\eps} \widetilde\Lambda( s)} 
{\rm e}^{i \widetilde p_\eps( s)\cdot (y-\widetilde q_\eps( s))}\vec  \varphi_{j_2}(s,y-\widetilde q_\eps( s)) \right)\\
&\qquad +\mathcal O((\eps\delta^{-2})^ {M+1}\delta^{-k-2\kappa_0})
\end{align*}
in $\Sigma ^ k_\eps$. 
We observe
\begin{align*}
&\widehat{\underline W_{j_1,2}(s)}\wp^\eps_{\zeta^\flat} \left(
{\rm e}^{\frac{i}{\eps} \widetilde\Lambda( s)} 
{\rm e}^{i \widetilde p_\eps( s)\cdot (y-\widetilde q_\eps( s))}\vec  \varphi_{j_2}(s,y-\widetilde q_\eps( s)) \right)\\
&\qquad = \wp^\eps_{\zeta^\flat} \left({\rm op}_1\left(\underline W_{j_1,2}(s,\zeta^\flat +\sqrt\eps z)^*\right)
{\rm e}^{\frac{i}{\eps} \widetilde\Lambda( s)} 
{\rm e}^{i \widetilde p_\eps( s)\cdot (y-\widetilde q_\eps( s))} \vec \varphi_{j_2}(s,y-\widetilde q_\eps( s)) \right)\\
&\qquad = \wp^\eps_{\zeta^\flat} \left(
{\rm e}^{\frac{i}{\eps} \widetilde\Lambda( s)} 
{\rm e}^{i \widetilde p_\eps( s)\cdot (y-\widetilde q_\eps( s))} \underline{\vec \varphi}_{j_1,j_2}(s,y-\widetilde q_\eps( s)) \right)
\end{align*}
with $\varphi_{j_1,j_2}(s)={\rm op}_1\left(\underline W_{j_1,2}(s,\tilde z_\eps(s)+\zeta^\flat +\sqrt\eps z)^*\right)\vec \varphi_{j_2}$. 
Reorganizing the terms and 
 Taylor expanding in $z$ the symbol of the preceding pseudodifferential operator, we obtain the asymptotics with the terms $Q^\eps_j(s)$, as in~\eqref{def:Qj}. 
We conclude the proof with the relation
\[
\varphi_{1,0}(0,y)={\rm op}_1\left(\underline W_{1,2}(0,\zeta^\flat +\sqrt\eps z)^*\right)\vec \varphi_{0}(0)= W_1(t^\flat, \zeta^\flat)^* \vec\varphi,
\]
whence~\eqref{def:Q0}
\end{proof}

We now go back to the action of the operators $\widetilde\Theta_{m}(s)$ (see~\eqref{eq_de_Theta}) on wave packets. By Lemma~\ref{lem_FLR},
 for all $\vec\varphi\in \mathcal S(\R^d,\C^m)$, we have in the Schwartz class,
\begin{align*}
\widetilde \Theta_1(s) {\rm WP}^\eps_{\zeta^\flat}(\vec \varphi)&=
- i \sum_{0\leq j\leq M} \eps^{j-1} {\rm WP}^\eps_{\zeta^\flat}\left(\int_{t^\flat-\delta}^s \begin{pmatrix} 0 & Q_j^\eps(s_1)^*\\
Q_j^\eps(s_1) & 0 
\end{pmatrix} \vec \varphi ds_1 \right)\\
&\qquad + \mathcal O((\eps\delta^{-2})^{\frac{M+1}2} \delta^{-2\kappa_0-k}
) .
\end{align*}
More generally, after reorganization of the terms, we can write for all $M\in\N^*$ and $\varphi\in \mathcal S(\R^d,\C^m)$,
\begin{equation}\label{resommation}
\sum_{1\leq m\leq M} \widetilde \Theta_m(s) =
\sum_{1\leq m\leq M}  \Theta_{m,M}^{\eps,\delta}(s) , \qquad s\in[t^\flat-\delta, t^\flat+\delta]
\end{equation}
with 
\begin{align}
\label{def:gros_theta}
& \Theta_{m',M}^{\eps,\delta}(s){\rm WP}^\eps_{\zeta^\flat}(\vec \varphi)
= \sum_{m=1}^M\; \sum_{j_1+\cdots+j_m +m=m'} 
\eps^{\frac{j_1+\cdots+j_m}{2}}\wp^\eps_{\zeta^\flat}\, ({\mathcal T}^{\eps,\delta}_{j_1,\cdots,j_m} \vec \varphi),\\
\label{def:tilde_transfer}
&{\mathcal T}^{\eps,\delta}_{j_1,\cdots,j_m}(s) \vec \varphi= (-i)^m 
\int_{t^\flat-\delta}^s ds_1 \int_{s_1}^{t^\flat+\delta} ds_2 \cdots \int_{s_{m-1}}^{t^\flat+\delta} ds_m T_{j_m}(s_m)\cdots T_{j_2}(s_2) T_{j_1}(s_1),\\
\nonumber
\mbox{and}\;\;&T_j(s)= \begin{pmatrix} 0 & Q_j^\eps(s)^*\\
Q_j^\eps(s) & 0 
\end{pmatrix}.
\end{align}
Note that when $m'=1$, there is only one term in the sum of~\eqref{def:gros_theta}, corresponding to $m=1$ and $j_1=0$.
The reorganization of the terms is made in view of the claim 
\begin{equation}\label{claim:T}
\|{\mathcal T}^{\eps,\delta}_{j_1,\cdots,j_m}(s) \vec \varphi\|_{\Sigma^k} \leq C_{j_1,\cdots,j_m,\varphi} \,\eps ^{\frac m2} \,|\log\eps|^{\max(0,m-1)}
, \;\;\forall \vec \phi\in\mathcal S(\R^{d},\C^m).
\end{equation}
Then, we define the transfer operators of~\eqref{eq:pluto1} by setting 
\begin{equation}\label{eq:pluto1bis}
\mathcal T^{\eps,\delta}_{m',M}=
\sum_{m=1}^M\; \sum_{j_1+\cdots+j_m +m=m'} 
 {\mathcal T}^{\eps,\delta}_{j_1,\cdots,j_m},\;\; 1\leq m'\leq M,
\end{equation}
and~\eqref{claim:T} implies
that~\eqref{norm_estimate} and~\eqref{eq:pluto3} hold.
Proving~\eqref{claim:T} (and thus~\eqref{norm_estimate} and~\eqref{eq:pluto3}) is the subject of the next section. 

\section{The transfer operators}
\label{sec:mtom+1}

In this section, we analyze the operators ${\mathcal T}^{\eps,\delta}_{j_1,\cdots , j_m} (s)$ introduced in~\eqref{def:tilde_transfer}. The key argument is the next Lemma.

\begin{lemma}\label{lem:asymptotics}
Let us fix $M\in\N$. \begin{enumerate}
    \item For all $1\leq j\leq M$ and  $k\in\N$,  there exists $c_{k,j}>0$ and $\ell\in\N$  such that for all time-dependent Schwartz function $\vec\varphi\in\mathcal C^\infty(\R,\mathcal S(\R^d,\C^m))$, $s\in [t^\flat-\delta,t^\flat+\delta]$, $\eps\in(0,\eps_0]$ and $k\in\N$,
    \begin{align*}
&\Bigl\|  \int_{t^\flat-\delta}^s  T_{j}(s') \vec\varphi(s')ds'\Bigr\|_{\Sigma^k}\\
&\qquad\leq     
       c_{k,j} \sqrt \eps \Bigl(
       \sup_{s'\in[t^\flat-\delta,t^\flat+\delta]}\| \vec\varphi(s')\|_{\Sigma^{k+\ell+1}} +|\log|\eps|| 
\sup_{s'\in[t^\flat -\delta,t^\flat+\delta]}   \| \partial_s \vec\varphi(s')  \|_{\Sigma^{k+\ell}} \Bigr) . 
    \end{align*}
    \item 
    There exists an operator 
$\tilde {T}_1^{\eps,\delta}
$ such that for all $k\in\N$, there are  constants $C_k$, $ c_{1,k}$ and $\ell\in\N$ such that for  all  $\vec\varphi\in \mathcal S(\R^d,\C^m)$, 
\begin{align*}
&\left\|\int_{t^\flat-\delta}^s T_1(s')\vec \varphi ds'- \sqrt\eps 
\begin{pmatrix} 0 & W_1(t^\flat,\zeta^\flat)  \mathcal T^\flat_{2\rightarrow 1} \\ 
W_1(t^\flat,\zeta^\flat) ^* \mathcal T^\flat_{1\rightarrow 2}
& 0 
\end{pmatrix} \vec \varphi -
 \sqrt\eps\, \tilde T^{\eps, \delta}_1\vec \varphi  \right\|_{\Sigma^{k}} \\
 &\qquad \leq C_{k} \sqrt\eps \left(\frac{\sqrt\eps}\delta\right)^{M+1}\|\vec \varphi\|_{\Sigma^{k+\ell}}.
\end{align*}
where  $\mathcal T^\flat_{1\rightarrow 2}$ is the transfer operator defined in~\eqref{transf1}, and with
 \[
  \|\tilde T^{\eps, \delta} _1 \vec \varphi \|_{\Sigma^k}\leq c_{1,k} \, \delta\,\|\vec \varphi\|_{\Sigma^{k+\ell}}.
 \]
\end{enumerate}
\end{lemma}

\begin{proof}[Proof of 
the claim~\eqref{claim:T}, and 
equations~\eqref{norm_estimate} and~\eqref{eq:pluto3}]
Note that in view of 
\[
 {\mathcal T} ^{\eps,\delta}_{j_1}(s)\,\vec \varphi
=-i  \int_{t^\flat-\delta}^s  T_{j}(s') \vec \varphi(s')ds',
\]
Lemma~\ref{lem:asymptotics}(2) applies directly to the first transfer operator. 
Besides, for $\vec\varphi$ independent on $s$, we have 
\[
\| {\mathcal T} ^{\eps,\delta}_{j_1}(s)\,\vec \varphi\|_{\Sigma^k}\leq c_{k,0} \,\sqrt\eps \| \vec \varphi\|_{\Sigma^{k+\ell(M,k)}}.
\]
We deduce that~\eqref{claim:T} holds for $m=1$.
We can now conclude the proof of the claim for all $m\geq 1$ by recursion. We observe 
\begin{align*}
{\mathcal T}^{\eps,\delta}_{j_1,\cdots,j_m}(s) \vec \varphi&= -i     \int_{t^\flat-\delta}^s  T_{j}(s_1) \vec\varphi(s_1)ds_1
\end{align*}
with 
\begin{align*}
&\vec\varphi(s)= {\mathcal T}^{\eps,\delta}_{j_2,\cdots,j_m} (t^\flat+\delta)\vec \varphi-
{\mathcal T}^{\eps,\delta}_{j_2,\cdots,j_m}(s) \vec \varphi,\\
&
\partial_{s}\vec\varphi(s)= -i {\mathcal T}^{\eps,\delta}_{j_3,\cdots,j_m} (t^\flat+\delta)\vec \varphi-
{\mathcal T}^{\eps,\delta}_{j_3,\cdots,j_m}(s) \vec \varphi,
\end{align*}
and Lemma~\ref{lem:asymptotics}(1) allows to conclude the recursion argument and to obtain~\eqref{claim:T}. 
\end{proof}

The proof of Lemma~\ref{lem:asymptotics} relies on the analysis of  the (scalar) integrals   defined  for $j\in\N$ and $\vec\varphi\in\mathcal C^\infty(\R,S(\R^d))$ scalar-valued by 
\begin{equation}\label{def:Ibazar}
I_{\varphi(.)}^{\eps,\delta} (s):= \int_{t^\flat-\delta }^{s} {\rm e}^{\frac{i}{\eps} \Lambda(s'-t^\flat)} 
{Q}_j(s'-t^\flat ) 
{\rm e}^{i p_\eps(s'-t^\flat)\cdot (y-q_\eps(s'-t^\flat))}  \vec\varphi(s',y-q_\eps(s'-t^\flat))  ds'.
\end{equation}
We will us this preparatory Lemma that simplifies the general form of the terms $I^{\eps,\delta}_{\varphi(\cdot)}(s)$.

\begin{lemma}\label{lem:prepa}
Set
$L= \beta^\flat\cdot y-\alpha ^\flat \cdot D_y$.
Then, in Schwartz space,
\[
I_{\varphi(.)}^{\eps,\delta}(s)= \sqrt\eps \int_{-\delta/\sqrt\eps}^{(s-t^\flat)/\sqrt\eps} {\rm e}^{i\phi^\eps(\sigma)} 
{ Q}_j(\sigma\sqrt\eps  ) 
{\rm e}^{i L\sigma } \vec\varphi(t^\flat +\sqrt\eps\sigma)
d\sigma  +\O(\sqrt\eps \delta)
\]
with
\begin{align*}
\phi^\eps(\sigma)& = \frac 1\eps \Lambda(\sqrt\eps\sigma)-\frac 1 2 p_\eps(\sigma)q_\eps(\sigma).
\end{align*}
\end{lemma}

In the following, we will use that by  Lemma~\ref{lem_FLR},
\begin{align*}
\phi^\eps(\sigma)
& =\frac 12 (\ddot \Lambda(0)-\dot p(0)\dot q(0))\sigma^2 + \sqrt\eps \sigma^3 \phi^\eps_1(\sqrt\eps \sigma)
\\
&=-\mu^\flat \sigma^2 + \sqrt\eps \sigma^3 \phi^\eps_1(\sqrt\eps \sigma)
\end{align*}
where $\sigma\mapsto \phi^\eps_1(\sigma)$  is a smooth family of operators bounded in the Schwartz class and with bounded derivatives on any compact of $\R$.
In particular, we have on one hand,
  $$
 \left\vert\frac{d}{d\sigma}\phi^\eps(\sigma) \right\vert \geq c_0\vert \sigma\vert,
  $$
  for some constant $c_0>0$; and on the other hand,  we have  
    $$
   {\rm e}^{i\phi^\eps(\sigma )}=  \frac{1}{\sigma} b^\eps(\sigma)
    \partial_\sigma{\rm e}^{i\phi^\eps(\sigma)} \;\;\mbox{with}\;\; b^\eps(\sigma) := \frac{\sigma}{i\partial_\sigma \phi^\eps(\sigma)}\;\;\mbox{and}\
    \left\vert\frac{d}{d\sigma}b^\eps(\sigma) \right\vert \leq \frac{c_1}{\vert \sigma \vert}
    $$
    for some constant $c_1>0$ as above. 
 Finally, we will also take advantage of the fact that the operator ${\rm e}^{i\sigma L}$ maps $\Sigma^k$ into itself continuously for all $s\in\R$ and $k\in\N$.
\smallskip 

Let us now prove Lemma~\ref{lem:prepa}.

\begin{proof}
    We first transform the integral by the change of variables $s'=t^\flat+\sigma\sqrt\eps $: 
\begin{align*}
I_{\varphi(.)}^{\eps,\delta}(s)=& \sqrt\eps \int_{-\delta/\sqrt\eps}^{(s-t^\flat)/\sqrt\eps} {\rm e}^{\frac{i}{\eps} \Lambda(\sigma \sqrt\eps )} 
Q_j(\sigma\sqrt\eps  ) 
{\rm e}^{i p_\eps(\sigma\sqrt\eps )\cdot (y-q_\eps(\sigma\sqrt\eps ))}\vec\varphi(t^\flat+\sigma\sqrt\eps,y-q_\eps(\sigma\sqrt\eps)  ) d\sigma. 
\end{align*}
One sees that the change of variable has exhibited a power of $\sqrt\eps$. 
 However,  even though the integrand is bounded, the size of the support of the integral is large: it is of size $2 \frac \delta{\sqrt\eps}$, which spoils that gain of $\sqrt\eps$. This integral will turn out to be  smaller  because of the oscillations of the phase outside $s=0$. The proof then consists in  integrations by parts. Before entering into the proof in the next sections, we make some preparatory work. 
 \smallskip 

 We start by studying the phase present in the integral close to $\sigma=0$.  
Following~\cite{FLR1},  
 we set 
 \[
 L^\eps(\sigma)= p_\eps(\sigma\sqrt\eps) \cdot y -q_\eps (\sigma\sqrt\eps ) \cdot D_y, 
 \]
 and we observe 
 \[
 {\rm e}^{i p_\eps(\sigma\sqrt\eps )\cdot (y-q_\eps(\sigma\sqrt\eps ))}\vec\varphi(t^\flat+\sigma\sqrt\eps,y-q_\eps(\sigma\sqrt\eps)  ) 
=  {\rm e}^{-\frac i2 p_\eps(\sigma\sqrt\eps )\cdot q_\eps(\sigma\sqrt\eps )} {\rm e}^{iL^\eps(\sigma) }\vec\varphi.
 \]
 Therefore, we have 
 \[
 {\rm e}^{\frac{i}{\eps} \Lambda(\sigma \sqrt\eps )} 
{\rm e}^{i p_\eps(\sigma\sqrt\eps )\cdot (y-q_\eps(\sigma\sqrt\eps ))}\vec\varphi(t^\flat+\sigma\sqrt\eps,y-q_\eps(\sigma\sqrt\eps)  ) =
{\rm e}^{\frac{i}{\eps} \Lambda(\sigma \sqrt\eps )-\frac i2 p_\eps(\sigma\sqrt\eps )\cdot q_\eps(\sigma\sqrt\eps )} {\rm e}^{iL^\eps(\sigma) }\vec\varphi.
 \]
 
 Using Taylor expansion in $\sigma=0$, we obtain 
$$
\frac 1\eps \Lambda(\sigma \sqrt \eps ) -\frac 1 2 q_\eps(\sigma\sqrt\eps)\cdot p_\eps(\sigma\sqrt\eps) = \mu^\flat \sigma ^2 + \sqrt \eps s^3 f_1(\sigma \sqrt\eps) $$
 with $s \mapsto f_1(s)$ bounded, together with its derivatives, for $s\in [t_0,t_0+T]$. In the following, the notation $f_j$ will denote functions that have the same property.
Then,  with 
$
L= \beta^\flat\cdot y-\alpha ^\flat \cdot D_y$,
 we have 
\[
L^\eps(\sigma) = \sigma L +\sqrt\eps \sigma ^2 L^\eps_1(\sigma\sqrt\eps ).
\]
The family of operator $s \mapsto L^\eps_1(s)$ maps $\mathcal S(\R^d)$ into itself, for $s\in [t_0,t_0+T]$.
Moreover,  
 the commutator $[ L, L_1(\sigma \sqrt\eps)]$ is a scalar so that we have 
 $$\frac 12[ L, L_1(\sigma \sqrt\eps)] =  f_2(\sigma \sqrt\eps) $$
  with the notation we have just introduced. 
  Therefore, by Baker-Campbell-Hausdorff formula
  $${\rm e}^{iL^\eps(\sigma )} = {\rm e}^{i\sigma L} {\rm e}^{i\sigma ^2\sqrt\eps L_1(\sigma \sqrt\eps)}{ \rm e} ^{i  \sqrt\eps \sigma ^3 f_2(\sigma \sqrt\eps)}.$$
  We will also use 
  $${\rm e}^{i\sqrt\eps s^2 L_1(\sigma \sqrt\eps)} ={\rm Id} + \sqrt\eps \sigma ^2 \Theta(\sigma \sqrt\eps)$$
  where the operator-valued map $ s\mapsto \Theta(s)$ is smooth and such that  for all $\sigma\in [t_0,t_0+T]$, the operator $\Theta(s)$ and its derivatives maps $\mathcal S(\R^d)$ into itself.
  Setting $f_3=f_1+f_2$, we deduce that $I^{\eps,\delta}_{\varphi(\cdot)}(s)$ writes 
   \begin{align*}
&I^{\eps,\delta}_{\varphi(\cdot)}(s)=\sqrt\eps  \int_{-\delta/\sqrt\eps }^{(s-t^\flat)/\sqrt\eps }
{\rm e}^{i\mu^\flat \sigma^2 + \sqrt\eps \sigma^3 f_3(\sigma\sqrt\eps)}  {Q}_j(\sigma\sqrt\eps ) 
{\rm e}^{i sL  }  ds + R^{\eps,\delta}(s)\\
\mbox{with}\;\;&R^{\eps,\delta}(s)= \eps  \int_{-\delta /\sqrt\eps }^{(s-t^\flat)/\sqrt\eps }
{\rm e}^{i\mu^\flat \sigma^2 + \sqrt\eps \sigma^3 f_3(\sigma \sqrt\eps)} 
{Q}_j (\sigma\sqrt\eps )
{\rm e}^{i \sigma L  }  
\sigma ^2 \Theta^\eps(\sigma \sqrt\eps)d\sigma.
\end{align*}
Let us analyze $R^{\eps,\delta}(s)$. For this, we perform an integration by parts. Indeed,  
 $$\partial_\sigma  ( \mu^\flat \sigma ^2 +\sqrt\eps \sigma ^3 f_3(\sigma \sqrt\eps))= 2 \mu^\flat \sigma ( 1+ \sigma\sqrt\eps f_4(\sigma\sqrt\eps))$$
 for some smooth bounded function $f_4$ with bounded derivatives. Moreover, since~$\delta$ is small, 
we have  $1+\sigma\sqrt\eps f_4(\sigma\sqrt\eps) > 1/2$
 for all $\sigma \in]-\delta/\sqrt\eps,+\delta/\sqrt\eps[$. 
 Therefore, we can write 
 \begin{align*}
 R^{\eps,\delta}(s) = &  \left[ \frac {\eps \sigma}{2i\mu^\flat ( 1+ \sigma\sqrt\eps f_4(\sigma\sqrt\eps))}
  {\rm e}^{i\mu^\flat \sigma^2 +i\sqrt\eps \sigma^3 f_3(\sigma\sqrt\eps)}  {Q}_j (\sigma\sqrt\eps )  {\rm e}^{i\sigma L}    \right]_{-\frac\delta{\sqrt\eps}}^{\frac{s-t^\flat}{\sqrt\eps}} \\
 &\qquad
- \frac{\eps}{2i\mu^\flat} \int_{-\frac{\delta}{\sqrt\eps}}^{\frac{s-t^\flat}{\sqrt\eps}} 
 {\rm e}^{i\mu^\flat \sigma^2+i\sqrt\eps \sigma^3 f_3(\sigma\sqrt\eps)}
 \frac{d}{d\sigma}\left(\frac{\sigma} {1+ \sigma\sqrt\eps f_4(\sigma\sqrt\eps)}  {Q}_j (\sigma\sqrt\eps ) \e ^{i\sigma L} \right) 
d\sigma.
 \end{align*}
 We deduce that for all  $\varphi\in {\mathcal S}(\R^d)$, we have in the Schwartz class
 $R^{\eps,\delta} (s)\varphi= O(\sqrt \eps\delta) + R^{\eps,\delta}_1(s)\varphi$
 with 
 $$R^{\eps,\delta}_1\varphi(s)=
 - \frac{\eps}{2i\mu^\flat} \int_{-\frac{\delta}{\sqrt\eps}}^{\frac{s-t^\flat}{\sqrt\eps}} 
 {\rm e}^{i\mu^\flat \sigma^2+i\sqrt\eps \sigma^3 f_3(\sigma\sqrt\eps)}
 \left(\frac{\sigma} {1+ \sigma\sqrt\eps f_4(\sigma\sqrt\eps)} {Q}_j (\sigma\sqrt\eps )  \e ^{i\sigma L} L\varphi \right) 
d\sigma.
$$
We then need another integration by parts to obtain that in Schwartz class 
$R^{\eps,\delta}_1 (s)\varphi= O(\sqrt \eps\delta) $.  Note that this additional integration by parts is required by the presence of a $\sigma$ without a coefficient $\sqrt\eps$ in  the integrand. We write 
\begin{align*}
R^{\eps,\delta}_1(s)\varphi&=
 - \frac{\eps}{(2i\mu^\flat)^2} \left[
 {\rm e}^{i\mu^\flat \sigma^2+i\sqrt\eps \sigma^3 f_3(\sigma\sqrt\eps)}
 \left(\frac{1} {(1+ \sigma\sqrt\eps f_4(\sigma\sqrt\eps))^2} { Q}_j (\sigma\sqrt\eps )  \e ^{i\sigma L} L\varphi \right) \right]_{-\frac{\delta}{\sqrt\eps}}^{\frac{s-t^\flat}{\sqrt\eps}} \\
&\;\;\;\; + 
\frac{\eps}{(2i\mu^\flat)^2} \int_{-\frac{\delta}{\sqrt\eps}}^{\frac{s-t^\flat}{\sqrt\eps}} 
 {\rm e}^{i\mu^\flat \sigma^2+i\sqrt\eps \sigma^3 f_3(\sigma\sqrt\eps)} \frac d {d\sigma}
 \left(\frac{1} {(1+ \sigma\sqrt\eps f_4(\sigma\sqrt\eps))^2} {\mathcal Q}^\eps(\sigma\sqrt\eps )  \e ^{i\sigma L} L\varphi \right) 
d\sigma \\
&= O(\delta\sqrt\eps)
\end{align*}
This terminates the proof of Lemma~\ref{lem:prepa}.
\end{proof}

 We devote the next three subsections to the proof of each part of Lemma~\ref{lem:asymptotics}, taking advantage of Lemma~\ref{lem:prepa} and focusing on the main part of $I^{\eps,\delta}_{\varphi(\cdot)}(s)$ exhibited in this statement. 
 
\subsection{Estimates on the transfer operators}

\begin{proof}[Proof of Lemma~\ref{lem:asymptotics}(1)]
    Because  the phase $\phi^\eps$ is oscillating far away from $0$, we 
use a smooth cut-off function $\chi$ identically equal to $1$ close to $0$ and to $1$ outside a  ball of radius $1$, and we write in Schwartz space
\[
I_{\varphi(.)}^{\eps,\delta}(s)=I^{1}(s)+I^{2}(s)+ \O(\sqrt\eps \delta)
\]
 with 
\begin{align*}
I^{1}(s)& = \sqrt\eps \int_{-\delta/\sqrt\eps}^{(s-t^\flat)/\sqrt\eps} {\rm e}^{i\phi^\eps(\sigma)} \chi(\sigma)
{Q}_j(\sigma\sqrt\eps  ) 
{\rm e}^{i L\sigma } \phi(t^\flat +\sigma\sqrt\eps) d\sigma, 
\\
I^{2}(s)& = \sqrt\eps \int_{-\delta/\sqrt\eps}^{(s-t^\flat)/\sqrt\eps} {\rm e}^{i\phi^\eps(\sigma)} (1-\chi)(\sigma)
{Q}_j(\sigma\sqrt\eps  ) 
{\rm e}^{i L\sigma } \phi(t^\flat +\sigma\sqrt\eps) d\sigma. 
\end{align*}
Of course, we have also used Lemma~\ref{lem:prepa}.
\smallskip 

\noindent {\it The compactly supported term}. 
The term $I^{1}$ is the easiest to deal with since it has compact support, independently of $\eps$ and $\delta$. Therefore, using that $Q_j(\sigma\sqrt\eps)$ maps $\mathcal S(\R^d,\C^m)$ into itself for each $1\leq j\leq M$, we obtain the existence of $\ell=\ell(M,k)$ such that 
\begin{align*}
\| I^{1}(s)\|_{\Sigma^k}\ & \leq C\,\sqrt\eps  \sup_{s'\in [t^\flat-\delta,t^\flat+\delta ]}\| \vec\varphi(s') \|_{\Sigma^{k+\ell(M,k)}}.
\end{align*}
\smallskip 

\noindent {\it The oscillating term}. 
For dealing with the term $I^2$, we take advantage of the oscillating phase for compensating the fact that the support is large and we perform integration by parts.
We write 
\begin{align*}
I^{2}(s) &= \sqrt\eps \int_{-\delta/\sqrt\eps}^{(s-t^\flat)/\sqrt\eps} \partial_\sigma \left( {\rm e}^{i\phi^\eps(\sigma)} \right)     \frac{1}{\sigma} b^\eps(\sigma)
(1-\chi)(\sigma)
{Q}_j(\sigma\sqrt\eps  ) 
{\rm e}^{i L\sigma } \vec\varphi(t^\flat + \sigma\sqrt\eps)d\sigma \\
& = - \sqrt\eps \int_{-\delta/\sqrt\eps}^{(s-t^\flat)/\sqrt\eps} {\rm e}^{i\phi^\eps(\sigma)}  \partial_\sigma \left(    \frac{1}{\sigma} b^\eps(\sigma) (1-\chi)(\sigma)
{Q}_j(\sigma\sqrt\eps  ) 
{\rm e}^{i L\sigma }  \vec\varphi(t^\flat+ \sigma\sqrt\eps)   \right)d\sigma\\
&\qquad \qquad + \sqrt\eps \left. \left( {\rm e}^{i\phi^\eps(\sigma)}   \frac{1}{\sigma} b^\eps(\sigma)
(1-\chi)(\sigma)
{Q}_j(\sigma\sqrt\eps  ) 
{\rm e}^{i L\sigma }\vec\varphi(t^\flat+\sqrt\eps \sigma)\right)\right|_{\sigma=-\frac\delta{\sqrt\eps}}^{\sigma= \frac{s-t^\flat}{\sqrt\eps}} \\
&= -\sqrt\eps \sum_{\varsigma=1}^3
\int_{-\delta/\sqrt\eps}^{(s-t^\flat)/\sqrt\eps} {\rm e}^{i\phi^\eps(\sigma)}  b^\eps_\varsigma (\sigma) 
 \vec\varphi(t^\flat+ \sigma\sqrt\eps)   d\sigma + \mathcal O(\delta^{-1}\eps \sup_{s'\in[t^\flat-\delta,t^\flat+\delta]}\| \vec\varphi(s')\|_{\Sigma_{k+\ell(M,k)}})
\end{align*}
with 
\begin{align*}
b^\eps_1(\sigma) & = \left( -\frac 1{\sigma^2}
b^\eps(\sigma)
(1-\chi)(\sigma) 
+ \frac 1\sigma
\partial_\sigma b^\eps(\sigma)
(1-\chi)(\sigma) 
-\frac 1\sigma b^\eps(\sigma)
\chi'(\sigma) \right)
{Q}_j(\sigma\sqrt\eps  ) {\rm e}^{i L\sigma } 
,\\
   b^\eps_2(\sigma)& =
   \frac {\sqrt\eps}{\sigma} b^\eps(\sigma)
(1-\chi)(\sigma)\left(\partial_s
{Q}_j(\sigma\sqrt\eps  ) {\rm e}^{i L\sigma } 
+ {Q}_j(\sigma\sqrt\eps  ) 
{\rm e}^{i L\sigma }\partial_s\right) ,\\
b^\eps_3(\sigma) &= 
\frac i\sigma b^\eps(\sigma)
(1-\chi)(\sigma)
{Q}_j(\sigma\sqrt\eps  ) {\rm e}^{i L\sigma } 
L .
\end{align*}
We deduce from the integrability of the coefficient of ${Q}_j(\sigma\sqrt\eps  ) 
$ in $b^\eps_1$
that 
\[
\int_{-\delta/\sqrt\eps}^{(s-t^\flat)/\sqrt\eps} {\rm e}^{i\phi^\eps(\sigma)}  b^\eps_1 (\sigma)  \vec\varphi(t^\flat+ \sigma\sqrt\eps)   d\sigma = \mathcal O(\sup_{s'\in[t^\flat-\delta,t^\flat+\delta]}\| \vec\varphi(s')\|_{\Sigma_{k+\ell(M,k)}}).
\]
Then, using the size of the support of the integral, we obtain 
\[
\int_{-\delta/\sqrt\eps}^{(s-t^\flat)/\sqrt\eps} {\rm e}^{i\phi^\eps(\sigma)}  b^\eps_2 (\sigma)   \vec\varphi(t^\flat+ \sigma\sqrt\eps)   d\sigma =  \mathcal O( \delta \sup_{s'\in[t^\flat-\delta,t^\flat+\delta]}(
\|  \vec\varphi(s')\|_{\Sigma_{k+\ell(M,k)}}+\| \partial_s \vec\varphi(s')\|_{\Sigma_{k+\ell(M,k)}}).
\]
Finally, the last term satisfies
\[
\int_{-\delta/\sqrt\eps}^{(s-t^\flat)/\sqrt\eps} {\rm e}^{i\phi^\eps(\sigma)}  b^\eps_3(\sigma) 
{\rm e}^{i L\sigma }  \vec\varphi(t^\flat+ \sigma\sqrt\eps)   d\sigma =  \mathcal O(|\log\eps|) \sup_{s'\in[t^\flat-\delta,t^\flat+\delta]}(
\|  \vec\varphi(s')\|_{\Sigma_{k+1+\ell(M,k)}},
\]
which terminates the proof of (1).
\end{proof}

\subsection{The leading order term of the transfers operator.}

 \begin{proof}[Proof of Lemma~\ref{lem:asymptotics}(2)]
We now study $I_{\varphi}^{\eps,\delta}(s)$ for some function $\vec\varphi$ independent of $s$ and 
 we  want to calculate the leading order term when $s=t^\flat +\delta$ and $j=1$.
We set 
\[
\mathfrak I_\varphi= \frac 1{\sqrt\eps} I_{\varphi}^{\eps,\delta}(t^\flat +\delta )- W_1(t^\flat,\zeta^\flat) \mathcal T^\flat_{1\rightarrow 2}.
\]
Following~\cite{FLR1} Section 5.3, we first transform the expression $ I_{\varphi}^{\eps,\delta}(t^\flat +\delta)$ by performing the change of variable
\[
z= \sigma(1+\sqrt\eps \sigma \tilde\phi(\sigma\sqrt\eps)/\mu^\flat)^{1/2}
\] 
and we observe that 
$\sigma=z(1+\sqrt\eps z g_1(z\sqrt\eps) ) \;\;\mbox{and} \;\; \partial_\sigma z = 1+\sqrt\eps z g_2( z\sqrt\eps)$
for some smooth bounded functions $g_1$ and $g_2$ with bounded derivatives. 
Note that
we have used that $\sigma\sqrt\eps$ is  of order $\delta$, thus small, in the domain of the integral. 
Besides, there exists a family of operators $\widetilde { Q}^\eps (z)$ such that 
${Q}^\eps (\sigma\sqrt\eps ) = \widetilde { Q}^\eps (z\sqrt\eps)$
with $\widetilde { Q}^\eps(0)= { Q}^\eps(0)$ and $ Q^0(0)=  W_1(t^\flat,\zeta^\flat)$ 
\begin{align*}
I_{\varphi}^{\eps,\delta}(t^\flat +\delta)= \sqrt\eps \int_{z_-}^{z_+}   {\rm e}^{-i\mu^\flat z^2 }
\widetilde {Q}^\eps (z\sqrt\eps ) {\rm e}^{iz(1+\sqrt\eps z g_1(z\sqrt\eps))\, L}\vec\varphi \frac{dz}{1+\sqrt\eps z g_2( z\sqrt\eps)} + \O(\sqrt\eps\delta)
\end{align*}
(the $\O(\sqrt\eps\delta)$ coming from Lemma~\ref{lem:prepa}).
Moreover $z_-=-\delta/\sqrt\eps+o(1)$ and $z_+=\delta/\sqrt\eps+o(1)$.
A Taylor expansion allows to write 
\begin{align*}
&\widetilde { Q}^\eps (z\sqrt\eps ) {\rm e}^{i\sqrt\eps z^2 g_1(z\sqrt\eps)\, L} \frac{1}{1+\sqrt\eps z g_2( z\sqrt\eps)}=
{ Q}_0(0)+  \sqrt\eps  z (\widetilde { Q}^\eps_1(z\sqrt\eps)+ z \widetilde { Q}^\eps_2(z\sqrt\eps))
\end{align*}
for some smooth operator-valued maps $z\mapsto \widetilde { Q}^\eps_j(z\sqrt\eps)$ mapping $\mathcal S(\R^d)$ into itself, such that for all $\varphi\in{\mathcal S}(\R^d)$ the family 
$\| {\widetilde {Q}^\eps_j(z\sqrt\eps)}\varphi\|_{\Sigma^k}\leq c_j\| \varphi\|_{\Sigma^{k+1+\ell(M,k)}}$ (because of the loss of regularity involved by $L$ and the $ Q^\eps_j$-s).
We obtain in Schwartz class 
\[
\mathfrak I_\varphi=\mathfrak I_\varphi^1+\mathfrak I_\varphi^2+\O(\delta)
\]
with
\begin{align*}
&\mathfrak I^1_\varphi=\sqrt \eps 
\int_{z_-}^{z_+} z \, {\rm e}^{-i\mu^\flat z^2 }  
(\widetilde { Q}^\eps_1 (z\sqrt\eps ) +z\widetilde {Q}^\eps_2(z\sqrt\eps) ){\rm e}^{izL}\,\vec\varphi dz,\\
&\mathfrak I_\varphi^ 2= Q_0(0) \int_{\R\setminus [z_-,z_+]}   {\rm e}^{-i\mu^\flat z^2 }  
\, 
{\rm e}^{izL}\ \vec\varphi dz.
\end{align*}
Let us study $\mathfrak I_\varphi^1$.
Arguing by integration by parts as previously, we obtain 
\begin{align*}
\mathfrak I^1_\varphi = &  \frac{\sqrt \eps}{2i\mu^\flat} 
\int_{z_-}^{z_+}   \, {\rm e}^{-i\mu^\flat z^2 }{\partial_z} \left( 
\widetilde {Q}^\eps_1 (z\sqrt\eps ) +z\widetilde{ Q}^\eps_2(z\sqrt\eps)  {\rm e}^{iz L} \right)\,\vec\varphi\,dz + \O(\delta\| \vec\varphi\|_{\Sigma^{k+\ell(M,k)}})\\
&=  \frac{ \eps}{2i\mu^\flat} 
\int_{z^-}^{z_+}   \, {\rm e}^{-i\mu^\flat z^2 }  
\partial_z \widetilde { Q}^\eps_1 (z\sqrt\eps )   {\rm e}^{iz L} \, \vec \varphi dz\\
&\qquad \qquad +  \frac{\sqrt \eps}{2i\mu^\flat} 
\int_{z_-}^{z_+}   
\, {\rm e}^{-i\mu^\flat z^2 }
 \partial_z \left(
 \widetilde{ Q}^\eps_2(z\sqrt\eps) +\sqrt\eps z\partial_z \widetilde{ Q}^\eps_2(z\sqrt\eps)  {\rm e}^{iz L}\right) \,\vec \varphi\, dz + \mathcal O(\delta\| \vec\varphi\|_{\Sigma^{k+\ell(M,k)}}).
\end{align*}
Since $|z\sqrt\eps|\leq \delta$ on the support of the integral, we obtain
\[
\mathfrak I^1_\varphi=
\mathcal O(\delta\| \vec\varphi\|_{\Sigma^{k+\ell(M,k)}}).
\]


It remains to analyze  $\mathfrak I_\varphi^2$. 
For this, we focus on 
$$
\mathfrak I_\varphi^{2,\infty}:=\int_{-\infty}^{z_-}
{\rm e}^{-i\mu^\flat s^2 } 
{\rm e}^{isL} \,\vec\varphi \, ds.
$$
We  claim that for 
any $k\geq 0$, and any $N\geq 1$, 
 there exists a constant $C_{k,N}>0$ and an integer $\ell(N,k)$ such that  
\beq\label{eq:ertransf}
 \Vert \mathfrak I_\varphi^{2,\infty}\Vert_{\Sigma^k} \leq C_{k, N} \|\varphi\|_{\Sigma^{\ell(N,k)}}\lambda^{-N},\;\; \lambda=\delta/\sqrt\eps.
\eeq
This estimate will conclude the proof, since a similar estimate holds when integrating from $z_+$ to~$+\infty$.
For proving~\eqref{eq:ertransf}, we 
 use a Littlewood-Paley  decomposition defined as follows (see \cite[p. 91]{AG}): 
Let $\chi_0\in C_0^\infty(\R)$ be such that $\chi_0(u) = 1$ if $\vert u\vert \leq 1/2$ 	and $\chi_0(u) =0$ if $\vert u\vert \geq 1$. Define $\chi_1(u) = \chi_0(u/2) - \chi_0(u)$ (supported in  
$\{1/2\leq \vert u\vert \leq 2\}$). Then we have
$$
1 = \chi_0(u) +\sum_{j\geq 0}\chi_1(2^{-j}u),\,\, \forall u\in\R.
$$
We then write 
 \begin{align}\label{pluto18}
\mathfrak I_\varphi^{2,\infty}
= &\int_{-\infty}^{z_-}\chi_0\left(\frac{s}{\lambda}\right){\rm e}^{-i\mu^\flat s^2 } 
{\rm e}^{isL}\,\vec \varphi  \,ds+ \sum_{j\geq 0}\int_{-\infty}^{z_-}\chi_1\left(\frac{s}{2^j\lambda}\right){\rm e}^{-i\mu^\flat s^2 } 
{\rm e}^{isL}\,\vec \varphi  \,ds.
\end{align}
We deduce that there exists $j_0$ such that for $j\geq j_0$, 
\begin{align*}
\int_{-\infty}^{z_-}\chi_1\left(\frac{s}{2^j\lambda}\right){\rm e}^{-i\mu^\flat s^2 } 
{\rm e}^{isL}\,\vec \varphi  \,ds
=&
\int_{-\infty}^{+\infty}\chi_1\left(\frac{s}{2^j\lambda}\right){\rm e}^{-i\mu^\flat s^2 } 
{\rm e}^{isL}\,\vec \varphi \, ds.
\end{align*}
The first $(j_0+1)$-th. terms in the right-hand side of~\eqref{pluto18}  they can be treated as $\mathfrak I ^1_\varphi$ via a unique integration by parts, generating boundary terms. We obtain 
\[
\mathfrak I_\varphi^{2,\infty}= 
\sum_{j\geq j_0+1}\int_{-\infty}^{+\infty}\chi_1\left(\frac{s}{2^j\lambda}\right){\rm e}^{-i\mu^\flat s^2 } 
{\rm e}^{isL}\,\vec \varphi \, ds+ 
\mathcal O(\delta\| \vec\varphi\|_{\Sigma^{k+\ell(M,k)}}).
\]
We then 
  integrate by parts using the differential operator $P:=\frac{i}{2\mu^\flat s}\frac{\partial}{\partial s}$. We  have for $j\geq 1$,
 $$
 \int_{-\infty}^{+\infty}\chi_1\left(\frac{s}{2^j\lambda}\right){\rm e}^{-i\mu^\flat s^2 } 
{\rm e}^{isL}\,\vec\varphi \, ds = \int_{-\infty}^{+\infty}{\rm e}^{-i\mu^\flat s^2 } 
(P^*)^M\left(\chi_1\left(\frac{s}{2^j\lambda}\right){\rm e}^{isL}\right)\,\vec\varphi \, ds.
$$
Besides, for any $N\in\R$, there exists a constant $C_N>0$ such that if $f\in\mathcal C^\infty_0(\R)$,
 $$
 | (P^*)^N f(s)| \leq C_N \langle s\rangle ^{-N}\sup_{1\leq p\leq N} |f^{(p)} (s)|.
 $$
 Therefore, 
   noticing that $\vert s\vert$ is of order $2^j\lambda$ on the support of the integrals,  we obtain
 \begin{align*}
\left\|  \int_{-\infty}^{+\infty}\chi_1\left(\frac{s}{2^j\lambda}\right){\rm e}^{-i\mu^\flat s^2 } 
{\rm e}^{isL}\,\vec\varphi\,  ds \right\|_{\Sigma^k} &\leq  C\|\vec\varphi \|_{\Sigma^{k+N}} \int _{ |s| \sim 2^j \lambda}
 |s| ^{-N}  ds\\
 &\leq C'\|\vec\varphi \|_{\Sigma^{k+N}} (2^j\lambda)^{-N+1} \int_{1/2}^2 \frac{ds}{s^2}.
\end{align*}
Therefore,  taking $N$ large enough, the  series  is convergent and we have~\eqref{eq:ertransf}.
\end{proof}

\subsection{Conclusion of the proof of Theorem~\ref{thm:through_cros}}

We can now conclude the proof of 
Theorem~\ref{thm:through_cros}. Via a Dyson series, we have constructed interaction operators in Section~\ref{sec:dyson}, leading to the relation~\eqref{eqdef:dyson_bis}. After resummation as in~\eqref{resommation}, we are left with operators $\Theta^{\eps,\delta}_{m,M}$ satisfying~\eqref{eqdef:dyson}. Moreover, their action on wave packets expresses as~\eqref{def:gros_theta}, which is studied in details in Section~\ref{sec:mtom+1}.
By Lemma~\ref{lem:asymptotics}, we obtain that the operators $\Theta^{\eps,\delta}_{m,M}$ satisfy~\eqref{eq:pluto1} with the properties~\eqref{norm_estimate}
and~\eqref{eq:pluto3}.

\section{Propagation of wave packets}
\label{sec:thm_main}

We prove Theorem~\ref{th:WPmain}. Let $k\in\N$ and let be $\psi^\eps_0$ a polarized wave packet as in~\eqref{def:initial_data}:
\[
\psi^\eps_0= \vec V_0\wp^\eps_{z_0} (f_0)\;\;\mbox{with} \;\;f_0\in{\mathcal S}(\R^d)\;\;\mbox{and}\;\;\vec V_0\in\C^m.
\]
Let $\delta>0$.
By Theorem~\ref{th:WPmain_hors_gap}, for $t\in[t_0,t^\flat-\delta]$, and in $\Sigma^k_\eps$, $\psi^\eps(t)$ is an asymptotic sum of wave packets and writes
\[
\psi^\eps(t) =\psi^{\eps,N}_1(t)+\psi^{\eps,N}_2(t) 
+\O\left(\left(\frac{\sqrt\eps}{\delta}\right)^{N+1}\delta^{-2\kappa_0}\right).
\]
with $\psi^{\eps,N}_\ell(t)$ given by~\eqref{eq:psiepsell}. 
\smallskip 

We now  take the vector 
\[
\underline \psi^{\eps,N}(t^\flat-\delta) = \, ^t\left( \underline \psi^{\eps,N}_1(t^\flat-\delta),\underline \psi^{\eps,N}_2(t^\flat-\delta)\right), \;\;
\underline \psi^{\eps,N}_\ell(t^\flat-\delta)=\psi^{\eps,N}_\ell(t^\flat -\delta),\;\ell=1,2,
\] 
as initial data in the system~\eqref{newsch}. It is a sum of $N$ wave packets.
By construction, in particular because of the linearity of the equation, we have  for all $t\in[t^\flat-\delta, t^\flat+\delta]$,
\[
\psi^\eps(t) = \underline \psi^{\eps,N}_1(t) + \underline \psi^{\eps,N}_2(t) +\O\left(\left(\frac{\sqrt\eps}{\delta}\right)^{N+1}\delta^{-2\kappa_0}\right)
\]
in $\Sigma^k_\eps$. When $t=t^\flat +\delta$, we deduce from Theorem~\ref{thm:through_cros}  that in $\Sigma^k_\eps$, 
\begin{align*}
\underline \psi^{\eps, N} (t^\flat +\delta)= &
\U^\eps_{\rm diag}(t^\flat +\delta,t^\flat  )\;\left(\1+\sum_{1\leq m\leq M}  \Theta^{\eps,\delta}_{m,M} \right)\; \U^\eps_{\rm diag}(t^\flat,t^\flat -\delta ) \underline \psi^{\eps, N}(t^\flat -\delta) \\
&\;\;+\O(\delta^{M+1}) +\O\left((\eps\delta^{-2})^{\frac{M+1}2}\delta^{-2\kappa_0-k}\right).
\end{align*}
By  Proposition~\ref{evadia} of the Appendix, $\underline \psi^{\eps, N} _\ell(t^\flat +\delta)$ is a sum of wave packets 
\[
\underline \psi^{\eps, N} _\ell(t^\flat +\delta)= \sum_{0\leq m\leq M} \eps^{\frac m 2} \underline \psi^{\eps, m,M,N} _\ell(t^\flat +\delta)
\]
where each term $\underline \psi^{\eps, m,M,N} _\ell(t^\flat +\delta)$ involves a term of action. When $m=0$ and $m=1$, these terms have been computed precisely:
\begin{itemize}
\item[(i)] If $m=0$, for $\ell=1,2$,
\begin{align*}
\underline \psi^{\eps, 0,M,N} _\ell(t^\flat +\delta) & 
= {\rm e}^{\frac i\eps S_\ell (t^\flat +\delta, t^\flat -\delta, \zeta^\flat)} 
{\rm WP}^\eps_{\Phi^{t^\flat+\delta,t_0} (z_0)} \left(
\mathcal M [ F_\ell (t^\flat +\delta, t_0,z_0)] \pi_\ell (t_0,z_0)  \vec V_0f_0
\right)
\end{align*}
where we have used the property of the scalar propagation of wave packets. 
\item[(ii)] If $m=1$, 
only the term with $\ell=2$ contributes and 
\begin{align*}
\underline \psi^{\eps, 1,M,N} _2(t^\flat +\delta) & = 
 {\rm e}^{\frac i\eps S_2(t^\flat+\delta, t^\flat ,\zeta^\flat)+\frac i\eps S_1(t^\flat, t_0,z_0)} 
  {\rm WP}^\eps_{\Phi^{t^\flat+\delta, t^\flat}_2(\zeta^\flat) } (\varphi_2)
\end{align*}
with 
\[
\varphi_2= \mathcal M[F_2(t^\flat+\delta,t^\flat, \zeta^\flat)] W_1(t^\flat, \zeta^\flat) ^*\mathcal T^\flat_{1\rightarrow 2} \mathcal M[F_1(t^\flat, t_0,z_0)]\pi_1(t_0,z_0) \vec V_0 f_0,
\] 
where $W_1$ is the off-diagonal matrix computed in~\eqref{def:W1}.
\end{itemize}
At that stage of the proof, we choose $M\ge N$ and have obtained that $\psi^\eps(t^\flat +\delta)$ is a sum of wave packets up to $\O\left(\left(\frac{\sqrt\eps}{\delta}\right)^{N+1}\delta^{-2\kappa_0-k}+\delta ^{M+1}\right)$ in $\Sigma^k_\eps$. Moreover, we know precisely the terms of order $\eps^0$ and $\eps^{\frac 12}$.
\smallskip 

For concluding, we take the vector 
\[
\psi^\eps_{\rm app}(t^\flat+\delta):=\underline \psi^{\eps,N}_1(t^\flat+\delta ) + \underline \psi^{\eps,N}_2(t^\flat +\delta )
\]
as initial data at time $t=t^\flat+\delta$  in the equation~\eqref{eq:sch}. The function $\psi^\eps_{\rm app}(t^\flat+\delta)$ is an approximation of $\psi^\eps (t^\flat+\delta)$ at order $\O\left(\left(\frac{\sqrt\eps}{\delta}\right)^{N+1}\delta^{-2\kappa_0-k}+\delta ^{M+1}\right)$ in $\Sigma^k_\eps$.
 By construction and because of the linearity of the equation, for all times $t\in[t^\flat+\delta, t_0+T]$, 
\[
\psi^\eps(t)=\mathcal U_H(t,t^\flat +\delta) \psi^\eps_{\rm app}(t^\flat+\delta)+ \O\left(\left(\frac{\sqrt\eps}{\delta}\right)^{N+1}\delta^{-2\kappa_0-k} +\delta^{M+1}\right).
\]
We then applies Theorem~\ref{th:WPmain_hors_gap} between times $t^\flat +\delta$ and $t$. Indeed, the classical trajectories involved in the construction do not meet $\Upsilon$ again and we are in an adiabatic regime, as in Theorem~\ref{th:WPmain_hors_gap}. This concludes the proof of  Theorem~\ref{th:WPmain}. 

\section{Transfer operator for Gaussian states}
\label{sec:trsf_Gauss}

In this section, we study the action of the transfer operator $\mathcal T_{1\to 2}^\flat$ defined in~\eqref{transf1} on Gaussian states. This will justify the transfer coefficient~\eqref{def:tau} used in the initial value representations. 
\smallskip 

Following~\cite{FLR1}, for $(\mu,\alpha,\beta)\in\R\times\R^{2d}$ and $\varphi\in\Sch(\R^d)$, we set 
  \beq\label{transf2}
{\mathcal T}_{\mu,\alpha, \beta}\varphi(y) =
 \left(\int_{-\infty}^{+\infty}{\rm e}^{i\mu s^2}{\rm e}^{is(\beta\cdot y-\alpha\cdot D_y)}ds\right)\varphi(y).
\eeq
The operator $\mathcal T_{1\to 2}^\flat$ then satisfies
\[
\mathcal T_{1\to 2}^\flat= \mathcal T_{-\mu^\flat, \alpha^\flat,\beta^\flat}.
\]
Indeed, by the Baker-Campbell-Hausdorff formula, we have 
$${\rm e}^{is\beta\cdot y}{\rm e}^{-is\alpha\cdot D_y}= {\rm e}^{is\beta\cdot y-is\alpha\cdot D_y+is^2\alpha\cdot\beta/2},$$ 
and we deduce the equivalent representation
\beq\label{transf1_bis}
 {\mathcal T}_{\mu,\alpha, \beta}\varphi(y) =
  \int_{-\infty}^{+\infty}{\rm e}^{i(\mu-\alpha\cdot\beta/2)s^2} {\rm e}^{is\beta\cdot y}\varphi(y-s\alpha)ds .
 \eeq
 An explicit computation also gives the  following  useful connection  with the Fourier transform
 \beq\label{FT}
 {\mathcal F}{\mathcal T}_{\mu,\alpha, \beta}= {\mathcal T}_{\mu+\alpha\cdot\beta, \beta, -\alpha}{\mathcal F}.
\eeq
The next proposition sums up the main information that we will use about these operators. 

\begin{proposition}\label{prop:T}
Let $(\mu,\alpha,\beta)\in \R^{2d+1}$.  
\begin{enumerate}
\item  The operator   $\mathcal T_{\mu,\alpha,\beta}$ maps ${\mathcal S}(\R^d)$ into itself if and only if $\mu\not=0$.
 \item Moreover, if $\mu\not=0$,   $\mathcal T_{\mu,\alpha,\beta}$ is  a metaplectic transformation in the Hilbert space
 $L^2(\R^d)$  multiplied by a complex number: 
 \begin{equation}\label{eq:T}
 {\mathcal T}_{\mu,\alpha, \beta} 
= \sqrt{\frac{i\pi}{\mu}}{\rm e}^{-\frac{i}{4\mu}(\beta\cdot y -\alpha\cdot D_y)^2}.
 \end{equation}
\item
If $\mu\not=0$, $\Gamma\in{\mathfrak S}^+(d)$ and $P\in{\mathcal C}^\infty(\R^{2d})$ is a polynomial function then there exists
$\Gamma_{\mu,\alpha,\beta,\Gamma}\in{\mathfrak S}^ +(d)$ such that
$$
{\mathcal T}_{\mu,\alpha,\beta}  ({\rm op}^w_1(P) g^\Gamma ) = 
\sqrt{\frac{i\pi}{\mu}} {\rm op}^w_1(P\circ \Phi_{\alpha,\beta} ((4\mu)^{-1} )g^{\Gamma_{\mu,\alpha,\beta,\Gamma}}
$$
where  $\Phi_{\alpha,\beta}$  is the symplectic $2d\times 2d$ matrix given by
\begin{align} \label{def:Phi}
 &\qquad \qquad \Phi_{\alpha,\beta}(t) =\begin{pmatrix}  \1 -2t\beta\otimes\alpha&2t\alpha\otimes\alpha, \\-2t\beta\otimes\beta & \1 +2t\alpha\otimes\beta\end{pmatrix}, \;\;t\in\R,
\end{align}
and 
\beq\label{Deq}
\Gamma_{\mu,\alpha,\beta,\Gamma} =    \Gamma  -\frac{(\beta-\Gamma\alpha)\otimes (\beta-\Gamma\alpha) }{2\mu-\alpha\cdot\beta+\alpha\cdot\Gamma\alpha}.
\eeq
\end{enumerate}
\end{proposition}

\begin{remark}
The matrix $\Gamma_{\mu,\alpha,\beta,\Gamma}$ is in $\mathfrak S^+(d)$ since $g^{\Gamma_{\mu,\alpha,\beta,\Gamma}}$ is proved to be Schwartz class. 
It is also important to notice that 
$2\mu-\alpha\cdot\beta+\alpha\cdot \Gamma\alpha$ is non zero because the imaginary part of $\Gamma$ is positive definite. 
We note that the normalisation of the Gaussian $g^{\Gamma_{\mu,\alpha,\beta,\Gamma}}$ follows \eqref{def:Gamma}.
\end{remark}

\begin{proof}
Point (1) is linked with Point (2) and comes from the formula~\eqref{transf1} and~\eqref{transf2}. Indeed, when $\mu\not=0$, 
equation~\eqref{eq:T} is an application of relation~\eqref{transf2} and of functional calculus  on the self-adjoint operator $(\beta\cdot y -\alpha\cdot D_y)^2$  and the Fourier-transform formula  of complex Gaussian functions:
\beq\label{FG}
\int_{-\infty}^{+\infty}{\rm e}^{is^2\mu}{\rm e}^{is\tau}ds  
=\sqrt{\frac{i\pi}{\mu}}{\rm e}^{\frac{\tau^2}{4i\mu}}
\;{\rm with}\;\arg(i\mu)\in]-\pi, \pi[.
\eeq
It remains to analyze the case where $\mu=0$. The computations are different whether $\alpha\cdot\beta=0$ or not. We assume $\alpha\not=0$ and we set
$$
\hat \alpha=\frac\alpha{|\alpha|},\;\; y= (y\cdot \hat\alpha)\,\hat\alpha +y_\perp.$$
Similar formulas can be obtained when $\beta\not=0$ using~\eqref{FT}. 
Let us first assume $\alpha\cdot\beta=0$.
\begin{align*}
\mathcal T_{0,\alpha,\beta} \varphi(y) & = \int \e^{is\beta \cdot y_\perp} \varphi( y\cdot \hat \alpha \,\hat \alpha -s\alpha +y_\perp) ds\\
&=|\alpha|^{-1} \int \e^{i|\alpha|^{-1} (y\cdot \hat \alpha-\sigma) (\beta\cdot y_\perp)} \varphi (\sigma \hat \alpha +y_\perp)d\sigma\\
&= |\alpha|^{-1} \e^{i |\alpha|^{-1} (y\cdot \hat \alpha) (\beta \cdot y_\perp)} \mathcal F_{\alpha} \varphi\left( \frac{\beta \cdot y_\perp}{|\alpha|} ,y_\perp\right)
\end{align*}
where $\varphi\in{\mathcal S}(\R^d)$, $y_\perp=y-{\hat\alpha\cdot y} \,\hat\alpha$ and $\mathcal F_\alpha$ is the partial Fourier transform in the direction~$\alpha$. \\
In the case where $\alpha\cdot\beta\not=0$, we write
\begin{align*}
\mathcal T_{0,\alpha,\beta} \varphi(y)& =
 (2\pi)^{-1} \int_{\R^2} \e^{-is^2 \frac{\alpha\cdot\beta}{2}+is\beta\cdot y+i\eta (y\cdot \hat\alpha -s|\alpha|) }\mathcal F_\alpha\varphi(\eta,y_\perp) d\eta ds\\
&
= \sqrt{\frac{1}{2 i\pi\beta\cdot  \alpha}} \int \e^{ i\frac {(\beta\cdot y -\eta|\alpha|)^2 } {2\alpha\cdot \beta} +i\eta y\cdot \hat \alpha} \mathcal F_\alpha\varphi(\eta,y_\perp) d\eta \\
&= \sqrt{\frac{1}{2 i\pi\beta\cdot  \alpha}} 
 \e^{i \frac{(\beta\cdot y)^2} { 2\beta\cdot \alpha}}
  \int 
  \e^{-i\eta \frac{\beta_\perp\cdot y_\perp} {\beta\cdot \hat \alpha}}
\e^{i \frac{\eta^2 |\alpha|^2} {2\beta\cdot\alpha}}
\mathcal F_\alpha\varphi(\eta,y_\perp) d\eta \\
& = \sqrt{\frac{1}{2i\pi\beta\cdot  \alpha}}  \e^{i \frac{(\beta\cdot y)^2} { 2\beta\cdot \alpha}} 
\int \e^{-i\eta \frac{\beta_\perp\cdot y_\perp} {\beta\cdot \hat \alpha} }
\mathcal F_\alpha\left(e^{i \frac{(D_{y}\cdot \alpha)^2} {2\beta\cdot\alpha} } \varphi\right)(\eta,y_\perp) d\eta\\
&= \sqrt{\frac{2\pi}{i\beta\cdot  \alpha}}  \e^{i \frac{(\beta\cdot y)^2} { 2\beta\cdot \alpha}}  
\left( e^{i \frac{(D_{y}\cdot\alpha)^2} {2\beta\cdot\alpha}}  \varphi\right)
\left(-\frac{\beta_\perp\cdot y_\perp} {\beta\cdot \hat \alpha}\hat\alpha+ y_\perp\right)
\end{align*}
This concludes the proof of Points~(1) and~(2). 

\medskip
 
Point (3) derives from  the formulation of~$\mathcal T_{\mu,\alpha,\beta}$ as a metaplectic transform,
\[
\mathcal T_{\mu,\alpha,\beta} = \sqrt{\frac{i\pi}{\mu}} \e^{-\frac{i}{4\mu} \hat K}
\]
associated with the quadratic Hamiltonian 
 \[
 K(y,\eta) = (\beta\cdot y-\alpha\cdot\eta)^2.
 \]
We use general  results concerning the action of a metaplectic transformation  on a Gaussian~$g^\Gamma$ (for details see \cite{corobook}[Chapter 3] and in particular \cite{corobook}[\S3.2, Theorem 16]).  With the quadratic Hamiltonian $\hat K$ one associates the   linear flow $\Phi_{\alpha,\beta}(t)=(\Phi_{ij}(t))_{1\leq i,j\leq 2}$ (in a $d\times d$ block form) given by~\eqref{def:Phi}. 
Besides, the Egorov theorem and the classical propagation of the Gaussian are both  exact: we have for~$P$ a smooth polynomial function 
 $$
{\rm e}^{-it\hat K}({\rm op}_1(P)g^\Gamma) = {\rm op}_1(P\circ \Phi_{\alpha,\beta}(t)){\rm e}^{-it\hat K} g^\Gamma=
({\rm op}_1(P\circ \Phi_{\alpha,\beta}(t)) g^{\Gamma_t},
$$
where, in view of~\eqref{def:Gamma}, the matrix $\Gamma_t\in \mathfrak S^+(d)$ is given by
\begin{equation}\label{eq:inconv_inv}
\Gamma_t = (\Phi_{21}(t) +\Phi_{22}(t)\Gamma)(\Phi_{11}(t)+\Phi_{12}(t)\Gamma)^{-1}
,\;\; c_{\Gamma_t} = 
c_\Gamma\, {\rm det}^{-1/2}(\Phi_{11}(t) +\Phi_{12}(t)\Gamma), 
\end{equation}
This induces the existence of the matrix $\Gamma_{\mu,\alpha,\beta,\Gamma}\in\mathfrak S^+(d)$ of 
Point (3) of the Proposition. Formula~\eqref{Deq} follows either by computing the inverse in \eqref{eq:inconv_inv} or by applying
\eqref{transf1_bis} to get
\[
   {\mathcal T}_{\mu,\alpha,\beta} g^\Gamma(y)
   = 
c_\Gamma \sqrt{\frac{2i\pi}{2\mu-\alpha\cdot\beta+\alpha\cdot\Gamma\alpha}}\;
   {\rm e}^{\frac{i}{2}\left(y\cdot\Gamma y -\frac{(y\cdot(\beta-\Gamma\alpha))^2}{2\mu-\alpha\cdot\beta+\alpha\cdot\Gamma\alpha}\right)}   .
\]
\end{proof}


\appendix

\chapter{Matrix-valued Hamiltonians}\label{app:A}

We explore here the implications of the technical assumptions on the Hamiltonian~$H^\eps$. We aim at motivating the set of Assumptions~\ref{hyp:growthH} and deriving their
 consequences. Roughly speaking, these assumptions ensure the existence of the propagators associated with the full matrix-valued Hamiltonian and with its eigenvalues. Secondly, they guarantee adequate properties of growth at infinity which are used in our analysis. 
 \smallskip 
 
 In this section, we work with $m\times m $ ($m\in\N$) matrix-valued  Hamiltonians $H$ that are  subquadratic:
 \begin{equation}\label{fifi1}
\forall \beta\in\N^{2d},\;\;\exists C_\beta >0,\;\;\forall (t,z)\in I\times \R^d,\;\; |\partial^\beta H(t,z)|_{\C^{m,m}}\leq C_\beta \langle z\rangle ^{(2-|\beta|)+}.
\end{equation}
For their eigenvalues we can deduce the following growth properties.

\begin{lemma}\label{lem:growth_1}
Assume that the matrix-valued function  $H\in\mathcal C^\infty(I\times \R^{2d}, \C^{m,m})$  satisfies~\eqref{fifi1}.
Assume that for all $(t,z)\in I\times\R^{2d}$, $H(t,z)$
has a smooth  eigenvalue $h(t,z)$ with a smooth eigenprojector $\pi(t,z)$ of constant rank for $|z|>R$, $R>0$. Then,
there exists a constant $C>0$ such that for all $(t,z)\in I\times\R^{2d}$ with $|z|>R$,
 \[
 \vert h(t,z)\vert \leq  C\langle z\rangle^2,\;\; \vert\nabla h(t,z)\vert\leq C\langle z\rangle.
 \]
\end{lemma}

\begin{proof}
 The relation $H\pi = h\pi$ implies $|h(t,z)|\le  C\langle z\rangle^2$. 
Moreover, writing 
\[H=h\pi+H(1-\pi)= h\pi+(1-\pi) H\]
 and differentiating these two relations,  we obtain for all $j\in\{1,\cdots 2d\}$, denoting $\partial_{z_j}$ by $\partial_j$
\begin{equation}\label{eq:partialH}
\partial_{j} H = \partial_{j} h\, \pi+ h\partial_{j} \pi + \partial_{j} H(1-\pi) - H\partial _{j}\pi=
\partial_{j} h \, \pi + h\partial_{j}\pi + (1-\pi )\partial_{j} H-\partial _{j}\pi H.
\end{equation}
Multiplying from the left and the right with $\pi$ and using that $\partial\pi$ is off-diagonal, 
we obtain the relation $\pi \partial _{j}H\pi = \partial _{j}h\pi$, whence with $c={\rm Rank} (\pi)=cte$, 
\[ \partial _{j}h= c\, {\rm Tr} (\pi\partial _{j}H \pi).\]
This implies $|\partial_{j} h(t,z)|\le C\langle z\rangle$.  
\end{proof}

This proof shows that the study of higher derivatives of the eigenvalues requires a control on the derivatives of the eigenprojectors. The following example shows that the situation may become very intricate, and one can have smooth subquadratic eigenvalues while the derivatives of the projectors are unbounded.  

\begin{example}
Assume  $d=1$. Let $\rho, \theta \in\mathcal C^\infty(\R)$ such that $\rho$ has  bounded derivatives and with
\[
\rho(x) = \frac 1 {x^3} \;\;\mbox{and}\;\; \theta(x)= x\ln (x) -x\;\;\mbox{for}\;\; |x|>1.
\]
Define
\[
 H(x)= \rho(x) \begin{pmatrix} \cos( \theta(x)) &  \sin( \theta(x))\\  \sin ( \theta(x)) & - \cos ( \theta(x)) \end{pmatrix}.
 \]
The eigenvalues of $H$ are $\pm \rho$ and the eigenprojector associated with the eigenvalue $\rho$ is  
\[
\pi(x)= \frac 12 \left( 1+\begin{pmatrix} \cos( \theta(x)) &  \sin( \theta(x))\\  \sin ( \theta(x)) & - \cos ( \theta(x)) \end{pmatrix}\right).
\]
Its derivative
\[
\pi'(x)= \frac 12 \theta'(x) \begin{pmatrix}- \sin( \theta(x)) &  \cos( \theta(x))\\  \cos ( \theta(x)) &  \sin ( \theta(x)) \end{pmatrix}
\]
is not bounded. On the other side, $\rho$ and $H$ are sub-quadratic. Indeed, for $|x|>1$, the derivatives of the coefficients of $H$ are of the form 
$$\frac 1 {x^3} \left(p_1\left(\frac 1 x, \ln x\right) \cos \theta(x) + p_2 \left(\frac 1x, \ln x\right) \sin \theta(x) \right)$$
 for $p_1$ and $p_2$ two polynomial functions of two variables. Thus, they are  bounded. 
\end{example}

A manner of controlling the growth of the derivatives of the eigenprojectors consists in requiring a lower bound on the gap function $f$ at infinity. 

\begin{lemma}
\label{lem:growth_2}
Let $\ell\in\{1,2\}$.
Assume that the matrix-valued function  $H\in\mathcal C^\infty(I\times \R^{2d}, \C^{m,m})$ is bounded together with its derivatives and  
has a smooth bounded eigenvalue $h$  with smooth associated eigenprojectors $\pi(t,z)$ of constant rank for $|z|>R$, $R>0$.
Assume there exists $C,n_0>0$ such that for $(t,z)\in I\times \R^d$ with $|z|>R$,
\[
{\rm dist}\left(h(t,z), {\rm Sp}(H(t,z) )\setminus\{ h(t,z)\}\right)  \geq C\langle z\rangle ^{-n_0}.
\]
 Then, for all $\gamma\in\N^{2d}$ with $|\gamma|\geq 2$, there exists a constant $C_\gamma >0$ such that 
\[
\forall (t,z)\in I\times \R^{2d},\;\; |\partial_z^\gamma \pi(t,z)|\leq C_\gamma \langle z\rangle^{ |\gamma| n_0}.
\]
\end{lemma}

\begin{proof}   
   We work for $|z|>R$ and set $\delta=\langle z\rangle^{-n_0}$ for this proof only. We use the formula from complex analysis 
   \[
   \pi(z)=-\frac1{2i\pi}\oint_{\mathcal C} (H(t,z)-\zeta)^{-1} d\zeta
   \]
   where ${\mathcal C}$ is a  closed contour in the complex plane with winding number one around~$h(t,z)$. More precisely, the contour ${\mathcal C}$ is constructed as follows.
   We set $a=\inf \,h(t,z)$, $b=\sup \,h(t,z)$ and for $\eta>0$, we consider the rectangles $\mathcal R_\eta$ based on the points $(a-\eta,-1-\eta)$, $(a-\eta,1+\eta)$, $(b+\eta,1+\eta)$ and $(b+\eta,-1-\eta)$.
We choose a smooth  curve ${\mathcal C}$, strictly included in $(\mathcal R_{\delta} \setminus \mathcal R_0)$ and consisting in four connected branches, ${\mathcal C}_\ell$, $1\leq \ell \leq 4$, that are two by two disjoint with 
\[
{\mathcal C}_1=\{ \zeta=a- \delta/2-it, \, -1\leq t\leq 1\}, \;\; {\mathcal C}_3=\{ \zeta=b+\delta/2 + it, \, -1\leq t\leq 1\}
\]
and ${\mathcal C}_2\subset \{1\leq  \Im(\zeta)\leq 1+\delta \}$,  ${\mathcal C}_4\subset \{-1-\delta \leq  \Im(\zeta)\leq -1\}$.
\smallskip

By resolvant estimates and Pythagoras' theorem, there exists a constant $c>0$ such that, 
\[
\forall \zeta \in{\mathcal C} _1\cup {\mathcal C}_3,\;\; 
| (H(t,z)-\zeta)^{-1}|_{\C^{m,m} }\leq c\frac 1{\sqrt{t^2+\delta^{2}}},
\]
while for all  $\zeta\in{\mathcal C}_2\cup{\mathcal C}_4$, this norm is uniformly bounded in $\delta$.
Similarly, 
\[
\forall \zeta \in{\mathcal C}_3,\;\; 
| (H(t,z)-\zeta)^{-1}|_{\C^{m,m}}\leq c\frac 1{\sqrt{t^2+\delta^{2}}},
\]
while for all  $\zeta\in{\mathcal C}_1\cup {\mathcal C}_2\cup{\mathcal C}_4$, this norm is uniformly bounded in $\delta$.

 We deduce from the boundedness of the derivatives of $H(t,z)$ that for all $\gamma\in\N^{2d}$, there exists $c_\gamma>0$ such that 
\begin{align*}
&\forall \zeta \in{\mathcal C} _1\cup {\mathcal C}_3,\;\; 
| \partial^\gamma_{z}\left( (H(t,z)-\zeta)^{-1}\right)|_{\C^{m,m}}\leq c_\gamma \frac 1{\left(\sqrt {t^2+\delta^2}\right)^{|\gamma|+1}},
\end{align*}
while these derivatives are uniformly bounded, independently of $\delta$ on the other branches of~${\mathcal C}$.
As a consequence, there exists  $C_\gamma>0$ associated with  $\gamma\in\N^{2d}$, and   such that 
\begin{align*}
| \partial^\gamma_{z} \pi|_{\C^{m,m}}
& \leq C_\gamma \int_{-1}^{+1} 
\frac {dt}{\left(\sqrt {t^2+\delta^2}\right)^{|\gamma|+1}}
C_\gamma  \delta^{-|\gamma|}\int_{0} ^{+\infty} \frac {du}{\left(\sqrt {u^2+1}\right)^{|\gamma|+1}}= \mathcal{O}(\delta^{-|\gamma|}),
\end{align*}
which concludes the proof.
\end{proof}

These two Lemmata allow to derive  the consequences of Assumption~\ref{hyp:growthH} for a Hamiltonian $H^\eps=H_0+\eps H_1$. 

\begin{proposition}\label{lem:growth_eigen_bis}
 Assume that $H^\eps=H_0+\eps H_1$ satisfies Assumption~\ref{hyp:growthH}. Then, for $\ell\in\{1,2\}$ we have the following properties:
  \begin{enumerate}
 \item 
  For all $\gamma\in\N^{2d}$ with $|\gamma|\geq 2$, there exists a constant $C_\gamma >0$ such that 
\[
\forall (t,z)\in I\times \R^{2d},\;\; |\partial_z^\gamma \pi_\ell (t,z)|\leq C_\gamma \langle z\rangle^{ |\gamma| n_0}.
\]
\item The Hamiltonian trajectories $\Phi^{t_0,t}_{h_\ell}(z)$ are globally defined  for all $z\in\R^{2d}$. Besides, there exists $C'>0$ such that  
 $$
 \vert  \Phi^{t_0,t}_{h_\ell}(z) \vert\leq C \vert z\vert{\rm e}^{C\vert t-t_0\vert}
 $$
 and the Jacobian matrices $F_\ell(t,z)=\partial_z \Phi^{t,t_0}_{h_\ell}(z)$ (see~\eqref{def:F}) satisfy
 \[
 \| F_\ell(t,z)\|_{\C^{2d,2d}}\leq C {\rm e}^{C\vert t-t_0\vert}.
 \]
\end{enumerate}
\end{proposition}

\begin{remark}
Recall the representation $H_0=v\1+f(\pi_2-\pi_1)$ with gap $f$ and average $v$ defined \Cref{def:smooth_cros}. 
Under the assumptions of Proposition~\ref{lem:growth_eigen_bis}, for all $j\in\{1,\cdots, 2d\}$, 
the matrices 
\[
f \partial_{z_j} \pi_1=-f\partial_{z_j} \pi_2= \frac 12 \left(\partial _{z_j} (H_0-v) -\partial_{z_j}  f(\pi_2-\pi_1)\right)
\]
are bounded. 
\end{remark}

\begin{proof}[Proof of Proposition~\ref{lem:growth_eigen_bis}] 
The eigenprojectors $\pi_\ell$ are those of $H_0-v\1_{\C^m}$. By (i) and (ii) of Assumptions~\ref{hyp:growthH}, the Hamiltonian  $H_0-v\1_{\C^m}$ satisfies the properties required for applying Lemma~\ref{lem:growth_2}, whence Point 1 of Proposition~\ref{lem:growth_eigen_bis}. 
Moreover, Point 2 of Proposition~\ref{lem:growth_eigen_bis} comes from (iii) of Assumptions~\ref{hyp:growthH}.
\end{proof}

We close this Section with the proof of Lemma~\ref{lem:trsp_par}. 

\begin{proof}[Proof of Lemma~\ref{lem:trsp_par}]
The map $t\mapsto \mathcal R_\ell(t,t_0,z)$ is valued in the set of unitary maps because 
the matrix $H^{{\rm adia}}_{\ell,1} $ is self adjoint. 
Besides, 
 the  map 
  \[
  (t,z)\mapsto Z_\ell(t,z)=  \pi_\ell (t,\Phi^{t,t_0}_{h_\ell}(z) ){\mathcal R}_\ell (t,t_0,z)\pi_\ell^\perp(t_0,z)
  \]
   satisfies the ODE
  \[ 
  i\partial_t Z_\ell(t,z)= \left(-i\left(\pi_\ell (\partial_t \pi_\ell +\{h,\pi_\ell \})\right)\circ \Phi^{t,t_0}_{h_\ell}\right)\!(z)\, Z_\ell(t,z),\;\;\;\; Z_\ell(t_0,z)=0,
  \]
 and thus coincides with the solution $Z_\ell(t,z)=0$.
 \smallskip 

 For proving~\eqref{eq:deriv_R}, we observe that $\omega\cdot\nabla _z\mathcal R(t,t_0,z)$ satisfies the ODE
 \begin{align*}
 i\partial_t(\omega\cdot\nabla _z\mathcal R(t,t_0,z)= &  H_{\ell,1}^{\rm adiab} (
 t,\Phi^{t,t_0}_{h_\ell}(z) )(\omega\cdot\nabla_z{\mathcal R}_\ell (t,t_0,z))\\
&\;\; +  (F_\ell(t,t_0,z) \omega)\cdot\nabla_z H_{\ell,1}^{\rm adiab} (
 t,\Phi^{t,t_0}_{h_\ell}(z) ){\mathcal R}_\ell (t,t_0,z),
 \end{align*}
 with initial data $\omega\cdot\nabla _z\mathcal R(t,t_0,z)=0$. 
 The formula then comes from Duhamel formula and the definition of $\mathcal R(t,t_0,z)$. 
\end{proof}

\chapter{Elements of symbolic calculus : 
the Moyal product}\label{prodest}

In this section, we revisit results about symbolic calculus, in particular regarding composition formula and the remainder estimate for the Moyal product. We aim at their extension to the setting of the symbol classes $\bdS_\delta^\mu(\mathcal D)$ and $\mathbf S_{\eps,\delta}^\mu(\mathcal D)$ that we have introduced in Definition~\ref{def:symbol_Sdelta}.

\section{Formal expansion}
We first recall the formal product rule for quantum observables 
with Weyl quantization. Let $A, B \in {\mathcal S}(\R^{2d},\C^{m,m})$. The Moyal product $C:=A\circledast B$ is the  
semi-classical observable 
$C$ such that $\widehat{A}\circ \widehat{B} = \widehat{C}$. Some computations with 
the Fourier transform give the following well known formula \cite[Theorems 18.1.8]{ho} \cite[Theorem~4.11(i)]{Zwobook}
\beq\label{moy1}
C(x,\xi) = 
\exp\left(\frac{i\eps}{2}\sigma(D_q,D_p;D_{q^\prime},D_{p^\prime})\right)A(q,p) 
B(q^\prime,p^\prime)\vert_{(q,p)=(q^\prime,p^\prime)=(x,\xi)}, 
\eeq 
where $\sigma$ is the symplectic bilinear form $\sigma((q,p), (q',p'))=p\cdot q'-p'\cdot q$ and 
$D=i^{-1}\nabla$. 
By expanding the exponential term, we obtain
\beq\label{prod2} 
C(x,\xi) = 
\sum_{j\geq 
0}\frac{\eps^j}{j!}\left(\frac{i}{2}\sigma(D_q,D_p;D_{q^\prime},D_{p^\prime}) \right)^j 
A(q,p)B(q^\prime,p^\prime)\vert_{(q,p)=(q^\prime,p^\prime)=(x,\xi)}. 
\eeq 
So that $C = \sum_{j\ge 0}\eps^j C_j$ is a formal power series in $\eps$ with coefficients 
given by~\eqref{fifi2}.

\section{Symbols with derivative bounds} \label{sec:app_fifi4}

For $\mu\geq 0$ denote by ${\bf P}(\mu)$ the linear space of matrix-valued $\mathcal C^\infty$ symbols $A:\R^{2d}\rightarrow\C^{m,m}$ such that for any $\gamma\in \N^{2d}$ with $\vert\gamma\vert\geq \mu$, there exists $C_\gamma>0$ such that 
$$
\vert\partial_z^\gamma A(z)\vert \leq C_\gamma\;\;\forall z\in\R^{2d}.
$$
Assuming that $A\in {\bf P}(\mu_A)$, $B\in {\bf P}(\mu_B)$ it is known that, for $\eps$ fixed, $A\circledast B= C$ where $C\in{\bf P}(\mu_C)$ with $\mu_C\geq \max\{\mu_A, \mu_B\}$ 
(see e.g. the proof \cite[Theorem~18.1.8]{ho})
\medskip 

We aim at having   better estimates for  small $\eps>0$ and a control of the derivatives of $C$ in terms of those of $A$ and $B$. The following estimate and its proof are a particular case of \cite[Theorem~A.1]{BR}.

\begin{theorem}\label{thm:moyalest}
For every $N\in\N$ and $\gamma\in\N^{2d}$, there  exists a constant $K_{N,\gamma}$ such that for any 
$A\in {\bf P}(\mu_A)$, $B\in {\bf P}(\mu_B)$ the Moyal remainder 
\beq 
R_N(A,B;z;\eps):= (A\circledast B)(z) - \sum_{0\leq j \leq N}\eps^j C_j(z)
\eeq 
satisfies for every $z\in\R^{2d}$ and $\eps\in(0,1]$,
\beq\label{remest1} 
\left\vert \partial_z^\gamma R_N(A,B;z;\eps)
 \right\vert \leq \eps^{N+1}K_{N,\gamma}
\sum_{N+1\leq \vert\alpha\vert, \vert\beta\vert \leq N+ \kappa_0+\vert\gamma\vert} 
\Vert\partial_z^\alpha A\Vert_{L^\infty}\Vert\partial_z^\beta B\Vert_{L^\infty},
\eeq
with $\kappa_0 = 4d+2$.
\end{theorem}

The estimate~\eqref{eq:norm_pseudo}  allows to evaluate the norms of the operators involved in Theorem~\ref{thm:moyalest} when considering symbols of the classes 
$\bdS_\delta^\mu(\mathcal D)$ and $\mathbf S_{\eps,\delta}^\mu(\mathcal D)$ of Definition~\ref{def:symbol_Sdelta}.
\begin{corollary}\label{cor:kappa0}
If $A\in\mathbf S_{\eps,\delta}^\mu(\R^{2d})$ and $B\in \mathbf S_{\eps,\delta}^{\mu'}(\R^{2d})$, then $A\circledast B\in \mathbf S^{\mu+\mu'}_{\eps,\delta}$ (see Remark~\ref{rem:Sdelta}) and  for all $N,k\in\N$, there exists a constant $C_{N,k}>0$ such that 
\[ 
\| {\rm op}_\eps(R_N(A,B;z;\eps))\|_{\Sigma^k_\eps} \leq C _{N,k} \,\eps^{N+1}\,  \delta ^{ \mu+\mu'-2(N+1+k+\kappa_0)}.
\]
\end{corollary}

The proof of Theorem~\ref{thm:moyalest} relies on the technical next lemma, the proof of which we postpone at the end of this paragraph.

\begin{lemma}\label{foi} 
There exists a constant $C_d>0$ such that for any $F\in{\mathcal S}(\R^{2d}\times \R^{2d},\C^{m,m})$ 
the integral 
\beq
I(\lambda) = \lambda^{2d} \int_{\R^{2d}\times \R^{2d}}\exp[-i\lambda 
\sigma(u, v)]F(u,v)dudv. 
\eeq 
satisfies  
\beq
\vert I(\lambda)\vert \leq C_d\sup_{u, v\in\R^{2d} \atop \vert\alpha\vert+\vert\beta\vert\leq 4d+1}
\vert \partial^\alpha_u\partial^\beta_v F(u,v)\vert.  
\eeq 
\end{lemma} 

We are now in position of proving Theorem~\ref{thm:moyalest}.

\begin{proof}[Proof of Theorem~\ref{thm:moyalest}] 
The application of Taylor formula at order $N$ gives the following formula for the remainder
\begin{align*}
R_N(A,B;z;\eps) &= \frac{1}{N!}\left(\frac{i\eps}{2}\right)^{N+1} 
\left(\int_0^1(1-t)^N 
\exp\left(\frac{it \eps}{2}\sigma(D_u;D_v)\right) dt \right)\\
&\qquad\times\, \sigma^{N+1}(D_u;D_v) A(u)B(v) \vert_{u=v=(x,\xi)}.
\end{align*}
By Fourier transform computations (see~\cite[Theorem~4.8(iii) or Theorem~4.11(ii)]{Zwobook},  we obtain
\beq\label{reprem} 
R_N(A,B;z;\eps) = \frac{1}{N!}\left(\frac{i\eps}{2}\right)^{N+1} 
\int_0^1(1-t)^NR_{N,t}(z;\eps)dt,
\eeq
where $R_{N,t}(z;\eps)$ is given by the oscillating integral 
\begin{align*}
&R_{N,t}(z;\eps) =
(2\pi\eps t)^{-2d} 
\int\int_{\R^{2d}\times \R^{2d}} 
\exp\left(-\frac{i}{2t\eps}\sigma(u;v)\right)\sigma^{N+1}(D_u;D_v) 
A(u+z)B(v+z)dudv. 
\end{align*}
We now use Lemma~\ref{foi} for $A, B\in{\mathcal S}(\R^{2d})$ with the integrand 
\[
F_{N,\gamma}(z;u,v) = \pi^{-2d} \ \partial_z^\gamma\left(\sigma^{N+1}(D_u;D_v) A(u+z)B(v+z)\right)
\] 
and the parameter $\lambda=1/(2t\eps)$. We then have 
\[
|\partial_z^\gamma R_{N,t}(z;\eps)| \le C_d \sup_{u,v\in\R^{2d}\atop |\alpha|+|\beta|\le 4d+1}|\partial^\alpha_u\partial^\beta_v F_{N,\gamma}(z;u,v)|.
\]
Moreover, there holds the elementary estimate
\[
\vert \sigma^{N+1}(D_u,D_v)A(u)B(v)\vert \leq 
(2d)^{N+1}\sup_{\vert \alpha\vert+\vert \beta\vert=N+1} 
\vert\partial_x^\alpha\partial_\xi^\beta 
A(x,\xi)\partial_y^\beta\partial_\eta^\alpha B(y,\eta)\vert. 
\]
Together with the Leibniz formula, we then get the claimed results with universal constants.
For symbols $A\in{\bf P}(\mu_A)$ and $B\in{\bf P}(\mu_B)$ we argue by localisation. We use 
$A_\eta(u) = {\rm e}^{-\eta u^2}A(u)$ and 
$B_\eta(v) = {\rm e}^{-\eta v^2}B(v)$ for $\eta 
>0$ and 
pass to the limit $\eta \rightarrow 0$.
\end{proof}

It remains to prove  Lemma~\ref{foi}.

\begin{proof}[Proof of Lemma~\ref{foi}] The lemma is proved in a standard way, using integration  by parts and stationary phase argument.
For the sake of completeness, we give here a  proof.  
We introduce a cut-off $\chi_0\in 
\mathcal C_0^\infty(\R)$ such that  
$$\chi_0(x) = 1\;\;\mbox{ for}\;\; \vert x\vert\leq 1/2\;\;\mbox{ and}\;\; \chi_0(x) = 0\;\;\mbox{for }\;\;\vert x\vert\geq1.$$
 We split $I(\lambda)$ into 
too pieces and write 
$I(\lambda)=I_0(\lambda) + I_1(\lambda) $ with 
\begin{align*}
I_0(\lambda)& = \lambda^{2d}\int\!\!\int_{\R^{2d}\times \R^{2d}}\exp[-i\lambda 
\sigma(u, v)] 
\chi_0(u^2+v^2)F(u,v)dudv, \\ 
I_1(\lambda) &= \lambda^{2d}\int\!\!\int_{\R^{2d}\times \R^{2d}}\exp[-i\lambda 
\sigma(u, v)] 
(1-\chi_0)(u^2+v^2))F(u,v)dudv.
\end{align*}
We notice that $(u,v)\mapsto \sigma(u,u)$ is a quadratic  non-degenerate real  form on $\R^{4d}$.\\
Let us estimate  $I_1(\lambda)$. For $I_1(\lambda)$, the integrand is supported outside  the ball of radius $1/\sqrt2$ in $\R^{4d}$. Therefore, we can   integrate by parts with the 
differential operator 
\[
L = \frac{i}{|u|^2+|v|^2}\left(Ju\cdot\frac{\partial}{\partial 
v}-Jv\cdot\frac{\partial}{\partial u}\right),
\]
using that $L \e^{-i\lambda\sigma(u,v)} = L \e^{-i\lambda Ju\cdot v} = \lambda\e^{-i\lambda\sigma(u,v)}$. 
 Performing $4d+1$ integrations by parts for gaining enough decay to ensure integrability in $(u,v)\in\R^{4d}$, we get a constant $c_d$ such that 
\[
\vert I_1(\lambda)\vert \leq c_d
\sup_{u,v \in \R^{2d}\atop \vert \mu\vert+\vert \nu\vert\leq 4d+1} 
\vert\partial^\mu_u \partial^\nu_vF(u,v)\vert. 
\]
To estimate $I_0(\lambda)$ we apply the stationary phase. 
The symmetric matrix of the quadratic form $\sigma(u,v)$ is 
$$
A_\sigma = \begin{pmatrix}0 & -J\\
J & 0\end{pmatrix}.
$$
So, by the stationary  phase Theorem (\cite{ho}, Vol.I, section 7.7), 
we obtain the existence of two constants $c_1,c_2>0$ such that 
{
\beq
\vert I_0(\lambda)- \lambda^{-2d} c_1\vert \leq c_2
\sup_{u,v \in \R^{2d}\atop |\alpha| \leq 2} 
\vert\partial^\alpha (\chi_0(u^2+v^2) F(u,v)\vert. 
\eeq 
}
 \end{proof}


\chapter{Propagation of wave packets by perturbation of scalar systems}\label{app:C}

In this Appendix, we revisit several well-known results concerning a Hamiltonian $\widehat K(t)$ with its symbol $K(t)$
valued in the set $\C^{m,m}$ of self-adjoint $m\times m$ matrices ($m\in\N$), and which is a perturbation of a scalar function $k(t)$.  
We consider an interval $I_\delta\subset \R$ that may depend on $\delta>0$, which is small, and assume that $K(t)$ is  defined on  $I_\delta$ and  of the form 
\beq\label{eq:K}
K(t) = k(t){\1}_m +\eps K_1(t)+\cdots + \eps^N K_N(t)
\eeq
with $k$ scalar-valued and ~$ k(t){\1}_m +\eps K_1(t)$ subquadratic on the time interval $I_\delta$ according to Definition~\ref{def:subquad}. 
 \smallskip

 We will work on an adequate domain $\mathcal D\subseteq I_\delta\times\R^{2d}$ and use the symbol class $\bdS_{\eps,\delta} ^{-1} (\mathcal D)$, where the additional parameter $\delta>0$ controls the growth of derivatives, see \Cref{def:symbol_Sdelta}. We assume 
\begin{equation}\label{ass_K_delta}
\frac 1{\eps^2}\left( K(t,\cdot)-k(t,\cdot)\1_m -\eps K_1(t,\cdot)\right )\in \bdS_{\eps,\delta
} ^{-1} (\mathcal D).
\end{equation}
In particular, we assume $K_2(t)\in \bdS_{\delta}^{-1}$, $(\mathcal D)$, $K_3(t)\in \bdS_{\delta}^{-3}(\mathcal D)$ and more generally $K_j(t)\in \bdS_{\delta}^{3-2j} (\mathcal D)$ for $j\ge 2$. 
We revisit standard propagation results and take care of the loss in $\delta^{-1}$ and control all the classical estimates with respect to this parameter.

\smallskip 
These properties are satisfied by the Hamiltonians that we consider in the adiabatic region. Indeed, the formally constructed adiabatic symbol $H^{{\rm adia},\eps}(t)$ satisfies  \eqref{ass_K_delta}, see property (3) of of Theorem~\ref{adia1}. Truncating the series at $N$ and introducing adequate temporal cut-offs, 
the Hamiltonians $\widetilde{H}_\ell^{{\rm adia},N,\eps} (t)$ defined for $\ell\in\{1,2\}$ in~\eqref{def:tilde_adia}  
satisfy both assumptions  \eqref{eq:K} and \eqref{ass_K_delta} on the intervals $[t_0,t^\flat-\delta]$ and  
$[t^\flat+\delta, t_0+T]$. 

\section{Egorov Theorem}\label{app:egorov}

We denote by $ \U_K^\eps (t,t_{\rm in})$  the unitary propagator associated with $\widehat K(t)$. It satisfies 
\begin{equation}\label{eq:propagatorK}
i\eps\partial_t \U_K^\eps (t,t_{\rm in})= \widehat K(t) \, \U_K^\eps (t,t_{\rm in}),\;\; \U_K^\eps (t_{\rm in},t_{\rm in})=\1_{m}.
\end{equation}
Moreover 
\[ \U_K^\eps (t,t_{\rm in})^*= \U_K^\eps (t_{\rm in},t),\; t_{\rm in},t\in I_\delta.\]
It is also important that $ \U_K^\eps (t,t_{\rm in})$ maps $\Sigma^k_\eps$ into itself for all $k\in\N$. Indeed, for all 
$\psi^\eps_0\in\Sigma^1_\eps$ and all $1\leq j\leq d$, the function $\psi^\eps(t) =  \U_K^\eps (t,t_{\rm in}) \psi^\eps_0$ satisfies
\begin{align*}
&i\eps\partial_t (x_j\psi^\eps(t)) -\widehat K(t) (x_j\psi^\eps(t)) = f^\eps_j(t),\\
&
i\eps\partial_t (\eps D_{x_j}\psi^\eps(t)) -\widehat K(t) (\eps D_{x_j}\psi^\eps(t)) = g^\eps_j(t),
\end{align*}
where
\begin{align*}
& f^\eps_j(t) = [x_j, \widehat K(t)] \psi^\eps(t) = -\tfrac{\eps}{i} \mathrm{op}^w_\eps(\partial_{\xi_j} K(t))\psi^\eps(t),\\
& g^\eps_j(t)=[\eps D_{x_j} ,\widehat K(t)] \psi^\eps(t) = \tfrac{\eps}{i} \mathrm{op}^w_\eps(\partial_{x_j} K(t))\psi^\eps(t)
\end{align*}
are uniformly bounded in $L^2$. Therefore, $w^\eps(t) = {}^t(x_j\psi^\eps(t),\eps D_{x_j}\psi^\eps(t))$ satisfies
\[
\|w^\eps(t) \|_{L^2} \le \|w^\eps(t_{\rm in})\|_{L^2} + C \int_{t_{\rm in}}^t \|w^\eps(s)\|_{L^2} \, ds
\]
for some constant $C>0$. By the Gronwall Lemma, one then deduces the $L^2$-boundedness of the families $(x_j\psi^\eps(t))$ and $(\eps D_{x_j}\psi^\eps(t))$, whence the boundedness of $\psi^\eps(t)$ in $\Sigma^1_\eps$ for all $t\in I_\delta$. The corresponding recursive process will give the boundedness of $\psi^\eps(t)$ in  $\Sigma^k_\eps$ for any $k\in\N$.

\smallskip 
The Egorov Theorem describes the evolution of an observable when it is conjugated by the propagator $ \U_K^\eps (t,t_{\rm in})$.  
Our aim is to consider the evolution of $ \U_K^\eps (t_{\rm in},t)\widehat A \, \U^\eps_K(t,t_{\rm in})$ for matrix-valued observables $A\in \bdS_{\delta}(\mathcal D)$ and in spaces $\Sigma^k_\eps$, with a precise estimate of the remainders.

 \begin{proposition}\label{EgD} With the above assumptions on $K(t)$, for any matrix-valued symbol $A\in \bdS_{\delta}(\mathcal D)$, 
 there exists a formal series 
 \[
 (t,t_{\rm in})\mapsto \sum_{j\geq 0} \eps^j  A_j(t,t_{\rm in})
 \]
 defined on $I_\delta\times I_\delta$ such that for any $J\geq 1$,  we have for all $t,t_{\rm in}\in I_\delta$,
 $$
 \U_K^\eps (t,t_{\rm in})^*\,\widehat A \; \U^\eps_K(t,t_{\rm in}) = \sum_{0\leq j\leq J}\eps^j\widehat A_j(t,t_{\rm in}) + \eps^{J+1} \widehat { R}^\eps_J(t,t_{\rm in})
 $$
 with $A_j(t,t_{\rm in})\in\bdS^{-2j}_{\delta}(\mathcal D)$. 
   Besides, for all $k\in\N$, there exists a constant $C_{k,J}>0$ such that 
\[
 \|  \widehat {R^\eps_J}(t,t_{\rm in})\|_{\mathcal L(\Sigma^k_\eps)} \leq C_{k,J} \delta^{-2(J+1+\kappa_0)-k},
\]
and the matrix $A_{0}(t,t_{\rm in})$ is given by
   \beq\label{ps2}
    A_{0}(t,t_{\rm in}, z) = \left({\mathcal R}(t,t_{\rm in})^*A{\mathcal R}(t,t_{\rm in})\right)\circ (\Phi_{k}^{t,t_{\rm in}})(z),
   \eeq
 where the unitary matrices ${\mathcal R}(t,t_{\rm in},z)$  solve the transport equation
  \beq\label{ps3}
  i  \partial_t{\mathcal R} (t,t_{\rm in},z) = K_{1} \left(t,\Phi_{k}^{t,t_{\rm in}}(z)\right){\mathcal R}(t,t_{\rm in},z),\; {\mathcal R}(t_{\rm in},t_{\rm in},z)=\1_m.
\eeq
  \end{proposition}

 \begin{remark} \begin{enumerate}\item Note that in the setting of Proposition~\ref{EgD}, 
 we have the propagation law of the supports:  
 $${\rm supp}(A_j(t,t_{\rm in})) = \Phi_{k}^{t_{\rm in},t}({\rm supp}(A))\;\;\mbox{  for any}\;\;j\geq 0.$$
\item
In the scalar time-independent case, that is for $K=k(z)$ and $K_j=0$ for $j\geq 1$, and with  $\delta=1$, the Egorov theorem
\[
 \U_K^\eps (t_{\rm in},t)\widehat A \, \U^\eps_K(t,t_{\rm in}) = \widehat {A\circ \Phi_k^{t,t_{\rm in}}}+\O(\eps^2)
 \]
is well-known (see~\cite{corobook,disj,Zwobook} for example) and has a vanishing sub-principal symbol $A_1(t,t_{\rm in}) = 0$.
\end{enumerate}
\end{remark}
 
 In the time-dependent  matrix-valued  case considered here, 
the dynamics on the observable is driven by the classical flow twisted by the precession ${\mathcal R}$ (see also Section~\ref{sec:parallel_trspt} where such terms appear).

 \begin{proof}
 We perform a recursive argument. The starting point comes from the analysis of an auxiliary map. 
 We consider the operator $\mathcal L^{t,\tau}$ defined by 
\[
\mathcal L^{t,\tau} a(z)= a(\Phi_k^{t,\tau}(z)),\;\; z\in\R^{2d},
\]
for any smooth function $a$. 
Then, the function $\tilde a(\tau,z)=\mathcal L^{t,\tau} a(z) $ satisfies $\tilde a(\tau,\Phi_k^{\tau,t}(z))=a(z)$ and does not depend on  $\tau$. Using that $\Phi^{t,\tau}_k$ is a diffeomorphism, we deduce that 
\[
\partial_\tau \mathcal L^{t,\tau} a = - \{k(\tau), \mathcal L^{t,\tau}a\},\;\; \mathcal L^{t,t}=a.
\]
 We now introduce 
for $\tau,t\in I_\delta$ the auxiliary map valued in $\mathbf S_\delta(\mathcal D)$ and defined  by 
 \[
 A\mapsto A(t,\tau):=\mathcal L^{t,\tau} (\mathcal R(\tau,t) A \mathcal R(\tau,t)^*)
 \]
 We deduce from the analysis of $\mathcal L^{t,\tau}$ the equation 
\[
\partial_\tau A(t,\tau) = -\{ k(\tau), A(t,\tau)\} +
\mathcal L^{t,\tau} (\mathcal \partial_\tau R(\tau,t) A \mathcal R(\tau,t)^*) +
\mathcal L^{t,\tau} (\mathcal R(\tau,t) A  \partial_\tau \mathcal R(\tau,t)^*)
\]
In view of~\eqref{ps3}
\begin{align*}
\mathcal L^{t,\tau} (\mathcal \partial_\tau R(\tau,t) A \mathcal R(\tau,t)^*)& =\frac 1i
\mathcal L^{t,\tau} (K_1(\tau, \Phi_k^{\tau,t}(z))R(\tau,t) A \mathcal R(\tau,t)^*) \\
&= \frac 1 i K_1(\tau, \Phi_k^{\tau,t}\circ \Phi_k^{t,\tau}(z))\mathcal L^{t,\tau} (R(\tau,t) A \mathcal R(\tau,t)^*)\\
&= \frac 1 i K_1(\tau,z) A(t,\tau).
\end{align*}
Arguing similarly for the last term, we obtain that $A(t,\tau)$ solves 
 \[
 \partial_\tau A(t,\tau,z) =- \{ k (\tau), A(t,\tau)\}(z) +  \frac 1i [K_1 (\tau,z), A(t,\tau,z)],\;\; A(t,t)=A.
 \]

Let us now start with the proof of the result for $J=0$. We choose $s,\tau, t\in I_\delta$ and consider the quantity 
\[
Q^\eps_A(t,s,\tau)= \mathcal U^\eps_K(\tau,s)^*\, \widehat A(t,\tau) \,\mathcal U^\eps_K(\tau,s).
\]
 The times $s,\tau,t$ can be understood as $s\leq \tau \leq t$ with $s$ an initial time (that will be taken as $s=t_{\rm in}$) and $t$ the time at which we want to prove the property. 
We then have the boundary properties 
\[
Q^\eps_A(t,s,t)=\mathcal U^\eps_K(t,s)^* \widehat A\, \mathcal U^\eps_K(t,s)\;\;\mbox{and}\;\;
Q^\eps_A(t,s,s)= \widehat A(t,s).
\]
Differentiating in $\tau$, we have 
\begin{align*}
\frac d{d\tau} Q^\eps_A(t,s,\tau) & =  \mathcal U^\eps_K(\tau,s)^* 
\left(\left[ -\frac 1{i\eps} \widehat K(\tau), 
\widehat A(t,\tau)\right] + \partial_\tau \widehat  A (t,\tau)
\right) \mathcal  U^\eps_K(\tau,s)\\
&=  \mathcal U^\eps_K(\tau,s)^* 
\left(\left[ -\frac 1{i\eps} \widehat K(\tau), 
\widehat A(t,\tau)\right] - \widehat{ \{k, A(t,\tau)\}} + \frac 1 i \widehat{[K_1(\tau), A (t,\tau)]}
\right) \mathcal  U^\eps_K(\tau,s)\\
&=\eps\,  \mathcal U^\eps_K(\tau,s)^*\,   \widehat {B_1^\eps}(t,\tau)\,  \mathcal  U^\eps_K(\tau,s)
\end{align*}
where the matrix $B^\eps_1\in\mathbf S^{-2-2\kappa_0}_{\eps,\delta}(\mathcal D)$ stems from the  Moyal product (see Corollary~\ref{cor:kappa0}). 
We deduce by integration between the times $s$ and $\tau$
\begin{equation}\label{egorov_step1}
Q^\eps_A(t,s,\tau)= Q^\eps_A(t,s,s) +\eps \int_s^\tau  \mathcal U^\eps_K(\tau',s)^*\,   \widehat {B_1^\eps}(t,\tau')\,  \mathcal  U^\eps_K(\tau',s) d\tau',
\end{equation}
which implies when $\tau=t$, 
\[
\mathcal U^\eps_K(t,s)^* \widehat A \,\mathcal U^\eps_K(t,s)= \widehat A(t,s) +\eps \int_s^t \mathcal U^\eps_K(\tau,s)^*\,   \widehat {B_1^\eps}(t,\tau)\,  \mathcal  U^\eps_K(\tau,s) d\tau.
\]
This gives the first step of the recursive argument.
\smallskip

We now assume that we have obtained for $J\geq 0$ 
\[
\mathcal U^\eps_K(t,s)^*\, \widehat A \,\mathcal U^\eps_K(t,s)= \sum_{j=0}^J \eps^j \widehat A_j(t,s) +\eps^{J+1} \int_s^t \mathcal U^\eps_K(\tau,s)^*\,   \widehat {B_{J+1}^\eps}(t,\tau)\,  \mathcal  U^\eps_K(\tau,s) d\tau
\]
with $B^\eps_{J+1}\in \mathbf S^{-2-2\kappa_0}_\delta(\mathcal D)$. 
We write 
\[
B^\eps_{J+1}= B_{J+1} + B^{1,\eps}_{J+2}
\]
 with $B_{J+1}\in \mathbf S_\delta^{-2(J+1)-2\kappa_0}(\mathcal D)$ and $B^{1,\eps}_{J+2}\in \mathbf S^{-2(J+2)-2\kappa_0}_{\eps,\delta}(\mathcal D)$. Then, the preceding equation writes 
\begin{align}\label{eq:UKst}
\mathcal U^\eps_K(t,s)^* \widehat A \,\mathcal U^\eps_K(t,s) & = \sum_{j=0}^J \eps^j \widehat A_j(t,s) \\
\nonumber 
& +\eps^{J+1} \int_s^t Q^\eps_{B_{J+1}}(t,s,\tau)  d\tau + \eps^{J+2} 
 \int_s^t \mathcal U^\eps_K(\tau,s)^*\,   \widehat {B_{J+2}^{1,\eps}}(t,\tau)\,  \mathcal  U^\eps_K(\tau,s) d\tau.
\end{align}
We focus on the term involving $Q^\eps_{B_{J+1}}(t,s,\tau) $ that we treat 
 as in the preceding step. We obtain by~\eqref{egorov_step1}
 \[
 Q^\eps_{B_{J+1}}(t,s,\tau)  = Q^\eps_{B_{J+1}}(t,s,s)
+\eps \int_s^\tau \mathcal U^\eps_K(\tau',s)^* \widehat B_{J+2}^{2,\eps} (\tau,\tau')\mathcal U^\eps_K(\tau',s) d\tau'
 \]
 with $B_{J+2}^{2,\eps} \in \mathbf S^{-2(J+2)-2\kappa_0}_{\eps,\delta}(\mathcal D)$.
 We set 
 \[
 A_{J+1}(s,t)= (t-s) Q^\eps_{B_{J+1}}(t,s,s)
 \]
 and 
 \[
 B^\eps_{J+2} (t,\tau)= B^{1,\eps}_{J+2} (t,\tau) +\int_s^{\tau} \mathcal U^\eps_K(\tau',\tau)^*\, B_{J+2}^{2,\eps} (\tau',\tau)\, \mathcal U^\eps_K(\tau',\tau) d\tau'\in\mathbf S^{-2(J+1)}_{\eps,\delta}(\mathcal D).
 \]
 The equation~\eqref{eq:UKst} then becomes 
 \[
  \mathcal U^\eps_K(t,s)^* \widehat A \,\mathcal U^\eps_K(t,s)  = \sum_{j=0}^{J+1} \eps^j \widehat A_j(t,s) \\
+\eps^{J+2} \int_s^t \mathcal U^\eps_K(\tau,s)^*\,   \widehat {B_{J+2}^{\eps}}(t,\tau)\,  \mathcal  U^\eps_K(\tau,s)  d\tau
\]
and we concludes the proof by Proposition~\ref{prop:fifi3}. 
    \end{proof}

 \section{Asymptotic behavior of the propagator}\label{app:wp}   
   In this section, we analyze the propagator $\mathcal U^\eps_K(t,t_{\rm in})$, and compare it with 
       $\U_{KS}^\eps(t,t_{\rm in})$, the propagator for 
       $KS(t) = k(t)\1_m +\eps K_1(t)$.

 \begin{lemma}\label{split} We assume \eqref{eq:K} and \eqref{ass_K_delta}. Then, 
 there exists a formal series
 \[
 \W^\eps(t,t_{\rm in}) = \sum_{j\ge 0}\eps^j\, \W_j(t,t_{\rm in})\in\bdS_{\eps,\delta}^{1}(\mathcal D)
 \]
 with $\W_0(t,t_{\rm in})=\1_m$ and $\W_j(t,t_{\rm in})\in \bdS_{\delta}^{1-2j}(\mathcal D)$ for $j\ge 1$, 
 such that for all $t\in I_\delta$, 
$$
     \U_K^\eps(t,t_{\rm in}) =  \, \widehat \W^\eps(t,t_{\rm in}) \,\U_{KS}^\eps(t,t_{\rm in}).
     $$  
 \end{lemma}

     \begin{proof}[Proof of Lemma~\ref{split}]
       If such a $\W^\eps(t,t_{\rm in})$ exists, it must satisfy   
       \[
       \widehat \W^\eps(t,t_{\rm in})= \U_K^\eps(t,t_{\rm in})\,  \U_{KS}^\eps(t,t_{\rm in})^*.
       \]
       As a consequence, the operator $\widehat \W^\eps(t,t_{\rm in})$ solves
\begin{align*}
       i\eps\partial_t\widehat \W^\eps(t,t_{\rm in}) &= \,\widehat K(t)\, \U_K^\eps(t,t_{\rm in})\,\U_{KS}^\eps(t,t_{\rm in})^*-  \U_K^\eps(t,t_{\rm in})\, \U_{KS}^\eps(t,t_{\rm in})^*\,\widehat{KS}(t)
       \\
       &= \widehat K(t)\, \widehat \W^\eps(t,t_{\rm in}) \,  -
       \, \widehat\W^\eps(t,t_{\rm in})\widehat{KS}(t) 
       \end{align*}
      with $\widehat \W^\eps(t_{\rm in},t_{\rm in})=\1_m$.
      We write 
      \begin{equation}\label{eq:Y}
      i\eps\partial_t\widehat \W^\eps(t,t_{\rm in})= \left[ \widehat k(t)\, \1_{m} +\eps \widehat{K_1}(t), \widehat \W^\eps(t,t_{\rm in})\right]+ 
      \sum_{j= 2}^N 
      \eps^j \widehat {K_j}(t) \widehat\W^\eps(t,t_{\rm in}),
      \end{equation}
      and we deduce a set of recursive equations. Looking to the leading order term, we obtain  
      \[
      \partial_t \W_0(t, t_{\rm in}) + \{k(t),\W_0(t,t_{\rm in})\} - \frac{1}{i}[K_1(t),\W_0(t,t_{\rm in})]=0,\;\; \W_0(t_{\rm in},t_{\rm in})=\1_{m}
      \]
      and therefore
      \[
      \W_{0}(t,t_{\rm in}) = \left({\mathcal R}(t,t_{\rm in})^*\1_m{\mathcal R}(t,t_{\rm in})\right)\circ (\Phi_{k}^{t,t_{\rm in}}) = \1_m.
      \]
      Then, the terms of order $\eps^{j+1}$ give  equations of the form 
      \begin{equation}\label{dY_jsymbol}
      \partial_t \W_{j+1}(t, t_{\rm in}) +\{k(t),\W_{j+1} (t,t_{\rm in})\} - \frac{1}{i}[K_1(t),\W_{j+1}(t,t_{\rm in})] = L_{j+1}(t, t_{\rm in}) ,\;\;\W_{j+1}(t_{\rm in},t_{\rm in})=0,
      \end{equation}
      where the source term 
      $L_{j+1}(t,t_{\rm in})$ depends on $\W_\ell(t, t_{\rm in}) $ for $0\leq \ell\leq j$. For example, for the first terms we have
      \begin{align*}
      L_1(t,t_{\rm in}) &= K_2(t) \W_0(t,t_{\rm in}) = K_2(t) \in\bdS^{-1}_\delta(\mathcal D)\quad\text{and thus}\quad
      \W_1(t,t_{\rm in})\in \bdS^{-1}_\delta(\mathcal D),\\
      L_2(t,t_{\rm in}) &= K_2(t)\W_1(t,t_{\rm in}) + K_3(t) \in\bdS^{-3}_\delta(\mathcal D)\quad\text{and thus}\quad
      \W_2(t,t_{\rm in})\in \bdS^{-3}_\delta(\mathcal D),
      \end{align*}
      where we have used the multiplication property of the symbol classes, see \Cref{rem:Sdelta}. More generally, we obtain
      \[
      L_{j+1}(t,t_{\rm in}) = \sum_{\ell=0}^{j} K_{2+\ell}(t) \W_{j-\ell}(t,t_{\rm in}) \quad\text{and thus}\quad \W_{j+1}(t,t_{\rm in})
      \in \bdS^{-1-2j}_\delta(\mathcal D).
      \]  
       \end{proof}

\begin{remark}
If we assume $\delta\ge\sqrt\eps$ and write 
\[
\W^\eps(t,t_{\rm in}) = \sum_{j=0}^J\eps^j\, \W_j(t,t_{\rm in})  + \eps^{J+1}R_J^\eps(t,t_{\rm in})
\]
with $J\in\N$, then  \Cref{prop:fifi3} implies that for all $j=1,\ldots,J$ and $k\in\N$, there exist constants 
$C_{j,k}>0$ and $C_{J,k}>0$ such that 
\begin{align}\label{eq:Wj}
\| {\rm op}_\eps ( \W_j(t,t_{\rm in}) )\|_{\mathcal L(\Sigma^k_\eps)} &\leq C_{j,k} \,\delta^{1-2j-k},\\
\label{eq:RJ}
\| {\rm op}_\eps ( R_J^\eps(t,t_{\rm in}))\|_{\mathcal L(\Sigma^k_\eps)} &\leq
C_{J,k} \,\delta^{-1-2J-k}
\end{align}
 \end{remark}

\begin{corollary}\label{coro:split}
Assume $\delta\ge\sqrt\eps$ and take  $J\in\N$. Denote $\W^{\eps, J}(t,t_{\rm in}) = \sum_{j=0}^J \eps^j Y_j(t,t_{\rm in})$. Let $k\in\N$. Then, there exists $C_{J,k}>0$ such that 
\[
\left\| \mathcal U^\eps_K(t,t_{\rm in})-
\widehat \W^{\eps, J}(t,t_{\rm in})\, \mathcal U^\eps_{KS}(t,t_{\rm in})\right\|_{\mathcal L(\Sigma^k_\eps)}  \leq C_{k,J}
|t-t_{\rm in}|\left(\frac \eps{\delta^2}\right)^{J+1} \delta ^{-\kappa_0-k}
\]
where $\kappa_0>0$ is the universal constant of Lemma~\ref{thm:prodest}.
 \end{corollary}

       \begin{proof}[Proof of Corollary~\ref{coro:split}]
       We use~\eqref{dY_jsymbol}, and we obtain $\partial_t Y_{j+1}\in S^{1-2(j+1)}_\delta(\mathcal D)$. Then, we estimate the remainder in~\eqref{eq:Y}. By Lemma~\ref{thm:prodest} and Proposition~\ref{prop:fifi3}, we obtain in $\mathcal L(\Sigma^k_\eps)$
           \begin{align*}
  &    i\eps\partial_t\widehat \W^{\eps,J}(t,t_{\rm in})- \left[ \widehat k(t)\, \1_{m} +\eps \widehat{K_1}(t), \widehat \W^{\eps,J}(t,t_{\rm in})\right]-\sum_{j= 2}^J 
      \eps^j \widehat {K_j}(t) \widehat\W^{\eps,J}(t,t_{\rm in})\\
      &\qquad = \O\left( \eps\left(\frac\eps{\delta^2}\right)^{J+1}\delta^{-\kappa_0-k}\right),
      \end{align*} 
 and we  deduce 
\[
i\eps\partial_t \left( \mathcal U^\eps_K(t,t_{\rm in})-
\widehat \W^{\eps, J}(t,t_{\rm in})\, \mathcal U^\eps_{KS}(t,t_{\rm in})\right)= \O\left( \eps\left(\frac\eps{\delta^2}\right)^{J+1}\delta^{-\kappa_0-k}\right),
\]
      which gives the result after integration from $t_{\rm in}$ to $t$. 
       \end{proof}

     \section{Propagation of wave packets}
  
       If $\psi^\eps_{\rm in}$ is a  wave packet and $K(t)$ a matrix-valued Hamiltonian of the form \eqref{eq:K} satisfying \eqref{ass_K_delta}, then the action of the unitary propagator $\mathcal U^\eps_{K} (t,t_{\rm in})$ on $\psi^\eps_{\rm in}$ can be described precisely. Following Section 14.2 of~\cite{corobook}, Theorem~77, we have the following result.

       \begin{theorem}\label{evadia}
       Assume $\delta\ge\sqrt\eps$ and  
       $\psi^\eps_{\rm in}=\wp_{z_0}^\eps(\vec f^\eps)$,
        $$\vec f^\eps=\sum_{0\leq j\leq N}\eps^{j/2}\vec f_j, \;\; \vec f_j\in {\mathcal S}(\R^d,\C^m) .
        $$
        There exists a family $(\vec U_j(t) )_{j\geq 0}$  defined on the interval $I_\delta$ such that 
        \begin{enumerate}
        \item[(i)] For all $j\in\N$ and $t\in I_\delta$,   $\vec U_j(t)\in\mathcal S(\R^d,\C^m)$,
        \item[(ii)] For all $k\in\N$ and $j\ge 2$, there exists a constant $C_{k,j}>0$ such that 
        \[
         \sup_{t\in I_\delta} \, 
         \sup_{|\alpha|+|\beta| \le k}
         \| x^\alpha \partial_x^\beta \vec U_j(t) \|_{L^\infty} \leq C_{k,j}\, \delta^{1-j-k}.
        \]
        \item [(iii)] For all  $k\in\N$ and $N\in\N$, there exists $C_{k,N}$ and $N_{k}$ such that for all  $t\in I_\delta$, we have 
\begin{align*}
&\left\|\mathcal U^\eps_K(t,t_{\rm in})\psi^\eps_{\rm in} -   {\rm e}^{\frac{i}{\eps}S(t,t_{\rm in},z_0)}\wp^\eps_{z_t} \left( \sum_{j=0}^{N}\eps^{\frac j 2}{\mathcal R}(t,t_{\rm in})\,
         {\mathcal M}[F(t,t_{\rm in})]  \vec U_j(t) \right) \right\| _{\Sigma^k_\eps} \\
         &\quad\leq C_{k,N}\  \left( \frac{\sqrt\eps}{\delta}\right)^{N+1} \delta^{-k-\kappa_0}.
\end{align*}
Here, $z_t= \Phi_k^{t,t_{\rm in}}(z_0)$, while $F(t,t_{\rm in}) = F(t,t_{\rm in},z_0)$ is the stability matrix  (see~\eqref{def:F}) and  ${\mathcal R}(t,t_{\rm in})={\mathcal R}(t,t_{\rm in},z_0)$ satisfies the equation~\eqref{ps3} for $z=z_0$.
\end{enumerate}
       Besides 
 \begin{align}\label{def:vecU0}
&\vec U_0 (t) = \vec f_0\;\;\mbox{and}\;\;
    \vec U_1 (t)    = \vec f_1+   {\bf b}_1(t,t_{\rm in})\vec f_0 
  \end{align}
  where
\begin{align}\label{b1'}
  {\bf b}_1(t,t_{\rm in})& = \frac 1i \sum_{\vert\alpha\vert=3}\frac{1}{\alpha!}\int_{t_{\rm in}}^t\partial_z^\alpha k(s,\Phi_k^{s,t_{\rm in}}(z_0)){\rm op}_1^w((F(s, t_{\rm in})z)^\alpha )ds \, \1_{m}\nonumber\\
 &\qquad + \frac 1i\int_{t_{\rm in}}^t {\mathcal R}(t_{\rm in},s) \partial_z K_1(s,\Phi_k^{s,t_{\rm in}}(z_0)) {\rm op}_1^w(F(s, t_{\rm in})z ) {\mathcal R}(s,t_{\rm in})\, ds.
\end{align}

 \end{theorem}

       Of course, the result holds for $K(t)=k(t)\1_{m}$, the scalar case, and for $K(t)=KS(t)=k(t)\1_m +\eps K_1(t)$, which corresponds to the simplest perturbation of the scalar case by adding a subprincipal term. 
       In order to emphasize the characteristics of each of these cases, we detail the proof along these three cases: scalar, scalar with subprincipal perturbation, scalar with general perturbation. The first two cases have been considered in~\cite{corobook}, the strategy of which we follow. We note, that only in the third case the $\delta$-dependence of the estimates arises.

       \subsection{The scalar case}\label{prop:scalar}

       In this section, we review the approximation for the standard scalar situation $K(t)=k(t)\1_{m}$ with $z\mapsto k(t,z)$ subquadratic for $t\in I_\delta$. 
       
       \begin{proof}[Proof of Theorem~\ref{evadia} in the scalar case]
       We 
       look for an approximation of the propagated wave packet $\mathcal U_{K}^\eps(t,t_{\rm in})\psi^\eps_{\rm in}$ by
 \begin{equation}\label{form_phi}
\varphi^\eps_N(t):=   {\rm e}^{\frac{i}{\eps}S(t,t_{\rm in},z_0)}\wp^\eps_{z_t} \left(\vec B^\eps_N(t) \right) 
 \end{equation}
 for some time-dependent Schwartz vector-valued and time dependent functions $z\mapsto \vec B^\eps_N(t,z)$.
 We have the two relations
 \begin{align*}
&  i\eps\partial_t \varphi^\eps_N(t) = {\rm e}^{\frac{i}{\eps} S(t,t_{\rm in},z_0)}\wp^\eps_{z_t} \left(
 -(\partial_t S(t,t_{\rm in},z_0) -p_t\cdot \dot q_t) \vec B_N^\eps(t)+\sqrt\eps\, \dot z_t \cdot (J\widehat z)\, \vec B_N^\eps(t) +i\eps \partial_t \vec B_N^\eps(t)
 \right)\\
& \widehat {k(t)} \varphi^\eps_N(t) =  {\rm e}^{\frac{i}{\eps}S(t,t_{\rm in},z_0) }\wp^\eps_{z_t} \left(
 \widehat {k(t,z_t+\sqrt\eps z)} \vec B_N^\eps(t).
 \right)
 \end{align*}
 We look for $\vec B^\eps_N(t)$ of the form
 \begin{equation}\label{form_Bj}
 \vec B^\eps_N(t)=\sum_{j=0}^N \eps^{\frac j2} \vec B_j(t).
 \end{equation}
 We perform a Taylor expansion of $z\mapsto k(t,z_t+\sqrt\eps z)$ 
 \begin{align}\label{eq:KSTaylor}
 k(t,z_t+\sqrt\eps z)&= k(t,z_t)\,  +\sqrt\eps \,\nabla k(t,z_t)\cdot  z \,  + \eps \frac 12 {\rm Hess}\, k(t,z_t)z\cdot z\,   \\
 \nonumber
 &\qquad +
 \sum _{j=3}^{N+2} \eps^{\frac j2}L_j(t) [z]^j + \eps^{\frac{N+3}{2}}R_N(t,z_t, \sqrt\eps z)[z]^{N+3}
 \end{align}
 for some coefficients $t\mapsto L_j(t)$ uniformly bounded on $I_\delta$  and a smooth function $z\mapsto R_N(t,z_t, z) $ satisfying uniform symbol estimates for $t\in I_\delta$.
 \smallskip

 We now construct the elements $\vec B_j(t)$, $0\leq j\leq N$, by solving recursive equations, that stem from  
 the classical flow and the action. We want to choose  the $\vec B_j(t)$ so that 
 \[
 i\eps\partial_t \varphi^\eps_N(t)= \widehat {k(t)} \varphi^\eps_N(t) + \O(\eps^{\frac{N+3}{2}}) .
 \]
Then, $\U^\eps_K(t,t_{\rm in})\psi^\eps_{\rm in} = \varphi^\eps_N(t) + \O(\eps^{\frac{N+1}{2}})$ in $\Sigma^k_\eps$ with an error constant independent of $\delta$. We now construct the approximation $\varphi^\eps_N(t)$ in several steps.
 \begin {enumerate}
 \item[(i)]
 The equation~\eqref{def:S} for the action writes $\partial_t S(t,t_{\rm in},z_0) =p_t\cdot \dot q_t- k(t,z_t)$ and allows to get rid of the terms of order $\eps^0$. \item[(ii)]
 The equation for the flow $\dot z_t = J \nabla k(t,z_t)$ removes the terms of order $\sqrt\eps $. 
 \item[(iii)] For the term of order $\eps$, we choose $\vec B_0(t)$ so that 
 \[
 i\partial_t \vec B_0(t)= \frac 12 {\rm Hess}\, k(t,z_t)\widehat z\cdot \widehat z\,   \vec B_0(t)
 \]
 with $\vec B_0(t_{\rm in})=\vec  f_0$, which is solved by
  \begin{equation*}
\vec B_0 (t)= 
         {\mathcal M}[F(t,t_{\rm in})] \vec f_0= {\mathcal M}[F(t,t_{\rm in})] \vec U_0,
  \end{equation*}
according to~\eqref{def:vecU0} since $\mathcal R(t,t_{\rm in})=\1_{m}$ in the scalar case. 
 \item [(iv)]
 Finally, we treat the terms in $\eps^{1+\frac {j}2}$ with $j\geq 1$
 by solving the equations  
  \[
 i\partial_t \vec B_j(t)= \frac 12 {\rm Hess}\, k(t,z_t)\widehat z\cdot \widehat z\,   \vec B_j(t) + \vec G_j(t)
 \]
with $\vec B_j(t_{\rm in})= \vec f_j$ and the source term $\vec G_j(t)$ depends on $\vec B_0(t),\ldots, \vec B_{j-1}(t)$. Thus,  by Duhamel formula:
\[
\vec B_j(t)=\mathcal M[F(t,t_{\rm in})]\vec f_j  + \frac 1i  \int_{t_{\rm in}} ^t 
\mathcal M[F(t,s)] \vec G_j(s) ds.
\]
 \end{enumerate}
 \smallskip 
 
In the case $j=1$, we have 
\[
\vec G_1(t)= \widehat {p(z)} \vec B_0(t) ,\;\; p(z)= \sum_{\vert\alpha\vert=3}\frac{1}{\alpha!}\partial_z^\alpha k(t,z_t)  [ z]^\alpha \1_{m}  .
\]
Moreover, by the (exact) Egorov theorem for quadratic Hamiltonians (see~\eqref{prop:metaplectic}), we have 
\begin{align*}
\mathcal M[F(t,s)] \, \widehat {p(z)}  \, \mathcal M[F(s,t_{\rm in})]
&= \mathcal M[F(t,t_{\rm in})]\,  \widehat {p(F(s,t_{\rm in})z)} 
\end{align*} 
This gives the expression of $\vec U_1(t)$ in~\eqref{def:vecU0}. 
\end{proof}

    \subsection{The scalar case with subprincipal perturbation}\label{prop_subprinc}
We now focus on the case $K(t)=KS(t)=k(t)\1_m +\eps K_1(t)$, a subquadratic symbol with scalar principal symbol and matrix-valued perturbation. 
 We claim 
   \begin{equation}\label{KSwp_prop}
 \U_{KS}^\eps(t,t_{\rm in}) \psi^\eps_{\rm in} =    
    {\rm e}^{\frac{i}{\eps}S(t,t_{\rm in},z_0)}\wp^\eps_{z_t} \left( \sum_{j=0}^{N}\eps^{\frac j 2} \vec B_j(t) \right)  +\O(\eps^{\frac{N+1}{2}})
    \end{equation}
    with 
 $\vec B_j(t)=\mathcal R(t,t_{\rm in})\mathcal M[F(t,t_{\rm in})]\vec U_j(t)$ as in~\eqref{def:vecU0} for $j=0,\ldots,N$, and  $\vec B_j(t)$ is determined by a recursive equation in terms of $\vec B_0(t),\ldots , \vec B_{j-1}(t)$. 
 \smallskip 

    \begin{proof}
        [Proof of Theorem~\ref{evadia} with simple subprincipal term.]
         We argue as in the scalar case and look for an approximate value of the form~\eqref{form_phi}.
 We now have the two relations
 \begin{align*}
&  i\eps\partial_t \varphi^\eps_N(t) = {\rm e}^{\frac{i}{\eps} S(t,t_{\rm in},z_0)}\wp^\eps_{z_t} \left(
 -(\partial_t S(t,t_{\rm in},z_0) -p_t\cdot \dot q_t) \vec B_j(t)+\sqrt\eps\, \dot z_t \cdot (J\widehat z)\, \vec B_j(t) +i\eps \partial_t \vec B_j(t)
 \right)\\
& \widehat {KS(t)} \varphi^\eps_N(t) =  {\rm e}^{\frac{i}{\eps}S(t,t_{\rm in},z_0) }\wp^\eps_{z_t} \left(
 \widehat {KS(t,z_t+\sqrt\eps z)} \vec B_j(t)
 \right)
 \end{align*}
and perform a Taylor expansion of $z\mapsto KS(t,z_t+\sqrt\eps z)$ 
 \begin{align*}
 KS(t,z_t+\sqrt\eps z)&= k(t,z_t)\, \1_{m} +\sqrt\eps\,\nabla k(t,z_t)\cdot  z \, \1_{m} + \eps \left(\frac 12 {\rm Hess}\, k(t,z_t)z\cdot z\, \1_{m}  +K_1(t,z_t)\right)\\
 &\qquad +
 \sum _{j=3}^{N+2} \eps^{\frac j2}L_j(t) [z]^j + \eps^{\frac{N+3}{2}} R_N(t,z_t, \sqrt\eps z)[z]^{N+3}.
 \end{align*}
  Here again, the function $z\mapsto R_N(t,z_t, z) $ is smooth and uniformly bounded fo $t\in I_\delta$  with bounded derivatives. 
  Similarly, the $L_j$-s are bounded on $I_\delta$.
 \smallskip

 We then verify that the properties of the classical flows and of the action allow to 
choose $B^\eps_N(t)$ of the form~\eqref{form_Bj}  so that 
 \[
 i\eps\partial_t \varphi^\eps_N(t)= \widehat {KS(t)} \varphi^\eps_N(t) +\O(\eps^{\frac{N+3}{2}}).
 \]
 The terms in $\eps^0$ and $\sqrt\eps$ are dealt with as in the scalar case. 
However, the terms in $\eps^{\frac j 2}$ with $j\geq 2$ yield modifications, due to the presence of the subprincipal term.  We thus have to revisit the steps (iii) and (iv) of the proof in the scalar case.

\begin{enumerate}
    \item[(iii)]
For the term of order $\eps$, the equation of $\vec B_0(t)$ now writes
 \[
 i\partial_t \vec B_0(t)= \left(\frac 12 {\rm Hess}\, k(t,z_t)\widehat z\cdot \widehat z\, \1_{m}  +K_1(t,z_t)\right) \vec B_0(t)
 \]
 with $\vec B_0(t_{\rm in})= \vec f_0$, which is solved by $\vec B_0(t)=\mathcal R(t,t_{\rm in})\mathcal M[F(t,t_{\rm in})]\vec U_0(t)$ given in~\eqref{def:vecU0}.
 \item[(iv)] The terms in $\eps^{1+\frac {j}2}$ with $j\geq 1$
yield the equations  
  \[
 i\partial_t \vec B_j(t)= \left(\frac 12 {\rm Hess}\, k(t,z_t)\widehat z\cdot \widehat z\, \1_{m}  +K_1(t,z_t)\right) \vec B_j(t) + \vec G_j(t)
 \]
with $\vec B_j(t_{\rm in})= \vec f_j$ and  source terms $\vec G_j(t)$ depending as before  on $\vec B_0(t),\cdots \vec B_{j-1}(t)$.  Duhamel formula now gives
\[
\vec B_j(t)=\mathcal R(t,t_{\rm in})\mathcal M[F(t,t_{\rm in})] \left(\vec f_j+ \frac 1i  \int_{t_{\rm in}} ^t 
\mathcal R(t_{\rm in},s)\mathcal M[F(t_{\rm in},s)] \vec G_j(s) ds\right).
\]
 Indeed, $\mathcal R(t,t_{\rm in}) =\mathcal R(t,t_{\rm in},z_0)$ is a matrix, independent of $z$, which commutes with the action of the scalar operator $\mathcal M[F(t,s)] = \mathcal M[F(t,s,z_0)]$ for all $s\in [t_{\rm in}, t]$. 
 \end{enumerate}
 \smallskip 
 
In the case $j=1$, we have 
\[
\vec G_1(t)= \widehat {p(z)} \vec B_0(t) ,\;\; p(z)= \sum_{\vert\alpha\vert=3}\frac{1}{\alpha!}\partial_z^\alpha k(t,z_t)  [ z]^\alpha \1_{m} +  \partial_z K_1(t,z_t)\cdot  z  .
\]
Moreover, by the exact Egorov theorem for quadratic Hamiltonians (see~\eqref{prop:metaplectic}), we have 
\begin{align*}
    \mathcal R(t_{\rm in},s) \mathcal M[F(t_{\rm in},s)] \widehat {p(z)}  \mathcal R(s,t_{\rm in}) 
    \mathcal M[F(s,t_{\rm in})] = \mathcal R(t_{\rm in},s) \widehat {p(F(s,t_{\rm in})z)} \mathcal R(s,t_{\rm in}), 
\end{align*}
where again we have used the commutation of the $z$-independent matrix $\mathcal R(t,t_{\rm in}) $ with the scalar operator $\mathcal M[F(t,t_{\rm in})]$. 
This gives the expression of $\vec B_1(t)$ in~\eqref{def:vecU0} and concludes the proof of the claim~\eqref{KSwp_prop}.
 \smallskip 
    \end{proof}

    \subsection{The scalar case with general perturbation}
    We can now prove Theorem~\ref{evadia} in full generality and we have to take into account the effects of the terms $K_j(t)$ with $j\geq 2$, in view of assumption~\eqref{ass_K_delta}.
    
\begin{proof}[Proof of Theorem~\ref{evadia}]
 We first reduce to the evolution of wave packets by the subprincipal propagator $\U_{KS}^\eps(t,t_{\rm in})$. Indeed, if $k,J\in \N$,

using Corollary~\ref{coro:split},
\begin{align*}
   \U_K^\eps(t,t_{\rm in}) \psi^\eps_{\rm in} & =  \, \widehat \W^{\eps,J}(t,t_{\rm in}) \,\U_{KS}^\eps(t,t_{\rm in}) \psi^\eps_{\rm in}
   + \O((\eps/\delta^2)^{J+1}\delta^{-\kappa_0-k}
   \end{align*}
      in $\Sigma^k_\eps$.
Combining this with the result of the preceding subsection, that is~\eqref{KSwp_prop}, 
we have to balance an approximation of the order $\eps^{J}$ with one of the order $\eps^{(N+1)/2}$. With a suitable choice of $J$, we  have
\[
\U_{K}^\eps(t,t_{\rm in}) \psi^\eps_{\rm in} =    
    {\rm e}^{\frac{i}{\eps}S(t,t_{\rm in},z_0)}
    \sum_{j=0}^J \sum_{n=0}^N \eps^{j+\frac{n}{2}} \,  \widehat \W_j(t,t_{\rm in}) \, 
    \wp^\eps_{z_t}(\vec B_n(t))  +
  \O\left(\left(\frac{\sqrt\eps}\delta\right)^{N+1} \delta^{-k-\kappa_0}\right).
\]
We now note that $\W_0(t,t_{\rm in})=\1_m$ and, for $j\ge1$, $n\ge 0$,
 \[
  \widehat \W_j(t,t_{\rm in}) \wp^\eps_{z_t} ( \vec B_n(t) )=
  \wp^\eps_{z_t} \left( {\rm op}_1\!\left(  \W_j(t,t_{\rm in},z_t+\sqrt\eps z) \right)\vec B_n(t) \right).
  \]
Sorting the double sum into powers of $\sqrt\eps$, we thus arrive at
\[
\U_{K}^\eps(t,t_{\rm in}) \psi^\eps_{\rm in} = {\rm e}^{\frac{i}{\eps}S(t,t_{\rm in},z_0)}\,
\wp^\eps_{z_t}\! \left( \vec B_0(t) +\sqrt\eps \vec B_1(t) + \sum_{n=2}^N \eps^{\frac{n}{2}} \vec C_n(t) \right) 
+   
\O\left(\left(\frac{\sqrt\eps}\delta\right)^{N+1} \delta^{-k-\kappa_0}\right)
\]
where each $\vec C_n(t)$, $n\ge 2$, is a correction of $\vec B_\ell(t)$ by terms involving $\W_j(t,t_{\rm in})$ with 
$2j\le n$. Therefore the estimate \eqref{eq:Wj} implies our claimed final result.
\end{proof}

   \section{Combined actions of  propagators on wave packets}\label{forwards_backwards}

We work here in the interval $I_\delta=(t^\flat-\delta, t^\flat +\delta)$ where $t^\flat$ plays the role of an initial time.
We focus on the action of $\U^\eps_{H_2}(t^\flat, t) \,\U^\eps_{H_1}(t,t^\flat)$ on a wave packet, where $H_1(t,z)$ and $H_2(t,z)$ are two symbols satisfying~\eqref{eq:K} and \eqref{ass_K_delta} on $I_\delta\times\R^{2d}$. We denote their scalar principal symbols by $h_1(t,z)$ and $h_2(t,z)$, respectively. 
\smallskip

Applying 
Theorem~\ref{evadia} to the operators $\widehat H_1(t), \widehat H_2(t)$, we obtain for all  $s\in[-\delta,\delta]$, in $\Sigma^ k_\eps$,
\begin{align}\label{proof_comp}
& \U_{H_2}^\eps(t^\flat, t^\flat+ s)\U_{H_1}^\eps(t^\flat+ s,t^\flat)
\wp_{\zeta^\flat}^\eps (\vec \varphi )\\\nonumber
&= 
{\rm e}^{\frac{i}{\eps}S( s)}
\sum_{j=1} ^ J \eps^ {\frac j2}
\wp_{\zeta( s)}^\eps \bigl({\mathcal M}( s)\vec \varphi_j(s)\bigr) 
+ \O\left(\left(\frac{\sqrt\eps}\delta\right)^{J+1} \delta^{-k-\kappa_0}\right)
\end{align}
where we  denoted the combined center, phase and metaplectic transform by
\begin{align}
\label{def:zeta_combined}
\zeta( s) &=\Phi_{h_2}^{t^\flat,t^\flat+  s}\big(\Phi_{h_1}^{t^\flat+ s,t^\flat}(\zeta^\flat)\big),\\
\label{def:S_combined}
S( s) &= 
S_1(t^\flat+ s,t^\flat, \zeta^\flat) +S_2(t^\flat,t^\flat+ s,\Phi_{h_1}^{t^\flat+ s, t^\flat}(\zeta^\flat)),\\
\label{def:M_combined}
{\mathcal M}( s) &= {\mathcal M}[F_2(t^\flat,t^\flat+ s,\Phi_{h_1}^{t^\flat+ s, t^\flat}(\zeta^\flat))] 
\,{\mathcal M}[F_1(t^\flat+ s,t^\flat,\zeta^\flat)].
\end{align}
Here, we have used the commutation of the scalar operators given by the metaplectic transforms with the fixed matrices of parallel transport, which are included in the $\vec \varphi_j$.  
Moreover, for all $s\in [\delta, \delta]$ and for $j\in\N$, $\vec \varphi_j(s)\in\mathcal S(\R^ d,\C^m)$ and 
\begin{equation}\label{def:vecphi0}
\vec \varphi_0(s)=\mathcal R_2 (t^\flat,t^\flat+ s,\Phi_{h_1}^{t^\flat+ s, t^\flat}(\zeta^\flat))
\mathcal R_1(t^ \flat +s, t^\flat, \zeta^\flat)\vec \varphi.
\end{equation}
Pushing forward the analysis as in~\cite{FLR1}[Section~5.2]; we obtain the following statement, of which we will give all the elements of proofs.

 \begin{lemma}\label{lem:composed}
    Assume $\delta\ge\sqrt\eps$. Let $k\in\N$, $\vec \varphi\in\mathcal S(\R^d,\C^m)$ and $(t^\flat, \zeta^\flat )\in I_\delta\times \R^{2d}$. Then, there exists a family $(\vec \varphi_j(s))_{j\in\N}$ of vector-valued Schwartz functions parametrized by $s\in [-\delta, \delta]$ and such that for all $s\in [-\delta,  +\delta]$ and $J\in\N$, we have in $\Sigma ^ k_\eps$,
\begin{align*}
\U_{H_2}^\eps(t^\flat, t^\flat+ s)\U_{H_1}^\eps(t^\flat+ s,t^\flat)
\wp_{\zeta^\flat}^\eps (\vec \varphi)
&=  \sum_{j=0}^ J \, \eps^{\frac j2} \, \wp^\eps_{\zeta^\flat} \left(
{\rm e}^{\frac{i}{\eps} \widetilde\Lambda( s)} 
{\rm e}^{i \widetilde p_\eps( s)\cdot (y-\widetilde q_\eps( s))} \vec \varphi_j(s,y-\widetilde q_\eps( s)) \right)\\
&\qquad  +
\O\left(\left(\frac{\sqrt\eps}\delta\right)^{N+1} \delta^{-k-\kappa_0}\right),
\end{align*}
with, using notation~\eqref{def:mu} and ~\eqref{def:alpha_beta}, and setting $\widetilde z_\eps(s)=(\tilde q_\eps(s),\tilde p_\eps(s))$,
\[
\widetilde z_\eps(0)  = 0,\quad \dot{\widetilde z}_\eps (0) = (\alpha^\flat,\beta^\flat),\;\;
\widetilde\Lambda(0)  =\dot{\widetilde\Lambda}(0)= 0\;\;\mbox{and}\;\; 
\frac 12\left(\ddot{\widetilde \Lambda} (0)- \dot p(0)\cdot \dot q(0)\right)= -\mu^\flat.
\]
 \end{lemma}

The proof of Lemma~\ref{lem:composed} relies on a careful analysis of the asymptotics of the functions $\zeta(s)$ and $S(s)$, and of the operator $\mathcal M(s)$. We will use the next Lemma, the proof of which is postponed to the end of the section.

 \begin{lemma}\label{lem:phasis}
 Let $\zeta$ and $S$ be defined as above, and write $\zeta( s) = \zeta^\flat + z( s)$ with $z(s)=(q(s),p(s))$. Consider  $\Lambda$ given for $s\in [-\delta,\delta]$ by
 \begin{align}\nonumber
\Lambda( s) &= S_1(t^\flat+ s,t^\flat, \zeta^\flat) +S_2(t^\flat,t^\flat+ s,\Phi_{h_1}^{t^\flat+ s, t^\flat}(\zeta^\flat)) - q( s)\cdot p^\flat.
\end{align}
 We have 
 \begin{align}
 \label{zeta(0)}
 \zeta(0)&
 =\zeta^\flat,\;\;\dot\zeta(0)= 
  (\alpha^\flat,\beta^\flat),\;\; S(0)=0,\;\; \dot S(0) = p^\flat \cdot\alpha^\flat,\\
 \label{lambda(0)}
 \Lambda(0)&=\dot\Lambda(0)= 0,\;\; 
\frac 12\left(\ddot\Lambda (0)- \dot p(0)\cdot \dot q(0)\right)= -\mu^\flat.
\end{align}
 \end{lemma}

 \begin{proof}[Proof of Lemma~\ref{lem:composed}]
We start with~\eqref{proof_comp}. 
 We then have by Lemma~\ref{lem:phasis}  $z(0) = 0$ and $\dot z(0) = (\alpha^\flat,\beta^\flat)$.
We use the properties of the wave packet transform 
\begin{align*}
&\wp^\eps_z={\rm e}^{-\frac i{2\eps}p\cdot q}\, \widehat {T}^\eps(z)\,  \Lambda^\eps, \\
\mbox{with}\;\;&\widehat {T}^\eps(z)=\e^{\frac i\eps (p\cdot \hat x-q\cdot \hat \xi)}\;\;\mbox{and} \;\;\Lambda^\eps\varphi (x)=\eps^{-\frac d4}\varphi(x/\sqrt\eps),\;\;x\in\R^d.
\end{align*}
We obtain
\begin{align*}
\wp^\eps_{\zeta( s)} &= {\rm e}^{-\frac{i}{\eps} p^\flat\cdot q( s)}
 \wp^\eps_{\zeta^\flat}\Lambda_\eps^{-1} \wp^\eps_{z( s)}\\*[1ex]
 &=  {\rm e}^{-\frac{i}{\eps} p^\flat\cdot q( s)}
 {\rm e}^{-\frac{i}{2\eps}p( s)\cdot q( s)} \wp^\eps_{\zeta^\flat} 
\Lambda_\eps^{-1} \widehat T^\eps(z( s)) \Lambda_\eps\\*[1ex]
 &=  {\rm e}^{-\frac{i}{\eps} p^\flat\cdot q( s)}
 {\rm e}^{-\frac{i}{2\eps}p( s)\cdot q( s)} \wp^\eps_{\zeta^\flat} 
\widehat T^1(z_\eps( s)), 
\end{align*}
with $z_\eps(s)= z(s)/\sqrt\eps$. By the translation properties 
of the metaplectic transform \cite[Section~3.3]{corobook}, we have
\[
\widehat T^1(z_\eps( s)) {\mathcal M}( s)=
 {\mathcal M}( s)\widehat T^1(\widetilde z_\eps( s)) 
\]
with the new center
\[
\widetilde z_\eps( s) =\widetilde z(s)/ \sqrt\eps,\;\; \widetilde z(s)=F_1(t^\flat+ s,t^\flat,\zeta^\flat)^{-1} 
F_2(t^\flat,t^\flat+ s,\Phi_{h_1}^{t^\flat+ s, t^\flat}(\zeta^\flat))^{-1} z( s)
\]
We observe that 
\[
\widetilde z(0) = z(0) = 0,\quad \dot{\widetilde z}(0) = \dot z(0) = (\alpha^\flat,\beta^\flat).
\]  
Since
\[
\widehat T^1(\widetilde z_\eps( s))\varphi(y) = 
{\rm e}^{\frac{i}{2}\widetilde q_\eps( s)\cdot\widetilde p_\eps( s)}
{\rm e}^{i \widetilde p_\eps( s)\cdot (y-\widetilde q_\eps( s))} \varphi(y-\widetilde q_\eps( s)),
\]
the phase $\widetilde\Lambda( s)$ 
\begin{align}\nonumber
\widetilde\Lambda( s) &= S( s) - p^\flat\cdot q( s) - \frac 12 \,p( s)\cdot q( s) + \frac 12\,
\widetilde p( s)\cdot \widetilde q( s),
\end{align}
yields the approximation of the Lemma. Indeed, 
we have $\widetilde\Lambda(0) = \dot{\widetilde\Lambda}(0) = 0$, 
$\ddot{\widetilde\Lambda}(0)=\ddot{\Lambda}(0)$ and we conclude 
by Lemma~\ref{lem:phasis}.
\end{proof}

It remains to prove the asymptotics of Lemma~\ref{lem:phasis}.

 \begin{proof}[Proof of Lemma~\ref{lem:phasis}]
 We begin with the function $\zeta$ and 
 we compute the Taylor expansion at the order~2 for $(q( s),p( s))=\zeta( s)-\zeta^\flat $ at $ s=0$.
Let  be $h=h_1, h_2$.  We have :
 \begin{align}\label {DTPhi}
 \Phi_h^{t,t_0}(z) = z &+(t-t_0)J\partial_z h(t_0,z) + \frac{(t-t_0)^2}{2}\left(J\partial^2_{t,z}h(t_0,z)+ J\partial^2_{z,z}h(t_0,z)J\partial_zh(t_0,z)\right)
 \\
 \nonumber
 & + \O(|t-t_0|^3).
\end{align} 
Applying this formula, we obtain (omitting the argument $(t^\flat, \zeta^\flat)$ in the functions $h_1$, $h_2$ and their derivatives)
\begin{align*} 
&\Phi_{h_1}^{t^\flat+ s,t^\flat}(\zeta^\flat)= \zeta^\flat+ s J\partial_z h_1 +\frac{ s^2}2 \left(J\partial^2_{t,z}h_1+ J\partial^2_{z,z}h_1J\partial_zh_1\right)
  + \O(| s|^3),\\
&  \zeta(t) = \Phi_{h_1}^{t^\flat+ s,t^\flat}(\zeta^\flat) - s J\partial_z h_2(t^\flat+ s,\Phi_{h_1}^{t^\flat+ s,t^\flat}(\zeta^\flat)) 
  +\frac{ s^2}2 \left(J\partial^2_{t,z}h_2+ J\partial^2_{z,z}h_2J\partial_zh_2\right)
  + \O(| s|^3).
  \end{align*}
  We deduce 
  \begin{align*} 
  \zeta(t) &= \zeta^\flat + s J\partial_z (h_1-h_2) + \O(| s|^3)\\
  &+\frac{ s^2}2 \left(J\partial^2_{t,z}(h_1-h_2)+ J\partial^2_{z,z}(h_1-h_2)J\partial_zh_1
  +J\partial^2_{z,z} h_2 J\partial_z (h_2-h_1)\right),
  \end{align*} 
and, for further use, the relation 
\begin{align}
\label{C5}
-p^\flat \dot q(0) =& -p^\flat \cdot\partial_q(h_1-h_2),\\
\label{C5''}
-p^\flat \cdot \ddot q(0)=&-p^\flat \cdot (\partial^2_{t,p}(h_1-h_2)+ \partial^2_{z,p}(h_1-h_2)J\partial_zh_1
  +\partial^2_{z,p} h_2 J\partial_z (h_2-h_1))
\end{align}

\medskip

 We continue with the function $\Lambda$  and we use Taylor expansion of the actions for general Hamiltonian $h$. In view of~\eqref{def:S} and~\eqref{DTPhi}, we have (omitting the argument $(t_0,z_0)$ in the terms of the form $\partial^\alpha h(t_0,z_0)$)
\begin{align*}
S(t,t_0,z_0) &=  \int_{t_0}^t (p_0-(s-t_0)\partial_qh)\cdot (\partial_p h+(s-t_0) (\partial^2_{t,p}h +\partial^2_{z,p}hJ\partial_z h) )ds\\
&\qquad
-\int_{t_0}^{t} (h+(s-t_0)\partial_t h ) ds  +\O((t-t_0)^3)\\
&= (p_0\cdot \partial_p h-h)(t-t_0) -\frac{(t-t_0)^2}{2} (\partial_t h+\partial_qh \cdot \partial_p h -p_0\cdot (\partial_{t,p}^2 h+ \partial^2_{z,p}hJ\partial_z h))\\
&\qquad
+\O((t-t_0) ^3).
\end{align*} 
We first apply the formula with $h=h_1$, $t=t^\flat+ s$, $t=t^\flat$ and $z=\zeta^\flat$, which gives (when the arguments of the functions are omitted, they are fixed to $(t^\flat,\zeta^\flat)$)
\begin{align*}
S_1(t^\flat+ s,t^\flat,\zeta^\flat) = &\,  s(p\cdot \partial_p h_1-h_1)\\
&-\frac{ s^2}{2} (\partial_t h_1+\partial_qh_1 \cdot \partial_p h_1 -p\cdot (\partial_{t,p}^2 h_1+ \partial^2_{z,p}h_1J\partial_z h_1))
+\O( s ^3).
\end{align*}
We now use the same formula with $h=h_2$, $t=t^\flat$, $t_0=t^\flat+ s$, $z_0=\Phi_{h_1}^{t^\flat+ s, t^\flat}(\zeta^\flat)$. We obtain
\begin{align*}
S_2(t^\flat,t^\flat+ s,\Phi_{h_1}^{t^\flat+ s,t^\flat}(\zeta^\flat)) &=\\
-  s &(p_1(t^\flat+ s,t^\flat,\zeta^\flat)\cdot \partial_p h_2(t^\flat+ s,\Phi_{h_1}^{t^\flat+ s,t^\flat}(\zeta^\flat)) -h_2(t^\flat+ s,\Phi_{h_1}^{t^\flat+ s,t^\flat}(\zeta^\flat)) )\\
& -\frac{ s^2}{2} (\partial_t h_2+\partial_qh_2 \cdot \partial_p h_2 -p\cdot (\partial_{t,p}^2 h_2+ \partial^2_{z,p}h_2J\partial_z h_2)) +\O( s^3)
\end{align*} 
The treatment of the term of order $ s$ has to be performed carefully for $S_2(t^\flat,t^\flat+ s,\Phi_{h_1}^{t^\flat+ s,t^\flat}(\zeta^\flat))$. 
We obtain 
\begin{align*}
&S_2(t^\flat,t^\flat+ s,\Phi_{h_1}^{t^\flat+ s,t^\flat}(\zeta^\flat)) =  - s (p\cdot \partial_p h_2 -h_2)
- s^2(-\partial_t h_2-\partial_q h_2\cdot \partial_p  h_1+p\cdot (\partial^2_{t,p} h_2 +\partial^2_{z,p} h_2 J\partial_z h_2))\\
&\qquad\qquad -\frac{ s^2}{2} (\partial_t h_2+\partial_qh_2 \cdot \partial_p h_2 -p\cdot (\partial_{t,p}^2 h_2+ \partial^2_{z,p}h_2J\partial_z h_1))
+\O( s^3)\\
&\;\;= \,  (p\cdot \partial_p h_2 -h_2) s 
+\frac{ s^2}{2} (\partial_t h_2+\partial_qh_2 \cdot \partial_p (2h_1-h_2) -p\cdot (\partial_{t,p}^2 h_2+ \partial^2_{z,p}h_2J\partial_z (2h_2-h_1)
+\O( s^3)
\end{align*}
As a consequence,
\begin{align*}
S_1(t^\flat+ s,t^\flat,\zeta^\flat) + & S_2(t^\flat,t^\flat+ s,\Phi_{h_1}^{t^\flat+ s,t^\flat}(\zeta^\flat))= s \, p\cdot \partial_p (h_1-h_2) 
 + \frac{ s^2} 2(\partial_t (h_2-h_1)\\
 &-\partial_q h_2\cdot \partial_p(h_2-h_1)+\partial_p h_1 \cdot \partial_q (h_2-h_1) \\
 &+p\cdot (\partial_{t,p}^2(h_1-h_2)+\partial_{z,p}^2 (h_1-h_2)J\partial_z h_1+\partial_{z,p}^2 h_2 J\partial_z(h_1-h_2)))+\O( s^3).
\end{align*}
Combining with~\eqref{C5''}, we obtain 
\begin{align*}
\Lambda( s) &=
\frac{ s^2} 2(\partial_t (h_2-h_1)-\partial_q h_2\cdot \partial_p(h_2-h_1)+\partial_p h_1 \cdot \partial_q (h_2-h_1)) +\O( s^3)\\
&=\frac{ s^2} 2(\partial_t (h_2-h_1)-\partial_q h_2\cdot \partial_p(h_2-h_1)+\partial_p h_1 \cdot \partial_q (h_2-h_1))+\O( s^3).
\end{align*}
Finally, we observe 
\begin{align*}
&\frac12(\ddot\Lambda(0) -{\dot p(0)\cdot \dot q(0)})\\
&= \frac{1}{2}\left( \partial_t(h_2-h_1) - \partial_q h_2\cdot \partial_p (h_2-h_1)+\partial_p h_1\cdot \partial_q(h_2-h_1) 
+\partial_p(h_2-h_1) \cdot\partial_q (h_2-h_1)\right)\\
&=  \frac{1}{2}\left( \partial_t(h_2-h_1) - \partial_q h_1\cdot \partial_p (h_2-h_1)+\partial_p h_1\cdot \partial_q(h_2-h_1)\right) \\
&=  \frac{1}{2}\left( \partial_t(h_2-h_1) +\left\{\frac{h_1+h_2}{2}, h_2-h_1\right\}\right)\\
&=- (\partial_t f +\{v,f\}).
\end{align*}
\end{proof}

\backmatter
\bibliographystyle{amsalpha}



\printindex

\end{document}